\documentclass[12pt,twoside]{preprint}
\usepackage{hyperref}
\usepackage{breakurl}

\usepackage{amssymb}
\usepackage{amsmath}
\usepackage{mhenvs}
\usepackage{mhequ}
\usepackage{mhsymb}
\usepackage{booktabs}
\usepackage{tikz}
\usepackage{mathrsfs}
\usepackage[full]{textcomp}
\usepackage[osf]{newtxtext}

\usepackage{microtype}
\usepackage{verbatim}
\usepackage{comment}

\usepackage{wasysym}
\usepackage{centernot}

\usetikzlibrary{calc}
\usetikzlibrary{external}

\let\d\partial

\makeatletter
\pgfdeclareshape{crosscircle}
{
  \inheritsavedanchors[from=circle] 
  \inheritanchorborder[from=circle]
  \inheritanchor[from=circle]{north}
  \inheritanchor[from=circle]{north west}
  \inheritanchor[from=circle]{north east}
  \inheritanchor[from=circle]{center}
  \inheritanchor[from=circle]{west}
  \inheritanchor[from=circle]{east}
  \inheritanchor[from=circle]{mid}
  \inheritanchor[from=circle]{mid west}
  \inheritanchor[from=circle]{mid east}
  \inheritanchor[from=circle]{base}
  \inheritanchor[from=circle]{base west}
  \inheritanchor[from=circle]{base east}
  \inheritanchor[from=circle]{south}
  \inheritanchor[from=circle]{south west}
  \inheritanchor[from=circle]{south east}
  \inheritbackgroundpath[from=circle]
  \foregroundpath{
    \centerpoint%
    \pgf@xc=\pgf@x%
    \pgf@yc=\pgf@y%
    \pgfutil@tempdima=\radius%
    \pgfmathsetlength{\pgf@xb}{\pgfkeysvalueof{/pgf/outer xsep}}%
    \pgfmathsetlength{\pgf@yb}{\pgfkeysvalueof{/pgf/outer ysep}}%
    \ifdim\pgf@xb<\pgf@yb%
      \advance\pgfutil@tempdima by-\pgf@yb%
    \else%
      \advance\pgfutil@tempdima by-\pgf@xb%
    \fi%
    \pgfpathmoveto{\pgfpointadd{\pgfqpoint{\pgf@xc}{\pgf@yc}}{\pgfqpoint{-0.707107\pgfutil@tempdima}{-0.707107\pgfutil@tempdima}}}
    \pgfpathlineto{\pgfpointadd{\pgfqpoint{\pgf@xc}{\pgf@yc}}{\pgfqpoint{0.707107\pgfutil@tempdima}{0.707107\pgfutil@tempdima}}}
    \pgfpathmoveto{\pgfpointadd{\pgfqpoint{\pgf@xc}{\pgf@yc}}{\pgfqpoint{-0.707107\pgfutil@tempdima}{0.707107\pgfutil@tempdima}}}
    \pgfpathlineto{\pgfpointadd{\pgfqpoint{\pgf@xc}{\pgf@yc}}{\pgfqpoint{0.707107\pgfutil@tempdima}{-0.707107\pgfutil@tempdima}}}
  }
}
\makeatother

\colorlet{symbols}{black}      
\colorlet{testcolor}{green!60!black}

\def\symbol#1{\textcolor{symbols}{#1}}
\def\Wick#1{\mathopen{{:}}#1\mathclose{{:}}}
\def\1{\mathbf{\symbol{1}}}

\usetikzlibrary{shapes.misc}
\usetikzlibrary{shapes.symbols}
\usetikzlibrary{shapes.geometric}
\usetikzlibrary{snakes}
\usetikzlibrary{decorations}
\usetikzlibrary{decorations.markings}

\tikzstyle{tinydots}=[dash pattern=on \pgflinewidth off 2*\pgflinewidth]

\def\drawx{\draw[-,solid] (-3pt,-3pt) -- (3pt,3pt);\draw[-,solid] (-3pt,3pt) -- (3pt,-3pt);}
\tikzset{
	n/.style={circle,fill=black!15,inner sep=0pt, minimum size=0.7mm},
	l/.style={inner sep=2pt,label=center:{...}},
	r/.style={circle,fill=red,inner sep=0pt, minimum size=1mm},
	simple/.style={circle,inner sep=0pt, minimum size=0.5mm},
	d/.style={circle,fill=black,inner sep=0pt, minimum size=0.7mm},
	c/.style={circle,draw=black,fill=white,inner sep=0pt, minimum size=1.2mm},
	root/.style={circle,fill=testcolor,inner sep=0pt, minimum size=2mm},
	dot/.style={circle,fill=black,inner sep=0pt, minimum size=1mm},
	var/.style={circle,fill=black!10,draw=black,inner sep=0pt, minimum size=2mm},
	dotred/.style={circle,fill=black!50,inner sep=0pt, minimum size=2mm},
	generic/.style={semithick,shorten >=1pt,shorten <=1pt},
	dist/.style={ultra thick,draw=testcolor,shorten >=1pt,shorten <=1pt},
	testfcn/.style={ultra thick,testcolor,shorten >=1pt,shorten <=1pt,<-},
	testfcnx/.style={ultra thick,testcolor,shorten >=1pt,shorten <=1pt,<-,
		postaction={decorate,decoration={markings,mark=at position 0.6 with {\drawx}}}},
	kprime/.style={semithick,shorten >=1pt,shorten <=1pt,densely dashed,->},
	kprimex/.style={semithick,shorten >=1pt,shorten <=1pt,densely dashed,->,
		postaction={decorate,decoration={markings,mark=at position 0.4 with {\drawx}}}},
	kernel/.style={semithick,shorten >=1pt,shorten <=1pt,->},
	multx/.style={shorten >=1pt,shorten <=1pt,
		postaction={decorate,decoration={markings,mark=at position 0.5 with {\drawx}}}},
	kernelx/.style={semithick,shorten >=1pt,shorten <=1pt,->,
		postaction={decorate,decoration={markings,mark=at position 0.4 with {\drawx}}}},
	kernel1/.style={->,semithick,shorten >=1pt,shorten <=1pt,postaction={decorate,decoration={markings,mark=at position 0.45 with {\draw[-] (0,-0.1) -- (0,0.1);}}}},
	kernel2/.style={->,semithick,shorten >=1pt,shorten <=1pt,postaction={decorate,decoration={markings,mark=at position 0.45 with {\draw[-] (0.05,-0.1) -- (0.05,0.1);\draw[-] (-0.05,-0.1) -- (-0.05,0.1);}}}},
	kernelBig/.style={semithick,shorten >=1pt,shorten <=1pt,decorate, decoration={zigzag,amplitude=1.5pt,segment length = 3pt,pre length=2pt,post length=2pt}},
	rho/.style={dotted,semithick,shorten >=1pt,shorten <=1pt},
	renorm/.style={shape=circle,fill=white,inner sep=1pt},
	labl/.style={shape=rectangle,fill=white,inner sep=1pt},
cumu2n/.style={inner sep=3pt},
cumu2/.style={draw=red!50,fill=red!20},
cumu3/.style={regular polygon, regular polygon sides=3,draw=red!50,rounded corners=3pt,fill=red!20,minimum size=5mm},
cumu4/.style={regular polygon, regular polygon sides=4,draw=red!50,rounded corners=3pt,fill=red!20,minimum size=7mm},
cumu5/.style={regular polygon, regular polygon sides=5,draw=red!50,rounded corners=3pt,fill=red!20,minimum size=5mm},
	xi/.style={circle,fill=symbols!10,draw=symbols,inner sep=0pt,minimum size=1.2mm},
	xix/.style={crosscircle,fill=symbols!10,draw=symbols,inner sep=0pt,minimum size=1.2mm},
	xib/.style={circle,fill=symbols!10,draw=symbols,inner sep=0pt,minimum size=1.6mm},
	xibx/.style={crosscircle,fill=symbols!10,draw=symbols,inner sep=0pt,minimum size=1.6mm},
	not/.style={circle,fill=symbols,draw=symbols,inner sep=0pt,minimum size=0.5mm},
	>=stealth,
	}
\makeatletter
\def\DeclareSymbol#1#2#3{\expandafter\gdef\csname MH@symb@#1\endcsname{\tikzsetnextfilename{symbol#1}\tikz[baseline=#2,scale=0.15,draw=symbols]{#3}}\expandafter\gdef\csname MH@symb@#1s\endcsname{\scalebox{0.7}{\tikzsetnextfilename{symbol#1s}\tikz[baseline=#2,scale=0.15,draw=symbols]{#3}}}}
\def\<#1>{\csname MH@symb@#1\endcsname}
\makeatother

\DeclareSymbol{Xi22}{0.5}{\draw (0,0) node[xi] {} -- (-1,1) node[not] {} -- (0,2) node[xi] {};}

\DeclareSymbol{Xi2}{-2}{\draw (0,-0.25) node[xi] {} -- (-1,1) node[xi] {};}
\DeclareSymbol{Xi3}{0}{\draw (0,0) node[xi] {} -- (-1,1) node[xi] {} -- (0,2) node[xi] {};}
\DeclareSymbol{Xi4}{2}{\draw (0,0) node[xi] {} -- (-1,1) node[xi] {} -- (0,2) node[xi] {} -- (-1,3) node[xi] {};}
\DeclareSymbol{Xi2X}{-2}{\draw (0,-0.25) node[xi] {} -- (-1,1) node[xix] {};}
\DeclareSymbol{XXi2}{-2}{\draw (0,-0.25) node[xix] {} -- (-1,1) node[xi] {};}

\DeclareSymbol{IXi2}{0}{\draw (0,-0.25) node[not] {} -- (-1,1) node[xi] {} -- (0,2) node[xi] {};}
\DeclareSymbol{IXi^2}{-1}{\draw (-1,1) node[xi] {} -- (0,0) node[not] {} -- (1,1) node[xi] {};}

\DeclareSymbol{XiX}{-2.8}{\node[xibx] {};}
\DeclareSymbol{Xi}{-2.8}{\node[xib] {};}
\DeclareSymbol{IXiX}{-1}{\draw (0,-0.25) node[not] {} -- (-1,1) node[xix] {};}

\DeclareSymbol{Xi3b}{-1}{\draw (-1,1) node[xi] {} -- (0,0) node[xi] {} -- (1,1) node[xi] {};}

\DeclareSymbol{IXi3}{2}{\draw (0,-0.25) node[not] {} -- (-1,1) node[xi] {} -- (0,2) node[xi] {} -- (-1,3) node[xi] {};}
\DeclareSymbol{IXi}{-2}{\draw (0,-0.25) node[not] {} -- (-1,1) node[xi] {};}
\DeclareSymbol{XiI}{-2}{\draw (0,-0.25) node[xi] {} -- (-1,1) node[not] {};}

\DeclareSymbol{Xi4b}{-1}{\draw(0,1.5) node[xi] {} -- (0,0); \draw (-1,1) node[xi] {} -- (0,0) node[xi] {} -- (1,1) node[xi] {};}
\DeclareSymbol{Xi4b'}{-1}{\draw(0,1.5) node[xi] {} -- (0,-0.2); \draw (-1,1) node[xi] {} -- (0,-0.2) node[not] {} -- (1,1) node[xi] {};}
\DeclareSymbol{Xi4c}{0}{\draw (0,1) -- (0.8,2.2) node[xi] {};\draw (0,-0.25) node[xi] {} -- (0,1) node[xi] {} -- (-0.8,2.2) node[xi] {};}
\DeclareSymbol{Xi4d}{-4.5}{\draw (0,-1.5) node[not] {} -- (0,0); \draw (-1,1) node[xi] {} -- (0,0) node[xi] {} -- (1,1) node[xi] {};}
\DeclareSymbol{Xi4e}{0}{\draw (0,2) node[xi] {} -- (-1,1) node[xi] {} -- (0,0) node[xi] {} -- (1,1) node[xi] {};}
\DeclareSymbol{Xi4e'}{0}{\draw (0,2) node[xi] {} -- (-1,1) node[xi] {} -- (0,-0.2) node[not] {} -- (1,1) node[xi] {};}

\newtheorem{assumption}[lemma]{Assumption}

\newcommand{\Mod}[1]{\ (\text{mod}\ #1)}

\let\I\CI

\def\s{\mathfrak{s}}
\def\c{\mathfrak{c}}

\newcommand{\V}  {\overline{V}}
\def\K{\mathfrak{K}}

\def\Ren{\mathscr{R}}

\def\${|\!|\!|}
\def\Z{\mathbb{Z}}
\def\T{\mathbb{T}}
\def\R{\mathbb{R}}
\def\N{\mathbb{N}}
\def\cA{\mathcal{A}}

\def\cC{\mathcal{C}}
\def\cD{\mathcal{D}}
\def\CG{\mathcal{G}}
\def\cH{\mathcal{H}}
\def\cK{\mathcal{K}}

\def\cT{\mathcal{T}}
\def\cI{\mathcal{I}}
\def\cU{\mathcal{U}}

\def\cP{\mathcal{P}}
\def\cS{\mathcal{S}}
\def\cF{\mathcal{F}}
\def\fx{\mathbf{x}}
\def\fy{\mathbf{y}}
\def\fz{\mathbf{z}}

\def\cN{\mathcal{N}}
\def\CCE{\mathbb{E}}
\def\CCV{\mathbb{V}}
\def\CCG{\mathbb{G}}
\def\fn{\mathbf{n}}
\def\fa{\mathbf{a}}
\def\fb{\mathbf{b}}

\def\fu{\mathbf{u}}
\def\emptyset{{\centernot\ocircle}}

\def\fT{\mathbf{T}}
\def\fs{\mathbf{s}}
\def\ft{\mathbf{t}}
\def\Cumus{\mathfrak{C}}
\def\Tree{\mathfrak{T}}
\def\Trees{\mathfrak{F}}
\def\pP{\mathfrak{P}}
\def\powerset{\mathscr{P}}

\def\tT{\mathtt{T}}

\def\Part{P}

\def\mod{\mathrm{mod}}
\def\gain{\mathrm{gain}}

\def\I{\mathrm{I}}
\def\II{\mathrm{I\kern-0.1emI}}
\def\III{\mathrm{I\kern-0.1emI\kern-0.1emI}}
\def\IV{\mathrm{I\kern-0.1emV}}
\def\V{\mathrm{V}}
\def\VI{\mathrm{V\kern-0.1emI}}

\def\f#1#2{\frac{#1}{#2}}

\def\St{\mathbb{S}}

\def\Vec{\mathop{\mathrm{Vec}}}

\DeclareSymbol{X}{-2.4}{\node[dot] {};}
\DeclareSymbol{1}{0}{\draw[white] (-.4,0) -- (.4,0); \draw (0,0)  -- (0,1.2) node[dot] {};}
\DeclareSymbol{2}{0}{\draw (-0.6,1.3) node[dot] {} -- (0,0) -- (0.6,1.3) node[dot] {};}

\DeclareSymbol{3}{0}{\draw (0,0) -- (0,1.2) node[dot] {}; \draw (-.7,1) node[dot] {} -- (0,0) -- (.7,1) node[dot] {};}

\DeclareSymbol{31}{-3}{\draw (0,0) -- (0,-1) -- (1,0) node[dot] {}; \draw (0,0) -- (0,1.2) node[dot] {}; \draw (-.7,1) node[dot] {} -- (0,0) -- (.7,1) node[dot] {};}
\DeclareSymbol{30}{-3}{\draw (0,0) -- (0,-1); \draw (0,0) -- (0,1.2) node[dot] {}; \draw (-.7,1) node[dot] {} -- (0,0) -- (.7,1) node[dot] {};}
\DeclareSymbol{32}{-3}{\draw (0,0) -- (0,-1) -- (1,0) node[dot] {}; \draw (0,0) -- (0,-1) -- (-1,0) node[dot] {}; \draw (0,0) -- (0,1.2) node[dot] {}; \draw (-.7,1) node[dot] {} -- (0,0) -- (.7,1) node[dot] {};}
\DeclareSymbol{22}{-3}{\draw (0,0.3) -- (0,-1) -- (1,0) node[dot] {}; \draw (0,0.3) -- (0,-1) -- (-1,0) node[dot] {};\draw (-.7,1) node[dot] {} -- (0,0.3) -- (.7,1) node[dot] {};}
\DeclareSymbol{20}{-3}{\draw (0,0) -- (0,-1);\draw (-.7,1) node[dot] {} -- (0,0) -- (.7,1) node[dot] {};}
\DeclareSymbol{12}{-3}{\draw (0,0.3) -- (0,-1) -- (1,0) node[dot] {}; \draw (0,0.3) -- (0,-1) -- (-1,0) node[dot] {};\draw (-.7,1) node[dot] {} -- (0,0.3);}
\DeclareSymbol{10}{-3}{\draw (0,0.3) -- (0,-1);\draw (-.7,1) node[dot] {} -- (0,0.3);}
\DeclareSymbol{21}{-3}{\draw (0,0.3) -- (0,-1) -- (1,0) node[dot] {};\draw (-.7,1) node[dot] {} -- (0,0.3) -- (.7,1) node[dot] {};}
\DeclareSymbol{21a}{-3}{\draw (0,0.3) -- (0,-1) -- (1,0) node[dot] {};\draw  (0,0.3) -- (.7,1) node[dot] {};}
\DeclareSymbol{211}{3}{
	\draw (0,0) -- (-2,3) node[dot] {};
	\draw (0,0)  -- (1,1) node[dot] {};
	\draw (-0.67,1)  -- (.33,2) node[dot] {};
	\draw (-1.33,2)  -- (-0.33,3) node[dot] {};
}
\DeclareSymbol{4}{-3}{
	\draw (-1,0)  -- (0,-1.2) -- (1,0) ;
	\draw (-1.5,1) node[dot] {} -- (-1,0) -- (-0.5,1) node[dot] {};  
	\draw (1.5,1) node[dot] {} -- (1,0) -- (0.5,1) node[dot] {};
}

\definecolor{Red}{rgb}{1,0,0}
\definecolor{Blue}{rgb}{0,0,1}
\definecolor{Olive}{rgb}{0.41,0.55,0.13}
\definecolor{Yarok}{rgb}{0,0.5,0}
\definecolor{Green}{rgb}{0,1,0}
\definecolor{MGreen}{rgb}{0,0.8,0}
\definecolor{DGreen}{rgb}{0,0.65,0}
\definecolor{Yellow}{rgb}{1,1,0}
\definecolor{Cyan}{rgb}{0,1,1}
\definecolor{Magenta}{rgb}{1,0,1}
\definecolor{Orange}{rgb}{1,.5,0}
\definecolor{Violet}{rgb}{.5,0,.5}
\definecolor{Purple}{rgb}{.75,0,.25}
\definecolor{Brown}{rgb}{.75,.5,.25}
\definecolor{Grey}{rgb}{.7,.7,.7}
\definecolor{Black}{rgb}{0,0,0}
\definecolor{dr}{rgb}{0.8,0,0}
\definecolor{db}{rgb}{0,0,0.8}

\newcommand{\row}[3]{
\foreach[count=\n] \i in {#1}
    \node at (\n,#2) [\i] (#3\n) {};
}
\newcounter{list_size}
\newcommand{\connect}[4][dr]{
\edef\myleftx{10000}
\edef\myrightx{-10000}
\edef\mycentery{0}
\setcounter{list_size}{0}
\foreach \j in {#2}{\stepcounter{list_size}};
\edef\n{\arabic{list_size}}
\ifthenelse{\n=1}
{
	\draw (#3#2) -- (#4#2);
}
{
\foreach \j in {#2} {
	\path (#3\j); \pgfgetlastxy{\XCoord}{\YCoord};
    \pgfmathsetmacro{\lx}{min(\myleftx,\XCoord)};
    \pgfmathsetmacro{\rx}{max(\myrightx,\XCoord)};
    \pgfmathsetmacro{\cy}{\mycentery+\YCoord/(2*\n)};
	\path (#4\j); \pgfgetlastxy{\XCoord}{\YCoord};
    \pgfmathsetmacro{\cy}{\cy+\YCoord/(2*\n)};
    \global\let\myleftx=\lx
    \global\let\myrightx=\rx
    \global\let\mycentery=\cy
	}
\foreach \j in {#2} {
	\draw (#3\j) -- (#4\j);
}
\draw[thick,#1] (\myleftx pt,\mycentery pt) -- (\myrightx pt,\mycentery pt);
\foreach \j in {#2} {
	\path (#3\j); \pgfgetlastxy{\XCoord}{\YCoord};
    \node at (\XCoord,\mycentery pt) [simple,fill=#1] {};
}
}
}

\begin{document} 
 
\title{A scaling limit of the parabolic Anderson model with exclusion interaction}
\author{Dirk Erhard$^1$ and Martin Hairer$^2$}
\institute{Universidade Federal da Bahia, Brazil, \email{erharddirk@gmail.com} 
\and Imperial College London, UK, \email{m.hairer@imperial.ac.uk}}

\maketitle

\begin{abstract}
We consider the (discrete) parabolic Anderson model $\partial u(t,x)/\partial t=\Delta u(t,x) +\xi_t(x) u(t,x)$, $t\geq 0$, $x\in \Z^d$, where the $\xi$-field is $\R$-valued and plays the role of a dynamic random environment, and $\Delta$ is the discrete Laplacian. We focus on the case in which $\xi$ is given by a properly rescaled  symmetric simple exclusion process under which it converges to an Ornstein--Uhlenbeck process. 
Scaling the Laplacian diffusively and restricting ourselves to a torus, we show that in dimension $d=3$ upon considering a suitably renormalised version of the above equation, the sequence of solutions converges in law. 

As a by-product of our main result we obtain precise asymptotics for the survival probability of a simple random walk that is killed at a scale dependent rate when meeting an exclusion particle.
Our proof relies on the discrete theory of regularity structures of~\cite{ErhardHairerRegularity} and on novel  sharp estimates of joint cumulants of arbitrary 
large order for the exclusion process.
We think that the latter is of independent interest and may find applications elsewhere.\\[.4em]
\end{abstract}

\setcounter{tocdepth}{2}
\tableofcontents

\section{Introduction}
\label{S1}
The (discrete) parabolic Anderson model is the partial differential equation
\begin{equ} 
\label{eq:PAM}
\partial_t u (x,t)=(\Delta u)(x,t) - \xi_t(x)u(x,t) \;.
\end{equ}
Here $x\in \Z^d$, $t\geq 0$, and $\Delta$ is the discrete Laplacian acting on $u$ as
\begin{equ}[e:delta]
\Delta u(x,t)= \sum_{y:\, x\sim y}[u(y,t)-u(x,t)],
\end{equ}
where $x\sim y$ means that $x$ and $y$ are nearest neighbours with respect to the Euclidean norm on $\Z^d$, while $\xi$ is an $\R$-valued random field that plays the role of a dynamic random environment. In this article we focus on the case in which $\xi$ is given by the symmetric simple exclusion process, i.e., $\xi=(\xi_t)_{t\geq 0}$ is the Markov process taking values in $\cH=\{0,1\}^{\Z^d}$ with generator $L$ acting on local functions 
$f\in\R^{\cH}$ via
\begin{equation}
\label{eq:genex}
(Lf)(\eta) = \sum_{x\sim y}
[f(\eta^{x,y})-f(\eta)].
\end{equation}
Here,
\begin{equation}
\label{eq:etaswapped}
\eta^{x,y}(z) = 
\left\{\begin{array}{ll}
\eta(z), &z\neq x,y\\
\eta(y), &z=x\\
\eta(x), &z=y.
\end{array}\right.
\end{equation}
Informally, \eqref{eq:genex} says that neighbouring states are swapped independently at 
rate~$1$. An equivalent way to construct the symmetric simple exclusion process is via its \emph{graphical representation}. Here, space is drawn sideways, time is
drawn upwards, and to each edge $e = \{v,w\}$ connecting two neighbouring sites $v$ and $w$ in $\Z^d$ one attaches a Poisson process $P(e)$ with intensity~$1$. Each event of 
such a Poisson process is drawn as a link above $e$. The configuration at 
time $t$ is then obtained
from the one at time~$0$ by transporting the local states along paths that move
upwards with time and sidewards along the links (see Figure~\ref{fig1}).


\begin{figure}
\setlength{\unitlength}{0.27cm}
\begin{center}
\tikzsetnextfilename{graphrep}
\begin{tikzpicture}[scale=0.7,baseline=0.85cm]
\draw[->] (0,0) -- (8,0);
\draw[black!25] (0,4) -- (8,4);
\foreach \x in {1,...,7}
{
	\draw (\x,0) -- (\x,4.5);
}
\draw[very thick,dr!80] (4,0) -- (4,1) -- (5,1) -- (5,2.5) -- (4,2.5) -- (4,3.2) -- (3,3.2) -- (3,4);
\foreach \x/\y in {1/2,1/3,2/1.1,3/2,3/3.2,4/2.5,4/1,5/0.5,6/1.8,6/2.2,6/3.7}
{
	\draw[tinydots] (\x,\y) -- (\x+1,\y);
}

\node[fill=white,below] at (4,0) {$x$};
\node[fill=white,above] at (3,4) {$y$};
\node[left] at (0,0) {$0$};
\node[left] at (0,4) {$t$};
\node[right] at (8,0) {$\Z^d$};
\node[inner sep=0pt,minimum size=1mm] at (4,0) [circle,fill=dr!80] {};
\node[inner sep=0pt,minimum size=1mm] at (3,4) [circle,fill=dr!80] {};
\end{tikzpicture}
\end{center}
\caption{Graphical representation. The dashed lines are links
and the thick line shows a path from $(x,0)$ to $(y,t)$ following links.}
               \label{fig1}
\end{figure}
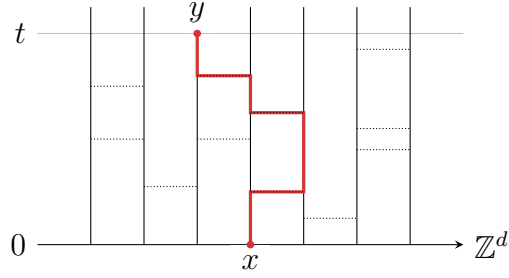

For $\rho\in (0,1)$, denote by $\nu_{\rho}$ the Bernoulli product measure on $\Z^d$ with parameter $\rho$. It is well known that $(\nu_{\rho})_{\rho \in [0,1]}$ forms a family of invariant measures for $\xi$ under which the dynamics is reversible. We assume throughout this article that $\xi_0\sim\nu_\rho$ for some fixed value of $\rho$.
Denote by $\cS(\R^d)$ the space of Schwartz functions and for $N\in\N$ define the fluctuation field $Y^N$ of $\xi$ via
\begin{equation}
Y_t^N(f)= 2^{-Nd}\sum_{x\in\Z^d}f(x/2^N)2^{Nd/2}[\xi_{t2^{2N}}(x)-\rho].
\end{equation}
It was shown in \cite{Ravi92} that for all $T>0$ the process $Y^N$ converges in law with respect to the Skorokhod topology in the space $\cD([0,T], \cS'(\R^d))$ to the
stationary generalised Ornstein--Uhlenbeck process $Y$ solving
\begin{equ}[eq:Y]
\d_t Y = \Delta Y + \sqrt{2\rho(1-\rho)} \,\div \hat \xi\;,
\end{equ}
where $\hat \xi$ denotes a vector-valued space-time white noise.
It is well known that when $d \ge 2$, $Y$ does not make sense as a random space-time function, 
but as a random function in time taking values 
in the Besov--H\"{o}lder space $\cC^{-d/2-}(\R^{d+1};\R)$.
This suggests that looking at the equation
\begin{equation}\label{eq:disunrenPAM}
\partial_t\hat u^N(t,x)=(\Delta_N \hat u^N)(t,x)-2^{Nd/2}\bar\xi^N_{t}(x)\hat u^N(t,x),
\end{equation}
where $\Delta_N =2^{2N}\Delta$ is the rescaled discrete Laplacian acting on functions defined on $2^{-N}\Z^d$, $x\in2^{-N}\Z^d$, $t\geq 0$, $\xi^N_t(x)=\xi_{2^{2N}t}(2^Nx)$, and $\bar \xi^N_t(x)=\xi_{2^{2N}t}(2^Nx)-\rho$ is its centred version, its solutions $\hat u^N$ should converge to the solution $\hat u$ of
\begin{equation}\label{eq:unrenPAM}
\partial_t \hat u(t,x)=(\Delta \hat u)(t,x)-Y_t(x) \hat u(t,x),\quad x\in\R^d,t\geq 0.
\end{equation}
Here, $\Delta$ denotes the usual continuous Laplacian. The problem with this equation is 
that as soon as $d\geq 2$ there is no canonical interpretation for the product between 
$Y$ and $\hat u$. Indeed, since $Y$ has regularity $-d/2^-$ (and no better) one would 
expect $\hat u$ to 
have regularity no better than $2-d/2^-$. The product then becomes ill-defined 
when $2-d/2^- -d/2^-\leq 0$, i.e.\ as
soon as $d\geq 2$. However, the theory of regularity structures developed in~\cite{Regularity} shows that in dimensions $d\in\{2,3\}$, it is possible to renormalise~\eqref{eq:unrenPAM} 
so that it becomes well-defined. This amounts to subtracting an infinite counterterm and 
the more honest way of writing~\eqref{eq:unrenPAM} is in fact
\begin{equation}\label{eq:renPAM}
\partial_t u = \Delta  u - (Y - \infty) u\;.
\end{equation}
Since $Y$ is Gaussian it follows indeed from the arguments in~\cite{WongZakai} and~\cite{Ajay} that there is a rigorous way of making sense of~\eqref{eq:renPAM} by replacing $Y$ first by a smooth approximation $Y_\eps$  and $\infty$ by an explicit constant $C_\eps$ that tends to infinity as $\eps$ tends to zero. This suggests that instead of looking at~\eqref{eq:disunrenPAM} we really should look at 
\begin{equ} 
\label{eq:PAMN}
\partial_t u^{N}(t,x)=(2^{2N}\Delta u^{N})(t,x) - [2^{Nd/2}\bar\xi^N_{t}(x)-C_N]u^{N}(t,x),
\end{equ}
for some choice of constants $C_N$ that tend to infinity. To show the convergence of $u^N$ one faces two problems:
\begin{itemize}
	\item[(1)] The original theory of regularity structures does not immediately apply to show convergence of discrete systems.
	\item[(2)] Unlike $Y$, the exclusion process is of course not a Gaussian process. In particular estimates on second moments do not immediately yield estimates on higher moments
	and higher order cumulants play a role in estimates.
\end{itemize}
The first problem can be circumvented by using a discrete theory of regularity 
structures~\cite{ErhardHairerRegularity}. The second problem however is more difficult to 
solve and requires much better control on the symmetric simple exclusion process than what
is currently known. For example, it seems unfortunately not be possible to take advantage of the 
fact that $\xi$ satisfies a log Sobolev inequality, since the log Sobolev constant is too 
large. Instead, we use a mixture of martingale techniques and combinatorial arguments to obtain estimates on joint cumulants of arbitrary high order of $\xi$ that are sufficient for our purposes. In particular, one of our main results can be formulated as follows.

\begin{theorem}\label{thm:cumu}
Let $\xi$ denote the stationary symmetric simple
exclusion process on $(t,x) \in \R\times \Z^d$ with $d \ge 3$ and $\E\xi(t,x) = \rho \in (0,1)$.
Then, for every $k \ge 2$ there exists a constant $C$ such that, for 
any collection $(t_i, x_i)_{i \in [k]}$ of $k$ space-time points, 
the joint cumulant satisfies
\begin{equ}\label{eq:cumu}
\E_c \{\xi(t_i,x_i)\}_{i \le k} \le C\sum_{\sigma}\prod_{i \in [k]} \bigl(1 + |t_{\sigma_{i+1}} - t_{\sigma_i}|
+ |x_{\sigma_{i+1}} - x_{\sigma_i}|^2\bigr)^{-d/4}\;,
\end{equ}
where the sum runs over all bijections $\sigma \colon [k] \to [k]$
and we use the convention $\sigma_{k+1}=\sigma_1$.
\end{theorem}

\begin{remark}
We expect this bound to be optimal whenever the points obey a near parabolic scaling
in the sense that  
$|t_{i} - t_{j}| \simeq |x_{i} - x_{j}|^2$, which
is the regime that interests us. If $|t_{i} - t_{j}| \ll  |x_{i} - x_{j}|^2$, one
should be able to obtain a better bound by exploiting the superexponential 
decay of the discrete heat kernel.
\end{remark}

Regarding convergence of the $u_N$'s as in \eqref{eq:PAMN}, we impose additional
restrictions: we replace $\Z^d$ by a large torus $\Z_N^d=(\Z/2^{N}\Z)^d\eqdef \St$ and
we only consider initial conditions for~\eqref{eq:PAMN} that are smooth at macroscopic scale:
see Section~\ref{sec:funcSpaces} below for the definitions of the function spaces used to formulate our
result.

\begin{theorem}
\label{thm}
Let $d=3$ and let $(u_{0}^N)_{N\in\N}$ be a sequence of initial conditions such that
there is $\eta\in (0,1)$ and $u_0\in\cC^\eta$ with
\begin{equation}
\lim_{n\to\infty}\|u_0;u_0^N\|_\eta=0\;.
\end{equation} 
Write $u^N$ for the solution to~\eqref{eq:PAMN}  defined on the rescaled torus $\T_N^d=2^{-N}\Z_N^d$ 
and $u$ for the solution to~\eqref{eq:renPAM}.
Then, there is a diverging sequence of constants $C_N$ such that the sequence  $u^N$  converges
in law to $u$ in $\cC_N^{\bar\eta,T}$ for all $\bar\eta \in (0,\frac12\wedge\eta)$. 
\end{theorem}

\begin{remark}
We expect Theorem~~\ref{thm} to be true also in dimension $2$. From a purely analytical point of view the $2$-dimensional case is easier than the $3$-dimensional case. However, it turns out that the stochastic estimates become harder in dimension $2$. If we would have Theorem~\ref{thm:cumu} for $d=2$ (which we strongly expect to be true), then the $2$-dimensional case would follow from a by now standard adaptation of our arguments. The only gap to be filled to establish~\eqref{eq:cumu} is to derive a gaussian upper bound on the discrete partial derivatives of the transition probability of a field of $n$ labelled exclusion particles. We refer to Remark~\ref{rem:2d} for a more precise formulation. 
\end{remark}
\begin{remark} It turns out that the constant $C_N$ actually can be written as a sum of three constants, namely it takes the form $C_N= c_N + c_N^{(1)} +c_N^{(2)}$. These constants will be introduced in Section~\ref{S3.3}.
In Proposition~\ref{prop:renconstant} we will show that $C_N\approx \alpha_N2^N+\beta_N2^{N/2} + \gamma_NN$ for some bounded sequences $\{\alpha_N\}_{N\in\N},\{\beta_N\}_{N\in\N},$ and $\{\gamma_N\}_{N\in\N}$.  
	\end{remark}
	
\emph{If it were the case that Theorem~\ref{thm} holds on the whole space $\Z^d$} then, translated in terms of the discrete parabolic Anderson model, our result could be 
restated as follows. Fix a realisation of the symmetric simple exclusion process $\eta$ 
with density $\rho$ on $\Z^d$ and let $X_t^\eps$ be the position of a simple random walker
on $\Z^d$ that is killed at rate $\eps \eta_t(X_t^\eps)$ but is otherwise independent of $\eta$.
Then, for $d =3$, in order to see non-trivial behaviour for $X_t^\eps$ when 
$\eps \ll 1$, one should
look at time scales of order $t = \tau \eps^{-2\beta}$ and space scales of order $X = x\eps^{-\beta}$
with $\beta = {2\over 4-d}=2$.
The probability that $X$ survives up to time $\tau \eps^{-2\beta}$ is then of order 
$S_\tau^{\eps}$ with 
\begin{equ}
	S_\tau^{\eps} = \eps^{\alpha_3 \tau} \exp(-\rho \eps^{-3} \tau + c\eps^{-2} \tau + \bar c\eps^{-1} \tau )\;.
\end{equ}
for some constants $c, \bar c, \alpha_3 \in \R$.
Furthermore, writing $P_t^\eta(x)$ for the probability that, given a realisation of the environment $\eta$,
the process $X$ survives up to time $t$ and can be found at location $x$, \emph{and provided that we had Theorem~\ref{thm} for the choice $u_{0_N}=\delta_0$, the Dirac delta at zero,} one has
\begin{equ}\label{eq:transitionsurvival}
\lim_{\eps \to 0} {P_{\tau \eps^{-2\beta}}^\eta(x\eps^{-\beta}) \over S_{\tau}^{\eps}}
= u(x,\tau)\;,
\end{equ}
where $u$ solves \eqref{eq:renPAM}.
Here, the convergence takes place in law,
jointly and uniformly for all values of $(\tau,x)$ belonging to any fixed compact region of space-time 
included in $\{(x,\tau)\,:\,\tau > 0\}$.

\begin{remark}
As already remarked above, we believe that one could adapt the techniques from~\cite{Cyril} 
to lift the two restrictions just mentioned above. 
\end{remark}

\begin{remark}
Using the same techniques, one easily shows that if one chooses
$\beta < 2/(4-d)$, then the result analogous to \eqref{eq:transitionsurvival} simply
yields a solution to the (deterministic) heat equation in the limit.
\end{remark}

\begin{remark}
As is often the case, while the constants $c$ and $\bar c$ are model-dependent 
(they will change for example if we change the
moves for the underlying random walk, the details of the symmetric simple exclusion process,
the shape of our grid, etc),
the exponent $\alpha_3$ is expected to only depend on the dimension. 
\end{remark}
\begin{remark}
The interpretation of the solution to~\eqref{eq:renPAM} as the survival probability of a simple random walk in a field of traps goes back at least to the celebrated works \cite{DV75,DV79} by Donsker and Varadhan. In these works, the authors considered the case in which the traps are given by a Poisson point field of balls and by an i.i.d.\ field of Bernoulli variables respectively, which are such that the Wiener sausage and random walk respectively are killed instantaneously upon stepping on a trap. They obtained asymptotics for the survival probability which were later refined in a number of works, see for example~\cite{Bolthausen94,Povel99,Sznitman91}. Whereas in the case of static traps we have by now a very good understanding of the behaviour of the underlying random walk, much less is known in the dynamic case. The works that are in this regard probably closest to ours are the works of Drewitz, G\"{a}rtner, Ramirez and Sun~\cite{DrewitzGartnerRamirezSun12}, Athreya, Drewitz and Sun~\cite{AthreyaDrewitzSun17}, and of Shen, Song, Sun and Xu~\cite{Shen20}. In the two former works the survival probability and the path behaviour of a random walk in a field of independent random walks starting from a Poisson field is studied. The main difference to our work is however, that the study of the path behaviour was restricted to the case in which there was an additional annealing with respect to the field of traps whereas we quench over the environment. Moreover, in contrast to~\cite{AthreyaDrewitzSun17} we look at a weak killing regime. In~\cite{Shen20} however the work~\cite{AthreyaDrewitzSun17} was extended to the quenched regime in one spatial dimension.
\end{remark}

\subsection{Notations}
Throughout this work $\s=(\s_0,\s_1,\ldots, \s_d)\in \N_{\geq 1}^{d+1}$ denotes a scaling of $\R^{d+1}$ and we 
associate to it the ``norm'' on $\R^{d+1}$ given by
\begin{equation}
\|z\|_\s :=\sup_{i\in\{0,\ldots, d\}}|z_i|^{1/\s_i}\;.
\end{equation}  
We define $|\s|=\sum_{i=0}^{d}|\s_i|$. Unless stated otherwise, $\s$ denotes the parabolic scaling in $\R^4$, i.e., $\s=(2,1,1,1)$. Given $\delta>0$ and $\varphi:\R^{d+1}\to\R$ we set $\CS_\s^{\delta}(z_0,\ldots, z_d)= (\delta^{-\s_0}z_0,\ldots, \delta^{-\s_d}z_d)$ and $(\CS_{\s,z}^{\delta}\varphi)(y)=\delta^{-|\s|}\varphi(\CS_\s^{\delta}(y-z))$. When the scaling is clear from the context we also use the notation $\varphi_z^{\delta}$. Moreover, $\|\cdot\|$ denotes the Euclidean norm. 

$P^N(e)$ denotes the Poisson process with intensity $2^{2N}$ attached to 
the edge $e$ in the graphical construction of the exclusion process. $\xi$ is the exclusion process and $\bar\xi$ is its centred version.
Given $\|e\|=1$, we denote by $\nabla_N^{e}$ the discrete derivative in direction $e$ that is defined via $\nabla_N^{e}p(x,w) = 2^{N}[p(x,w+e)-p(x,w)]$.
$\Delta_N$ denotes the usual discrete Laplacian sped up by a factor $2^{2N}$.
$\{\cF_r\}_{r\geq 0}$ denotes the natural filtration induced by the exclusion process.
The notation $\Z_N^d$ denotes the $d$-dimensional discrete torus of side length $2^N$, and will sometimes 
be denoted by $\St$ (standing for ``state space'') and $\T_N^d=2^{-N}\Z_N^d$ is its rescaled version.
The notation $\K$ is reserved to denote some compact subset of $\R^4$, the notation $\K_N$ will be used to denote a compact
subset of $\R^4$ with diameter bounded by $2^{-N+1}$. Moreover, $\bar{\K}$ denotes the 1-fattening of $\K$. 

\subsection{Discrete function spaces}\label{sec:funcSpaces}

For $\eta\in (0,1)$, we define discrete H\"older spaces $\cC_N^\eta(\T_N^d,\R)$ as the space of all elements $f\in\R^{\T_N^d}$, with norm 
\begin{equation}
\|f\|_{\cC_N^\eta}\overset{\text{def}}{=} \sup_{x\in\T_N^d}|f(x)| + \sup_{x\neq y\in\T_N^d}
\frac{|f(x)-u(y)|}{|x-y|^\eta}.
\end{equation}
Let $\eta\in (0,1)$. To compare an element $f\in\cC^\eta(\T^d,\R)$ in the usual H\"older space with an element $f^N\in \cC_N^\eta(\T_N^d,\R)$ we introduce the distance
\begin{equation}\label{eq:discreteHolder}
\begin{aligned}
\|f;f^N\|_{\eta}\overset{\text{def}}{=}
&\sup_{x\in\T_N^d}|f(x)-f^N(x)| + \sup_{x\neq y\in\T_N^d}
\frac{|(f(x)-f(y))-(f^N(x)-f^N(y))|}{|x-y|^\eta}\\
 &+\sup_{\substack{x,y\in\R^d:\, |x-y|<2^{-N}}}
\frac{|f(x)-f(y)|}{|x-y|^\eta}.
\end{aligned}
\end{equation}
The reason for this choice is that denoting by $\tilde f^N$ the function that is extended to all of $\T^d$ by linear interpolation of $f^N$, then at scales of order $\eps=2^{-N}$ or larger it may be possible to approximate $f$ by $\tilde f^N$, but this approximation may break down at smaller scales.
To compare functions $f\in\cC_\s^\eta ([0,T]\times \T^d,\R)$ and $f^N\colon [0,T]\times \T_N^d \to \R$, we define a ``distance'' by
\begin{equs}
\|f;f^N\|_{\cC_N^{\eta,T}}&\eqdef \sup_{(t,x)\in [0,T]\times \T_N^d}|f(t,x)-f^N(t,x)|+ \sup_{\substack{(t,x), (s,y)\in [0,T]\times\T_N^d\\ \|(t,x)-(s,y)\|_\s < 2^{-N}}}\frac{|f(t,x)-f(s,y)|}{\|(t,x)-(s,y)\|_\s^\eta}\\
& + \sup_{\substack{(t,x), (s,y)\in [0,T]\times\T_N^d\\ \|(t,x)-(s,y)\|_\s \geq 2^{-N}}} \frac{|(f(t,x)-f(s,y))-(f^N(t,x)-f^N(s,y))|}{\|(t,x)-(s,y)\|_\s^\eta}.
\end{equs}

\subsection*{Acknowledgements}

{\small
MH gratefully acknowledges support from the Royal Society through a research professorship. DE gratefully acknowledges financial support 
from the National Council for Scientific and Technological Development - CNPq via a 
Universal grant 409259/2018-7, and a Bolsa de Produtividade 303520/2019-1.
}

\section{Regularity structures}
\label{S2}
The proof of Theorem~\ref{thm} relies on the theory of discrete regularity structures 
introduced in \cite{ErhardHairerRegularity}, which builds up on~\cite{Regularity}. We refer to this source for a discussion of the theory that goes beyond the description we are giving here.
Recall that a regularity structure consists of a graded vector space $\mathcal{T}= \bigoplus_{\alpha\in A}\mathcal{T}_{\alpha}$, where the index set $A\subseteq \R$ is locally finite and bounded from below, as well as a group $\mathcal{G}$ consisting of
continuous block upper unitriangular linear transformations on $\mathcal{T}$.
We say that elements $\tau\in \mathcal{T}_{\alpha}$ are homogeneous of degree $\alpha$
and we write $\deg\tau = \alpha$. 

In \cite[Sec.~8]{Regularity} and \cite[Sec.~5]{BHZalg} a general procedure is explained to construct a regularity structure to describe a class of subcritical SPDEs. 
In our setting, we start with the regularity structure generated by the polynomials
on $\R^{d+1}$ with parabolic scaling and an additional symbol $\Xi$ with
$\deg\Xi= -\f32-\kappa$ (for some sufficiently small $\kappa > 0$) 
that will be used to represent a suitably 
rescaled copy of the environment, i.e.\ a symmetric simple exclusion process. 
We add to this an integration operator $\cI$ of degree $2$, i.e.\ whenever a symbol
$\tau$ belongs to our regularity structure, we also include a symbol $\cI(\tau)$
with $\deg \cI(\tau)=2+\deg\tau$.

We then define two collections $\cU$ and $\cF$ of formal expressions as the smallest
collections such that 
$X^k\in\cU$ for every multi-index $k$, $\cU \subset \cF$, and furthermore
\begin{equation}
\label{eq:U}
\tau\in\cU\;\Longrightarrow\; \tau\Xi\in\cF\;,\qquad
\tau\in\cF\;\Longrightarrow\; \cI(\tau)\in\cU\;.
\end{equation} 
We then set $\cT = \Vec \cF$, endowed with the grading given by the aforementioned notion of degree.
It follows from \cite[Lem.~8.10]{Regularity} that, provided that $\kappa< \f12$, $\cF$ contains only finitely many elements of degree less than $\alpha$ for any $\alpha\in\R$, so that 
each $\cT_{\alpha}$ is finite-dimensional. We also impose that $\CI(X^k) = 0$ for every $k \in \N^{d+1}$.

In particular, using the same graphical notation as in \cite{WongZakai},
the list of all symbols of negative degree is given by
$\{\<Xi>, \<Xi2>, \<Xi3>, \<XiX>, \<Xi2X>, \<Xi4>\}$,
with $\deg \<Xi> = -\f32-\kappa$, $\deg \<Xi2> = -1-2\kappa$, etc.
The last ingredient of the regularity structure, namely the structure group $\mathcal{G}$, can also be described 
explicitly, but since its precise description does not matter for the purpose of this article we refrain from giving it here. We refer the interested reader to~\cite[Sec.~8]{Regularity} for the general construction of $\mathcal{G}$,
 as well as to~\cite[Sec.~3]{WongZakai} where a similar case is studied. 
 
Given a regularity structure, a crucial concept is that of a model for it, and it is 
there where the discrete theory developed in \cite{ErhardHairerRegularity} comes in handy.
We set 
$\CX_N= \CD(\R,\R^{\T_N^3})$ where, in general, $\CD(\R,E)$ denotes the space of all c\`{a}dl\`{a}g functions from $\R$ 
to $E$.
We define an inclusion map $\iota_N:\CX_N\to \CS'(\R^{1+3})$ via
\begin{equation}
\label{eq:iota}
(\iota_N f)(\varphi)= 2^{-3N}\sum_{x\in\T_N^3}\int f(t,x)\varphi(t, x)\, dt.
\end{equation}
We further endow $\CX_N$ with a family of seminorms $\|\cdot\|_{\alpha;\K_N;z;N}$. Here, the parameter $\alpha$ ranges over $\R$, $\K_N$ ranges over all compact subsets in $\R^4$ with diameter bounded by $2^{-N+1}$, and $z=(t,x)\in\R\times \T_N^3$ is such that $z\in\K_N$. We further fix an integer $r>|{\min A}|$ and $\lambda\in (0,2^{-N}]$ and we denote by $\Phi_{N,z}^{\lambda}$ the set of all functions $\phi:\R\to\R$ with $\|\phi\|_{\CC^r}\leq 1$ and support contained in the unit ball so that moreover $\supp \CS_{2,t}^{\lambda}\phi\subset\{s\in\R:\, (s,x)\in\K_N\}$. We then define
\begin{equation}
\label{eq:seminorm}
\|f\|_{\alpha;\K_N;z;N}= \sup_{\lambda\in (0,2^{-N}]}\sup_{\phi\in \Phi_{N,z}^{\lambda}}
\lambda^{-\alpha}\Big|\int f(s,x)(\CS_{2,t}^{\lambda}\phi)(s)\, ds\Big|.
\end{equation}
We denote by $\|\cdot\|_\ell$ the norm of the $\ell$-th component in $\CT$ and we use $\$\cdot\$_{\ell;\K;N}$ to denote a seminorm on the space of functions $f:\R^4\to \CT_{<\ell}$. The expression $\$ f\$_{\ell;\K;N}$ for such a function $f$ is only allowed
to depend on the values of $f$ in a neighbourhood of size proportional to $2^{-N}$ around $\K$. 
Further below we give the precise expression of the norm we are actually interested in.
A discrete model then consists of a collection of maps $z\mapsto \Pi_z^N\in \CL(\CT, \CX_N)$ and $\Gamma^N:\R^{4}\times\R^{4}\to \CG$ such that
\begin{claim}
\item $\Gamma^N_{zz}= \mathrm{id}$, the identity operator, and $\Gamma^N_{xy}\Gamma^N_{yz}=\Gamma^N_{xz}$ for $x,y,z\in\R^4$;
\item One has $\Pi_z^N= \Pi_y^N\Gamma^N_{yz}$ for all $y,z\in\R^4$.
\end{claim}
Furthermore, for any compact set $\K\subset\R^d$ and every $\tau\in \CT_{\ell}$ one imposes
the analytical estimates
\begin{equ}
\label{eq:Pi}
\bigl|\bigl(\iota_N\Pi^{N}_z \tau\bigr)(\CS_{\s,z}^{\lambda}\phi)\bigr| \lesssim \|\tau\|_\ell\lambda^{\ell}\;,\quad 
\|\Pi^{N}_z \tau\|_{\ell;\K_N,z;N} \lesssim \|\tau\|_\ell\;,
\end{equ} and
\begin{equation}
\label{eq:Gamma}
\|\Gamma^N_{zz'}\tau\|_m\lesssim \|\tau\|_\ell \|z-z'\|_\s^{\ell-m}\;,\quad
\$ z\mapsto \Gamma^{N}_{zz'}\tau\$_{\ell;\K; N}\lesssim \|\tau\|_\ell.
\end{equation}
For some fixed $\gamma > 0$, these bounds are assumed to be uniform over $\lambda \in (2^{-N},1]$, all $\phi\in\CC^r$ with $\CC^r$-norm bounded by one and support contained in the unit ball, all $\ell\in A$ such that $\ell <\gamma$, and $m<\ell$, all $\tau\in \CT_{\ell}$, and locally uniform over $z,z'\in\K$ such that $\|z-z'\|_\s\in (2^{-N},1]$ and over all compact subsets $\K_N$ of diameter bounded by $2^{-N+1}$.

In what follows we denote by $\|\Pi^N\|_{\gamma;\K}^{(N)}$ the smallest proportionality constant in~\eqref{eq:Pi}
and \eqref{eq:Gamma}.
There is a specific class of models that are of importance, the class of admissible models. To describe these we first note that as a consequence of~\cite[Sec.~5]{MatetskiDiscrete} and~\cite[Sec.~A.1]{KPZJeremy} it is possible to decompose the properly rescaled Green's function $2^{3N}G^N$ of $\partial_t-\Delta_N$ as
\begin{equation}
2^{3N}G^N= K^N+R^N\;,
\end{equation}
where the function $K^N$ can be decomposed as $K^N= \sum_{n=1}^{N}K_n$,
where each $K_n$ has a support contained in the set $\{z\in\R^4: \|z\|_\s \lesssim 2^{-n}\}$,
and annihilates polynomials of scaled degree less than or equal to two. Moreover, for each $n\in\{1,\ldots, N-1\}$,
and each multi-index $|k|_\s\leq 2$,
\begin{equation}
\label{eq:Knest}
|D^kK_n(z)|\lesssim 2^{n(3+|k|_\s)}
\end{equation}
uniformly in all parameters.
If $n=N$ in Equation~\eqref{eq:Knest}, then the same estimate holds for $k=0$. The meaning in~\eqref{eq:Knest} should be understood in the following way: each $K_n$ can be extended to a function living on $\R^4$ and its derivatives up to second order are bounded.
Furthermore, $R^N$ is compactly supported, is anticipative, i.e., $R^N(\cdot,t)=0$ for $t<0$, and $\|R^N\|_{\CC^2}$ is bounded uniformly in $N$. The latter property has to be understood in a similar way as above.
Let $\zeta\in\R$ and $n\leq N$, we define
\begin{equation}
(T_{n,\zeta+2}^N F)(z)= \sum_{|k|_\s<\zeta+2}\frac{X^k}{k!}\CQ_k((T_{n,\zeta+\beta}^{N} F)(z)),
\end{equation}
where for $z=(t,x)\in \R^4$, $n\in\{1,\ldots, N\}$, and $|k|_\s <\zeta+2$,
\begin{equation}
\CQ_k((T_{n,\zeta+2}^{N} F)(z))= 
\begin{cases}
2^{-3N}\sum_{y\in\T_N^3}\int D_1^k K_n( z,(s,y))F(y,s)\, ds, &\text{if }n<N,\\
\delta_{k=0} 2^{-3N}\sum_{y\in\T_N^3}\int K_N( z,(s,y))F(y,s)\,ds, &\text{otherwise},
\end{cases}
\end{equation}
for all $F\in\CX_N$ for which the above expression makes sense. 
Here, the $2$ in the indices above is only there to make the notation consistent with~\cite[Section 4]{ErhardHairerRegularity}.
\begin{remark}
We also use the abbreviations
\begin{equation}
T_{\zeta+2}^N F= \sum_{n=1}^{N}T_{n,\zeta+2}^N F,\quad\mbox{and}\quad 
\CQ_k(T_{\zeta+2}^N F)(\cdot))=\sum_{n=1}^{N}\CQ_k(T_{n,\zeta+2}^N F)(\cdot)).
\end{equation}
\end{remark}
One may now show that Assumptions 4.3 and 4.4 in~\cite[Sec.~4]{ErhardHairerRegularity} are satisfied for $\sigma=2$. Indeed, Assumption 4.3 follows from Taylor's formula in a way similar to the arguments used in~\cite[Sec.~5]{Regularity}. Assumption 4.4 is a consequence of the fact that the $K_n$ annihilate polynomials of degree two, and of our definition of the $\CQ_k$'s. 
Given $K^N$ as above, the set of admissible models $\mathscr{M}$ then consists of all models such that, for every multi-index $k$,
\begin{equation}
\label{eq:Xk}
(\Pi_{(t,x)}^N X^k)(s,y)= ((s,y)-(t,x))^k
\end{equation}
and such that
\begin{equation}
\label{eq:I}
\begin{aligned}
(\Pi_{(t,x)}^N \CI\tau)(s,y)= &2^{-3N}\sum_{\bar x\in\T_N^3}\int K^N((s,y),(\bar s,\bar x))(\Pi_{(t,x)}^{N}\tau)(\bar s,\bar x)\, d\bar s\\
&-\sum_{|k|_\s<|\tau|+2} \frac{((s,y)-(t,x))^k}{k!}Q_k\big((T_{|\tau|+2}^{N}\Pi_{(t,x)}^N \tau )(t,x)\big).
\end{aligned}
\end{equation}
The following construction gives rise to an admissible model, see~\cite{Regularity}. First define for $(t,x)\in\R^4$ and $(s,y)\in\T_N^3\times\R$,
\begin{equation}
(\Pi_{(t,x)}^N\Xi)(s,y)=\bar\xi_s^N(y), 
\end{equation}
and let $\Pi_{(t,x)}^NX^k$ be as in~\eqref{eq:Xk}, and extend this definition using recursively~\eqref{eq:I} and the relation
\begin{equation}
\label{eq:product}
\Pi_{(t,x)}^N\tau\bar{\tau}= \Pi_{(t,x)}^N\tau\,\Pi_{(t,x)}^N\bar{\tau}.
\end{equation}
Finally, the construction of $\Gamma^N$ in the continuous case is outlined in~\cite[Sec.~8]{Regularity}. Since we do not make use of its precise form in this article we do not comment on how to adapt it to the current setting. The model constructed in this way is also called the canonical model. It was then shown in~\cite{ErhardHairerRegularity} how to associate an~\emph{abstract fixed point problem} to~\eqref{eq:PAM}.
This fixed point problem is posed in a certain space $\CD_N^{\gamma,\eta}$ of $\CT_{<\gamma}$-valued 
functions (here, $\CT_{<\gamma}$ denotes the space $\bigoplus_{\alpha <\gamma} \CT_{\alpha}$), which can 
be thought of as a H\"{o}lder space of order $\gamma$ that allows for a blow-up of order $\eta$ near $t=0$. 
More precisely, we define the ``$t=0$''-hyperplane via
\begin{equation}
\label{eq:P}
P=\{(t,x)\in\R^4:\, t=0\}.
\end{equation}
We further let 
$\|z\|_{P}= 1\wedge \inf_{y \in P} \|z-y\|_\s$ and $\|y,z\|_{P}= \|y\|_{P}\wedge \|z\|_{P}$,
as well as
\begin{equation}
\label{eq:kp}
\K_P=\{(y,z)\in (\K\setminus P)^2:\, y\neq z,\, \|y-z\|_\s\leq \|y,z\|_{P}\}.
\end{equation}
Consider now for $\gamma>0$ a function $f:\R^4\setminus P\to\CT_{<\gamma}$. We define for any compact set $\K\subset\R^4$,
\begin{equation}
\label{eq:smallscale}
\$ f\$_{\gamma,\eta;\K;N}=\sup_{\substack{z\in\K\setminus P\cap\R\times \T_N^3\\ \|z\|_P < 2^{-N}}}\sup_{\beta <\gamma}\frac{\|f(z)\|_\beta}{\|z\|_{P}^{(\eta-\beta)\wedge 0}}
+ \sup_{\substack{y,z\in\K_P\cap\R\times \T_N^3,\\ \|y-z\|_\s<2^{-N}}}
\sup_{\beta <\gamma}\frac{\|f(z)-\Gamma^{N}_{zy} f(y)\|_\beta}{\|y-z\|_\s^{\gamma-\beta}\|y,z\|_{P}^{\eta-\gamma}}.
\end{equation}
We then have the following definition.
\begin{definition}
\label{def:weightedDgamma}
Fix a regularity structure $\CT$ and a discrete model $(\Pi^{N},\Gamma^{N})$ and let $\eta\in\R$. We define the space 
$\CD^{\gamma,\eta}_N$ as the space of all functions $f:\R^d\setminus P\to\CT_{<\gamma}$ such that $\$ f\$_{\gamma,\eta;\K}^{N}<\infty$, where $\$ f\$_{\gamma,\eta;\K}^{N}$ is given via 
\begin{equation}
\label{eq:weightedDgamma}
\sup_{\substack{z\in\K\setminus P\\ \|z\|_P\geq 2^{-N}}}\sup_{\beta <\gamma}\frac{\| f(z)\|_{\beta}}{\| z\|_{P}^{(\eta-\beta)\wedge 0}} + \sup_{\substack{(y,z)\in\K_P,\\ 2^{-N}\leq \|y-z\|_\s\leq 1}}\sup_{\beta <\gamma}
\frac{\|f(z)-\Gamma^{N}_{zy}f(y)\|_{\beta}}{\|y-z\|_\s^{\gamma-\beta}\|y,z\|_{P}^{\eta-\gamma}}
+\$f\$_{\gamma,\eta;\K;N}.
\end{equation}
\end{definition}
In our particular case, we formulate the abstract fixed point problem in $\CD_N^{\gamma,\eta}$ for $\gamma>3/2+\kappa$ and $\eta\in[0,1)$.
However, before we are able to formulate the fixed point problem, we introduce the \emph{reconstruction operator}, a key concept in the theory of regularity structures. The reconstruction operator $\CR^N:\CD_N^{\gamma,\eta}\to\CX_N$ is simply defined by
\begin{equation}
\label{eq:rec}
(\CR^N f)(z)= (\Pi_z^Nf(z))(z),\quad z=(t,x)\in\R\times\T_N^3\;.
\end{equation}
Let $\alpha=-3/2-\kappa$, which coincides with the minimal degree of our regularity structure. We claim that one has the bounds:
\begin{equs}
\|\CR^{N} f - \Pi^{N}_z f(z)\|_{\gamma;\K_N;z;N}&\lesssim
\|\Pi^N\|_{\gamma;\bar\K_N}^{(N)} \$ f\$_{\gamma,\eta;\K_N}\;,\\
\|\CR^N f\|_{\alpha;\K_N;z;N}&\lesssim \$ f\$_{\gamma,\eta;\K_N;N}\;,
\end{equs}
locally uniformly in the parameters above.
In the first bound we further assume that $z\in\K_N$, and that $\K_N$ has a distance of order one from the hyperplane $P$ (this is just a reformulation of Assumption 3.1 in~\cite{ErhardHairerRegularity}.
The validity of these bounds is a consequence of the reconstruction theorem in~\cite{Regularity}. To see this, denote by $x$ the spatial coordinate of $z$ above, which we assume to be fixed for the moment. Then, only imposing the second inequality in~\eqref{eq:Pi} we are in the framework of~\cite{Regularity} (in this case time is the only coordinate and there is no ``space''), so that for each $x$ there is a reconstruction operator $\CR_x^N$ satisfying these bounds. Moreover, it was shown in~\cite{Regularity} and~\cite[Rem.~13.29]{FrizHairer} that~\eqref{eq:rec} holds at each point $z=(t,x)$ such that $(\Pi_{z_0}^N \tau)(\cdot)$ is continuous at $z$ for all $\tau\in\CT$. (This property is independent of the fixed base point $z_0$.) In our particular case, each finite time interval only contains a finite number of discontinuities, so that we may impose~\eqref{eq:rec} at every location without violating the two bounds. The local uniformity is a consequence of the fact that in~\eqref{eq:Pi} we imposed this as well.

It was then shown in~\cite[Sec.~3]{ErhardHairerRegularity}, that these local estimates carry over to global estimates of the same type.
The fixed point problem now takes the form
\begin{equation}
\label{eq:fixedpoint}
U^N= \CG^N\one_{t>0}U^N\Xi + G^Nu_0^N.
\end{equation}
Here, $\CG^N$ represents the convolution operator with the discrete heat kernel in the space $\CD_N^{\gamma,\eta}$ and it satisfies the relation
\begin{equation}
\CR^N\CG^N=G^N.
\end{equation}
Indeed, this relation is a consequence of~\cite[Thm~4.9 and Rem.~4.14]{ErhardHairerRegularity}.
The term $G^N u_0$ is the spatial convolution of $G^N$ with $u_0^N$ and is imbedded into $\CD_N^{\gamma,\eta}$ via
\begin{equation}
(G^Nu_0^N)(t,x)= \sum_{|k|_\s<\gamma}\frac{X^k}{k!}\CQ_k((G^N u_0^N)(t,x),
\end{equation}
where $\CQ_k((G^N u_0^N)(t,x))$ is given by
\begin{equation}
\begin{cases}
2^{-3N}\sum_{n<N}\sum_{y\in\T_N^3}\int D_1^k[ K_n+ R^N]((t,x),(s,y))\, u_0^N(y)\, ds, &\mbox{if $|k|_\s >0$},\\
2^{-3N}\sum_{y\in\T_N^3}\int [K^N+ R^N]((t,x),(s,y))\, u_0^N(y)\, ds, &\mbox{otherwise.}
\end{cases}
\end{equation}
Fix $\gamma\in(3/2+\kappa,2)$ and $\eta\in(0,1)$. In the same way as in~\cite[Thm~3.9]{WongZakai} one can show that for every admissible model and for every initial condition $u_0^N\in\CC_N^{\eta}$ the fixed point problem~\eqref{eq:fixedpoint} has a unique solution $U^N\in \CD_N^{\gamma,\eta}([0,T],\T_N^3)$, and this solution depends continuously on the underlying model. Moreover, if the model is canonical, then $\CR^N U^N$ coincides with the solution to~\eqref{eq:PAM}.
Thus, to show convergence of the sequence of equations under consideration, it would be enough to prove convergence of the sequence of canonical models. Unfortunately, this is not the way to go, since this sequence does not converge. To go around it, we will consider continuous transformations $\hat{M}^N$ on the space of admissible models $\mathscr{M}$ such that the sequence of \emph{renormalised models}
$(\hat{\Pi}^N,\hat{\Gamma}^N)= \hat{M}^N(\Pi^N,\Gamma^N)$ do converge.
The construction of the underlying \emph{renormalisation group} is mostly algebraic and has been worked out in~\cite[Sec.~8]{Regularity} for special cases and in~\cite{BHZalg} in the general case. Since it is not of 
great relevance for the analysis performed in this article we only provide a glimpse of its construction 
in the next section.
All constructions outlined above can also be done in the continuous setting and the same kind of results are true, see~\cite{Regularity}.
Hence, once we obtain the convergence of the renormalised models, the function $u=\CR U$
with $\CR$ being the reconstruction map associated with the limiting model is the
limiting solution stated in Theorem~\ref{thm}. The H\"{o}lder regularity of the solution $u$
is given by the minimum of the regularity of the initial condition $u_0$ and the degree of
$\<IXi>$ (the element of lowest degree describing $U$ that differs from the Taylor monomials),
which is $\f12-\kappa$.

\subsection{Renormalisation}
\label{S3}

We now give an explicit description of the renormalisation maps $\hat{M}_N$ alluded to in the previous section.
These maps are parametrised by linear maps $M_N:\CT\to\CT$ belonging to the \emph{renormalisation group} 
$\mathfrak{R}$, 
which was introduced in~\cite[Sec.~8.3]{Regularity}. A large subgroup $\CG_- \subset \mathfrak{R}$ 
is described in \cite[Sec.~6.4]{BHZalg} (see in particular Theorem~6.29 there), parametrised by
$\Vec \cF_-$, where $\cF_- \subset \cF$ denotes the symbols of strictly negative degree. 
We will restrict ourselves to the three-dimensional subgroup of $\CG_-$ isomorphic to
$\Vec \{\<Xi2>, \<Xi3>, \<Xi4>\}$ endowed with addition as its group operation.

Given a generic element $M = c\<Xi2>+ c^{(1)}\<Xi3> +c^{(2)} \<Xi4>$, its action $M\tau$ on 
an element $\tau \in \cF$ is obtained by iterating over all ways of contracting 
(possibly multiple) occurrences of 
$\<Xi2>$, $\<Xi3>$ and $\<Xi4>$ as subtrees of $\tau$ and replacing them by the corresponding 
constant (with a minus sign by convention). For example, we have
\begin{equ}
M \<Xi4> = \<Xi4> - c \bigl(\<IXi2> + \<Xi22>\bigr) - c^{(1)} \<IXi> - c^{(2)} \1\;.
\end{equ}
(The reason why we don't have an additional term $-c^{(1)}\<XiI>$ for example
is that this vanishes since $\CI(\1) = 0$.)

This group acts on the space of admissible models in a natural way, see \cite[Thm~6.29]{BHZalg}:
given a canonical model $(\Pi^N,\Gamma^N)$ and
\begin{equ}
M_N = c_N\<Xi2>+ c_N^{(1)}\<Xi3> +c_N^{(2)} \<Xi4> \in \CG_-\;,
\end{equ}
we write $(\hat{\Pi}^N,\hat{\Gamma}^N)= M_N(\Pi^N,\Gamma^N)$ 
for the unique admissible model such that, for $\tau \in \cF_-$, one has
\begin{equs}
\hat{\Pi}_z^N \tau &=
\Pi_z^N M_N\tau\;,\qquad \tau\neq \<Xi4>,\\
\hat{\Pi}_z^N \<Xi4> &= \Pi_z^N M_N\<Xi4>  + c_N\sum_{|k|_\s=2}\frac{(\cdot -z)^k}{k!}f_z^N(\CJ_k(\<IXi>))\,\Pi_z^N\<Xi>\;,
\end{equs}
where
\begin{equation}
f_z^N(\CJ_k(\<IXi>))= - \CQ_k(T_{|\<IXi>|+2}^N\Pi_z^N\<IXi>)(z)\;.
\end{equation}
From now on we solely work with the subgroup of $\mathfrak{R}$ just described.

One can then show in very much the same way as in~\cite[Sec.~4.4]{WongZakai}, that if $\hat{\CR}^N$ is the reconstruction operator associated to $(\hat{\Pi}^N,\hat{\Gamma}^N)$ and $U^N$ solves the fixed point problem~\eqref{eq:fixedpoint}, then $u^N= \hat{\CR}^N U^N$ solves the renormalised equation
\begin{equation}
\partial_t u^N(t,x)= \Delta_N u^N(t,x) -(2^{3N/2}\bar\xi_{2^{2N}t}(2^Nx)- c_N - c_N^{(1)}-c_N^{(2)})u^N(t,x).
\end{equation}
In particular, the renormalisation constant $C_N$ in \eqref{eq:PAMN} should be chosen as the sum of the 
three constants defining the element $M_N \in \CG_-$ guaranteeing that the renormalised model converges.

\subsection{Joint cumulants and Wick products}
\label{S3.1}

In this section we fix an index set $\cA$ and a collection of random variables $\{X_a\}_{a\in \cA}$. For $B\subseteq \cA$ we write 
\begin{equation}
\label{eq:X_B}
X_B=\{X_a:\, a\in B\}\quad\mbox{and}\quad X^B=\prod_{a\in B} X_a.
\end{equation}
Furthermore, we write $\powerset(B)$ for the powerset of $B$ and $\cP(B)\subset \powerset(\powerset(B))$ 
for the set of all partitions of $B$, i.e., $\pi\in\cP(B)$ if and only if $\pi$ is a collection 
of disjoint non-empty subsets of $B$ whose union equals $B$.
(Note that by convention $\cP(\emptyset) = \{\emptyset\}$.)
\begin{definition}
\label{def:cum}
Fix a finite subset $B\subseteq \cA$. We define the cumulant $\E_c(X_B)$ inductively over $|B|$ by 
$\E_c(X_B) = \E(X_a)$ if $B=\{a\}$ and
\begin{equation}
\label{eq:cum}
\begin{aligned}
&\E(X^{B}) = \sum_{\pi\in\cP(B)}\prod_{\bar{B}\in\pi}\E_c(X_{\bar{B}}),\quad \mbox{if } |B|\geq 2.
\end{aligned}
\end{equation}
\end{definition}
\begin{remark}
\label{rem:inversioncumulants}
Inverting the above formula we obtain \cite[Prop.~3.2.1]{PeccatiTaqqu}
\begin{equation}
\label{eq:inversioncumulants}
\E_c(X_B) = \sum_{\pi\in\cP(B)}(|\pi|-1)!(-1)^{|\pi|-1}\prod_{\bar{B}\in\pi}\E(X^{\bar{B}})\;.
\end{equation}
\end{remark}
With Definition~\ref{def:cum} at hand we may now introduce the concept of Wick products.
\begin{definition}
\label{def:Wick}
Let $A\subseteq\cA$. The Wick product $\Wick{X_A}$ is defined via $\Wick{X_{\emptyset}}=1$ and then recursively extended postulating the relation
\begin{equation}
\label{eq:Wick}
X^{A}= \sum_{B\subseteq A}\Wick{X_B}\sum_{\pi\in\cP(A\setminus B)}\prod_{\bar{B}\in \pi}
\E_c(X_{\bar{B}}).
\end{equation}
\end{definition}
\begin{remark}
\label{rem:Wick}
We note the following three facts that follow immediately from Definition~\ref{def:Wick}. First of all, as soon as $A\neq\emptyset$ we have that $\E(\Wick{X_A}) = 0$. Moreover, if we take the expectation on both sides of~\eqref{eq:Wick}, then all terms with $B\neq \emptyset$ vanish and we obtain the identity~\eqref{eq:cum}. Finally, we could have also written  $X^A$ above as $\sum_{B\subseteq A}\Wick{X_B}\E(X^{A\setminus B})$, but~\eqref{eq:Wick} is what will conceptually be used throughout this article. In Lemma~\ref{lem:diagram} below we show an extension of this formula which is called ``diagram formula'' in the literature. A proof can be found in~\cite{Surgailis}.
\end{remark}
Before we state the next lemma we need to introduce a special class of partitions.
\begin{definition}
\label{def:partition}
Let $M$ and $P$ be two sets and fix a subset $D\subseteq M\times P$. We say that $\pi\in\cP_M(D)$ if $\pi\in \cP(D)$ and for every $B\in\pi$, there exist $(i,k), (i',k') \in B$ such that $k\neq k'$. 
\end{definition}
In plain words, $\pi\in\cP_M(D)$ if it is a partition of $D$ and for each $B\in\pi$, the elements in $B$ can not have all the same "$P$-index". Note that in particular one must have $|B|\geq 2$.
One also has $\cP_M(D) = \emptyset$ if $|P| = 1$ but $\cP_M(D) = \{\emptyset\}$ if $D = \emptyset$.
\begin{lemma}
\label{lem:diagram}
Fix $m,p\in\N$. Define $M=\{i:\, 1\leq i\leq m\}$ and $P=\{k:\, 1\leq k\leq p\}$. Let $\{X_{(i,k)}\}_{i\in M, k\in P}$ be a collection of random variables all having bounded moments of all orders. One has the ``diagram formula"
\begin{equation}
\label{eq:diagram}
\E\bigg(\prod_{k=1}^{p}\Wick{\prod_{i=1}^{m}X_{(i,k)}}\bigg)
= \sum_{\pi\in\cP_M(M\times P)}\prod_{B\in\pi}\E_c(X_{B}).
\end{equation}

\end{lemma}
\subsection{Values of the constants}
\label{S3.3}
In this section we provide exact formulas for the renormalisation constants $c_N, c_N^{(1)}$ and $c_N^{(2)}$, and we determine their asymptotics. The final renormalisation constant $C_N$ is then simply given by $C_N=c_N + c_N^{(1)}+c_N^{(2)}$.
Using the notation $\xi^N(z)=\xi_t^N(x)$ for a time-space point $z=(t,x)\in\R_+\times\T_N^3$, where $\xi^N_t(x)= \xi_{2^{2N}t}(2^Nx)$, we have the formulas
\begin{equation}
\begin{aligned}\label{eq:renconst1}
c_N&= \int K^N(-z)\,\E_c(\xi^N(z),\xi^N(0))\, dz,\\
c_N^{(1)}&= \int K^N(-z_1)K^N(z_1-z_2)\, \E_c(\xi^N(z_1),\xi^N(z_2),\xi^N(0))\, dz,
\end{aligned}
\end{equation}
and $c_N^{(2)}= c_N^{(2,1)}+c_N^{(2,2)}+c_N^{(2,3)}$, where
\begin{equs}
c_N^{(2,1)}&= \int K^N(-z_1)K^N(z_1-z_2)K^N(z_2-z_3)\, \E_c(\xi^N(z_0),\ldots,\xi^N(z_3)),\\
c_N^{(2,2)}&=\int K^N(-z_1) \E_c(\xi^N(0),\xi^N(z_2))K^N(z_1-z_2)\E_c(\xi^N(z_1),\xi^N(z_3))K^N(z_2-z_3),\\
c_N^{(2,3)}&=\int K^N(-z_1) \Ren Q^N(z_1-z_2)K^N(z_2-z_3)\E_c(\xi^N(0),\xi^N(z_3)),\label{eq:c1c22c23}
\end{equs}
where in the first line $z_0=0$, and $\Ren Q^N$ denotes the renormalised kernel defined by $\Ren Q^N(z)= K^N(z) \E_c(\xi^N(0),\xi^N(z))- c_N 2^{3N}\delta_0(z)$. The integration above is implicitly assumed to be over all time-space points $z_i$ that appear, and each integral over space is a Riemann sum
$2^{-3N}\sum_{x\in \T_N^3}(\ldots)$.
Given two sequences $(a_N)_{N\in\N},$ and $(b_N)_{N\in\N}$ we use the notation $a_N\approx b_N$ to denote that they differ by at most a constant that is bounded in $N$.
The main result of this section then reads as follows.
\begin{proposition}
\label{prop:renconstant}
There exist constants $\alpha_N, \beta_N$ and $\gamma_N$ that are bounded in $N$ such that
\begin{equation}
c_N= \alpha_N 2^N,\quad c_N^{(1)}=\beta_N 2^{N/2},\quad \mbox{and}\quad c_N^{(2)}=\gamma_N N\;.
\end{equation}	
Moreover, $c_N^{(2,1)}$ and $c_N^{(2,3)}$ are themselves bounded in $N$.
\end{proposition}
The proof of Proposition~\ref{prop:renconstant} is provided in Appendix~\ref{C}.

\section{Proof of Theorem~\ref{thm:cumu}}
\label{S4}
The goal of this section is to derive estimates on the joint cumulants $\E_c(\{\xi_{t_a}^N(x_a)\,:\, a\in \cA\})$ for any finite collection of space-time points $(t_a,x_a)_{a\in \cA}$.
In order to achieve this goal we proceed in three steps.
In Section~\ref{S4.1} we will show that the joint transition probability of $n$ labelled exclusion particles admits
a Gaussian bound. In Section~\ref{S4.2} we use martingale techniques to get a decomposition of these transition probabilities. Finally, in Section~\ref{S4.3} we use the results from Sections~\ref{S4.1}--\ref{S4.2} to obtain bounds that turn out to be sufficient for our purposes. 
Recall that we also use the notation $\St$ in place of $\Z_N^d$.
Many of the arguments developed in this section actually turn out to not depend on the dimension. Therefore, until Section~\ref{sec:coalescence} we will make no restrictions on $d$.

\subsection{A priori bounds for the labelled exclusion process}
\label{S4.1}
Let $\cA$ be a finite index set and consider particles labelled by elements of $\cA$ that start from $|\cA|$ 
distinct sites and that evolve through stirring, i.e., according to the graphical representation 
defined in Section~\ref{S1}.
The state space of this Markov process is $\St^{\cA}$, defined by
\begin{equ}
\St^{\cA}= \{\fx = (\fx_a)_{a\in\cA}:\,\fx_a\in\St\text{ for all }a\in\cA,\, \fx_a\neq \fx_b\, \text{ for all } a\neq b\}
\end{equ}
and its generator is given by
\begin{equation}
\label{eq:labelledexclusion}
(L_\cA f)(\fx) = \sum_{\{x,y\}}[f(\sigma^{x,y}\fx)-f(\fx)]\;,
\end{equation}
where the sum runs over unoriented bonds $\{x,y\}$ between any pair of neighbouring sites $x,y\in\St$.
Here, for $\fx\in\St^\cA$ the configuration $\sigma^{x,y}\fx$ is given by
\begin{equation}
\label{eq:sigmaswapped}
(\sigma^{x,y}\fx)_a = 
\left\{\begin{array}{ll}
y, &\mbox{if }\fx_a= x,\\
x, &\mbox{if }\fx_a=y,\\
\fx_a, &\mbox{otherwise}.
\end{array}\right.
\end{equation}

\begin{remark}\label{rem:couplingX}
We denote the labelled exclusion process by $X^{\fx}$ and the exclusion particle started at position $x\in\St$ by $X^x$. Whenever we consider different instances of the process $X^{\fx}$ for possibly different starting points $\fx$, they are coupled through the same 
instance of the stirring process defined at the beginning of Section~\ref{S1}.
In this sense, the graphical construction defines not just a Markov process on $\St^\cA$, but a stochastic flow / random dynamical system.
\end{remark}

Let $\fx, \fy\in\St^\cA$ and denote by $p_t^{\fx}(\fy)$ the probability that the labelled exclusion process starting at $\fx$ at time $0$ is at $\fy$ at time $t$.
We further let $\Phi$ be the Legendre transform of $u\mapsto u^2\cosh u$, i.e., 
$\Phi(u) = \sup_{w\in\R}\{uw-w^2\cosh w\}$ and we note that $\Phi$ asymptotically behaves like $u^2$ for $u$ small and $u\log u$ for $u$ large.
We denote by $\|x-y\|_{\St}=\min_{z\in 2^N\Z^d}\|x-y-z\|_1$ the $\ell^1$-distance on the torus. 
\begin{lemma}
\label{lem:transitionprob}
Fix $\fx, \fy \in\St^\cA$ for some index set $\cA$. There are finite positive constants $C_1$ and $C_2$ 
depending on $d$ and $|\cA|$ such that
\begin{equation}
\label{eq:transitionprob}
p_t^{\fx}(\fy) \leq C_1\prod_{i \in \cA}\bar{p}_t^{x_i}(y_i),
\end{equation}
where for all $t>C_1$,
\begin{equation}
\label{eq:transitionproblarget}
\bar{p}_t^{x_i}(y_i) = 
\frac{1}{(1+t)^{d/2}}\sum_{z_i\in 2^N\Z^d}\exp\Big\{-\frac{C_2t}{2|\cA|(\log t)^2}\Phi\Big(\frac{\|x_i-y_i-z_i\|\log t}{C_2^2t}\Big)\Big\},
\end{equation}
whereas for $t\leq C_1$,
\begin{equation}
\label{eq:transitionprobsmallt}
\bar{p}_t^{x_i}(y_i)=
\sum_{z_i\in 2^N\Z^d}e^{-\|x_i-y_i-z_i\|_1}\;.
\end{equation}
\end{lemma}
\begin{remark}
Typically the number of particles that are in the exclusion process is proportional to $N$, so that one might get worried that the above estimates are not sharp enough for our purposes. However, in order to prove Theorem~\ref{thm} we only need to estimate cumulants of a fixed (but possibly large) order $n$, which turn out to depend only on exclusion processes that start with $n$ particles. 
\end{remark}

\begin{remark}
Actually the bounds \eqref{eq:transitionproblarget} and \eqref{eq:transitionprobsmallt} also hold
with the sum replaced by its largest term. 
\end{remark}
\begin{proof}
Fix $\fx,\fy\in\St^\cA$ and denote the transition probabilities for the exclusion process indexed by $\cA$ 
by $p_{t,\Z^d}^{\fx}(\fy)$. It follows from Landim~\cite[Thm~3.1]{Landim} that there are constants $C_1, C_2$ as in the formulation of the statement such that for all $t>C_1$,
\begin{equation}
\label{eq:landim}
p_{t,\Z^d}^{\fx}(\fy)\leq \frac{C_1}{(1+t)^{nd/2}}
\exp\Big\{-\frac{C_2 t}{2(\log t)^2}\Phi\Big(\frac{\|\fx-\fy\|\log t}{C_2^2 t}\Big)\Big\},
\end{equation}
where $\|\fx-\fy\|^2= \sum_{i\in \cA}\|x_i-y_i\|^2$. Consequently, $\|\fx-\fy\|\leq \sum_{i\in\cA} \|x_i-y_i\|$, and the monotonicity of $\Phi$ shows that
\begin{equation}
\Phi\Big(\frac{\|\fx-\fy\|\log t}{C_2^2 t}\Big)\geq \frac1n\sum_{i\in\cA}\Phi\Big(\frac{\|x_i-y_i\|\log t}{C_2^2 t}\Big).
\end{equation}
This property together with the representation
\begin{equation}
\label{eq:torusvsZd}
p_{t}^{\fx}(\fy) = \sum_{\bar{\fy}\equiv \fy \Mod{2^{N}}}p_{t,\Z^d}^{\fx}(\fy)
\end{equation}
yields~\eqref{eq:transitionproblarget}.

Let now, $t\leq C_1$ and assume that $n$ labelled exclusion particles are initially located at $\fx$. 
Note that the least number of jumps that need to occur for a single particle starting in $x_i$ and ending up at $y_i$ is $\|x_i-y_i\|_\St$. However, the total number of Poisson clocks that need to ring is not the sum over all the $\|x_i-y_i\|_\St$ since two particles may use the same Poisson clock to swap places. (But no
more than two since our state space forces particles to be located at distinct sites.)
Therefore, we see that the total number of such Poissonian events is at least $\frac12\sum_{i\in\cA}\|x_i-y_i\|_\St$.
To get a more precise handle on the probability of this event we consider the following alternative construction of the labelled exclusion process.
We denote the field of labelled exclusion particles by $(X^{x_a})_{a\in \cA}$. Each particle $X^{x_i}$ carries $2d$ exponential clocks all  having rate one. We describe these exponential clocks by independent Poisson processes $P_{i, j}$, with $i \in \cA$ and $j=1,\ldots, 2d$. We also write $e_j$ for the $j$-th unit vector if $1\leq j\leq d$ and $e_j= -e_{2d+1-j}$ if $d+1\leq j\leq 2d$. If the process $P_{i,j}$ jumps at time $s$, then
we set $X_s^{x_i}= X_{s^{-}}^{x_i}+e_j$ unless the site $X_{s^{-}}^{x_i}+e_j$ is already occupied by some other 
particle (say particle $\ell$) in which case the two particles $i$ and $\ell$ either swap positions or stay in place with 
respective probabilities ${1\over 2}$.

Setting $P_i = \sum_j P_{i,j}$, it then follows from Chebychev's inequality and the independence of the
$P_{i,j}$ that  
\begin{equation}
p_{t}^{\fx}(\fy) \le \P\Big(\sum_{i\in\cA}P_i(t)\geq \frac12\sum_{i\in\cA}\|x_i-y_i\|_\St\Big)
\leq \E\big(e^{N_1(t)}\big)^n\exp\Big\{-\frac12\sum_{i\in\cA}\|x_i-y_i\|_\St\Big\}\;,
\end{equation}
which yields the desired bound \eqref{eq:transitionprobsmallt}.
\end{proof}
The next lemma gives bounds on scaled versions of the kernels appearing in \eqref{eq:transitionproblarget}--\eqref{eq:transitionprobsmallt},
namely we set  $\bar{p}_t^{N,x}(y)=\bar{p}_{2^{2N}t}^{2^Nx}(2^Ny)$ for  $x,y\in\T_N^d$.
In the lemma that follows it is implicitely assumed that $x$ and $y$ are such that $\|x-y\|_1=\|x-y\|_{\St}$. We moreover use $\int_{\T_N^d} (\ldots)\,dx$ as a shorthand for
$2^{-dN}\sum_{x\in \T_N^d}(\ldots)$.
\begin{lemma}
\label{lem:scaling}
Let $t\in[0,1]$, $x,y\in\T_N^d$ and $\zeta\in[0,d]$. Define a scaled version of $\bar p^N$ via
\begin{equation}
\label{eq:scaling}
\bar{p}_{t,\zeta}^{N,x}(y)= 2^{\zeta N}\bar{p}_{t}^{N,x}(y).
\end{equation}
Then the quantity
\begin{equation}
\label{eq:orderpbar}
(\sqrt{t}+\|x-y\| + 2^{-N})^{\zeta}\bar{p}_{t,\zeta}^{N,x}(y)
\end{equation}
is bounded uniformly over $N \ge 0$, $t \in [0,1]$ and $x,y \in \T_N^d$.
Moreover, the bounds
\begin{equation}
\label{eq:upperboundsum}
\int_{\T_N^d}\bar{p}_{t,\zeta}^{N,x}(y)\,dx\lesssim 2^{(\zeta -d)N}\quad\mbox{and}\quad
\int_{\T_N^d}\bar{p}_{t,\zeta}^{N,x}(y)\,dy\lesssim 2^{(\zeta-d)N}.
\end{equation}
hold uniformly over $N \ge 0$, $t \in [0,1]$ and  $x,y \in \T_N^d$.
\end{lemma}
\begin{proof}
Combining the bound
$2^{(\zeta -d)N}(\sqrt{t}+\|x-y\|+2^{-N})^{-d}\leq (\sqrt{t}+\|x-y\|+2^{-N})^{-\zeta}$ with the way $\bar{p}_{\cdot,\zeta}^{N}$ is scaled, we see that it suffices to prove the statements for $\zeta=d$. 
The first bound is equivalent to the bound
\begin{equ}[e:wanted]
|\bar{p}_{t}^{x}(y)| \lesssim (\sqrt{t}+\|x-y\| + 1)^{-d}\;.
\end{equ}
It is obvious for $t\leq C_1$, so we consider the case $t>C_1$. 
For $\|x-y\| \le t/\log t$ (say), $\bar{p}_{t}^{x}(y)$ is bounded by the standard heat kernel, since $\Phi(u)\sim u^2$ for small $u$, for which
\eqref{e:wanted} holds, so it remains to consider the case $1 \le t/\log t \le \|x-y\|$.
In this regime the bound \eqref{eq:transitionproblarget} holds and the argument of $\Phi$ is greater than some
constant, so we can make use of the bound $\Phi(u) \gtrsim u$, yielding the existence of $c>0$ such that
\begin{equs}
|\bar{p}_{t}^{x}(y)|
&\lesssim (1+ \sqrt t)^{-d} \exp(-c\|x-y\|/\log t) \label{e:firstBoundpt}\\
&\lesssim (1+ \sqrt t + \|x-y\|)^{-d}  \|x-y\|^d \exp(-\|x-y\|^{1/2})\;.
\end{equs}
The required bound follows since $u \mapsto u^d e^{-\sqrt u}$ is bounded.

We now turn to the proof of~\eqref{eq:upperboundsum}.
For $t \le C_12^{-2N}$ it trivially follows from \eqref{eq:transitionprobsmallt}.
For $t \ge C_12^{-2N}$, it suffices to note that both the standard heat kernel and
the right hand side of \eqref{e:firstBoundpt} are summable in $x$, uniformly over $t$ and 
the underlying grid size. 
\end{proof}

Lemma~\ref{lem:scaling} motivates the following definition.
\begin{definition}\label{def:orderofkernel}
Fix $\zeta\in\R$. A kernel $p$ defined on $\R\times\T_N^d$  is said to be of order $\zeta$ if
\begin{equation}
\label{eq:orderofkernel}
\sup_{(t,x)\in\R\times \T_N^d}
(\sqrt{|t|}+\|x\|+2^{-N})^{\zeta}p_{t}(x) < \infty\;.
\end{equation}
If $p$ is instead defined on $\R^n\times \T_N^d$, then it is said to be of order $\zeta$ if
\begin{equ}
\sup_{(t,x)\in\R^n\times \T_N^d}
\Big(\Big(\sum_{i=1}^{n}|t_i|\Big)^{\tfrac12}+\|x\|+2^{-N}\Big)^{\zeta}p_{t}(x) < \infty\;.	
	\end{equ} 
If in one of the two cases above the kernel $p$ is defined on a subset of $\R\times\T_N^d$ or $\R^n\times\T_N^d$ respectively, then it is of order $\zeta$ if the supremum in the two equations above is taken over that subset.
\end{definition}
\begin{remark}
\label{rem:heatkernel}
Denote by $p^N$ the transition kernel of simple random walk on $\T_N^d$ jumping at rate $2d\times 2^{2N}$. It was shown in~\cite{MatetskiDiscrete} that $p_{\cdot,d}^{N}= 2^{dN}p^{N}$ is of order $d$. As in the proof of Lemma~\ref{lem:scaling} one may deduce from that, that for all $\zeta\in[0,d]$ the kernel $p_{\cdot,\zeta}^{N}=2^{\zeta N}p^N$ is of order $\zeta$. The bounds in~\eqref{eq:upperboundsum} are a consequence of the fact that $p^N$ is a transition probability.
With a slight abuse of notation we also denote by $p^N$ the transition probability for the labelled exclusion process jumping at rate $2d\times 2^{2N}$. This will not cause any ambiguity since the law of the (labelled) exclusion process is equal to that of a simple random walk when there is only one particle in the system.
\end{remark}
\begin{lemma}
\label{lem:convolutioninspace}
Let $p_1$, $p_2$ be two kernels of orders $\zeta_1$, $\zeta_2$ respectively
 and satisfying the bounds~\eqref{eq:upperboundsum}. Define for $x,z\in\T_N^d$ and $t = (t_1,t_2) \in \R_+^2$ the kernel
\begin{equation}
\label{eq:nfold}
p_{t}^{\otimes 2}(z)
= \int_{\T_N^d} p_{1,t_i}(x-y)p_{2,t_2}(y-z)\,dy\;,
\end{equation}
Then, provided that $\zeta \eqdef \zeta_1 + \zeta_2 - d>0$ the kernel $p^{\otimes 2}$ is of order $\zeta$ and, for all $t \in \R_+^2$ and $x,z\in\T_N^d$,
\begin{equation}
\label{eq:nfoldsummability}
\int_{\T_N^d}p_{t}^{\otimes 2}(x)\,dx \lesssim 2^{(\zeta - d)N}\;,
\end{equation}
with proportionality constants independent of $t$ and $N$.
\end{lemma}
\begin{proof}
We note that~\eqref{eq:nfoldsummability} is a consequence of~\eqref{eq:upperboundsum}.
To establish that $p^{\otimes 2}$ is of order $\zeta$, take $t_1,t_2\geq 0$ and $x,z\in\T_N^d$ such that 
$\|(t_1+t_2,x-z)\|\geq 2^{-N}$ and write the right-hand side of~\eqref{eq:nfold} as $\I+\II$, where,
\begin{equation}
\label{eq:IandII}
\begin{aligned}
\I&= 2^{-dN}\sum_{\substack{y\in\T_N^d:\\ \|x-y\|> \|y-z\|}}p_{1,t_1}(x-y)p_{2,t_2}(y-z)\;,\\
\II&= 2^{-dN}\sum_{\substack{y\in\T_N^d:\\ \|x-y\|\leq  \|y-z\|}}p_{1,t_1}(x-y)p_{2,t_2}(y-z)\;.
\end{aligned}
\end{equation}
To estimate $\I$ and $\II$ we may assume that $t_1>t_2$, so that $\sqrt{t_1+t_2}\leq \sqrt{2}\sqrt{t_1}$.
Note that if $\|x-y\|>\|y-z\|$, then $\|x-z\|\leq 2\|x-y\|$.
Since, $p_1$ is of order $\zeta_1$, we may conclude that under the above assumption on $y$, the chain of inequalities $p_{1,t_1}(x-y)\lesssim (\sqrt{t_1}+ \|x-y\|)^{-\zeta_1} \lesssim (\sqrt{t_1+t_2}+\|x-z\|)^{-\zeta_1}$ holds. Consequently,
\begin{equation}
\label{eq:estI}
\I\lesssim (\sqrt{t_1+t_2}+\|x-z\|+2^{-N})^{-\zeta_1}2^{-dN}\sum_{\substack{y\in\T_N^d:\\
\|x-y\|>\|y-z\|}}p_{2,t_2}(y-z)\;,
\end{equation}
and the bound on $\I$ follows from an application of~\eqref{eq:upperboundsum} to the rightmost term in~\eqref{eq:estI}. To bound $\II$, we distinguish between two cases.
\begin{enumerate}
\item $\sqrt{t_1+t_2}\geq \|x-z\|$. In this case we have
$p_{1,t_1}(x-y)\lesssim t_1^{-\zeta_1/2}\lesssim (\sqrt{t_1+t_2}+ \|x-z\|)^{-\zeta_1}$ and we may proceed as for $\I$.
\item $\sqrt{t_1+t_2} < \|x-z\|$. Since,
$\|x-z\|\leq 2\|y-z\|$ we may estimate in this case
$p_{2,t_2}(y-z)\lesssim (\sqrt{t_1+t_2}+ \|x-z\| )^{-\zeta_2}$. Thus, with~\eqref{eq:upperboundsum} at hand we may conclude the required bound for $\II$.
\end{enumerate}
The estimates in the case $\|(t_1+t_2,x-z)\|_{\mathfrak{s}}\leq 2^{-N}$ follow
from the boundedness of $p_{i}$ by $2^{\zeta_iN}$ and~\eqref{eq:upperboundsum}, which concludes the proof.
\end{proof}
We have the following immediate corollary of Lemma~\ref{lem:convolutioninspace}.
\begin{corollary}
\label{cor:nfold} Fix $n\in\N$, $\zeta_1,\ldots, \zeta_n\in[0,d]$ and for each $i\in\{1,\ldots, n\}$ let $p_i$ be either given by $p_{\cdot,\zeta_i}^N$, as in Remark~\ref{rem:heatkernel}, or by $\bar{p}_{\cdot,\zeta_i}^N$. Then, $p^{\otimes n}$ defined from one of these two kernels as in Lemma~\ref{lem:convolutioninspace} is of order $(\sum_{i=1}^{n} \zeta_i) -d(n-1)$.
\end{corollary}

\subsection{Decomposition of transition probabilities for the labelled exclusion process}
\label{S4.2}
In Section~\ref{S4.1} we derived estimates on the transition probabilities for the labelled exclusion process. In this section we use these estimates to show that even more holds, namely we show that for any $\fx,\fy\in\St^\cA$ one has an identity of the form ``$p_t^{\fx}(\fy) = \prod_{i \in \cA}p_t(x_i,y_i)$ + expectation of sum of product of martingales'', 
where $p_t(x,y)$ denotes the discrete heat kernel, see Proposition~\ref{prop:decompositionofproba} below.
To specify these products, we start by collecting and proving some facts about martingales.

\subsubsection{Facts about martingales}

We first recall that a semimartingale is a process $X$ of the form $X=X_0+M+A$, where $X_0$ is finite-valued and $\mathcal{F}_0$-measurable (for some filtration $\cF_t$, $t\geq 0$), $M$ is a local martingale and $A$ is of finite variation. The local martingale part may be further decomposed. Indeed, by Jacod and Shiryaev \cite[Theorem 1.4.18]{JacodShiryaev}, any local martingale $M$ admits a unique decomposition $M=M_0+M^c+M^d$, where $M^c$ is a continuous local martingale. 
\begin{definition}
\label{def:opvar}
Let $A$ be a finite index set and consider c\`adl\`ag semimartingales $X^i, i\in A$. The optional variation is the limit in probability of
\begin{equation}
\label{eq:opvar}
[X_A]_t \stackrel{\text{def}}{=} \lim_{\mathrm{max}(t_{i+1}-t_i)\downarrow 0}
\sum_{i} \prod_{j\in A}(X^j_{t_{i+1}}-X^{j}_{t_{i}}),
\end{equation}
where $0= t_0, t_1,\ldots$ is a partition of $[0,t]$. If $A=\emptyset$, then we define
$[X_A]=0.$ 
\end{definition}
For any c\`adl\`ag process $X$, we write $\Delta_r X_r = X_r-X_{r^{-}}$ for the size of the jump of $X$ at time $r$.
The following result taken from \cite[Prop.~1]{Mykland} clarifies the structure of the optional variation.

\begin{proposition}
\label{prop:repopvar}
Let $A$ be a finite index set. The optional variations are well defined for any collection of c\`adl\`ag semimartingales
$X^i, i\in A$, in the sense that the limit in \eqref{eq:opvar} is independent  of the sequence of partitions and has a c\`adl\`ag modification. The optional variations are semimartingales and they possess the representation,
\begin{equation}
\label{eq:repopvar}
[X_A]_t= 
\left\{
\begin{array}{cl}
X^{a}_{t}- X^{a}_{0},&\quad \text{if $A=\{a\}$,}\\
\langle X^{a,c},X^{b,c}\rangle_{t} + \sum_{r\leq t}\Delta_r X^{a}_{r}\Delta_r X^{b}_{r},&\quad \text{if $A=\{a,b\}$,}\\
\sum_{r\leq t} \prod_{i\in A}\Delta_r X^{i}_{r},&\quad \text{otherwise.}
\end{array}\right.
\end{equation}
Here, $X^{a_1,c}$ and $X^{a_2,c}$ denote the continuous martingale part of $X^{a_1}$ and $X^{a_2}$ respectively
and $\langle\cdot,\cdot\rangle$ denotes the predictable quadratic variation.\qed
\end{proposition}
It is well known that the quadratic variation of two semimartingales $X^1, X^2$ is also characterised via
\begin{equation}
\label{eq:quadvar}
 X^{1}_tX^{2}_t  = X^{1}_0X^{2}_0 +\int_{0}^{t}X^{1}_{r^-}\, dX^{2}_r
+\int_{0}^{t}X^{2}_{r^-}\, dX^{1}_r + [X^1,X^2]_t.
\end{equation}
This identity can be generalised to the case of arbitrary finite products.
Recall the notation $X^{A}$ from~\eqref{eq:X_B}.
\begin{lemma}
\label{lem:genprod}
Let $A$ be a finite index set and let $X^{i},\, i\in A$, be c\`adl\`ag semimartingales.
Then,
\begin{equation}
\label{eq:genprod}
\begin{aligned}
X^A_t = X^A_0 + \sum_{B \subsetneq A} \int_0^t X^B_{r^-} d[X_{A\setminus B}]_r\;,
\end{aligned}
\end{equation}
where $B=\emptyset$ is allowed in the above sum.
\end{lemma}
\begin{proof}
We write $A=\{1,\ldots,n\}$ and we proceed by induction over $n$. The case $n=1$ is obvious and the case $n=2$ is precisely \eqref{eq:quadvar}. Assume now that we have shown \eqref{eq:genprod} for some $n\geq 2$ and
write $Z^n_t = \prod_{i=1}^{n}X_t^{i}$.
Since c\`adl\`ag semimartingales are stable under products, $Z^n$ is again a c\`adl\`ag semimartingale
and we may apply \eqref{eq:quadvar} to obtain 
\begin{equation}
\label{eq:induction}
Z^{n+1}_t = Z^n_t X_t^{n+1}
=Z^{n+1}_0+ \int_0^{t}Z^n_{r^-}\, dX_r^{n+1} + \int_{0}^{t}X_{r^{-}}^{n+1}\,
dZ^n_r + [Z^n,X^{n+1}]_t.
\end{equation}
The induction hypothesis then yields
\begin{equ}
\int_{0}^{t}X_{r^{-}}^{n+1}\,
dZ^n_r = \sum_{B \subsetneq A} \int_{0}^{t}X_{r^{-}}^{n+1}\,X^B_{r^-}\,
d[X_{A\setminus B}]_r\;,
\end{equ}
as well as
\begin{equ}{}
[Z^n,X^{n+1}]_t = \sum_{B \subsetneq A} \Bigl[\int_0^\cdot X^B_{r^-} d[X_{A\setminus B}]_r,X^{n+1}\Bigr]_t\;,
\end{equ}
where we used the fact that the optional covariation of a constant process with any other semimartingale
always vanishes.
Combining this with the fact that for any collection of semimartingales $Z$ and $\{X^a\}_a$
and any finite index set $A$, one has
\begin{equ}
\Bigl[\int_0^\cdot Z_{r^-} d[X_{A}]_r,X^{a}\Bigr]_t =
\int_0^t Z_{r^-} d[X_{A\cup\{a\}}]_r\;,
\end{equ}
(combine for example \cite[Eq.~4.54 and Thm~4.47]{JacodShiryaev})
concludes the proof.
\end{proof}

\subsubsection{Application to the labelled exclusion process}
\label{S:binarytrees}

Let $\cA$ be a finite index set, fix $\fx,\fy,\fz\in\St^\cA$, define $\one_{\fy}(\cdot)= \one\{\fy=\cdot\}$ and 
consider a labelled exclusion process $X^\fx$ started at $\fx$.
We denote the particle started at $\fx_i$ by $X^{x_i}$. Since the generator of this process is given by $L_\cA$,
there is a martingale $M^{\fx}(\fy)$ such that
\begin{equation}
\label{eq:MP}
\one_\fy(X_t^{\fx})
= \one_\fy(X_0^{\fx})
+\int_0^t (L_\cA \one_{\fy})(X_s^\fx)\, ds + M_t^{\fx}(\fy)\;.
\end{equation}
It follows from Duhamel's formula, combined with the fact that $\one_\fy(\fz)$
is of the form $f(\fy-\fz)$ and the identity 
$\big(L_\cA f(\fy - \cdot)\big)(\fz) = \bigl(L_\cA f\bigr)(\fy - \fz)$, that 
\begin{equation}
\label{eq:Duhamel}
\one_\fy(X_t^{\fx})
=p_t^{\fx}(\fy) + \int_0^{t} \sum_{\fz\in\St^\cA} p_{t-s}^{\fz}(\fy)\, dM_s^{\fx}(\fz)\;.
\end{equation}
For $t>0$ we define martingales $M_{t,\cdot}^{\fx,\fy}$ by
\begin{equation}
\label{eq:Mfx}
M_{t,r}^{\fx,\fy} = \int_0^r \sum_{\fz\in\St^\cA} p_{t-s}^{\fz}(\fy)\, dM_s^{\fx}(\fz)\;.
\end{equation}
Since for all $r\geq t$, one has the identity $M_{t,r}^{\fx,\fy}= M_{t,t}^{\fx,\fy}$ (we comply with the 
usual convention that $p_t = 0$ for $t \le 0$), we call $t$ the terminal time of 
$M_{t,\cdot}^{\fx,\fy}$.

\begin{proposition}
	\label{prop:decompositionofproba}
	Fix $\fx,\fy\in\St^\cA$. Then,
	\begin{equation}
	\label{eq:decompositionofproba}
	p_t^{\fx}(\fy)=
	\sum_{J\subseteq\cA}\Big(\prod_{j\in J}p_{t}^{x_j}(y_j)\Bigr)
	\E\Big[\prod_{i\notin J} M_{t,t}^{x_i,y_i}\Big],
	\end{equation}
	where $J=\emptyset$ is allowed in the sum.
\end{proposition}

\begin{proof}
	With~\eqref{eq:Mfx} at hand we can write
	\begin{equation}
	\label{eq:decomposition}
	\begin{aligned}
	p_t^{\fx}(\fy)
	= \E\Big[\prod_{i \in \cA}\one\{X_t^{x_i}=y_i\}\Big]
	= \E\Big[\prod_{i\in\cA}[p_t^{x_i}(y_i)+M_{t,t}^{x_i,y_i}]\Big],
	\end{aligned}
	\end{equation}
	whence the claim follows by expanding the product.
\end{proof}
Proposition~\ref{prop:decompositionofproba} shows how $p_t^{\fx}(\fy)$ and $\prod_{i\in\cA}p_t^{x_i}(y_i)$ are related. Yet, this relation is only meaningful when we have a handle on the expectations of the various products of martingales appearing in~\eqref{eq:decompositionofproba}. The rest of this section is devoted to the analysis of these martingale products.

For this, it will be convenient to introduce the following notations. Given $\fx \in \St^A$ and $\bar A \subset A$,
we write $\fx_{\bar A} \in \St^{\bar A}$ for the restriction of $\fx$ to $\bar A$.
Given two disjoint index sets $A, B$ and $t \in \R_+^2$, we also write
\begin{equ}
P_t^{A,B}(\fx,\fy) = p_{t_1}^{\fx_A}(\fy_A)p_{t_2}^{\fx_B}(\fy_B)\;,\qquad \fx,\fy \in \St^{A\sqcup B}\;.
\end{equ}
Given furthermore two elements $i\in A$ and $j \in B$, we write
\begin{equ}[e:grad]
\nabla_{i,j} P_t^{A,B}(\fx,\fy) = 
\big(p_{t_1}^{\fx_A^{ij}}(\fy_A) - p_{t_1}^{\fx_A}(\fy_A)\big)\big(p_{t_2}^{\fx_B^{ij}}(\fy_B)-p_{t_2}^{\fx_B}(\fy_B)\big)\;,\quad 
\|\fx_i - \fx_j\| = 1\;,
\end{equ}
and $\nabla_{i,j} P_t^{A,B}(\fx,\fy) = 0$ otherwise. Here, we write
$\fx^{ij}$ for the state with indices $i$ and $j$ swapped. Finally, we write $\mathbb{E}_N$ for the set of edges in $\St$.

\begin{lemma}
	\label{lem:prodoflabelledexclusion}
	Fix a finite index set $\cA = \cA^1 \sqcup \cA^2$, as well as $\fx,\fy\in\St^{\cA}$ and write
	$\fx^i = \fx_{\cA^i}$. Then,
	\begin{equation}
	\label{eq:optvaroflabelledexclusion}
d{[M_{t_1,\cdot}^{\fx^1,\fy^1}, M_{t_2,\cdot}^{\fx^2,\fy^2}]_{s}}
	= \sum_{e\in \mathbb{E}_N} \sum_{\substack{i\in\cA^1,\\ j\in\cA^2}}\one\{X_{s^-}^{\fx}(i,j) = e\} 
\nabla_{i,j}P^{\cA^1,\cA^2}_{t-s}(X_{s^-}^{\fx},\fy)\, dN_s(e)\,,
	\end{equation}
	where on the right hand side $t=(t_1,t_2)$ and $s$ is identified with the $2$-dimensional vectors with both components equal to $s$.
	Moreover, the optional variation of three or more martingales as above vanishes almost surely. Recall that $P_{\cdot}(e)$ denotes the Poisson clock of rate $1$ attached to the edge $e$ introduced in the graphical construction of the exclusion process in Section~\ref{S1}.
\end{lemma}

\begin{proof}
Since, by~\eqref{eq:MP}, the jumps of $M_s^{\fx}(\fz)$ are given by
\begin{equ}
\Delta_s M_s^{\fx}(\fz) = \Delta_s \one_\fz(X_s^\fx)\;,
\end{equ}
we have
\begin{equ}
\Delta_s M_{t,s}^{\fx,\fy} = \sum_{\fz} p_{t-s}^\fz(\fy) \Delta_s \one_\fz(X_s^\fx)
= p_{t-s}^{X_s^\fx}(\fy) - p_{t-s}^{X_{s-}^\fx}(\fy)\;.
\end{equ}
Since $\fx \in \St^{\cA}$ so that the $\fx(i)$ are all distinct, it follows that the only 
times at which $M_{t_1,\cdot}^{\fx^1,\fy^1}$ and
$M_{t_2,\cdot}^{\fx^2,\fy^2}$ can jump simultaneously are those Poisson events in $P(e)$,
where $e$ is an edge connecting $X_{s^-}^{\fx}(i)$ to $X_{s^-}^{\fx}(j)$ for some 
$i \in \cA^1$ and $j \in \cA^2$,
which leads to \eqref{eq:optvaroflabelledexclusion}. For the same reason, three such martingales 
corresponding to completely distinct points $\fx$ cannot jump simultaneously. 
\end{proof}

With Lemmas~\ref{lem:genprod} and~\ref{lem:prodoflabelledexclusion} at hand we are now able to derive an expression for $\E[\prod_{i\in\cA}M_{t,t}^{x_i,y_i}]$ which can be most conveniently be represented in 
terms of labelled rooted binary trees. 
Before introducing these we start with an illustrative example.
For $\fx, \fy \in\St^d$, we are interested in $\E[\prod_{i\le 3}M_{t,t}^{\fx_i,\fy_i}]$. To that end we note that, with hopefully obvious notation, by Lemmas~\ref{lem:genprod} and~\ref{lem:prodoflabelledexclusion}, and the fact that we are dealing with martingales starting at zero,
\begin{equation}
\label{eq:nequals3}
\E\Big[\prod_{i=1}^{3}M_{t,t}^{\fx_i,\fy_i}\Big]
=\sum_{i=1}^{3}\E\Big[\int_0^t M_{t,s^{-}}^{\fx_i,\fy_i}\, d[M_{t,\cdot}^{\fx_j,\fy_j}, j\neq i]_s\Big].
\end{equation}
We focus solely on the case $i=3$. The other cases can be treated in the same manner.
First of all, by Lemma~\ref{lem:prodoflabelledexclusion},
\begin{equation}
\label{eq:nequals3optvar}
d[M_{t,\cdot}^{\fx_1,\fy_1}, M_{t,\cdot}^{\fx_2,\fy_2}]_s = 
\sum_{e\in\mathbb{E}_N} \one\{X_{s^{-}}^{\fx_{1,2}}=e\}\nabla_{1,2}P^{\{1\},\{2\}}_{t-s}\big(X_{s^{-}}^{x_1,x_2}(y_1,y_2)\big)\, dN_s(e)\;.
\end{equation}
For every choice of $e\in\mathbb{E}_N$ the map $s\mapsto N_s(e)-s$ defines a zero mean martingale. Therefore, the summand with $i=3$ in~\eqref{eq:nequals3} equals
\begin{equation}
\label{eq:firsttimeoptvar}
\sum_{\fz\in\St^2}\E\Big[\int_0^t M_{t,s}^{\fx_3,\fy_3} \one\{X_{s^{-}}^{\fx_{\{1,2\}}}=\fz\}\nabla_{1,2}P^{\{1\},\{2\}}_{t-s}\big(\fz,\fy_{\{1,2\}}\big)\,ds\Big].
\end{equation}
Note that $s\mapsto  M_{t,s}^{x_3,y_3} \one\{X_{s^{-}}^{\fx_{\{1,2\}}}=\fz\}$ is not a martingale anymore, so that it seems that martingale techniques break down at this point. To overcome this problem we write in a similar way as in~\eqref{eq:Duhamel},
\begin{equation}
\label{eq:firstmerge}
\begin{aligned}
 \one\{X_{s^{-}}^{\fx_{\{1,2\}}}=\fz\}
=p_s^{\fx_{\{1,2\}}}(\fz)+ M_{s,s}^{\fx_{\{1,2\}}, \fz}.
\end{aligned}
\end{equation}
Inserting the transition kernel into~\eqref{eq:firsttimeoptvar} yields zero since $s\mapsto M_{t,s}^{\fx_3,\fy_3}$ is a mean zero martingale, so it remains to plug the martingale term into~\eqref{eq:firsttimeoptvar}. 
Using Lemmas~\ref{lem:genprod} and~\ref{lem:prodoflabelledexclusion} we conclude that \eqref{eq:firsttimeoptvar}
equals
\begin{equation}
\label{eq:finishexample}
\sum_{\fu \in \St^d\atop \fz \in \St^2} \sum_{j=1}^2 \int_0^t\int_0^s p_r^\fx(\fu) \nabla_{3,j} P^{\{3\},\{1,2\}}_{t-r,s-r}(\fu, \fy_{\{3\}} \sqcup \fz)\,\nabla_{1,2}  P^{\{1\},\{2\}}_{t-s}\big(\fz,\fy_{\{1,2\}}\big)\,ds\,dr
\end{equation}

\begin{remark}
	\label{rem:example}
	Below we introduce the class of labelled rooted binary trees alluded to above. We mention however already now that in this context~\eqref{eq:finishexample} can be represented as
	\begin{equ}[e:picTrees]
	\tikzsetnextfilename{treebin1}
	\begin{tikzpicture}[scale=0.35,baseline=0.5cm]
	\node at (-2,0) [dot] (left) {};
	\node at (0,0)  [dot] (middle) {};
	\node at (3,0)  [dot] (right) {};
	\node at (-1,1) [dot] (top1) {};
	\node at (1,3)  [dot] (top2) {};
	\node[below] at (-2,0) {\tiny 1};
	\node[below] at (0,0) {\tiny 2};
	\node[below] at (3,0) {\tiny 3};
	\node[left] at  (top1) {\tiny 1};
	
	\draw (left) to (top1);
	\draw (middle) to (top1); 
	\draw (top1) to (top2);
	\draw (right) to (top2);
	\end{tikzpicture}\;+
	\tikzsetnextfilename{treebin2}
	\begin{tikzpicture}[scale=0.35,baseline=0.5cm]
	\node at (-2,0) [dot] (left) {};
	\node at (0,0)  [dot] (middle) {};
	\node at (3,0)  [dot] (right) {};
	\node at (-1,1) [dot] (top1) {};
	\node at (1,3)  [dot] (top2) {};
	\node[below] at (-2,0) {\tiny 1};
	\node[below] at (0,0) {\tiny 2};
	\node[below] at (3,0) {\tiny 3};
	\node[left] at (top1)  {\tiny 2};
	
	\draw (left) to (top1);
	\draw (middle) to (top1); 
	\draw (top1) to (top2);
	\draw (right) to (top2);
	\end{tikzpicture}\;.
	\end{equ}
Reading~\eqref{eq:finishexample} backwards in time, the interpretation of the tree on the left hand side is that first particle $1$ and $2$ swap places.  Thereafter particle $1$ swaps places with particle $3$ and finally all three particles move to $x_1, x_2$ and $x_3$ respectively.
\end{remark}

We now introduce a class of labelled binary trees that will serve our purposes.
We fix a (non-empty) index set $\cA$ equipped with a partition $\hat \cA \neq \{\cA\}$ 
and we consider 
a forest $T$ of subsets of $\cA$ having the elements of $\hat \cA$ as its minimal sets. 
In other words, $T \subset \cP(\cA)$ is a collection of subsets (which we interpret as `nodes' of
the forest $T$) such that
\begin{claim}
\item One has $\hat \cA \subset T$ and, for any $u,v \in T$, $u \cap v \in \{\emptyset,u,v\}$.
\item For $u \in \hat\cA$ and $v \in T$, $v \subset u$ implies that $v = u$.
\item For any $u \in \hat\cA$, there exists $v \in T$ with $u \subset v$ and $v \neq u$.
\item For any $v \in T \setminus \hat \cA$, there exist $v_1,v_2 \in T$ such that $v = v_1 \cup v_2$.
\end{claim}
The first property implies that the Hasse diagram for the poset $(T, \subseteq)$ 
is a forest, the second property implies that its leaves are given by the elements of $\hat \cA$, 
the third that its roots are \textit{not} given by elements of $\hat \cA$, and the 
last property implies that it is binary. We write $T_\rho \subset T$ for its set of maximal elements and we set
\begin{equ}
T_\downarrow = T \setminus T_\rho\;,\qquad 
T^\uparrow = T \setminus \hat \cA\;. 
\end{equ}
We then endow $T$ with a map $\fa \colon T_\downarrow \to \cA$ such that $\fa_v \in v$ for every $v \in T$
and we denote by $\Trees_\cA$ the set of all such pairs $(T,\fa)$. 
We also write $\Tree_\cA \subset \Trees_\cA$ for the pairs $(T,\fa)$
such that $T$ consists of a single tree, i.e.\ those such that $T_\rho = \{\cA\}$. 

\begin{remark}
Note that $\Trees_\cA = \emptyset$ if $\hat \cA = \{\cA\}$ while $\Trees_\emptyset = \{(\emptyset,e)\}$
with $e$ the empty function.
\end{remark}

We now fix sites
$\fx, \fy \in \St^\CA$ as well as positive numbers $\ft = (t_u)_{u\in\hat \cA}$ and some $t>0$ with $t\leq t_u$ for all $u\in\hat\cA$.
As before, we write $\fx_A$ for the restriction of $\fx$ to $A \subset \cA$.
Given $(T,\fa) \in \Trees_\cA$, we then assign to each $v \in T$
an integration variable $(s^v,\fz^v) \in \cD_t \eqdef [0,t] \times \St^v$.
For $v \in \hat \cA$, we furthermore set
\begin{equ}[e:leavevariables]
	s^v = t_{v}\;,\qquad \fz^v = \fy_{v}\;.
\end{equ}
For every $v \in T^\uparrow$, we then define a function
$P_v$ in the following way. Writing (uniquely) $v = v_1 \cup v_2$ with $v_i \in T$, we set
\begin{equ}[e:defPv]
	P_v(s,\fz) \eqdef \nabla_{\fa_{v_1},\fa_{v_2}} P^{v_1,v_2}_{s^{v_1} - s^v,s^{v_2} - s^v}(\fz^v, \fz^{v_1}\sqcup \fz^{v_2})\;. 
\end{equ}
(Note that this is well-defined since it is invariant under $v_1 \leftrightarrow v_2$.)
We then define
\begin{equ}[e:defI]
	\frak{I}_t^\cA(T,\fa,\fx,\fy,\ft)
	= \int_{\cD_t^{T^{\uparrow}}} \prod_{u \in T_\rho} p_{s_u}^{\fx_u}(\fz^u) \prod_{v \in T^\uparrow} P_v(s,\fz)\,ds\,\hat d\fz\;,
\end{equ}
where $\hat d\fz$ simply denotes counting measure (i.e.\ without normalising factor $2^{-dN}$).
In the special case $\cA = \emptyset$, it is natural to set 
$\frak{I}_t^{\emptyset}(\emptyset,e,e,e,e) = 1$. (Again, $e$ denotes the empty function.)

With these notations at hand, the expectation of products of martingales can be expressed in the
following way.
\begin{proposition}
	\label{prop:productofmartingales}
	For $\cA$, $\hat \cA$, $\fx$, $\fy$, $\ft$ and $t$ as above, one has
	\begin{equation}
	\label{eq:productofmartingales}
	\E\Big[\prod_{u\in \hat \cA}M_{t_u,t}^{\fx_u,\fy_u}\Big]
	= \sum_{(T,\fa)\in\Trees_\cA}\frak{I}_t^{\cA}(T,\fa,\fx,\fy,\ft)\;.
	\end{equation}
	Furthermore, their joint cumulant is given by 
	\begin{equation}
	\label{eq:cumulantmartingales}
	\E_c\Big\{M_{t_u,t}^{\fx_u,\fy_u}\,:\, u\in \hat \cA\Big\}
	= \sum_{(T,\fa)\in\Tree_\cA}\frak{I}_t^{\cA}(T,\fa,\fx,\fy,\ft)\;.
	\end{equation}
\end{proposition}

\begin{remark}
	We will really only ever be interested in the special case of Proposition~\ref{prop:productofmartingales}
	when $\hat \cA =\{\{a\}\,:\, a \in \cA\}$.
	The reason for giving this more general statement is that it lends itself better to a proof by induction.
\end{remark}
\begin{remark}\label{rem:replacebyprocess}
	Given random variables $X_1,\ldots, X_n$ with $n\geq 2$, their joint cumulant coincides with that of $\bar X_1,\ldots,\bar X_n$, where $\bar X_i=X_i-\E[X_i]$. As a consequence, if we have
$t_u = t$ for all $u$ in \eref{eq:cumulantmartingales}, then each instance of $M_{t,t}^{\fx_u,\fy_u}$ 
can be replaced by $\one_{\fy_u}(X_t^{\fx_u})$.
\end{remark}
\begin{proof}
We first show that \eqref{eq:cumulantmartingales} follows immediately from 
\eqref{eq:productofmartingales}. Indeed, the definition of $\frak{I}_t^{\cA}$ implies that 
if $T$ is not connected, then $\frak{I}_t^{\cA}$ factorises over its connected components
as
\begin{equ}
\frak{I}_t^{\cA}(T,\fa,\fx,\fy,\ft)
= 
\prod_{u \in T_\rho} \frak{I}_t^{u}(T_u,\fa,\fx_{u},\fy_{u},\ft_{u})\;,
\end{equ}
where $T_u$ denotes the connected component of $T$ with root $u$.
Note that one necessarily has $T_u \in \Tree_{u}$.
It follows that we can rewrite \eqref{eq:productofmartingales} as 
\begin{equ}
\E\Big[\prod_{u\in \hat \cA}M_{t_u,t}^{\fx_u,\fy_u}\Big]
= \sum_{\pi \ge \hat \cA}
\prod_{B \in \pi} \Bigl(\sum_{(T,\fa)\in\Tree_{B}}\frak{I}_t^{B}(T,\fa,\fx_B,\fy_B,\ft_B)\Bigr)\;,
\end{equ}
where, the outer sum runs over all partitions $\pi$ of $\cA$ that are coarser than $\hat \cA$.  
Comparing this to the definition of joint cumulants, \eqref{eq:cumulantmartingales} follows at once.

It is convenient to consider triples $(T,\fa,\le)$, where $\le$ is a
total order on $T^\uparrow$ such that $u \subseteq v$ implies $u \le v$. We write $\Trees_\cA^\le$
for the set of such triples and we write $\cD_t^{T,\le}$ for set of $(s,\fz) \colon T \to \cD_t$
such that $s^u \le s^v$ whenever $u \le v$. We then define $\frak{I}_t^{\cA,\le}$ as in
\eqref{e:defI} but with $\cD_t^T$ replaced by $\cD_t^{T,\le}$. Since $\cD_t^T = \bigcup \cD_t^{T,\le}$
with the union running over all total orders $\le$ compatible with $T$, it follows that we can replace 
$\Trees_\cA$ and $\frak{I}_t^{\cA}$ by $\Trees_\cA^\le$ and $\frak{I}_t^{\cA,\le}$ respectively in \eqref{eq:productofmartingales}.

In order to streamline notations, we consider $\fx$ to be fixed once and for all and we write
\begin{equ}
\pP^{\fy,\ft}_{\cA, t} \eqdef \prod_{u\in \hat \cA}M_{t_u,t}^{\fx_u,\fy_u}\;,
\end{equ}
where we consider $\cA$ as the underlying set endowed with its partition $\hat \cA$.
(We also use the convention $\pP^{\fy,\ft}_{\emptyset, t} = 1$.)
We then proceed by induction over the number of elements in the partition $\hat \cA$. 
If $|\hat \cA| = 0$, then $\cA = \emptyset$ and the statement is trivial since 
both sides of \eqref{eq:productofmartingales} equal $1$. The case $|\hat \cA| = 1$ so that $\cA \neq \emptyset$ and
$\hat \cA = \{\cA\}$ is equally trivial since the left-hand side vanishes, while $\Trees_\cA = \emptyset$
so that the right-hand side also vanishes.

%
Assume now that $|\hat \cA| \ge 2$.
Using Lemma~\ref{lem:genprod}, followed by Lemma~\ref{lem:prodoflabelledexclusion}, and finally~\eqref{eq:optvaroflabelledexclusion} combined with~\eqref{eq:Duhamel}, we then obtain
\begin{equs}
\E \pP^{\fy,\ft}_{\cA, t}
&= \E \pP^{\fy,\ft}_{\cA, 0}
+ \sum_{\hat B\subsetneq \hat \cA} \E \int_0^t \pP^{\fy,\ft}_{B, s^-}
d \bigl[\{M_{t_v,s^-}^{\fx_v,\fy_v}\}_{v\in \hat \cA \setminus \hat B}\bigr]_s \\ \label{e:expandP}
&= \sum_{\{a,b\}\subset \hat \cA} \E \int_0^t \pP^{\fy,\ft}_{\cA\setminus(a\cup b), s^-}
d \bigl[M_{t_a,\cdot}^{\fx_a,\fy_a}, M_{t_b,\cdot}^{\fx_b,\fy_b}\bigr]_s\\
&= \sum_{\{a,b\}\subset \hat \cA} \E \int_0^t 
\pP^{\fy,\ft}_{\cA\setminus (a\cup b), s^-} \sum_{i\in a\atop j \in b}
\sum_{\fz \in \St^{a\cup b}} \bigl(p_s^{\fx_{a\cup b}}(\fz) + M_{s,s}^{\fx_{a\cup b},\fz}\bigr) \\
&\qquad\times \nabla_{i,j} P^{a,b}_{t_a-s,t_b-s}(\fz,\fy_{a\cup b})\,ds\;.
\end{equs}
This suggests to define $\cA^{a,b}$ which, as a set, equals $\cA$ but 
is endowed with the coarser equivalence relation under which elements of $a$ are 
also equivalent to those of $b$, so that $a \cup b \in \hat \cA^{a,b}$.
Similarly, we define $\ft^{a,b}_s$ on $\hat \cA^{a,b}$ by setting $\ft^{a,b}_s(a\cup b) = s$
and $\ft^{a,b}_s(u) = \ft(u)$ otherwise. We also define $\fy^{a,b}_\fz$ 
by setting $\fy^{a,b}_\fz(u) = \fz(u)$ for $u \in a\cup b$ and
$\fy^{a,b}_\fz(u) = \fy(u)$ otherwise.

With these notations, it thus follows from \eqref{e:expandP} and the induction hypothesis
that
\begin{equs}
\E \pP^{\fy,\ft}_{\cA, t}
&= \sum_{\{a,b\}\subset \hat \cA}  \sum_{(T,\fa,\le)\in \Trees^\le_{\cA\setminus (a\cup b)}} 
\sum_{i\in a\atop j \in b}
\sum_{\fz \in \St^{a\cup b}}
\int_0^t \frak{I}_{s}^{\cA\setminus (a\cup b),\le}(T,\fa,\fx,\fy,\ft)\\
&\qquad\times  p_s^{\fx_{a\cup b}}(\fz) \nabla_{i,j} P^{a,b}_{t_a-s,t_b-s}(\fz,\fy_{a\cup b})\,ds\\
&+ \sum_{\{a,b\}\subset \hat \cA}  \sum_{(T,\fa,\le)\in \Trees^{\le}_{\cA^{a,b}}} 
\sum_{i\in a\atop j \in b}
\sum_{\fz \in \St^{a\cup b}}
\int_0^t \frak{I}_{s}^{\cA^{a,b},\le}(T,\fa,\fx,\fy^{a,b}_\fz,\ft^{a,b}_s)\\
&\qquad\times \nabla_{i,j} P^{a,b}_{t_a-s,t_b-s}(\fz,\fy_{a\cup b})\,ds\;.
\end{equs}
At this stage, given $a, b \in \hat \cA$ and $i \in a$, $j \in b$, as well as 
$(T,\fa,\le)\in \Trees^\le_{\cA\setminus (a\cup b)}$, we define
$(T_{a,b},\fa_{i,j},\le) \in \Trees^\le_{\cA}$ by setting $T_{a,b} = T \cup \{a,b,a\cup b\}$, 
$\fa_{i,j}(a) = i$, $\fa_{i,j}(b) = j$, and by setting $a \cup b \le u$ for all
$u \in T^\uparrow$.
Given 
$(T,\fa,\le)\in \Trees^\le_{\cA^{a,b}}$, we similarly define
$(T^{a,b},\fa_{i,j},\le) \in \tilde \Trees_{\cA}$ by setting $T^{a,b} = T \cup \{a,b\}$
and by setting $a \cup b \le u$ for all
$u \in T^\uparrow$. (Note that $a\cup b$ is a leaf of $T$ so it is not a priori
ordered with respect to any element in $T$.)

Using this notation, it follows from \eqref{e:defI} that we can rewrite the above expression as
\begin{equs}
\E \pP^{\fy,\ft}_{\cA, t}
&= \sum_{\{a,b\}\subset \hat \cA}  \sum_{(T,\fa,\le)\in \Trees^\le_{\cA\setminus (a\cup b)}} 
\sum_{i\in a\atop j \in b}
\frak{I}_{t}^{\cA,\le}(T_{a,b},\fa_{i,j},\fx,\fy,\ft)\\
&+ \sum_{\{a,b\}\subset \hat \cA}  \sum_{(T,\fa,\le)\in \Trees^\le_{\cA^{a,b}}} 
\sum_{i\in a\atop j \in b}
\frak{I}_{t}^{\cA,\le}(T^{a,b},\fa_{i,j},\fx,\fy,\ft)\;.
\end{equs}
It remains to note that
\begin{equs}
\Trees_{\cA}^\le
&= \bigcup_{\{a,b\}\subset \hat \cA}\bigcup_{i\in a\atop j \in b} \Big(\{(T^{a,b},\fa_{i,j},\le)\,:\, (T,\fa,\le)\in \Trees^\le_{\cA\setminus (a\cup b)}\} \\
&\qquad\qquad\cup \{(T_{a,b},\fa_{i,j},\le)\,:\, (T,\fa,\le)\in \Trees^\le_{\cA^{a,b}}\}\Big)\;.
\end{equs}
This is because, given $(S,\fb,\le) \in \Trees_\cA^\le$, if $v$ is the minimal element in $S^\uparrow$
for the order $\le$, one can write it uniquely as $v = a\cup b$ with $a,b \in \hat \cA$.
If $a \cup b$ is a root of $S$, we then have $S = (T^{a,b},\fa_{i,j},\le)$ for
$T = S \setminus \{a,b, a\cup b\} \in \Trees^\le_{\cA\setminus (a\cup b)}$, while otherwise
we have $S = (T_{a,b},\fa_{i,j},\le)$ for
$T = S \setminus \{a,b\} \in \Trees^\le_{\cA^{a, b}}$.
\end{proof}

\subsubsection{An upper bound on $(T,\fa)\in\Tree_\cA$}
\label{Sec:upperboundonbinary}
Fix a labelled tree $(T,\fa)\in\Tree_\cA$, and assume moreover that $\hat \cA = \{\{a\}\,:\, a \in \cA\}$, so that $\hat \cA$ can be canonically identified with $\cA$. Our goal is to integrate over the collection of ``free'' spatial points $\{\fz^v_b:\,  b\notin \{\fa_{v_1},\fa_{v_2}\}\}$. Here, ``free'' refers to the fact that the former collections of spatial points do not impose any restrictions on each other, as opposed to $\fz^{v}_{\fa_{v_1}}$ and $\fz^{v}_{\fa_{v_2}}$, which need to be at distance one, since the kernel in~\eqref{e:defPv} vanishes otherwise. Our goal is not to derive an identity but an upper bound. We start with an illustrative example which we then turn into a proof. Assume that $(T,\fa)$ is given by the following tree (the coloured edges can be ignored for now)
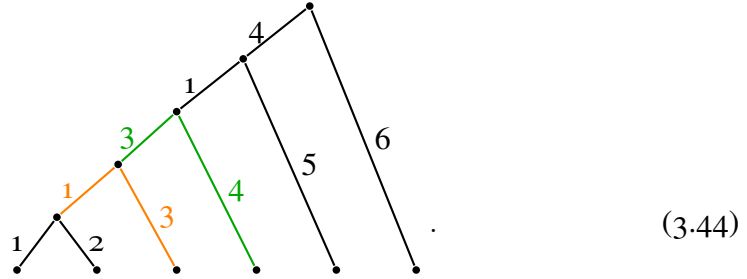
\begin{equ}[e:picTreesex1]
	\tikzsetnextfilename{treebig1}
	\begin{tikzpicture}[scale=0.35,baseline=0.5cm,thick]
	\node at (-5,0) [dot] (1) {};
	\node at (-2,0) [dot] (2) {};
	\node at (1,0) [dot] (3) {};
	\node at (4,0)  [dot] (4) {};
	\node at (7,0)  [dot] (5) {};
	\node at (10,0)  [dot] (6) {};
	\node at (-3.5,2) [dot] (top1) {};
	\node at (-1.2,4) [dot] (top2) {};
	\node at (1,6)  [dot] (top3) {};
	\node at (3.5,8)  [dot] (top4) {};
	\node at (6,10)  [dot] (top5) {};

	\draw (1)   to node[midway,left] {1} (top1);
	\draw (2) to node[midway,right] {2} (top1);
	\draw (3)[orange] to node[midway,right] {3} (top2);
	\draw (top1)[orange] to node[midway,left] {1} (top2); 
	\draw (top2)[DGreen] to node[midway,left] {3} (top3); 
	\draw (4)[DGreen] to node[midway,right] {4}  (top3); 
	\draw (top3) to node[midway,left] {1} (top4); 
	\draw (5) to node[midway,right] {5} (top4); 
	\draw (top4) to  node[midway,left] {4}(top5); 
	\draw (6) tonode[midway,right] {6} (top5); 
	\end{tikzpicture}\;.
\end{equ}
Here, for each node $v$ we moved the label assigned by $\fa_v$ to $v$ upwards towards the edge going out of $v$. In particular the leaf set  in this example is $\{1\}, \{2\},\ldots$. 
The way to think about the above tree is as follows. We see that the inner node that is most distant from the root joins the two edges with labels $1$ and $2$ respectively. The meaning is that first particle $1$ and $2$ interact by swapping places. The next interaction is between particle $1$ and $3$. This is signalled by the labels $1$ and $3$ on the orange edges. The labels $3$ and $4$ on the green edges mean that then particle $3$ and $4$ swap places. Afterwards the labels $1$ and $5$ mean that particle $1$ and $5$ swap places and the final interaction is between particle $4$ and $6$. 
In particular after its interaction with particle $1$, particle $2$ does not interact with any other particle. This suggests some kind of independence which would imply that, after its interaction with particle $1$, particle $2$ goes directly to $x_2$. If our kernels \emph{did} factorise, i.e., if we had the identity
\begin{equs}[e:factorisation]
p_t^{\fx}(\fy)= \prod_{a\in\cA} p_t^{\fx_a}(\fy_a)\;,
	\end{equs}
for any $\fx,\fy\in \St^{\cA}$ (which would mean that the exclusion particles are independent of each other), then this could be made rigorous.

Although~\eqref{e:factorisation} is unfortunately not true, Lemma~\ref{lem:transitionprob} provides a partial replacement. Indeed, applying Lemma~\ref{lem:transitionprob} to the expression on the right hand side of~\eqref{e:grad}, and estimating the gradient by the sum of the transition kernels shows that
\begin{equs}\label{eq:gradestAB}
|\nabla_{i,j} &P_t^{A,B}(\fx,\fy)|\\
&\leq |p_{t_1}^{\fx_A^{ij}}(\fy_A)+p_{t_1}^{\fx_A}(\fy_A)||p_{t_2}^{\fx_B^{ij}}(\fy_B)+p_{t_2}^{\fx_B}(\fy_B)|\\
&\lesssim \bigg(\prod_{\substack{k\in A,\\ k\neq i}}\bar{p}_{t_1}^{\fx_k}(\fy_k)\bigg)(\bar{p}_{t_1}^{\fx_j}(\fy_i)+\bar{p}_{t_1}^{\fx_i}(\fy_i))\bigg(\prod_{\substack{k\in B,\\ k\neq j}}\bar{p}_{t_2}^{\fx_k}(\fy_k)\bigg)(\bar{p}_{t_2}^{\fx_i}(\fy_j)+\bar{p}_{t_2}^{\fx_j}(\fy_j))\\
&\leq \bigg(\prod_{\substack{k\in A\cup B\\ k\neq i,j}}
\hat{p}_{t_k}^{\fx_k}(\fy_k)\bigg) \hat{p}_{t_1}^{\fx_i}(\fy_i)\hat{p}_{t_2}^{\fx_j}(\fy_j)\,,
	\end{equs}
where $t_k=t_1$ if $k\in A$ and $t_k=t_2$ otherwise. Here, we used the notation
$\hat{p}_{t}^{x}(y)= \sum_{z:\, \|z-x\|\leq 1} \bar{p}_t^{z}(y)$.
Since $\nabla_{i,j} P_t^{A,B}$ vanishes unless $\|\fx_i-\fx_j\|=1$, one actually has the stronger bound
\begin{equ}[e:gradupper]
|\nabla_{i,j} P_t^{A,B}(\fx,\fy)| \lesssim \bigg(\prod_{\substack{k\in A\cup B\\ k\neq i,j}}
\hat{p}_{t_k}^{\fx_k}(\fy_k)\bigg) \hat{p}_{t_1}^{\fx_i}(\fy_i)\hat{p}_{t_2}^{\fx_j}(\fy_j)\one\{\|\fx_i-\fx_j\|=1\}\;.
\end{equ}

Recall that in our notation each node coincides with the union of the nodes contained in the subtree having that node as a root. For example the node that joins the leaves $\{1\}$ and $\{2\}$ is identified with the set $\{1,2\}$, whereas the inner node with label $3$ is identified with the set $\{1,2,3\}$. Abbreviating $A=\{1,2\}$, $B=\{3\}$, $C=\{1,2,3\}$, $D=\{4\}$, and $E=\{1,2,3,4\}$ and going back to~\eqref{e:picTreesex1} we see that the kernel associated to the two orange edges can be estimated with the help of~\eqref{e:gradupper} by
\begin{equs}[e:grad13]
|\nabla_{1,3}&P_{s^A-s^C, s^B-s^C}^{A,B}(\fz^{C},\fz^A\sqcup \fy_3)|\\	
&\lesssim \hat{p}_{s^A-s^C}^{\fz_2^C}(\fz_2^A)\hat{p}_{s^A-s^C}^{\fz_1^C}(\fz_1^A)\hat{p}_{s^B-s^C}^{\fz_3^{C}}(y_3)\one\{\|\fz_1^C-\fz_3^C\|=1\}\,.
	\end{equs}
In the same way we see that the kernel associated to the two green edges in~\eqref{e:picTreesex1} satisfies
\begin{equs}[e:grad34]
|\nabla_{3,4}&P_{s^C-s^E, s^D-s^E}^{C,D}(\fz^{E},\fz^C\sqcup \fy_4)|\\
&\lesssim \hat{p}_{s^C-s^E}^{\fz_1^E}(\fz_1^C)\hat{p}_{s^C-s^E}^{\fz_2^E}(\fz_2^C)
\hat{p}_{s^C-s^E}^{\fz_3^E}(\fz_3^C)\hat{p}_{s^D-s^E}^{\fz_4^E}(\fy_4)
\one\{\|\fz_3^E-\fz_4^E\|=1\}\,.
	\end{equs}
Recall that \eqref{e:defI} implies that we need to multiply the kernels above. 
In particular one of the products that appears is
\begin{equ}
\hat{p}_{s^A-s^C}^{\fz_2^C}(\fz_2^A)\hat{p}_{s^C-s^E}^{\fz_2^E}(\fz_2^C)	\,.
	\end{equ}
The kernels above fall into the framework of Lemma~\ref{lem:convolutioninspace}, so that we can integrate over $\fz_2^C$ to obtain a new kernel connecting the spatial variables $\fz_2^A$ and $\fz_2^E$. Since this new kernel is also amenable to an application of Lemma~\ref{lem:convolutioninspace} this allows us to repeat the argument with the two edges that are on the shortest path that connect $\{1,2,3,4\}$ with the root. To be more precise, the remaining kernels that are encoded in the tree above are of the form $\nabla_{1,5}, \nabla_{4,6}$ and $p_{s^G}^{\fx}(\fz^G)$, where $G=\{1,2,3,4,5,6\}$ denotes the root of the tree. Applying the estimate~\eqref{e:gradupper} to these kernels and singling out the kernels corresponding to particle $2$, the following product appears
\begin{equ}
\hat{p}_{s^A-s^C}^{\fz_2^C}(\fz_2^A)\hat{p}_{s^C-s^E}^{\fz_2^E}(\fz_2^C)
\hat{p}_{s^E-s^F}^{\fz_2^F}(\fz_2^E)\hat{p}_{s^F-s^G}^{\fz_2^G}(\fz_2^F)\hat{p}_{s^G}^{x_2}(\fz_2^G)\,,
	\end{equ}
where $F=\{1,2,3,4,5\}$.
We now define
\begin{equs}
K_{s^A}&^{\{1,2\},x_2}(\fz_2^A,x_2) \\
&= \sum_{\fz_2^L:\, L \in\{C,E,F,G\}}	\hat{p}_{s^A-s^C}^{\fz_2^C}(\fz_2^A)\hat{p}_{s^C-s^E}^{\fz_2^E}(\fz_2^C)
\hat{p}_{s^E-s^F}^{\fz_2^F}(\fz_2^E)\hat{p}_{s^F-s^G}^{\fz_2^G}(\fz_2^F)\hat{p}_{s^G}^{x_2}(\fz_2^G)\,,
	\end{equs}
and $K^{\{1,2\},x_2}$ is the kernel assigned to the red edge in~\eqref{e:picTreesex2} connecting $\{1,2\}$ to $x_2$. Note at that point that the kernel $K^{\{1,2\},x_2}$ again satisfies the conditions of Lemma~\ref{lem:convolutioninspace} for any $\zeta\in [0,3]$.
The same argument cannot be applied to the kernels associated to particle $1$ in~\eqref{e:grad13} and~\eqref{e:grad34}, since $\fz_1^C$ is coupled with $\fz_3^C$.  This shows that~\eqref{e:picTreesex1} is indeed bounded by a proportionality constant times~\eqref{e:picTreesex2}.
In the same way as we did for particle $2$ we can argue that after particle $3$ interacted with particle $4$ it has no further interaction with any other particle. Hence, after that last interaction, repeating the arguments we applied to particle $2$ it goes all the way up to $x_3$, suggesting the following graphical representation
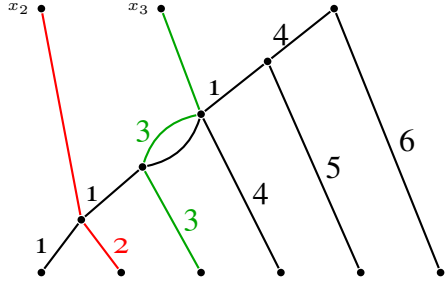
\begin{equ}
	\tikzsetnextfilename{treebigexpand1}
	\begin{tikzpicture}[scale=0.35,baseline=0.5cm,thick]
	\node at (-5,0) [dot] (1) {};
	\node at (-2,0) [dot] (2) {};
	\node at (1,0) [dot] (3) {};
	\node at (4,0)  [dot] (4) {};
	\node at (7,0)  [dot] (5) {};
	\node at (10,0)  [dot] (6) {};
	\node at (-3.5,2) [dot] (top1) {};
	\node at (-1.2,4) [dot] (top2) {};
	\node at (1,6)  [dot] (top3) {};
	\node at (3.5,8)  [dot] (top4) {};
	\node at (6,10)  [dot] (top5) {};
	\node at (-5,10) [dot] (x2)  {};
	\node at (-0.5, 10) [dot] (x3) {};
	\node[left] at (x2)  {\tiny $x_2$};
	\node[left] at (x3) {\tiny $x_3$};
	
	\draw (1) to node[midway,left] {1} (top1);
	\draw (2)[red] to node[midway,right] {2} (top1);
	\draw (top1)[red] to (x2);
	\draw (3)[DGreen] to node[midway,right] {3} (top2);
	\draw (top1) to node[midway,left] {1}  (top2); 
	\draw[bend left=30] (top2)[DGreen] to node[midway,left] {3} (top3); 
	\draw[bend right=30] (top2) to  (top3); 
	\draw (4) to node[midway,right] {4} (top3); 
	\draw (top3) to node[midway,left] {1} (top4); 
	\draw (5) to node[midway,right] {5} (top4); 
	\draw (top4) to node[midway,left] {4} (top5); 
	\draw (6) to node[midway,right] {6} (top5); 
	\draw (top3)[DGreen] to (x3);
	\end{tikzpicture}\;
\end{equ}
where the kernel connecting $\{1,2,3,4\}$ with $x_3$ is constructed in the same way as the kernel connecting $\{1,2\}$ with $x_2$, i.e., it is a multiple convolution of the $\hat{p}'$s.
Note that there is a black line parallel to the green line connecting the nodes labelled $3$ and $1$, which represents 
particle $1$. Turning our attention to that particle, we see that it first interacts with particle $2$, then directly with $3$ before it has its last interaction with $5$. Hence, in the same way as above we see that a convenient graphical notation is given by
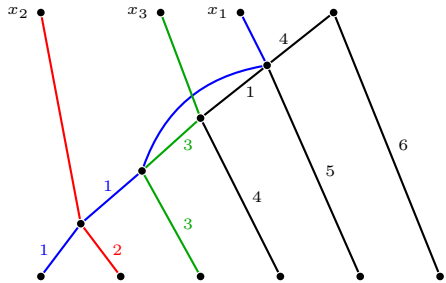
\begin{equ}
	\tikzsetnextfilename{treebigexpand2}
	\begin{tikzpicture}[scale=0.35,baseline=0.5cm,thick]
	\node at (-5,0) [dot] (1) {};
	\node at (-2,0) [dot] (2) {};
	\node at (1,0) [dot] (3) {};
	\node at (4,0)  [dot] (4) {};
	\node at (7,0)  [dot] (5) {};
	\node at (10,0)  [dot] (6) {};
	\node at (-3.5,2) [dot] (top1) {};
	\node at (-1.2,4) [dot] (top2) {};
	\node at (1,6)  [dot] (top3) {};
	\node at (3.5,8)  [dot] (top4) {};
	\node at (6,10)  [dot] (top5) {};
	\node at (-5,10) [dot] (x2)  {};
	\node at (-0.5, 10) [dot] (x3) {};
	\node at (2.5, 10) [dot]  (x1) {};
	\node[left] at (x2)  {\tiny $x_2$};
	\node[left] at (x3) {\tiny $x_3$};
	\node[left] at (x1) {\tiny $x_1$};
	
	\draw (1)[blue] to node[midway,left] {\tiny$1$} (top1);
	\draw (2)[red] to node[midway,right] {\tiny$2$} (top1);
	\draw (top1)[red] to (x2);
	\draw (3)[DGreen] to node[midway,right] {\tiny$3$}  (top2);
	\draw (top1)[blue] to node[near end,left] {\tiny$1$}  (top2); 
	\draw (top2)[DGreen] to node[midway,right] {\tiny$3$} (top3); 
	\draw (top2)[blue] to[bend left]  (top4);
	\draw (top3) to node[midway,right] {\tiny$1$} (top4);
	\draw (4) to node[midway,right] {\tiny$4$} (top3); 
	\draw (5) to node[midway,right] {\tiny$5$} (top4); 
	\draw (top4) to  node[midway,left] {\tiny$4$} (top5); 
	\draw (6) to node[midway,right] {\tiny$6$} (top5); 
	\draw (top3)[DGreen] to (x3);
	\draw (top4)[blue] to (x1);
	\end{tikzpicture}\;.
\end{equ}
Continuing in this way for the remaining three particles, we obtain the following graphical description
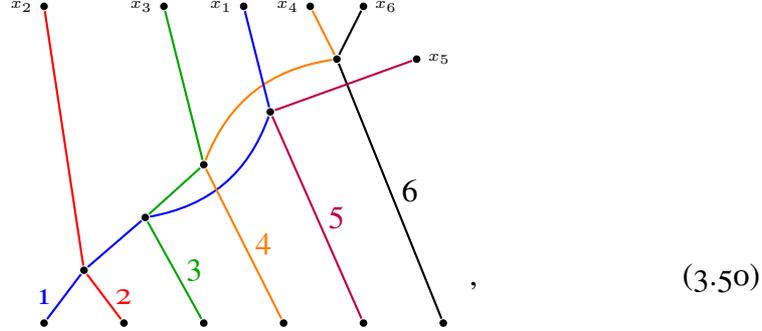
\begin{equ}[e:picTreesex4]
	\tikzsetnextfilename{treebigexpand3}
	\begin{tikzpicture}[scale=0.35,baseline=0.5cm,thick]
	\node at (-5,0) [dot] (1) {};
	\node at (-2,0) [dot] (2) {};
	\node at (1,0) [dot] (3) {};
	\node at (4,0)  [dot] (4) {};
	\node at (7,0)  [dot] (5) {};
	\node at (10,0)  [dot] (6) {};
	\node at (-3.5,2) [dot] (top1) {};
	\node at (-1.2,4) [dot] (top2) {};
	\node at (1,6)  [dot] (top3) {};
	\node at (3.5,8)  [dot] (top4) {};
	\node at (6,10)  [dot] (top5) {};
	\node at (-5,12) [dot] (x2)  {};
	\node at (-0.5, 12) [dot] (x3) {};
	\node at (2.5, 12) [dot]  (x1) {};
	\node at (5,12)    [dot]  (x4) {};
	\node at (7,12)    [dot]  (x6) {};
	\node at (9,10)    [dot]  (x5) {};
	\node[left] at (x2)  {\tiny $x_2$};
	\node[left] at (x3) {\tiny $x_3$};
	\node[left] at (x1) {\tiny $x_1$};
	\node[left] at (x4) {\tiny $x_4$};
	\node[right] at (x5) {\tiny $x_5$};
	\node[right] at (x6)  {\tiny $x_6$};
	
	\draw (1)[blue] to node[midway,left] {1}  (top1);
	\draw (2)[red] to node[midway,right] {2} (top1);
	\draw (top1)[red] to (x2);
	\draw (3)[DGreen] to node[midway,right] {3} (top2);
	\draw (top1)[blue] to  (top2); 
	\draw (top2)[DGreen] to  (top3); 
	\draw (top2)[blue] to[bend right] (top4);
	\draw (4)[orange] to node[midway,right] {4} (top3); 
	\draw (5)[purple] to  node[midway,right] {5} (top4); 
	\draw (6) to node[midway,right] {6} (top5); 
	\draw (top3)[DGreen] to (x3);
	\draw (top3)[orange] to [bend left] (top5);
	\draw (top5)[orange] to (x4);
	\draw (top4)[purple] to (x5);
	\draw (top4)[blue] to (x1);
	\draw (top5) to (x6);
	\end{tikzpicture}\;,
\end{equ} 
where for the sake of readability we omitted the labels along some of the edges.
To generalise the above computations fix 
 $(T,\fa)\in\Tree_\cA$, as well as $\fx,\fy,\ft$ and $t$ as in the previous section.
 We define $\hat T$ to be the graph with vertex set $\hat T = T\cup\cA$ and edge set $E$ given as follows. 
 For every $a \in \cA$, write $V_a \subset \hat T$ for the
 set of vertices $v$ such that either $v = \{a\}$, or $v = a$, or one can write 
 $v = v_1 \cup v_2$ with $v_i \in T$ and $\fa_{v_1} = a$. 
The partial order on $T$ induces a total order on $V_a$ if we furthermore postulate that the
element $a$ is maximal.
We then postulate that $(v,w) \in E$ if and only if there exists $a \in \cA$ such that $v,w \in V_a$,
$v < w$ (for the order on $V_a$), and there is no $u \in V_a$ with $v < u < w$.

To each $v\in \hat T$ we assign an integration variable $(s^v,z^v)\in \hat\cD_t\overset{\text{def}}{=} [0,t]\times \St$. For $v\in \hat{\cA}$ we set $s^v=t_v$ and $z^v=\fy_v$ and finally to $v\in\cA$ we assign $s^v=0$ and $z^v=\fx_v$. Fix an edge $(v,w)\in E$, and write $v_0,v_1,\ldots, v_k$ with $v_0=v$ and $v_k=w$ for the shortest path connecting $v$ to $w$ with respect to the tree structure of $T$. We then define two kernels, where the first simply serves to simplify the proof, whereas the second is the one we are really interested in. The first, $K^{(v,w)}$, is defined as follows
\begin{equ}[e:Kvw]
K^{(v,w)}_{s^v-s^w}(x,y)= \sum_{z_1,\ldots, z_{k-1}\in\St} \prod_{i=1}^{k}\hat{p}_{s^{v_{i-1}}-s^{v_i}}^{z_i}(z_{i-1})\,,
	\end{equ}
where $z_0=x$ and $z_k=y$. The other kernel, denoted by $\hat K$, is defined via
\begin{equ}
\hat K_{s}(x,y)= (\sqrt{|s|} +\|x-y\| +1)^{-d}.	
	\end{equ}
We note that by Lemma~\ref{lem:convolutioninspace} one has for any $\zeta\in[0,d]$ that the rescaled kernel $2^{\zeta N} K_{2^{2N\cdot}}^{(v,w)}(2^N\cdot)$ defined on $\R\times \T_N^d$ is of order $\zeta$. Therefore, one has in particular that $ K^{(v,w)}\lesssim \hat K$.
Setting
\begin{equ}[e:JThat]
\frak{I}_t^{\cA}(\hat T,\fa,\fx,\fy,\ft)
=\int_{\hat\cD_t^{ T^{\uparrow}}}\prod_{(v,w)\in E}\hat K_{s^v-s^w}(z^v,z^w)\, ds\, \hat dz\,,
\end{equ}
the following result is the main point of this construction. 
\begin{lemma}\label{lem:upperoncumu}
	For $\cA,\hat{\cA},\fx,\fy,\ft, t$ as above such that in particular $\hat\cA=\{\{a\}:\, a\in\cA\}$ and $(T,\fa)\in\Tree_\cA$ one has that
	\begin{equ}[e:upperoncumu]
		\frak{I}_t^\cA(T,\fa,\fx,\fy,\ft)\lesssim	\frak{I}_t^\cA(\hat T,\fa,\fx,\fy,\ft).
	\end{equ}	
\end{lemma}
\begin{proof}
Recall that the integrand on the right hand side of~\eqref{e:defI} is
\begin{equ}
p_{s^\rho}^{\fx}(\fz^\rho) \prod_{v \in T^\uparrow} P_v(s,\fz).	
	\end{equ}
Using~\eqref{e:gradupper} we see that
\begin{equs}[e:hatPdef]
|P_v(s,\fz)|&\lesssim \prod_{\substack{k\in v,\\ k\neq \fa_{v_1},\fa_{v_2}}}\hat{p}_{s^{v_k}-s^v}^{\fz^v_k}(\fz^{v_k}_k)\hat{p}_{s^{v_1}-s^v}^{\fz^v_{\fa_{v_1}}}(\fz^{v_1}_{\fa_{v_1}})\hat{p}_{s^{v_2}-s^v}^{\fz^v_{\fa_{v_2}}}(\fz^{v_2}_{\fa_{v_2}})\one\{\fz_{\fa_{v_1}}^{v}=\fz_{\fa_{v_2}}^v\}\\
&\eqdef \hat P_v(s,\fz).
\end{equs}
Here we wrote $v=v_1\cup v_2$ and we used the notation $v_k=v_1$ if $k\in v_1$ and $v_k=v_2$ otherwise. It is then immediate that 
\begin{equ}
\frak{I}_t^\cA(T,\fa,\fx,\fy,\ft) \lesssim \frak{I}_t^\cA(T,\fa,\fx,\fy,\ft, \hat P)\,,	
	\end{equ}
where
\begin{equ}[e:defIhatP]
\frak{I}_t^\cA(T,\fa,\fx,\fy,\ft, \hat P)	
= \int_{\cD_t^{T^{\uparrow}}} \prod_{w\in \rho} \hat{p}_{s^\rho}^{\fx_w}(\fz^\rho_w) \prod_{v \in T^\uparrow} \hat P_v(s,\fz)\,ds\,\hat d\fz\;.
	\end{equ}
To proceed we introduce the quantity
\begin{equ}\label{e:quantityKvw}
\frak{I}_t^{\cA}(\hat T,\fa,\fx,\fy,\ft,\cK)
=\int_{\hat\cD_t^{ T^{\uparrow}}}\prod_{(v,w)\in E}K_{s^v-s^w}^{(v,w)}(z^v,z^w)\, ds\, \hat dz\,.
	\end{equ}
We next prove that the right hand side of~\eqref{e:defIhatP} is bounded from above by some proportionality constant times the right hand side of~\eqref{e:quantityKvw}. Since $K^{(v.w)}\lesssim \hat K$, this will be enough to finish the proof.
To that end we proceed by induction over $|\cA|$, i.e., over the number of leaves. If $|\cA|=2$, then it follows from the definitions of $\cK$ and $\hat {P}$'s that we actually have 
\begin{equ}
\frak{I}_t^\cA(T,\fa,\fx,\fy,\ft, \hat P) = \frak{I}_t^\cA(\hat T,\fa,\fx,\fy,\ft, \cK).	
	\end{equ}
We continue with the induction step. To that end write $\rho=\rho_1\cup\rho_2$, where $\rho$ denotes the root of $T$. Note that $\rho_1$ and $\rho_2$ are roots of two subtrees $T_{\rho_1}$ and $T_{\rho_2}$ of $T$ and their leaf sets $\hat\cA_1$ and $\hat\cA_2$ satisfy $\hat\cA=\hat\cA_1\cup\hat\cA_2$, so that $T=T_{\rho_1}\cup T_{\rho_2}\cup \rho$. Hence, we can write 
\begin{equ}
\prod_{w\in \rho} \hat{p}_{s^\rho}^{\fx_w}(\fz^\rho_w) \prod_{v \in T^\uparrow} \hat P_v(s,\fz)
= \prod_{w\in \rho} \hat{p}_{s^\rho}^{\fx_w}(\fz^\rho_w)\prod_{v \in T_{\rho_1}^\uparrow} \hat P_v(s,\fz)\prod_{v \in T_{\rho_2}^\uparrow} \hat P_v(s,\fz) \hat P_\rho(s,\fz)\,.
	\end{equ}
Using the expression in~\eqref{e:hatPdef} in the place of $\hat P_\rho(s,\fz)$ above we see that
\begin{equs}
&\frak{I}_t^\cA(T,\fa,\fx,\fy,\ft, \hat P) \\
&= \int_{\cD_t}\, ds^\rho\, \hat d\fz^\rho  	
	\prod_{w\in \rho} \hat{p}_{s^\rho}^{\fx_w}(\fz^\rho_w))
	\one\{\fz^{\rho}_{\fa_{\rho_1}}= \fz^{\rho}_{\fa_{\rho_2}}\}
	\int_{\cD_t^{T_{\rho_1}^{\uparrow}}}\, ds\, \hat d\fz \prod_{w\in\rho_1} \hat{p}_{s^{\rho_1}-s^\rho}^{\fz_w^\rho}(\fz_w^{\rho_1})\prod_{v \in T_{\rho_1}^\uparrow} \hat P_v(s,\fz)\\
	&\quad\quad \times \int_{\cD_t^{T_{\rho_2}^{\uparrow}}}\, ds\, \hat d\fz \prod_{w\in\rho_2} \hat{p}_{s^{\rho_2}-s^\rho}^{\fz_w^\rho}(\fz_w^{\rho_2})\prod_{v \in T_{\rho_2}^\uparrow} \hat P_v(s,\fz)\,.
	\end{equs}
Fix $i\in\{1,2\}$. For each $v\in\cD_t^{T_{\rho_i}^{\uparrow}}$ we do a change of variables of the form $\hat s^v= s^v-s^{\rho}$, which reduces the domain of integration to $\cD_{t-s^\rho}^{T_{\rho_i}^{\uparrow}}$. Hence, we see that the integral over $\cD_t^{T_{\rho_i}}$ actually equals
\begin{equ}
\frak{I}_{t-s^{\rho}}^{\cA_i}(T_{\rho_i},\fa_{|T_{\rho_i}},\fz^{\rho}_{\rho_i},\fy_{\cA_i},\ft_{\cA_i}, \hat P)\,.
	\end{equ}
Here, $\fa_{|T_{\rho_i}}$ denotes the restriction of $\fa$ to $T_{\rho_i}$.
Hence, by the induction hypothesis we have the estimate
\begin{equs}[e:inductionupperP]
\frak{I}_t^\cA(T,\fa,\fx,\fy,\ft, \hat P)
\lesssim 	\int_{\cD_t}\, ds^\rho\, &\hat d\fz^\rho  	
\prod_{w\in \rho} \hat{p}_{s^\rho}^{\fx_w}(\fz^\rho_w)
\one\{\fz^{\rho}_{\fa_{\rho_1}}= \fz^{\rho}_{\fa_{\rho_2}}\}\\
&\times\prod_{i=1}^{2}\frak{I}_{t-s^\rho}^{\cA_i}(\hat T_{\rho_i},\fa_{|T_{\rho_i}},\fz^{\rho}_{\rho_i},\fy_{\cA_i},\ft_{\cA_i}, \cK_i)\,.
	\end{equs}
Here, for $i\in\{1,2\}$ we wrote $\cK_i=\{K^{(v,w)}\}_{(v,w)\in E_i}$, where $E_i$ denotes the edge set of $\hat T_{\rho_i}$.
We can write the integrand of $\frak{I}_t^{\cA_i}(\hat T_{\rho_i},\fa_{|T_{\rho_i}},\fz^{\rho}_{\rho_i},\fy_{\cA_i},\ft_{\cA_i}, \cK_i)$ as
\begin{equ}
\prod_{(v,w)\in E_i}K_{s^v-s^w}^{(v,w)}(z^v,z^w)
= \prod_{w\in\cA_i}\prod_{v:\, (v,w)\in E_i}K_{s^v-s^w}^{(v,w)}(z^v,\fz^{\rho}_w) \prod_{w\notin\cA_i}\prod_{v:\, (v,w)\in E_i}K_{s^v-s^w}^{(v,w)}(z^v,z^w)\,.
	\end{equ}
We then note that by the construction of $\hat T_{\rho_i}$ for every $w\in\cA_i$ there exists exactly one $v$ such that $(v,w)\in E_i$, and moreover $\rho=\cA_1\cup\cA_2$. Therefore, plugging this expression into the right hand side of~\eqref{e:inductionupperP}, integrating over $\{\fz^{\rho}_w\}_{w\neq \fa_{\rho_1},\fa_{\rho_2}}$, and recalling~\eqref{e:Kvw} yields the following expression
\begin{equs}
&\int_{\hat\cD_t}\, ds^{\rho} \hat d\fz^{\rho}_{\fa_{\rho_1}} \prod_{i=1}^{2}\hat{p}_{s^\rho}^{\fx_{\fa_{\rho_i}}}(\fz^{\rho}_{\fa_{\rho_i}})\one\{\fz^{\rho}_{\fa_{\rho_1}}= \fz^{\rho}_{\fa_{\rho_2}}\}
\int_{\hat\cD_{t-s^{\rho}}^{T_{\rho_1}^{\uparrow}}\cup\hat\cD_{t-s^{\rho}}^{T_{\rho_2}^{\uparrow}} }\, ds\, \hat dz \prod_{\substack{w\in\cA,\\w\neq\fa_{\rho_1},\fa_{\rho_2}}}\\
&\times\prod_{v:(v,w)\in E}K^{(v,w)}_{s^v+s^{\rho}}(z^v,\fx_w) \prod_{i=1}^{2}K_{s^v+s^{\rho}}^{(v,\fa_{\rho_i})}(z^v,\fz_{\fa_{\rho_i}}^{\rho})\prod_{w\notin\cA}\prod_{v:\, (v,w)\in E}K_{s^v-s^w}^{(v,w)}(z^v,z^w)\,.
	\end{equs}
A change of variables of the form $\hat s^v= s^v+s^{\rho}$ for any $v\in T_{\rho_i}^{\uparrow}$ for $i=1,2$ yields the claim.
	\end{proof}

\subsection{Cumulants for the exclusion process}
\label{S4.3}

Let $\cA$ be a finite index set with $|\cA| \ge 2$ and let $\fx \in \St^\cA$, 
$\ft \in \R^\cA$. Since all the quantities we will consider are continuous in $\ft$, we can (and will) assume
without loss of generality that $\ft\colon \cA \to \R$ is injective, so that it equips $\cA$ with 
the induced total order. We also introduce the relation $a \to b$ for pairs $(a,b)$ such that $t_b > t_a$ and
there exists no $c \in \cA$ with $t_b > t_c > t_a$. (Note that the relation $\to$ depends of course on
the choice of $\ft$, but this is suppressed in our notation.)

Fix now furthermore a partition $\pi$ of $\cA$ and define $Z_\pi^{\fx}$ as the set of all maps $\fz\colon \cA \to \St^\pi$
with the property that $\fz(a)_B = x_a$ if $a \in B$. With this notation at hand, we define
\begin{equ}[e:defptpi]
	p^\pi_\ft(\fx) = \sum_{\fz \in Z_\pi^{\fx}} \prod_{a \to b} p^{\fz(a)}_{t_b - t_a}(\fz(b))\;.
\end{equ}
Another way of interpreting this quantity is that
\begin{equ}[e:crucial]
	p^\pi_\ft(\fx) = \P(\Omega^\pi_{\ft,\fx})\;,
\end{equ} 
where $\Omega^\pi_{\ft,\fx}$ is the event on which, for every $B \in \pi$, 
the graphical construction of the exclusion process given by Figure~\ref{fig1} 
contains a trajectory going through the space-time points $(t_a,x_a)$ for every $a \in B$.
We have the following preliminary result 
for the moments of the exclusion process.

\begin{proposition}
	\label{prop:nthmoment}
	For any collection of space-time points as above and for $\xi$ a stationary symmetric simple
	exclusion process with fixed-time distribution that is Bernoulli with intensity $q$, one has
	the identity
	\begin{equation}
	\label{eq:nthmoment}
	\E\Big[\prod_{a \in \cA}\xi_{t_a}(x_a)\Big]
	= \sum_{\pi\in\cP(\cA)} C_\pi(q) p^\pi_\ft(\fx)\;,\qquad
	C_\pi(q) \eqdef \prod_{A \in \pi}\kappa_{|A|}(B_q)\;,
	\end{equation}
	where $B_q$ denotes a Bernoulli random variable with parameter $q$ and $\kappa_p$ denotes the $p$th cumulant.
\end{proposition}

\begin{proof}
	Write $\tilde p^\pi_\ft(\fx)$ for
	\begin{equ}
		\tilde p^\pi_\ft(\fx) = \P(\tilde \Omega^\pi_{\ft,\fx})\;,
	\end{equ} 
	where $\tilde \Omega^\pi_{\ft,\fx}\subset \Omega^\pi_{\ft,\fx}$ is the event where 
	we furthermore impose that the trajectories corresponding to distinct elements of $\pi$ do
	not intersect.
	The crucial observation then is that
	\begin{equ}
		p^\sigma_\ft(\fx) = \sum_{\pi \ge \sigma} \tilde p^\pi_\ft(\fx)\;,
	\end{equ}
	where $\pi \ge \sigma$ denotes that $\sigma$ is a refinement of $\pi$.
	We can therefore write
	\begin{equs}
		\E\Big[\prod_{a \in \cA}\xi_{t_a}(x_a)\Big]
		&= \sum_{\pi\in\CP(\cA)} \Bigl(\prod_{A\in \pi} \E B_q^{|A|}\Bigr) \tilde p^\pi_\ft(\fx) 
		= \sum_{\pi\in\CP(\cA)}\sum_{\sigma \le \pi} \Bigl(\prod_{A\in \sigma} \kappa_{|A|}(B_q)\Bigr) \tilde p^\pi_\ft(\fx)\\
		&= \sum_{\sigma\in\CP(\cA)} \Bigl(\prod_{A\in \sigma} \kappa_{|A|}(B_q)\Bigr) \sum_{\pi \ge \sigma} \tilde p^\pi_\ft(\fx) 
		= \sum_{\sigma\in\CP(\cA)} \Bigl(\prod_{A\in \sigma} \kappa_{|A|}(B_q)\Bigr) p^\sigma_\ft(\fx)\;,
	\end{equs}
	as  claimed.
\end{proof}

Our aim however is to obtain an expression for the joint \textit{cumulants} of these observables, which
are expected to exhibit better decay at large scales.
For this, we first rewrite the quantity $p^\pi_\ft(\fx)$ from \eqref{e:defptpi} in a slightly different
way. We use the notation $\cA_\to$ for the set of pairs $e = (e_-,e_+)$ such that $e_- \to e_+$.
For every $e \in \cA_\to$, we write $\pi_{e} \subset \pi$ for those elements $A \in \pi$
which straddle $e$ in the sense that there exist $a_-, a_+ \in A$ with $a_- \le e_- < e_+ \le a_+$.
For $c \in \cA$, we also write $\pi_c = \bigcup \bigl\{\pi_{e}\,:\, c \in e\bigr\}$, and we denote
by $E_\pi \subset \cA \times \pi$ the set of pairs $(a,A)$ such that $A \in \pi_a$.
We also write
$\St^\pi_\star$ for the maps $z\colon E_\pi \to \St$
and $Z_{\star,\pi}^{\fx}\subset \St^\pi_\star$ for the set of such maps with $\fz(a,A) = x_a$ if $a \in A$. 

Given a pair $e \in \cA_\to$ and $\fz \in Z_{\star,\pi}^\fx$ we also write $\fz_-(e)$ for the 
restriction of $\fz(e_-,\cdot)$ to $\pi_{e}$ and $\fz_+(e)$ for the restriction of $\fz(e_+,\cdot)$ to $\pi_{e}$.
With these notations one has the alternative expression
\begin{equ}[e:defptpialt]
	p^\pi_\ft(\fx) = \sum_{\fz \in Z_{\star,\pi}^{\fx}} \prod_{e \in \cA_\to} p^{\fz_-(e)}_{t_e}(\fz_+(e))\;,
\end{equ}
where we set $t_e = t_{e_+} - t_{e_-}$, as a consequence of the consistency of the kernels $p_t^{\fx}(\fy)$.

Graphically, \eqref{eq:nthmoment} can be depicted as follows. In the example where $\cA$ has four elements,
then we can write the corresponding fourth moment of the exclusion process as
\begin{equ}[e:pbasic]
	\E\Big[\prod_{a \in \cA}\xi_{t_a}(x_a)\Big] = 
	\tikzsetnextfilename{partition0}
	\begin{tikzpicture}[scale=0.4,baseline=0.85cm]
	\row{d}{1}{x}
	\row{d}{2}{y}
	\row{d}{3}{z}
	\row{d}{4}{w}
	\connect{1}{x}{y}
	\connect{1}{y}{z}
	\connect{1}{z}{w}
	\end{tikzpicture}
	+
	\tikzsetnextfilename{partition1}
	\begin{tikzpicture}[scale=0.4,baseline=0.85cm]
	\row{d,c}{1}{x}
	\row{c,d}{2}{y}
	\row{c,d}{3}{z}
	\row{d,c}{4}{w}
	\connect{1,2}{x}{y}
	\connect{1,2}{y}{z}
	\connect{1,2}{z}{w}
	\end{tikzpicture}
	+
	\tikzsetnextfilename{partition2}
	\begin{tikzpicture}[scale=0.4,baseline=0.85cm]
	\row{d,c}{1}{x}
	\row{d,c}{2}{y}
	\row{c,d}{3}{z}
	\row{c,d}{4}{w}
	\connect{1,2}{x}{y}
	\connect{1,2}{y}{z}
	\connect{1,2}{z}{w}
	\end{tikzpicture}
	+
	\tikzsetnextfilename{partition3}
	\begin{tikzpicture}[scale=0.4,baseline=0.85cm]
	\row{d,c,c}{1}{x}
	\row{c,d,c}{2}{y}
	\row{c,c,d}{3}{z}
	\row{d,c,c}{4}{w}
	\connect{1,2,3}{x}{y}
	\connect{1,2,3}{y}{z}
	\connect{1,2,3}{z}{w}
	\end{tikzpicture}
	+ \ldots
\end{equ}
where the four displayed graphs correspond to the terms $C_\pi(q) p^\pi_\ft(\fx)$ 
for the partitions
$\{\{1,2,3,4\}\}$, $\{\{1,4\},\{2,3\}\}$, $\{\{1,2\},\{3,4\}\}$ and $\{\{1,4\},\{3\},\{2\}\}$.
The interpretation of these pictures should be as follows. Time runs vertically and
each black dot represents
one of the variables $(t_a,x_a)$ for $a \in \cA$. Each circle appearing on the same row represents 
a variable $(t_a,z)$ with an implicit summation over all possible values of $z$.
Black dots appearing in the same column belong to the same element of the partition $\pi$ and each such
column also represents a factor $\kappa_k(B_q)$, where $k$ is the number of black dots in the column.
Vertical black line segments joining a coordinate $(s,x)$ to a coordinate $(t,y)$ that do not connect to a
horizontal red line represent factors of $p_{t-s}(y-x)$. 
Finally, each occurrence of 
\begin{equ}
	\tikzsetnextfilename{multiparticle}
	\begin{tikzpicture}[scale=0.7,baseline=0.9cm]
	\draw[rounded corners=2mm,fill=black!15,draw=none]
	(0.8,0.2) rectangle (3.2,1.2);
	\draw[rounded corners=2mm,fill=black!15,draw=none]
	(0.8,1.8) rectangle (3.2,2.8);
	\row{d,l,d}{1}{x}
	\row{d,l,d}{2}{y}
	\connect{1,2,3}{x}{y}
	\end{tikzpicture}
\end{equ}
represents an instance of $p_{t-s}^\fx(\fy)$ where $s$ and $t$ are the time coordinates of the 
two rows of coordinates and $\fy$ and $\fx$ are their spatial coordinates.

With this graphical notation, the terms in \eqref{e:pbasic} are obtained for any given partition
$\pi$ in the following way:
\begin{claim}
	\item Create a node for each element of $\cA \times \pi$ arranged in a two-dimensional array
	such that the vertical coordinate corresponds to $\cA$ and is ordered accordingly.
	\item Draw each node $(a,A)$ as a black dot if $a \in A$ and as a circle otherwise.
	\item Connect any vertically adjacent nodes
	by a black line and then connect all black lines at the same level by a red line.
\end{claim}
This graphical notation is of course very redundant since there are much simpler ways of 
graphically representing a partition $\pi$ and an order on $\cA$, but it will allow us to perform graphical
manipulations corresponding to different ways of rewriting the terms in \eqref{e:pbasic}.

Note that $p_{t-s}^\fx(\fy)$ is the probability that the labelled exclusion process $X^{\fx}$ started at $\fx$ is at position $\fy$ at time $t-s$. Thus, since a circle corresponds to summing over all possible positions we see that one has the simplification
\begin{equ}
	\tikzsetnextfilename{left_gen}
	\begin{tikzpicture}[scale=0.7,baseline=0.9cm]
	\draw[rounded corners=2mm,fill=black!15,draw=none]
	(0.8,0.2) rectangle (3.2,1.2);
	\draw[rounded corners=2mm,fill=black!15,draw=none]
	(0.8,1.8) rectangle (4.2,2.8);
	\row{d,l,d,c}{1}{x}
	\row{d,l,d,d}{2}{y}
	\connect{1,...,4}{x}{y}
	\end{tikzpicture}
	=
	\tikzsetnextfilename{right_gen}
	\begin{tikzpicture}[scale=0.7,baseline=0.9cm]
	\draw[rounded corners=2mm,fill=black!15,draw=none]
	(0.8,0.2) rectangle (3.2,1.2);
	\draw[rounded corners=2mm,fill=black!15,draw=none]
	(0.8,1.8) rectangle (4.2,2.8);
	\row{d,l,d,n}{1}{x}
	\row{d,l,d,d}{2}{y}
	\connect{1,2,3}{x}{y}
	\end{tikzpicture}
	\;.
\end{equ}
(It is crucial here that the variable that is being summed over is not 
connected to anything else
but one single red line as shown in the figure.) Applying this repeatedly,
we see that \eqref{e:pbasic} can be simplified to
\begin{equ}[e:pbasicsimple]
	\E\Big[\prod_{a \in \cA}\xi_{t_a}(x_a)\Big] = 
	\tikzsetnextfilename{irred_partition0}
	\begin{tikzpicture}[scale=0.4,baseline=0.85cm]
	\row{d}{1}{x}
	\row{d}{2}{y}
	\row{d}{3}{z}
	\row{d}{4}{w}
	\connect{1}{x}{y}
	\connect{1}{y}{z}
	\connect{1}{z}{w}
	\end{tikzpicture}
	+
	\tikzsetnextfilename{irred_partition1}
	\begin{tikzpicture}[scale=0.4,baseline=0.85cm]
	\row{d,n}{1}{x}
	\row{c,d}{2}{y}
	\row{c,d}{3}{z}
	\row{d,n}{4}{w}
	\connect{1}{x}{y}
	\connect{1,2}{y}{z}
	\connect{1}{z}{w}
	\end{tikzpicture}
	+
	\tikzsetnextfilename{irred_partition2}
	\begin{tikzpicture}[scale=0.4,baseline=0.85cm]
	\row{d,n}{1}{x}
	\row{d,n}{2}{y}
	\row{n,d}{3}{z}
	\row{n,d}{4}{w}
	\connect{1}{x}{y}
	\connect{2}{z}{w}
	\end{tikzpicture}
	+
	\tikzsetnextfilename{irred_partition3}
	\begin{tikzpicture}[scale=0.4,baseline=0.85cm]
	\row{d,n,n}{1}{x}
	\row{c,d,n}{2}{y}
	\row{c,n,d}{3}{z}
	\row{d,n,n}{4}{w}
	\connect{1}{x}{y}
	\connect{1}{y}{z}
	\connect{1}{z}{w}
	\end{tikzpicture}
	+ \ldots\;,
\end{equ}
which precisely describes the corresponding terms in \eqref{e:defptpialt}.
(Here, the isolated black dots just contribute a factor $q = \kappa_1(B_q)$, independently of
their space-time coordinates.) In terms of the above graphical construction,
it means that we discard those circles that don't have a dot both above and below them,
which is precisely how the set $E_\pi$ used in \eqref{e:defptpialt} is defined.
Choosing an element $\fz \in Z_{\star,\pi}^{\fx}$ is then the same as choosing an element
of $\St$ for every `dot' remaining in the simplified graphical representation with the constraint
that the elements associated to black dots equal the corresponding coordinate of $\fx$.

In order to obtain a formula for the joint cumulants of the exclusion process, 
one would like to have an expression of the form 
\begin{equ}
	\E\Big[\prod_{a \in \cA}\xi_{t_a}(x_a)\Big]
	= \sum_{\pi\in\cP(\cA)} A^\pi_\ft(\fx)\;,
\end{equ}
for some quantity $A^\pi$ that factors over the elements of $\pi$, with each factor
only depending on the coordinates $(t_a,x_a)$ with $a$ belonging to the corresponding element of $\pi$. 
While \eqref{eq:nthmoment} looks like this at first sight and $C_\pi(q)$ does indeed factor
over $\pi$, the quantity $p^\pi_\ft(\fx)$ unfortunately does not. We remedy this by rewriting it
in terms of the ``connected transition probabilities''
\begin{equ}\label{eq:connectedtransitionproba}
	p_{c;t}^{\fx}(\fy)
	\eqdef \E_c\Big\{\one\{X_t^{x_i}=y_i\}\,:\,i \in \cA\Big\}\;.
\end{equ}
We then define a set
$\Cumus_\cA$ of all elements $\Cumus = (\pi, K)$ where 
$\pi \in \CP(\cA)$ and $K\colon \cA_\to \to \bigsqcup_{e \in \cA_\to} \CP(\pi_e)$ with the property that
$K_{e} \in \CP(\pi_{e})$ and, given $\Cumus = (\pi, K)$, we set
\begin{equ}\label{eq:cumuproba}
	p^\Cumus_{c;\ft}(\fx) = \sum_{\fz \in Z_{\star,\pi}^{\fx}} \prod_{e \in \cA_\to}  
	\prod_{A \in K_{e}} p^{\fz_-(e)_A}_{c;t_e}(\fz_+(e)_A)\;,
\end{equ}
so that 
\begin{equ}\label{eq:prodrewrite}
	\E\Big[\prod_{a \in \cA}\xi_{t_a}(x_a)\Big]
	= \sum_{\Cumus \in \Cumus_\cA} C_\Cumus(q) p^\Cumus_{c;\ft}(\fx)\;.
\end{equ}
Here we assume as usual that $\cA$ is endowed with the order given by the $t_a$
and we set $C_\Cumus(q) = C_\pi(q)$ for $\Cumus = (\pi,K)$. Indeed, using~\eqref{eq:cum}, we can conclude that
\begin{equation}
p_{\ft}^{\pi}(\fx)=
\sum_{\fz \in Z_{\star,\pi}^{\fx}} \prod_{e \in \cA_\to}\sum_{K_e\in\cP(\pi_e)}\prod_{A\in K_e}
p_{c;t_e}^{\fz_-(e)_A}(\fz_+(e)_A),
\end{equation} 
and~\eqref{eq:prodrewrite} follows.
We can use a similar graphical notation as before for these terms, the difference 
being that, for any given $\Cumus$, we only connect two black lines at the same 
level $e \in \cA_\to$ if they belong to the same element of $K_e$. Since these connected lines
now represent a factor $p^{\fx}_{c;t}(\fy)$ instead of a factor $p^{\fx}_{t}(\fy)$ as before,
we draw them in blue instead of red. For example, \eqref{e:pbasicsimple} can be rewritten
in this new notation as
\begin{equ}[e:pbasicsimple2]
	\E\Big[\prod_{a \in \cA}\xi_{t_a}(x_a)\Big] = 
	\tikzsetnextfilename{irred_connected0}
	\begin{tikzpicture}[scale=0.4,baseline=0.85cm]
	\row{d}{1}{x}
	\row{d}{2}{y}
	\row{d}{3}{z}
	\row{d}{4}{w}
	\connect{1}{x}{y}
	\connect{1}{y}{z}
	\connect{1}{z}{w}
	\end{tikzpicture}
	+
	\tikzsetnextfilename{irred_connected1}
	\begin{tikzpicture}[scale=0.4,baseline=0.85cm]
	\row{d,n}{1}{x}
	\row{c,d}{2}{y}
	\row{c,d}{3}{z}
	\row{d,n}{4}{w}
	\connect{1}{x}{y}
	\connect{1}{y}{z}
	\connect{2}{y}{z}
	\connect{1}{z}{w}
	\end{tikzpicture}
	+
	\tikzsetnextfilename{irred_connected2}
	\begin{tikzpicture}[scale=0.4,baseline=0.85cm]
	\row{d,n}{1}{x}
	\row{c,d}{2}{y}
	\row{c,d}{3}{z}
	\row{d,n}{4}{w}
	\connect{1}{x}{y}
	\connect[db]{1,2}{y}{z}
	\connect{1}{z}{w}
	\end{tikzpicture}
	+
	\tikzsetnextfilename{irred_connected3}
	\begin{tikzpicture}[scale=0.4,baseline=0.85cm]
	\row{d,n}{1}{x}
	\row{d,n}{2}{y}
	\row{n,d}{3}{z}
	\row{n,d}{4}{w}
	\connect{1}{x}{y}
	\connect{2}{z}{w}
	\end{tikzpicture}
	+
	\tikzsetnextfilename{irred_connected4}
	\begin{tikzpicture}[scale=0.4,baseline=0.85cm]
	\row{d,n,n}{1}{x}
	\row{c,d,n}{2}{y}
	\row{c,n,d}{3}{z}
	\row{d,n,n}{4}{w}
	\connect{1}{x}{y}
	\connect{1}{y}{z}
	\connect{1}{z}{w}
	\end{tikzpicture}
	+ \ldots\;.
\end{equ}
Every
$\Cumus = (\pi, K)$ defines naturally a partition $\hat \Part_\Cumus \in \CP(\pi)$
by postulating that $\hat \Part_\Cumus$ is the finest partition such that, 
for every $e \in \cA_\to$, its projection onto $\CP(\pi_e)$ is coarser than $K_e$.
We write $\Part_\Cumus \in \CP(\cA)$ for the
coarsening of $\pi$ induced by $\hat \Part_\Cumus$.
In the graphical representation of $\Cumus$, $\Part_\Cumus$ is nothing but the 
partition of $\cA$ induced by decomposing $\Cumus$ into its connected components.

Writing $\Cumus_\cA^c = \{\Cumus \in \Cumus_\cA\,:\, \Part_\Cumus = \{\cA\}\}$ for the connected
elements, we have the following result for the $n$-point cumulants of the symmetric simple 
exclusion process.

\begin{theorem}\label{thm:cumulantofexclusion}
In the setting of Proposition~\ref{prop:nthmoment}, one has
the identity
\begin{equation}
\label{eq:nthcumulant}
\E_c\Big\{\xi_{t_a}(x_a)\,:\, a \in \cA\Big\}
= \sum_{\Cumus \in \Cumus_\cA^c} C_\Cumus(q) p^\Cumus_{c;\ft}(\fx)\;.
\end{equation}
\end{theorem}

\begin{proof}
Write $\CH = \Vec \big( \bigsqcup_{\emptyset \neq A \subset \cA} \Cumus_A\big)$
and let $\CG\colon \CH \to \R$ be the linear map given by
\begin{equ}
\CG(\Cumus) = C_\Cumus(q) p^\Cumus_{c;\ft}(\fx)\;.
\end{equ}
For $A \subset \cA$, we also write $\CH_A \subset \CH$ for the subspace spanned by $\Cumus_A$.
Note that $\CH_A \cap \CH_B = \{0\}$ as soon as $A \neq B$, so that in particular 
$\CH_\cA \subsetneq \CH$. Setting $\one = \sum\{\Cumus \in \Cumus_\cA\} \in \CH_\cA$,
 Proposition~\ref{prop:nthmoment} can then be rewritten as
\begin{equ}
\E\Big[\prod_{a \in \cA}\xi_{t_a}(x_a)\Big]
= \CG(\one)\;.
\end{equ}

We also define a ``product'' $\star$ on $\CH$ which however only allows to multiply
elements $\Cumus^A \in \Cumus_A$ and $\Cumus^B \in \Cumus_B$ for $A\cap B = \emptyset$.
To define $\star$, given $A$ and $B$ as above, we write
\begin{equ}
\Cumus_{A,B} = \bigl\{\Cumus \in \Cumus_{A\cup B}\,:\, \Part_\Cumus \le \{A,B\}\bigr\}\;,
\end{equ}
and we consider the projection $\Pi \colon \Cumus_{A,B} \to \Cumus_A \times \Cumus_B$
defined as follows. Given $\Cumus = (\pi,K) \in \Cumus_{A,B}$, we set
$\Pi \Cumus = \bigl((\pi^A,K^A),(\pi^B,K^B)\bigr)$, where $\pi^A$ and $\pi^B$ are 
the only partitions of $A$ and $B$ such that $\pi^A \cup \pi^B = \pi$ and $K^A$
is given by
\begin{equ}\label{eq:defKA}
K^A_{a,b} = \powerset(A) \cap \bigvee \Bigl\{K_{c,d}\,:\, (c,d) \in (A\cup B)_\to,\, [c,d] \subset [a,b]\Bigr\}\;,
\end{equ}
where the interval $[a,b]$ is viewed in $A \cup B$.
The definition of $K^B$ is the same with the roles of $A$ and $B$ reversed.
With this definition at hand, we set
\begin{equ}[e:defProduct]
\Cumus^A \star \Cumus^B = \sum \{\Cumus \in \Cumus_{A,B}\,:\, \Pi \Cumus = (\Cumus^A,\Cumus^B)\}\;.
\end{equ}
One can readily verify that this partial product is commutative, associative, and that 
$\Cumus^\emptyset$ contains a single element which acts as a unit for $\star$.
Taking for example $\cA = \{1,\ldots,6\}$ and $A = \{1,2,4,5\}$, $B = \{3,6\}$ one has
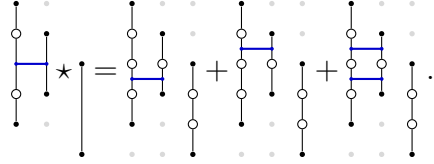
\begin{equ}
\tikzsetnextfilename{starprod1}
\begin{tikzpicture}[scale=0.4,baseline=1.4cm]
\row{d,n}{2}{x}
\row{c,d}{3}{y}
\row{c,d}{5}{z}
\row{d,n}{6}{w}
\connect{1}{x}{y}
\connect[db]{1,2}{y}{z}
\connect{1}{z}{w}
\end{tikzpicture}
\star 
\tikzsetnextfilename{starprod2}
\begin{tikzpicture}[scale=0.4,baseline=1.4cm]
\row{d}{1}{x}
\row{d}{4}{y}
\connect{1}{x}{y}
\end{tikzpicture}
=
\tikzsetnextfilename{starprod3}
\begin{tikzpicture}[scale=0.4,baseline=1.4cm]
\row{n,n,d}{1}{xx}
\row{d,n,c}{2}{x}
\row{c,d,c}{3}{y}
\row{c,c,d}{4}{yx}
\row{c,d,n}{5}{z}
\row{d,n,n}{6}{w}
\connect{3}{xx}{x}
\connect{3}{x}{y}
\connect{3}{y}{yx}
\connect{1}{x}{y}
\connect[db]{1,2}{y}{yx}
\connect{1}{yx}{z}
\connect{2}{yx}{z}
\connect{1}{z}{w}
\end{tikzpicture}
+
\tikzsetnextfilename{starprod4}
\begin{tikzpicture}[scale=0.4,baseline=1.4cm]
\row{n,n,d}{1}{xx}
\row{d,n,c}{2}{x}
\row{c,d,c}{3}{y}
\row{c,c,d}{4}{yx}
\row{c,d,n}{5}{z}
\row{d,n,n}{6}{w}
\connect{3}{xx}{x}
\connect{3}{x}{y}
\connect{3}{y}{yx}
\connect{1}{x}{y}
\connect[db]{1,2}{yx}{z}
\connect{1}{yx}{y}
\connect{2}{yx}{y}
\connect{1}{z}{w}
\end{tikzpicture}
+
\tikzsetnextfilename{starprod5}
\begin{tikzpicture}[scale=0.4,baseline=1.4cm]
\row{n,n,d}{1}{xx}
\row{d,n,c}{2}{x}
\row{c,d,c}{3}{y}
\row{c,c,d}{4}{yx}
\row{c,d,n}{5}{z}
\row{d,n,n}{6}{w}
\connect{3}{xx}{x}
\connect{3}{x}{y}
\connect{3}{y}{yx}
\connect{1}{x}{y}
\connect[db]{1,2}{y}{yx}
\connect[db]{1,2}{yx}{z}
\connect{1}{z}{w}
\end{tikzpicture}
\;.
\end{equ}

For any $A \subset \cA$ and any partition $\pi$ of $A$,
we then set
\begin{equ}
\one_A^\pi = \sum\{\Cumus \in \Cumus_A\,:\, \Part_\Cumus = \pi\}\;,
\end{equ}
so that in particular $\one = \sum_{\pi \in \CP(\cA)} \one_\cA^\pi$.
It is then immediate from \eqref{e:defProduct} that one has the  identity
\begin{equ}
\one_A^\pi \star \one_B^\sigma = \one_{A\cup B}^{\pi\cup \sigma}\;.
\end{equ}

We also claim that $\CG$ is multiplicative with respect to $\star$, which is the
content of Proposition~\ref{prop:Gmultiplicative} below.
As a consequence, one has
\begin{equ}
\CG(\one)
= \sum_{\pi \in \CP(\cA)} \CG(\one_\cA^\pi) = \sum_{\pi \in \CP(\cA)} \prod_{A \in \pi} \CG(\one_A^{\{A\}})\;.
\end{equ}
Since $\one_A^{\{A\}}$ is precisely the indicator function of the connected elements, the
claim follows.
\end{proof}

\subsubsection{Proof of the multiplicativity of $\CG$}

The goal of this section is to show that the map $\CG$ defined in the proof of Theorem~\ref{thm:cumulantofexclusion} is 
indeed multiplicative with respect to $\star$. To that end, fix $A, B\subset \cA$ such that $A\cap B=\emptyset$,
$\Cumus^A=(\pi^A,K^A)\in \Cumus_A$ and $\Cumus^B=(\pi^B,K^B)\in\Cumus_B$. For $e\in A_\to$ and 
$L\in K_{e}^{A}$ (in the graphical representation, $L$ is one of the ``blue lines''), we then 
consider the set $\CK_{e,L}$ of all maps $K:(A\cup B)_\to^{e} \to \cP(L)$, where 
$(A\cup B)_\to^{e}=(A\cup B)_\to \cap [e_-,e_+]$.
Fix $\fz\in Z_{\star,\pi^A}^{\fx}$. We let $Z_{\star,e,L}^{\fz}$ consist of all maps $\hat \fz \colon E_{e,L}^{A\cup B}\to \St$ with $E_{e,L}^{A\cup B}\eqdef ((A\cup B) \cap [e_-,e_+]) \times L$, such that
$\hat \fz(\tilde a,c)=\fz(\tilde a, c)$ for all $\tilde a\in\{e_-,e_+\}$
and $c\in L$. Given a subset $C \subset L$, the notation $\hat \fz(\tilde a,C)$ 
should be interpreted as a shorthand for the tuple
$\{\hat\fz(\tilde a,c)\,:\, c\in C\}$.

Before we formulate the next lemma we associate to $K \in \CK_{e,L}$ the finest partition $\Part_{K} \in \CP(L)$ such that $K(\tilde e) \le \Part_K$ for all $\tilde e\in (A\cup B)_\to^{e}$. We then define the connected elements by $\CK_{e,L}^{c}=\{K\in\CK_{e,L}:\, \Part_{K}=\{L\}\}$. 
\begin{lemma}\label{lem:insertingpoints}
In the notation just introduced
\begin{equation}
p_{c;t_e}^{\fz_{-}(e)_{L}}(\fz_{+}(e)_{L})=
\sum_{K\in\CK_{e,L}^{c}}\sum_{\hat\fz\in Z_{\star,e,L}^{\fz}}
\prod_{\tilde e\in (A\cup B)_\to^{e}}\prod_{C\in K(\tilde e)}
p_{c;t_{\tilde e}}^{\hat\fz(\tilde e_-,C)}(\hat\fz(\tilde e_+,C)).
\end{equation}
\end{lemma}
The formulation of the above lemma unfortunately needs a fair amount of notation and therefore looks quite heavy. In terms of the graphical notation the meaning however is quite simple. Indeed, the above lemma for instance implies the identity
\begin{equ}
	\tikzsetnextfilename{insert1}
\begin{tikzpicture}[scale=0.4,baseline=0.7cm]
\row{d,d,d}{1}{x}
\row{d,d,d}{3}{y}
\connect[db]{1,2,3}{x}{y}
\end{tikzpicture}	
=
\tikzsetnextfilename{insert2}
\begin{tikzpicture}[scale=0.4,baseline=0.7cm]
\row{d,d,d}{1}{x}
\row{c,c,c}{2}{y}
\row{d,d,d}{3}{z}
\connect[db]{1,2}{x}{y}
\connect[db]{2,3}{x}{y}
\connect{1}{y}{z}
\connect{2}{y}{z}
\connect{3}{y}{z}
\end{tikzpicture}
+\tikzsetnextfilename{insert3}
\begin{tikzpicture}[scale=0.4, baseline=0.7cm]
\row{d,d,d}{1}{x}
\row{c,c,c}{2}{y}
\row{d,d,d}{3}{z}
\connect[db]{1,2}{y}{z}
\connect[db]{2,3}{x}{y}
\connect{1}{x}{y}
\connect{2}{x}{y}
\connect{3}{y}{z}
\end{tikzpicture}
+
\tikzsetnextfilename{insert4}
\begin{tikzpicture}[scale=0.4, baseline=0.7cm]
\row{d,d,d}{1}{x}
\row{c,c,c}{2}{y}
\row{d,d,d}{3}{z}
\connect[db]{1,2}{x}{y}
\connect[db]{2,3}{y}{z}
\connect{1}{y}{z}
\connect{2}{y}{z}
\connect{3}{x}{y}
\end{tikzpicture}
+
\tikzsetnextfilename{insert5}
\begin{tikzpicture}[scale=0.4,baseline=0.7cm]
\row{d,d,d}{1}{x}
\row{c,c,c}{2}{y}
\row{d,d,d}{3}{z}
\connect[db]{1,2}{y}{z}
\connect[db]{2,3}{y}{z}
\connect{1}{x}{y}
\connect{2}{x}{y}
\connect{3}{x}{y}
\end{tikzpicture}\,.
	\end{equ}
Thinking of blue lines as segments connecting two vertical lines, 
Lemma~\ref{lem:insertingpoints} states that if one inserts one additional row of 
space-time points, then one needs to sum over all ways of putting each blue line segment above 
or below the newly inserted row.
Note that if one starts from a connected graph, the graph one obtains in this way is 
again connected. 

This explains the necessity to sum over $\cK_{e,L}^{c}$. Morever, we see that the points on the bottom and the top never change, hence the restriction $\hat\fz (\tilde a,c) =\fz(\tilde a,c)$ for $\tilde a\in\{e_-,e_+\}$ and $c\in L$ in the definition of $Z_{\star,e,L}^{\fz}$. For the last step of the proof we also notice at that point that even when summing over all spatial variables, i.e., over all circles, in the last graph above the two components factorise into two independent sums.

\begin{proof}
To simplify the notation we assume that $t_{e_-}=0$. We can write using the Markov property,
\begin{equation}\label{eq:transitionexpectation}
\begin{aligned}
\E[\one\{X_{t_{e_+}}^{\fz_{-}(e)_{L}}=\fz_{+}(e)_{L}\}]
&=\sum_{\hat \fz\in Z_{\star,e,L}^{\fz}}
\prod_{\tilde e \in (A\cup B)_{\to}^{e}}
p_{t_{\tilde e}}^{\hat\fz_-(\tilde e)_L}(\hat\fz_+(\tilde e)_L).
\end{aligned}
\end{equation}
The transition probability can further be  rewritten in terms of cumulants as
\begin{equation}
p_{t_{\tilde e}}^{\hat\fz_-(\tilde e)_L}(\hat\fz_+(\tilde e)_L)
=\sum_{\pi_{\tilde{e}}\in \cP(L)}
\prod_{C\in\pi_{\tilde{e}}} p_{c;t_{\tilde e}}^{\hat\fz(\tilde e_-,C)}(\fz(\tilde e_+,C)).
\end{equation}
In particular, as suggested by the notation, that the partition $\pi_{\tilde e}$ depends on the edge $\tilde{e}\in  (A\cup B)_\to^{e}$. Therefore, when plugging this expression back in into the right hand side of~\eqref{eq:transitionexpectation} we see that
\begin{equs}
\E[\one\{X_{t_{e_+}}^{\fz_{-}(e)_{L}}=\fz_{+}(e)_{L}\}]
=\sum_{K\in\CK_{e,L}}\sum_{\hat\fz\in Z_{\star,e,L}^{\fz}}
\prod_{\tilde e\in (A\cup B)_\to^{e}}\prod_{C\in K(\tilde e)}
p_{c;t_{\tilde e}}^{\hat\fz(\tilde e_-,C)}(\hat\fz(\tilde e_+,C)).
\end{equs}
Assume now that $K\notin \CK_{e,L}^{c}$, so that we can write $\cP_K=\{L_1,\ldots, L_n\}$ for some $n\geq 2$. In that case the sum over $\hat\fz\in Z_{\star,e,L}^{\fz}$ above factorises as
\begin{equ}
\prod_{i=1}^{n}\sum_{\hat\fz\in Z_{\star,e,L_i}^{\fz}}
\prod_{\tilde e\in (A\cup B)_\to^{e}}\prod_{C\in K(\tilde e)}
p_{c;t_{\tilde e}}^{\hat\fz(\tilde e_-,C)}(\hat\fz(\tilde e_+,C))
	\end{equ}
where $\hat\fz\in Z_{\star,e,L_i}^{\fz}$ is defined as $ Z_{\star,e,L}^{\fz}$ with the only restriction that we have $\hat \fz(\tilde a,c)=\fz(\tilde a,c)$ for any $\tilde a\in\{e_-,e_+\}$ and $c\in L_i$.
Thus, we may conclude as in the proof of~\eqref{eq:cumulantmartingales}.
\end{proof}

\begin{proposition}\label{prop:Gmultiplicative}
The map $\CG$ defined in the proof of Theorem~\ref{thm:cumulantofexclusion} is multiplicative with respect to $\star$.
\end{proposition}
\begin{proof}
Fix $\Cumus^A=(\pi^A,K^A)\in\Cumus_A$, and $\Cumus^B=(\pi^B,K^B)\in\Cumus_B$. We need to introduce more notation. Let $E_{\pi^A,B}\subset (A\cup B)\times\pi^A$ be such that 
\begin{itemize}
	\item[(1)] $(a,C)\in E_{\pi^A,B}$ if $a\in A$ and $C\in\pi_a^A$,
	\item[(2)] or $a\in B$ and $C\in\pi_e^A$. Here $e\in A_\to$ is such that $e_-< a< e_+$.
\end{itemize}
We then define 
\begin{equ}
Z_{\star,\pi^A,A\cup B}^{\fx}=\{\fz^A:E_{\pi^A,B}\to \St:\, \fz^A(a,C)=x_a\text{ if }	a\in C\}.
	\end{equ}
The space $\CK^A$ is defined as the space of all sections
$K:\bigcup_{e\in \cA_{\to}}(A\cup B)_\to^e \to \sqcup_{e\in\cA_\to}\sqcup_{\tilde{e}\in (A\cup B)_\to^e}\cP(\tilde \pi_{\tilde{e}}^A)$ satisfying
satisfy the relation
\begin{equ}\label{eq:KandKA}
	K^A_{e} =  \bigvee \Bigl\{K_{\bar e}\,:\, \bar e \in (A\cup B)_\to^{e}\Bigr\}\;\quad\text{for all }e\in\cA_\to\,.
\end{equ}
Here to define $\tilde\pi_{\tilde{e}}^A$ we distinguish between two cases.
\begin{itemize}
	\item[(1)] If $\tilde{e}\in A_\to$, then $\tilde\pi_{\tilde{e}}^A=\pi_{\tilde{e}}^A$.
	\item[(2)] If $\tilde{e}\in (A\cup B)_\to\setminus A_\to$, then $\tilde \pi_{\tilde{e}}^A=\pi_{e}^{A}$, where $e\in A_\to$ is such that $ e_- \leq  \tilde{e}_- < \tilde{e}_+\leq  e_+$.
\end{itemize} 
To see the relation of that definition with Lemma~\ref{lem:insertingpoints} note that if for some $e\in\cA_\to$ one has that $K_e^A=\{L_1,\ldots, L_n\}$, then 
$P_{K|(A\cup B)_\to^e}=\{L_1,\ldots, L_n\}$. 
For later use we already note now that if $\tilde{e}$ does not satisfy any of the two conditions in the two items above, then we define $\tilde{\pi}_{\tilde{e}}^A=\emptyset$.
%
With this notation at hand, we see that, using Lemma~\ref{lem:insertingpoints} in the second equality,
\begin{equs}
p_{c;\ft}^{\Cumus^A}(\fx)&=\sum_{\fz^A\in Z_{\star,\pi^A}^{\fx}}\prod_{e\in A_\to}\prod_{L\in K_e^{A}} p_{c;t_e}^{\fz^A_-(e)_{L}}(\fz^A_+(e)_{L})\\
&= \sum_{\fz^A\in Z_{\star,\pi^A}^{\fx}}\prod_{e\in A_\to}\prod_{L\in K_e^{A}}\sum_{K\in\cK_{e,L}^{c}}\sum_{\hat\fz\in Z_{\star,e,L}^{\fz^A}}\prod_{\tilde e\in (A\cup B)_\to^{e}}\prod_{C\in K(\tilde e)}
p_{c;t_{\tilde e}}^{\hat\fz(\tilde e_-,C)}(\hat\fz(\tilde e_+,C))\\
&= \sum_{K\in\CK^A}\sum_{\fz^A\in Z_{\star,\pi^A,A\cup B}^{\fx}}
\prod_{\tilde{e}\in\cA_\to}\prod_{e\in (A\cup B)_\to^{\tilde{e}}}\prod_{C\in K_e}p_{c;t_e}^{\fz^A(e_-,C)}(\fz^A(e_+,C)).
\end{equs} 
We can therefore conclude that
\begin{equs}
&p_{c;\ft}^{\Cumus^A}(\fx)p_{c;\ft}^{\Cumus^B}(\fx)\\
&=\sum_{K^1, K^2}\sum_{\fz^A,\fz^B}
\prod_{e\in (A\cup B)_\to}\prod_{\substack{C_1\in K_e^1,\\ C_2 \in K_e^2}}
p_{c;t_e}^{\fz^A(e_-,C_1)}(\fz^A(e_+,C_1))p_{c;t_e}^{\fz^B(e_-,C_2)}(\fz^A(e_+,C_2)),
\end{equs}
where the two outer sums are over the sets $\{\fz^A\in Z_{\star,\pi^A,A\cup B}^{\fx}, \fz^B\in Z_{\star,\pi^B,A\cup B}^{\fx}\}$ and $\{K^1\in \CK^A, K^2 \in\CK^B\}$.
Now note that for any $e\in (A\cup B)_\to$ we have that $(K_e^1,K_e^2)\in\cP(\tilde \pi_e^A)\sqcup\cP(\tilde \pi_e^B)=\cP(\pi_e)$, where $\pi=\pi^A\cup \pi^B$. 
Thus, defining $K$ on $(A\cup B)_\to$ via $K_e=(K_e^1,K_e^2)$, and recalling~\eqref{eq:KandKA} we see that
$\Pi\Cumus=\Pi(\pi,K)=(\Cumus^A,\Cumus^B)$. Moreover, any $K$ satisfying $\Pi\Cumus=\Pi(\pi,K)=(\Cumus^A,\Cumus^B)$ is of the form just described. Hence,
\begin{equation}
p_{c;\ft}^{\Cumus^A}(\fx)p_{c;\ft}^{\Cumus^B}(\fx)=p_{c;\ft}^{\Cumus^A\star\Cumus^B}(\fx),
\end{equation}
and since moreover $C_\Cumus(q)=C_\pi(q)$ factors over $\pi$, we can conclude the proof.
\end{proof}

\subsubsection{Bounding cumulants in terms of coalescence trees}\label{sec:coalescence}
In this section we use Theorem~\ref{thm:cumulantofexclusion} to estimate the cumulants of the exclusion process. We assume throughout this and the next section that $d\geq 3$. The estimates we obtain are inspired by~\cite[Appendix A]{Ajay}. To that end let $\CV=\{v_0,v_1,\ldots, v_p\}$ be an index set with $p\geq 1$ (in many cases of interest we have $\CV=\CA$, with $\CA$ as in the previous subsections). We also define $\CV_0=\CV\setminus\{v_0\}$. We are further given a collection of space time points $(y_v)_{v\in\CV}$ indexed by $\CV$ such that each $y_v\in \R\times 2^{-N}\St$ has norm bounded by one with respect to $\|\cdot\|_\s$. We then consider the collection $\CU_{\CV}$ of all labelled binary trees $(\fT,\fs)$, where
\begin{itemize}
	\item[\rm{(1)}] $\fT$ is a binary tree with leaf set $\CV$.
	\item[\rm{(2)}] $\fs:\fT_i\to \N$ is a label on the set of inner nodes of $\fT$ such that for any two inner nodes $\nu < \omega$, one has $\fs(\omega) < \fs(\nu)$. Here, $\nu\leq \omega$ if $\nu$ is a descendant of $\omega$.
\end{itemize}
We furthermore fix a measurable map $\CT:(\R\times 2^{-N}\St)^{\CV}\to \CU_{\CV}$ with the additional property
that, given $y \in (\R\times 2^{-N}\St)^{\CV}$, the associated labelled tree $(\fT,\fs) = \CT(y)$
is such that, for any two leaves $v,w$, one has
\begin{equ}[e:defDistTree]
	2^{-\fs(v\wedge w)}\leq \|y_v - y_w\|_\s\vee 2^{-N}\leq C_{|\CV|} 2^{-\fs(v\wedge w)}\;,
\end{equ}
for some constant $C_{|\CV|}$ depending only on the number of elements of $\CV$, but not on $N$.
Such maps exist: take for example for $\CT(y)$ the binary tree obtained by 
applying Kruskal's algorithm, see Appendix A.2 in~\cite{KPZJeremy}. 
For our purposes the precise construction is not important so that we will not further dwell on it. 
Given $(\fT,\fs)$ we finally define a map $c_\fT:\fT_i\to\R$ by
\begin{equ}
	c_\fT(a)=\begin{cases}
		d,\, \mbox{if }a=\rho_\fT,\\
		{d\over 2},\, \mbox{otherwise,}
	\end{cases}
\end{equ}
where $\rho_\fT$ denotes the root.
We also write $\c=(c_\fT)_\fT$ for the collection of maps $c_\fT$ defined as above indexed by the set of binary labelled trees.
Given a binary labelled tree $(\fT,\fs)$ we define
\begin{equ}
	\langle c_\fT,\fs\rangle =\sum_{a\in\fT_i} c_\fT(a)\fs(a).
\end{equ}
The following definition is adapted from Definition A.5 in~\cite{Ajay}.
\begin{definition}\label{def:boundedBy}
	Let $N\in\N$. We say that a function $F_N\colon (\R\times 2^{-N}\St)^{\CV} \to \R$ is bounded by $\c$ if one has the bound
\begin{equation}\label{supremum bound} 
	\| 
	F_N
	\|_{\c}
	\eqdef
		\sup_{\substack{y \in (\R\times 2^{-N}\St)^{\CV}\\ \|y\|_\s\leq 1}} 
	2^{-\langle c_{\fT},\fs\rangle}\big| F_N(y) \big|
	< \infty,
\end{equation}
	where the pair $(\fT,\fs)$ in the definition is such that $\CT(y)=(\fT,\fs)$. Here, the condition $\|y\|_\s\leq 1$ means that for all $v\in\CV$ one has that $\|y_v\|_\s\leq 1$.
	Given a family of functions $(F_N)_{N\in\N}$ that are bounded by $\c$, we say that this family is bounded uniformly by $\c$ if the supremum above is bounded uniformly in $N$.
\end{definition}
The key result of this section is as follows.
\begin{theorem}\label{thm:cumuboundedbyc} Fix $d\geq 3$.
	Let $\xi$ be the simple symmetric exclusion process with fixed-time distribution that is Bernoulli with parameter $q\in (0,1)$. For $t\geq 0$ and $x\in 2^{-N}\St$ define a rescaled version of $\xi$ via $\xi^N_t(x)= 2^{dN/2}\xi_{2^{2N}t}(2^Nx)$. Let $\cA$ be an index set with $|\cA|\geq 2$. Then, the family of functions $F=(F_N)_{N\in\N}$ given by
	\begin{equation}
	F_N((t_a,x_a),\, a\in\cA)= \E_c\Big\{\xi^N_{t_a}(x_a)\,:\, a \in \cA\Big\}
	\end{equation}
	defined on $(\R\times 2^{-N}\St)^{\cA}$ is bounded uniformly by $\c$.
\end{theorem}

\begin{proof}
Using Theorem~\ref{thm:cumulantofexclusion}, we write
	\begin{equation}\label{eq:FN}
	F_N((t_a,x_a),\, a\in\cA)
	=2^{dN|\cA|/2}\sum_{\Cumus \in \Cumus_\cA^c} C_\Cumus(q) p^{N,\Cumus}_{c;\ft}(\fx)\;,
	\end{equation}
	where $p^{N,\Cumus}_{c;\ft}(\fx)$ is given by
	\begin{equ}
		p^{N,\Cumus}_{c;\ft}(\fx) = p^{\Cumus}_{c;2^{2N}\ft}(2^{N}\fx)\,, \quad\text{for } \fx\in 2^{-N}\St^\cA\,.
	\end{equ}
	Since the collections of constants $ \{C_\Cumus(q):\,\Cumus \in \Cumus_\cA^c\}$ and $|\Cumus_\cA^c|$ are bounded from above uniformly in $N$, it suffices to estimate each term
	$2^{dN|\cA|/2}p^{N,\Cumus}_{c;\ft}(\fx)$ appearing above separately. To that end fix $\Cumus=(\pi,K)\in\Cumus_\cA^c$ for the rest of the proof and recall~\eqref{eq:cumuproba}. Then, for each $e\in\cA_\rightarrow$, $A\in K_e$ with $|A|\geq 2$ and $\fz\in 2^{-N}Z_{\star,\pi}^{\fx}$ we can write using Proposition~\ref{prop:productofmartingales} and Remark~\ref{rem:replacebyprocess}
	\begin{equation}
	p^{N,\fz_-(e)_A}_{c;t_e}(\fz_+(e)_A)=\sum_{(T,\fa)\in\Tree_A}\frak{I}_{t_e}^{A,N}(T,\fa,\fz_-(e)_A,\fz_+(e)_A,\ft_e)\;,
	\end{equation}
	where $\ft_e$ is simply the vector $\ft_e=(t_u)_{u\in A}$ such that $t_u=t_e$ for all $u\in A$, and
	\begin{equation}
	\frak{I}_{t_e}^{A,N}(T,\fa,\fz_-(e)_A,\fz_+(e)_A,\ft_e)= \frak{I}_{2^{2N}t_e}^{A}(T,\fa,2^N\fz_-(e)_A,2^N\fz_+(e)_A,2^{2N}\ft_e)\,.
	\end{equation} 
	Writing $\Tree$ for the set of all sections 
	\begin{equation}
	\mathtt{T}:\{(e,A):\, e\in\cA_\rightarrow,\, A\in K_e\}\to \bigsqcup_{\substack{e\in\cA_\rightarrow\\
			A\in K_e}}\Tree_A\;,
	\end{equation}
	and applying Lemma~\ref{lem:upperoncumu}, we write
	\begin{equs}\label{eq:pCumu}
	p^{N,\Cumus}_{c;\ft}(\fx)&= \sum_{\fz\in 2^{-N}Z_{\star, \pi}^{\fx}} \sum_{\mathtt{T}\in\Tree}
	\prod_{e\in\cA_\rightarrow}\prod_{A\in K_e} \frak{I}_{t_e}^{A,N}(\mathtt{T}(e,A),\fz_-(e)_A,\fz_+(e)_A,\ft_e)\\
	&\lesssim \sum_{\fz\in 2^{-N}Z_{\star, \pi}^{\fx}} \sum_{\mathtt{T}\in\Tree}
	\prod_{e\in\cA_\rightarrow}\prod_{A\in K_e} \frak{I}_{t_e}^{A,N}(\hat {\mathtt{T}}(e,A),\fz_-(e)_A,\fz_+(e)_A,\ft_e).
	\end{equs}
Above, the label $\fa$ does not appear in the notation anymore since fixing $\mathtt{T}\in \Tree$ one automatically fixes the label $\fa$ in each tree $\mathtt{T}(e,A)$.
Note that in the above expression there are actually two types of sums over points in $\St$, 
over all circles in the graphical notation introduced in~\eqref{e:pbasicsimple2}, and over all inner nodes in $\hat {\mathtt{T}}(e,A)$, the graph constructed in Section~\ref{Sec:upperboundonbinary}.  We first sum over the former. By the discussion in Section~\ref{S4.3} after~\eqref{e:pbasicsimple} each circle has a dot above and below, and the way a circle shows up is at the connection point between two cumulant terms like so:
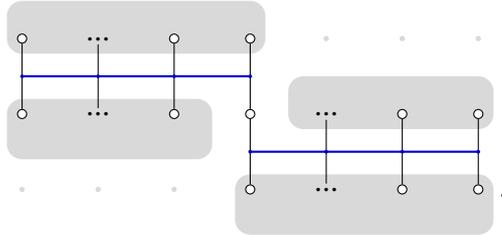
\begin{equ}
	\tikzsetnextfilename{cumullinked}
\begin{tikzpicture}[scale=1,baseline=0.9cm]
\draw[rounded corners=2mm,fill=black!15,draw=none]
(0.8,3.5) rectangle (4.2,2.8);
\draw[rounded corners=2mm,fill=black!15,draw=none]
(3.8,0.4) rectangle (7.2,1.2);
\draw[rounded corners=2mm,fill=black!15,draw=none]
(4.5,1.8) rectangle (7.2,2.5);
\draw[rounded corners=2mm,fill=black!15,draw=none]
(0.8,1.4) rectangle (3.5,2.2);
\row{n,n,n,c,l,c,c}{1}{w}
\row{c,l,c,c,l,c,c}{2}{x}
\row{c,l,c,c,n,n,n}{3}{y}
\connect[db]{1,...,4}{x}{y}
\connect[db]{4,5,6,7}{x}{w}
\end{tikzpicture}\,.
	\end{equ}
Now recall that each of these cumulant terms was replaced in~\eqref{eq:pCumu} by the corresponding graph constructed in Section~\ref{Sec:upperboundonbinary}. In particular there are exactly two kernels attached to the circle above, and both these kernels are constructed via~\eqref{e:Kvw}. Thus, summing over the circle above is simply a spatial convolution of two of such kernels, and doing so for all circles we see that we are only left with kernels that are all of the form~\eqref{e:Kvw}. To facilitate the understanding of the following construction we provide one more illustration. Assume we encounter the following situation
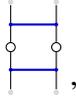
\begin{equ}
	\tikzsetnextfilename{simpleconnect}
	\begin{tikzpicture}[scale=0.6, baseline=0.7cm]
	\row{n,n}{1}{x}
	\row{c,c}{2}{y}
	\row{n,n}{3}{z}
	\connect[db]{1,2}{x}{y}
	\connect[db]{1,2}{y}{z}
	\end{tikzpicture}\;,
\end{equ}
i.e.\ two cumulant terms are attached to each other. First replacing each cumulant term by one of the graphs constructed in Section~\ref{Sec:upperboundonbinary} (which in this case is simply a binary tree) and then summing over the `circles' yields the following
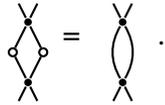
\begin{center}
	\tikzsetnextfilename{simplifyconnect1}
	\begin{tikzpicture}[scale=0.2, baseline=0.5cm, thick]
	\node [] (1) at (-3,-2) {};
	\node [] (2) at (-1,-2) {};
	\node [style=dot] (e1) at (-2,0) {};
	\node [circle,draw=black, fill=white, inner sep=0pt,minimum size=3pt] (c1) at (-3,2) {};
	\node [circle,draw=black, fill=white, inner sep=0pt,minimum size=3pt] (c2) at (-1,2) {};
    \node [style=dot] (e2) at (-2,4) {};
    \node [] (3) at (-3,6) {};
    \node [] (4) at (-1,6) {};

	\draw (1) -- (e1); 
	\draw (2) -- (e1);
	
	\draw (e1) -- (c1);
	\draw (e1) -- (c2);
	
	\draw (c1) -- (e2);
	\draw (c2) -- (e2);
	
	\draw (e2) -- (3);
	\draw (e2) -- (4);
	\end{tikzpicture}\;
	=
	\tikzsetnextfilename{simplifyconnect2}
		\begin{tikzpicture}[scale=0.2, baseline=0.5cm, thick]
	\node [] (1) at (-3,-2) {};
	\node [] (2) at (-1,-2) {};
	\node [style=dot] (e1) at (-2,0) {};
	\node [style=dot] (e2) at (-2,4) {};
	\node [] (3) at (-3,6) {};
	\node [] (4) at (-1,6) {};

	\draw (1) -- (e1); 
	\draw (2) -- (e1);
		
	\draw (e1) to [bend left] (e2);
	\draw (e1) to [bend right] (e2);
		
	\draw (e2) -- (3);
	\draw (e2) -- (4);
	\end{tikzpicture}\;.
\end{center}
On the right hand side above each edge represents one of the kernels constructed in~\eqref{e:Kvw}, and most importantly they fall into the framework of Lemma~\ref{lem:convolutioninspace}.
For each $\mathtt{T}\in\Tree$, 
the right hand side in~\eqref{eq:pCumu} may then be bounded from above by a generalised convolution obtained as follows. Consider a collection $\CG(\tT)$ of directed graphs $G$ whose common vertex set consists of two types of nodes, those indexed by $\cA$, and those indexed by the nodes $\cU$ of $\{\hat{\mathtt{T}}(e,A)\}_{e\in\cA_\to, A\in K_e}$ that are not circles. We moreover impose that the cardinality of $\cU$ is bounded from above by $|\cA|(|\cA|-1)$ (this upper bound comes from a simple counting argument). With this
vertex set fixed, $\CG(\tT)$ then consists of all graphs whose (directed) edge set obeys the following rules:
\begin{itemize}
	\item[(1)] Write $\pi=\{\cA_1,\ldots,\cA_\ell\}$ and consider a vertex $a\in\cA_i$ such that there are $a',a''\in\cA_i$ with $t_{a'} < t_a <t_{a''}$. Then, $a$ has exactly one incoming and one
	outgoing edge. If $t_a$ is maximal (minimal) among all $(t_{a'})_{a'\in\cA_i}$ then $a$ 
	has exactly one incident edge which is incoming (outgoing).
	\item[(2)] If $u\in\cU$, then $u$ has exactly two incoming and two outgoing edges.
(See Section~\ref{Sec:upperboundonbinary}.)
	\item[(3)] We allow double edges connecting $u,u'\in\cU$. Triple or quadruple edges however are not allowed.
	\item[(4)] The resulting graph is connected (see Theorem~\ref{thm:cumulantofexclusion}) and does not contain any directed loops.
\end{itemize}

\begin{remark}
The first item above reflects~\eqref{eq:cumuproba}, while the second reflects the construction 
of Section~\ref{Sec:upperboundonbinary}. Regarding the third item, we saw above an example of a 
double edge. Triple or quadruple edges however are not possible because of the binary nature 
of the trees in Proposition~\ref{prop:productofmartingales}.	
\end{remark}	

To each $u\in\cU$ we assign a time-space point $(s_u,y_u)\in [0,t]\times 2^{-N}\St$, where $t=\max_{a\in\cA} t_a$. We denote the collection of all time-space points $(s_v,y_v)_{v\in\cU}$ by $\fu$. To each edge $(v,w)\in G$ we assign a kernel $K$, as in Section~\ref{Sec:upperboundonbinary}, but this time defined on $\R\times 2^{-N}\St$ via
\begin{equ}
K_t(x) = (\sqrt{|t|}+\|x\| +2^{-N})^{-d}\,,	
\end{equ}	
Note that $K$ has the property that $2^{\zeta N} K$ is of order $\zeta$ for any $\zeta\in[0,d]$.
We then write
\begin{equ}\label{eq:FN2}
\cF^N(\mathtt{T})(\ft,\fx,\fu) = \sum_{G \in \CG(\CT)}\prod_{(v,w)\in G} K(z_w-z_v)\,,
	\end{equ}
where $z_v$ is the time-space point assigned to the vertex $v$ (which either may be of the form $(s_v,y_v)$ if $v\in\cU$ or of the form $(t_v,x_v)$ if $v\in \cA$.
Note that for fixed $\hat{\tT}(e,A)$ the integration over time is restricted to the 
interval $[0,2^{2N}t_e]$. Thus, using a change of variables in time to reduce each 
such integration to $[0,t_e]$ shows that the left hand side of~\eqref{eq:pCumu}
is bounded from above by some proportionality constant times
	\begin{equation}
	\label{eq:pCumu2}
	\sum_{\mathtt{T}\in\Tree} 2^{\tfrac{d}{2}N|\cA| + (d+2)N|\cU|}\int \cF^N(\tT)(\ft,\fx,\fu)\, d\fu\,,
	\end{equation}
	where the integral over space should be understood as a Riemann sum. Here, the factor $2^{(d+2)N|\cU|}$ comes from the change of variables in time providing a factor $2^{2N}$ 
	for each node in $\cU$ and from the fact that changing each sum $\sum_{u\in 2^{-N}\St}$ to a Riemann sum $2^{-dN}\sum_{u\in 2^{-N}\St}$ yields an additional factor $2^{dN}$.

\begin{remark}
We interpret the prefactor $2^{\tfrac{d}{2}N|\cA|+(d+2)N|\cU|}$ by saying that nodes $a\in\cA$ and 
$u\in\cU$ are each equipped with a weight $\tfrac{d}{2}$, and $(d+2)$ respectively. In the remainder of the section we explain how to redistribute these weights in order to obtain a good bound
on $\cF(\tT)$. 
\end{remark}

We now fix $\tT\in\Tree$ and a graph $G \in \CG(\tT)$ and obtain a bound on the integral  
over $\fu$ of the term $2^{\tfrac{d}{2}N|\cA|+(d+2)N|\cU|} \prod_{(v,w) \in G} K(z_w - z_v)$ appearing in  
\eqref{eq:pCumu2}.
To that end define
$\CV=\cA\cup\cU$, and apply the procedure outlined at the beginning of this section to the set $\CV$ to obtain a labelled binary tree $(\fT',\fs')$.

\begin{remark}\label{rem:identif}
Since the leaves of $(\fT,\fs)$ are identified with $\CA$, which is a
subset of the set $\CV$ of leaves of $(\fT',\fs')$, the nodes of $(\fT,\fs)$ are 
naturally identified with those nodes of $(\fT',\fs')$ that are of the form $u\vee v$
for some $u,v \in \CA$. 
\end{remark}

We now construct a weight assignment $\{c'(\nu,\fT')\,:\, \nu \in \fT'\}$
with the property that $\sum_{\nu \in \fT'} c'(\nu,\fT') = \tfrac{d}{2}|\CA| + (d+2)|\CU|$ 
in the following way. 
Given two leaves $\nu, \nu' \in \CV$, write 
$\nu \vee \nu' \in \fT'$ for their (unique) common ancestor furthest from the root. 

For any $\nu \in \CV$, consider the set $\CN_\nu = \{e \in G\,:\, \nu \in e\}$ of neighbouring edges 
in $G$ and write $e_\nu^{(j)}$
with $j = \{1,\ldots,|\CN_\nu|\}$ for an ordering of $\CN_\nu$ by ``distance''
in the sense that 
$e_\nu^{(j)} \in \CN_\nu$ and $j \le k$ implies 
$\fs'(\nu \vee \nu^{(j)}) \le \fs'(\nu \vee \nu^{(k)})$, where 
$\nu^{(j)}$ is the node such that $e_\nu^{(j)} = (\nu,\nu^{(j)})$
or $e_\nu^{(j)} = (\nu^{(j)},\nu)$.
(Such an ordering isn't unique, but $\nu \vee \nu^{(j)}$ is uniquely determined since
$\fs'$ is assumed to be injective on any path linking a leaf to the root.)
For any edge $e \in G$, we then write $\delta_e \colon G \to \R$ for the map
such that $\delta_e(e') = 1$ if $e = e'$ and $0$ otherwise. 

With these notations at hand, we then define a map $w_{\fT'}\colon G \to \R$ by
setting
$w_{\fT'} = \sum_{\nu \in \CV} w_{\fT'}(\nu, \cdot)$,
where 
\begin{equs}[2][e:defwT]
w_{\fT'}(\nu, \cdot) &= \tfrac{d}{2} \delta_{e_\nu^{(1)}} \;, &\quad \nu &\in \CA\;, \\
w_{\fT'}(\nu, \cdot) &= \tfrac{d}{2} \bigl(\delta_{e_\nu^{(1)}} +\delta_{e_\nu^{(2)}}\bigr)
+\delta_{e_\nu^{(3)}}  + \delta_{e_\nu^{(4)}} \;, &\quad \nu &\in \CU\;.
\end{equs}
Note that this definition makes sense since elements of $\CU$ have four neighbours
in $G$ (property~2), while elements of $\CA$ have at least one (property~1).
We also note at this point that since $d\geq 3$, we always have that $1\leq \tfrac{d}{2}$, which will soon play a crucial role.
We then ``push'' these weight assignments from $G$ onto the tree $\fT'$ in the natural
way by setting
\begin{equ}[e:defc']
w_{\fT'}^\vee(\nu, u)
 = \sum \{w_{\fT'}(\nu, e)\,:\, e^\vee = u\}\;,
\end{equ}
where $e^\vee = u \vee v$ for an edge $e = (u,v)$.
We also set $c'(\cdot ,\fT') = w_{\fT'}^\vee = \sum_{\nu\in \CV} w_{\fT'}^\vee(\nu, \cdot)$.
In particular, we note that the connectedness of $G$ implies 
that $c'(\rho_{\fT'},\fT') \ge d$ since there must exist an edge
$e = (u,v) \in G$ with $e^\vee = \rho_{\fT'}$ (otherwise the leaves of the 
two non-empty subtrees attached to the root would yield two connected 
components for $G$) and, since $\fs(\rho_{\fT'})$ is minimal, 
one then has $e_u^{(1)} = e_v^{(1)} = e$ so that $w_{\fT'}(e) = d$
so that $c'(\rho_{\fT'},\fT') \ge d$ by \eqref{e:defc'}.

The analysis so far results in the following lemma.
	\begin{lemma}\label{lem:estimateofF}
		Fix $\tT\in\Tree$, then for any collection of time space points $\{(t_a,x_a)\}_{a\in\cA}$ and $\{u_\nu\}_{\nu\in\cU}$ such that $\CT(\ft,\fx,\fu) =(\fT',\fs')$ one has the bound
		\begin{equ}[e:boundFT]
		2^{dN|\cA|/2}2^{(d+2)N|\cU|}\cF(\tT)(\ft,\fx,\fu)\lesssim 2^{\langle c'(\cdot,\fT'),\fs'\rangle}.
		\end{equ}
	\end{lemma}
\begin{proof}
Fix $G\in \CG(\tT)$. The main remark is that $w_{\fT'}(e) \in [0,d]$ since every vertex assigns a weight at most $\tfrac{d}{2}$ to any of its neighbouring edge, and any edge is of course adjacent to exactly two
vertices. As a consequence, since $|K(z)| \lesssim 2^{-dN}|z|^{-d}$, it follows from \eqref{e:defDistTree} that
\begin{equ}
\prod_{e \in G} 2^{w_{\fT'}(e) N} |K(z_u - z_v)| \lesssim \prod_{e \in G} 2^{w_{\fT'}(e) \fs(u\vee v)}\;.
\end{equ}
Summing over all finitely many choices of $G$, the left-hand side is nothing but 
the left-hand side of \eqref{e:boundFT} since $\sum_{e \in G} w_{\fT'}(e) = \tfrac{d}{2} |\CA| + (d+2)|\cU|$
by \eqref{e:defwT}. The right-hand side also equals that of \eqref{e:boundFT}
by \eqref{e:defc'}.
\end{proof}
In the next step we integrate over all variables indexed by $u\in\cU$. For this, the following
notation will be useful.
For any $\nu \in \fT'$, we write $\fT'_L(\nu) = \{v \in \CV\,:\, v \le \nu\}$ and
\begin{equ}[e:defAU]
\CA_\nu \eqdef |\fT'_L(\nu) \cap \CA|\;,\qquad
\CU_\nu \eqdef |\fT'_L(\nu) \cap \CU|\;.
\end{equ}
We also set $\fT'(\cU) = \fT' \setminus \fT$ (see Remark~\ref{rem:identif} for the 
inclusion $\fT \subset \fT'$)
and write $\fT'_\rho(\cU) \subset \fT'(\cU)$ for those elements $\nu \in \fT'$
such that $\CA_\nu = \emptyset$ but $\CA_{\nu^\uparrow} \neq \emptyset$.
In particular, it follows from these definitions that 
\begin{equation}\label{eq:disjointU}
	\cU = \bigcup_{\nu \in \fT'_\rho(\cU)}\cU_\nu\;,\quad \fT'(\cU) = \bigcup_{\nu \in \fT'_\rho(\cU)} \fT'_\nu\;,\quad \fT'_\nu = \{\nu^\uparrow\} \cup \{\mu \le \nu\}\;,
\end{equation}
and both unions are disjoint.

\begin{lemma}\label{lem:domainestimate}
		With the notation introduced above we have the bound
		\begin{equation}
		\int\one\{{\CT(\ft,\fx,\fu)=(\fT',\fs')}\} \, d\fu
		\lesssim \prod_{\nu\in\fT'(\cU)} 2^{-(d+2)\fs'(\nu)},
		\end{equation}
		where $d\fu$ denotes the Riemann sum over the collection $\{u_\nu\}_{\nu\in\cU}$.
\end{lemma}
\begin{proof}
For any $\nu \in \fT'_\rho(\cU)$ and any fixed values of 
$\{\fu_\mu\,:\, \mu \in \cU \setminus \cU_\nu\}$, we claim that the semi-discrete 
Lebesgue measure of
\begin{equ}[e:boundZnu]
		\CZ_\nu \eqdef \{\fu_\mu:\, \mu\in \cU_\nu\,, \CT(\ft,\fx,\fu)=(\fT',\fs')\} \subset (\R\times2^{-N}\St)^{\cU_\nu}
\end{equ}
is bounded from above by some fixed multiple of
$\prod_{\mu\in\fT'_\nu}2^{-(d+2)\fs'(\mu)}$.
Indeed, if we fix the value of $\fu_{\bar \nu}$ for some fixed arbitrary $\bar\nu \in \cU_\nu$,
then the measure of the remaining values is bounded by 
$\prod_{\mu \le \nu}2^{-(d+2)\fs'(\mu)}$ as in \cite[Lem.~A.13]{KPZJeremy}. 
Since we can find $a \in \CA$
such that $a \vee \bar \nu = \nu^\uparrow$ by the definition of $\fT'_\rho(\cU)$, 
the constraint $\|\fu_{\bar \nu} - (t_a,x_a)\| \le 2^{-\fs'(\nu^\uparrow)}$ guarantees that the
measure of all possible values for $u_{\bar \nu}$ is bounded by 
$2^{-(d+2)\fs'(\nu^\uparrow)}$, yielding \eqref{e:boundZnu}.

The claim then follows at once by integrating over all the $\cU_\nu$'s one by one.
\end{proof}
With the lemma at hand, defining 
\begin{equ}[e:defcbar]
	\bar c(\nu)=
	\begin{cases}
	c'(\nu, \fT'),\, &\mbox{if } \nu\notin \fT'(\cU),\\
	c'(\nu, \fT')- (d+2),\, &\mbox{if }\nu\in\fT'(\cU),
	\end{cases}
\end{equ}
it follows from Lemmas~\ref{lem:estimateofF} and~\ref{lem:domainestimate} that 
\begin{equation}
	2^{dN|\cA|/2+(d+2)N|\cU|}\int\one\{ \CT(\ft,\fx,\fu)=(\fT',\fs')\}|\cF^N(\tT)(\ft,\fx,\fu)| \, d\fu
	\lesssim 2^{\langle \bar c(\nu,\fT'),\fs'\rangle}.
\end{equation}
To finish the proof it remains to analyse $\bar c(\cdot,\fT')$ and to compare it to $c_\fT$ as defined at the beginning of this section.

For this, we define a weight $c_{\mod,\eps}$ by setting
\begin{equ}
c_{\mod,\eps}(\nu) = \hat c_\fT(\nu) + c_{\CU,\eps}(\nu) + 
c_{\gain}(\nu)\;,
\end{equ}
where the three terms are defined as follows. First, we set
\begin{equ}[e:defcthat]
\hat c_\fT(\nu)= \tfrac{d}{2} \big(\one_{\nu = \rho_{\fT'}} + \one_{\nu \in \fT}\big)\;,
\end{equ}
where we used the identification $\fT \subset \fT'$ of Remark~\ref{rem:identif}.
The term $c_{\CU,\eps}$ is then defined by 
setting $c_{\CU,\eps} = \sum_{\nu \in \fT'_\rho(\CU)} c_{\CU,\eps}^\nu$ with
\begin{equ}
c_{\CU,\eps}^\nu(\mu) = 
\left\{\begin{array}{cl}
	(|\CU_\nu|-1)\eps & \text{if $\mu = \nu^\uparrow$,} \\
	-\eps & \text{if $\mu \le \nu$,} \\
	0 & \text{otherwise.}
\end{array}\right.
\end{equ}
Note in particular that 
\begin{equ}[e:zerocU]
\sum_{\mu \le \nu} c_{\CU,\eps}(\mu) = 0\;,\qquad 
\forall \nu \in \fT \cup \{\bar \nu^\uparrow\,:\, \bar \nu \in \fT'_\rho(\CU)\}\;.
\end{equ}
The definition of $c_{\gain,\eps}$ is slightly more subtle.

Given any $\nu \in \fT$ (including its leaves), let $\bar \nu \in \fT'$
denote the maximal element with $\nu \le \bar \nu$ such that the
segment $S_\nu = [\nu,\bar \nu] \subset \fT'$ satisfies $S_\nu \cap \fT = \{\nu\}$. 
We then write 
\begin{equ}[e:defSnu]
S_\nu = \{\nu = \nu_0 < \nu_1 < \ldots < \nu_m = \bar \nu\}\;.
\end{equ}
(The case $\bar \nu = \nu$ and therefore $m = 0$ is allowed.)
Note also that for every $i > 0$, there exists exactly one 
element $\mu_i \in \fT_\rho'(\CU)$ such that $\mu_i^\uparrow = \nu_i$.

For every $\mu \in \fT'$ we then write $G_\mu \subset G$ for the (not necessarily connected) 
subgraph induced by $\fT'_\mu$. Its boundary $\d G_\mu \subset G$ consists of those edges 
with exactly one incident vertex in $\fT'_\mu$. 
With this notation at hand, we then define
\begin{equ}[e:defnustar]
m_\star(\nu) = m \wedge \inf\{i \ge 0 \,:\, |\d G_{\nu_i}| \le 2\}\;,\qquad \nu_\star = \nu_{m_\star(\nu)}\;.
\end{equ}
We then set $c_{\gain} = \sum_{\nu \in \fT} c_{\gain}^\nu$ with
\begin{equ}[e:defcgain]
c_{\gain}^\nu(\mu) = 
-\tfrac{d}{2} \one_{\mu = \nu} + \tfrac{d}{2}\one_{\mu = \nu_\star} - \tfrac{d}{2}\one_{\mu = \nu_\star^\uparrow} + \tfrac{d}{2}\one_{\mu = \bar \nu^\uparrow} \;.
\end{equ}
(We use the convention that $\rho_{\fT'}^\uparrow = \rho_{\fT'}$.)
The important feature of $c_{\gain}^\nu$ is that one has
\begin{equ}[e:boundcgain]
\sum_{\mu \le \tilde \nu} c_{\gain}^\nu(\mu) =
\left\{\begin{array}{cl}
	-\tfrac{d}{2} & \text{if $\tilde \nu \in S_\nu \setminus \{\nu_\star\}$,} \\
	0 & \text{otherwise,}
\end{array}\right.
\end{equ}
with the possible exception of the case $\tilde \nu = \rho_{\fT'}$, in which case
the sum always vanishes.

Since one has
\begin{equ}[e:numberAs]
|\{\mu \le \nu\,:\, \mu \in \fT\}| = 
\left\{\begin{array}{cl}
	0 & \text{if $|\CA_\nu| = 0$,} \\
	|\CA_\nu|-1 & \text{otherwise,}
\end{array}\right.
\end{equ}
(leaves aren't counted here) this implies that, for every $\nu \in \fT'$, one has 
\begin{equ}[e:numberUs]
|\{\mu \le \nu\,:\, \mu \in \fT'(\CU)\}| = 
\left\{\begin{array}{cl}
	|\CU_\nu|-1 & \text{if $|\CA_\nu| = 0$,} \\
	|\CU_\nu| & \text{otherwise.}
\end{array}\right.
\end{equ}
We then have the following.

\begin{lemma}\label{lem:boundweights}
Let $c_{\mod,\eps}$ be as above and let $\eps |\CU| \le 1$.
Then, for every $\nu \in \fT'$, one has $\sum_{\mu \le \nu} \bar c(\mu) \le \sum_{\mu \le \nu} c_{\mod,\eps}(\mu)$.
\end{lemma}
	
\begin{proof}
When $\nu$ equals the root, both expressions equal $\tfrac{d}{2} |\CA|$.
Fix now some $\nu \in \fT' \setminus \{\rho_{\fT'}\}$. With the 
notation \eqref{e:defAU} at hand, we distinguish the following two cases.

\medskip
\noindent\textbf{The case $|\CA_\nu| = 0$.}
In this case, only $c_{\CU,\eps}$ contributes to $c_{\mod,\eps}$ and one has
\begin{equ}
\sum_{\mu \le \nu} c_{\mod,\eps}(\mu) = -\eps(|\CU_\nu|-1) > - \eps |\CU| \ge -1\;.
\end{equ}
On the other hand, it follows from \eqref{e:defcbar} and \eqref{e:numberUs} that
\begin{equ}[e:wantedsum]
\sum_{\mu \le \nu} \bar c(\mu) = (d+2) - \sum_{u \in \fT'_L(\nu)} \sum_{\bar \nu > \nu}
w_{\fT'}^\vee(u,\bar \nu)\;.
\end{equ}
Here, the number $(d+2)$ appears by subtracting the weight 
$(d+2)|\{\mu \le \nu\,:\, \mu \in \fT'(\CU)\}|$ appearing in \eqref{e:defcbar} from the 
total weight $(d+2)|\CU_\nu|$ distributed by all the elements of $\CU_\nu$.

Since edges in $G$ are directed and $G$ does not contain any loops, we can define a 
partial order on $\CV$ by setting $v \prec u$ if there exists a 
path connecting $u$ to $v$ in $G$. Write $\underline u \in \CU_\nu$ for one of  
its minimal elements, and $\bar u \neq \underline u$ for one of its maximal elements.

By $G$'s property~2, $\underline u$ has two outgoing
edges $\underline e_{1,2}$, but since 
$\underline u$ is minimal in $\CU_\nu$, these must belong to $\d G_\nu$.
Furthermore, any edge $e$ in $\d G_\nu$ 
must satisfy $e^\vee > \nu$, so the only way in which the ``shortest'' edge
adjacent to $\underline u$ can be among $\underline e_{1,2}$ is if \textit{all} four
edges adjacent to $\underline u$ connect to outside elements. In any case, it follows
that $\underline u$ gives weight $\tfrac{d}{2}$ to at least two edges $e$ with $e^\vee > \nu$.
A virtually identical argument goes for $\bar u$ so that $\sum_{u \in \fT'_L(\nu)} \sum_{\bar \nu > \nu}
w(u,\bar \nu) \ge 2d$, thus bounding the right-hand side of \eqref{e:wantedsum} by $-1$ as required.

\medskip
\noindent\textbf{The case $|\CA_\nu| \ge 1$.}
In this case, there exists a unique $\underline \nu \in \fT$ such that 
$\nu \in S_{\underline \nu} = [\underline \nu, \bar \nu]$ as in \eqref{e:defSnu}.
By \eqref{e:zerocU} and \eqref{e:boundcgain} this shows that 
\begin{equ}[e:wantedUpper]
\sum_{\mu \le \nu} c_{\mod,\eps} = \tfrac{d}{2} (|\CA_\nu|-1) - \tfrac{d}{2}\one_{\nu \neq \nu_\star}\;,
\end{equ}
with $\nu_\star$ as in \eqref{e:defnustar}. 
Furthermore, one has
\begin{equ}[e:prop1]
\sum_{\mu \le \nu} \bar c(\mu) = \tfrac{d}{2} |\CA_\nu| - \sum_{u \in \fT'_L(\nu)} \sum_{\bar \nu > \nu}
w_{\fT'}^\vee(u,\bar \nu)\;.
\end{equ}
Since $G_\nu \neq G$ and $G$ is connected, one has $\d G_\nu \neq \emptyset$,
so that by the same reasoning as above one always has the ``trivial'' lower bound 
$\sum_{u \in \fT'_L(\nu)} \sum_{\bar \nu > \nu} w(u,\bar \nu) \ge \tfrac{d}{2}$.
Comparing this to \eqref{e:wantedUpper}, it remains to show that this bound can
always be improved upon when $\nu \neq \nu_\star$.

When $\nu < \nu_\star$, it follows from \eqref{e:defnustar} that 
$|\d G_\nu| \ge 3$ so that it follows as above that $\sum_{u \in \fT'_L(\nu)} \sum_{\bar \nu > \nu} w(u,\bar \nu) \ge \tfrac{3d}{2}$.

When $\nu > \nu_\star$, we consider the subgraph $\bar G_\nu \subset G$ induced by
\begin{equ}
\bigcup \{\CU_\mu\,:\, \mu^\uparrow \in (\nu_\star, \nu]\}\;.
\end{equ}
Since all the vertices of $\bar G_\nu$ are in $\CU$, we know that 
$|\d \bar G_\nu| \ge 4$ as before. On the other hand, we also know by definition of
$\nu_\star$ that $|\d G_{\nu_{\star}}| \le 2$ (the only possible exception is
when $m_\star = m$ in \eqref{e:defnustar}, but this cannot happen 
when $\nu > \nu_\star$). This implies
that there are at least two edges in $\d \bar G_\nu$ that connect to
the complement of $G_\nu$, and we can finish as above.
\end{proof}
\begin{remark}\label{rem:2d}
	We note that we used in the conclusion of the case $|\cA_\nu|=0$ the fact that $d\geq 3$. Indeed, with our current estimates if $d=2$ the right hand side of~\eqref{e:wantedsum} can only be bounded from above by $0$. If we would be able to bound it by some small negative number $-\kappa$, then there would be no problem in adapting our arguments. To achieve that one would need to replace $\tfrac{d}{2}$ in~\eqref{e:defwT} by $\tfrac{d+\kappa}{2}$. This can be done if one would be able to show that the kernel on the left hand side of~\eqref{eq:gradestAB} satisfies
	\begin{equation}
	\begin{aligned}
		&2^{Nd(|A| + |B|-2)}2^{2N(d+\kappa)} |\nabla_{i,j} P_{2^{2N}t}^{A,B}(2^N\fx,2^N\fy)|	\lesssim 
		\prod_{\substack{k\in A\cup B\\ k\neq i,j}}(\sqrt{|t|} +\|\fx_k-\fy_k\| + 2^{-N})^{-d}\\
		& \times(\sqrt{|t|} +\|\fx_i-\fy_i\| + 2^{-N})^{-d-\kappa} (\sqrt{|t|} +\|\fx_j-\fy_j\| + 2^{-N})^{-d-\kappa}\,.
		\end{aligned}
	\end{equation}
	Thus, this estimate is the only missing ingredient to establish Theorem~\ref{thm:cumu} and~\ref{thm} for $d=2$.
\end{remark}

\begin{corollary}
\label{cor:cmod}
For every $\kappa\in (0,1)$ there exists $\eps \in(0,1)$ such that 
$2^{\langle \bar c,\fs'\rangle}\leq 2^{\langle c_{\mod,\eps},\fs'\rangle}$.\qed
\end{corollary}

We are now able to complete the proof of Theorem~\ref{thm:cumuboundedbyc}.
To that end, we note that by \eqref{eq:pCumu2}, Lemmas~\ref{lem:estimateofF},~\ref{lem:domainestimate}, Corollary~\ref{cor:cmod}, 
the fact that there exist only finitely many binary trees
with leaf set $\CV$, and finitely many graphs $G$ it is sufficient to obtain the bound
\begin{equation}\label{eq:sumoverTprime}
	\sum_{\fs'}
	2^{\langle c_{\mod,\eps},\fs'\rangle} \lesssim 2^{\langle c_{\fT},\fs\rangle} \;,
\end{equation}
where $\fs'$ runs over all labels for the fixed binary tree $\fT'$ such that 
furthermore $\fs'(\nu) = \fs(\nu)$ for all $\nu \in \fT$ (see Remark~\ref{rem:identif}).
For this, we first introduce the tree $\fT'' = \{\nu \in \fT' \,:\, \CA_\nu \neq \emptyset\}$
(which is not binary anymore!) and show that 
\begin{equ}[e:firstBound]
\sum_{\fs'}
	2^{\langle c_{\mod,\eps},\fs'\rangle} \lesssim
	\sum_{\fs''}
	2^{\langle \hat c_{\fT} + c_{\gain},\fs''\rangle}\;,
\end{equ}
where similarly to above $\fs''$ runs over all labels for $\fT''$ such that 
$\fs''(\nu) = \fs(\nu)$ for all $\nu \in \fT$. The reason why this is the case
is that, given $\fs''$ extending $\fs$, we have
\begin{equs}
\sum_{\fs'\,|\, \fs''}
	2^{\langle c_{\mod,\eps},\fs'\rangle} &=
2^{\langle \hat c_{\fT} + c_{\gain},\fs''\rangle}  \sum_{\fs'\,|\, \fs''}
	 \prod_{\nu \in \fT'_\rho(\CU)} \Big(
2^{(|\CU_\nu|-1)\eps \fs''(\nu^\uparrow)} \prod_{\mu \le \nu} 2^{-\eps \fs'(\mu)}\Big) \\
&\le 
2^{\langle \hat c_{\fT} + c_{\gain},\fs''\rangle}
	 \prod_{\nu \in \fT'_\rho(\CU)} \Big(
2^{(|\CU_\nu|-1)\eps \fs''(\nu^\uparrow)} \prod_{\mu \le \nu} \sum_{k > \fs''(\nu)} 2^{-\eps k}\Big) \\
&\lesssim 
2^{\langle \hat c_{\fT} + c_{\gain},\fs''\rangle}\;,
\end{equs}
as claimed, where the last bound comes from the fact that $|\{\mu\,:\, \mu\le \nu\}| = |\CU_\nu|-1$. We now show that, for any $\nu \in \fT$ such that  $\bar \nu \not \in\{\nu, \rho_{\fT'}\}$ and any labelling $\fs$ of $\fT$, one has
\begin{equ}
\sum_{\fs(\nu) > n_1 >\ldots > n_m > \fs(\bar \nu^\uparrow)} \prod_{i=0}^{m+1} 2^{n_i c_{\gain}^\nu(\nu_i)}
\lesssim 1 \;,
\end{equ}
where we made use of the same notations as in \eqref{e:defSnu}, as well as the
conventions
\begin{equ}
\nu_0 = \nu\;,\qquad n_0=\fs(\nu)\;,\qquad \nu_{m+1} = \bar \nu^\uparrow\;,\qquad n_{m+1}=\fs(\bar \nu^\uparrow).
\end{equ}
This is because, by \eqref{e:defcgain}, the left-hand side is bounded by 
\begin{equ}
\Big(\sum_{\fs(\nu) > n_1 >\ldots > n_{m_\star}} 2^{\tfrac{d}{2}(n_{m_\star} - n_0)} \Big)\Big(\sum_{n_{m_\star+1} >\ldots > n_{m} > \fs(\bar \nu^\uparrow)} 2^{\tfrac{d}{2}(n_{m+1} - n_{m_\star+1})} \Big)
\lesssim 1 \;,
\end{equ}
as claimed. In the particular case of $m > 0$ and $\nu_m = \rho_{\fT'}$ (or equivalently
$\nu = \rho_\fT$), we instead use the bound
\begin{equs}
\sum_{\fs(\nu) > n_1 >\ldots > n_m} &\prod_{i=0}^m 2^{n_i (\hat c_\fT + c_{\gain}^\nu)(\nu_i)} \\
&\lesssim \Big(\sum_{\fs(\nu) > n_1 >\ldots > n_{m_\star}} 2^{\tfrac{d}{2}n_{m_\star}} \Big)\Big(\sum_{n_m < \fs(\nu)}\sum_{n_{m_\star+1} >\ldots > n_{m}} 2^{- \tfrac{d}{2}n_{m_\star+1} + d n_m} \Big)\;,
\end{equs}
where the weight $d$ at the root comes from the weight $\tfrac{d}{2}$ given by the last term
of \eqref{e:defcgain} and in \eqref{e:defcthat}.
A similar calculation to above then shows that this is bounded by 
some fixed multiple of 
\begin{equ}
2^{(d - \tfrac{d}{2} + \tfrac{d}{2}) \fs(\nu)} = 2^{c_\fT(\nu)\fs(\nu)}\;.
\end{equ}
One furthermore verifies that in both cases the boundary cases 
$m_\star = 0$ and $m_\star = m$ are handled correctly when interpreting the 
corresponding missing factors as being equal to $\tfrac{d}{2}$. 

Combining these bounds with \eqref{e:firstBound}, the claim now follows at once.
\end{proof}

\subsubsection{Bounding cumulants by cycles}

In this section we derive the kind of bounds on the joint cumulants of the exclusion process
that we employ to finally prove Theorem~\ref{thm}. In this section, given a finite set $\CA$,
we write $\sigma(\CA)$ for all maps $\sigma\colon \N \to \CA$ that are surjective and 
periodic  with period $|\CA|$.

\begin{theorem}
	\label{thm:cycle}
	Fix $d\geq 3$.
	Let $\xi$ be the simple symmetric exclusion process with fixed-time distribution that is Bernoulli with parameter $q\in (0,1)$. For $t\geq 0$ and $x\in 2^{-N}\St$ define a rescaled version of $\xi$ via $\xi^N_t(x)= 2^{dN/2}\xi_{2^{2N}t}(2^Nx)$. Let $\cA$ be an index set with $|\cA|\geq 2$. One then has the bound
	\begin{equation}
	\E_c\Big\{\xi^N_{t_a}(x_a)\,:\, a \in \cA\Big\}
	\lesssim \sum_{\sigma \in \sigma(\CA)}
	\prod_{i=1}^{|\CA|}(\|z_{\sigma(i+1)} - z_{\sigma(i)}\|_\s\vee 2^{-N})^{-\tfrac{d}{2}},
	\end{equation}
uniformly over $N>0$ and over all collections of time-space points $z_a = (t_a,x_a)$ 
indexed by $a\in\cA$.
\end{theorem}
\begin{proof}
	Given $\cA$ and the collection $\{z_a\}$, we construct a labelled binary 
	tree $(\fT,\fs)$ as in the previous section.
It then suffices to consider any depth-first exploration of $\fT$ starting at the root
and to note that if $\sigma_\fT$ denotes the ordering of $\CA$ given by
the order in which the leaves are traversed by the exploration then, as a consequence 
of the upper and lower bound \eqref{e:defDistTree}, one has
\begin{equ}
2^{\scal{\bar c_{\fT},\fs}} \lesssim \prod_{i=1}^{|\CA|}(\|z_{\sigma_\fT(i+1)} - z_{\sigma_\fT(i)}\|_\s\vee 2^{-N})^{-\tfrac{d}{2}}\;,
\end{equ} 
and the claim follows from Theorem~\ref{thm:cumuboundedbyc}.
\end{proof}

\begin{remark}
An alternative way of obtaining one of the orderings $\sigma_\fT$ coming from a breadth-first exploration 
is to consider a planar embedding of $\fT$ and to traverse the leaves in clockwise order. 
\end{remark}

\begin{remark}
In terms of the original symmetric simple exclusion process on the torus $\St = \Z^d / 2^N \Z^d$,
one can rewrite this bound as
\begin{equ}
	\E_c\Big\{\xi_{t_a}(x_a)\,:\, a \in \cA\Big\}
	\lesssim \sum_{\sigma \in \sigma(\CA)}
	\prod_{i=1}^{|\CA|}(\|z_{\sigma(i+1)} - z_{\sigma(i)}\|_\s\vee 1)^{-\tfrac{d}{2}}\;,
\end{equ}
where $(t_a,x_a) \in \R \times \St$. Since this holds uniformly in $N$, it does in
particular also hold for the symmetric simple exclusion process on all of $\Z^d$ by
a straightforward approximation argument.
\end{remark}

\section{Bounds on renormalised models}
\label{SBRM}
Throughout this section we assume that $d=3$.
In this section we derive uniform moment bounds on the renormalised models
$(\hat\Pi^N,\hat\Gamma^N)$ constructed in Section~\ref{S3}. 
We moreover will derive bounds on the distance between the model $(\hat\Pi^N,\hat\Gamma^N)$ and a ``smooth'' version of it which we will denote by $(\hat\Pi^{\delta,N},\hat\Gamma^{\delta,N})$ and which we will define now. 
Let $\rho$ be a smooth, compactly supported function defined on $\R^4$ that integrates to one.
For $\delta\in (2^{-N},1]$ we define a rescaled version $\rho^{\delta}(t,x)\overset{\text{def}}{=}\delta^{-5}\rho(\delta^{-2}t, \delta^{-1}x)$, and also
$\rho^{\delta,N}(t,x)\overset{\text{def}}{=}2^{-3N}\int\rho^{\delta}(t,y)\, \one\{|y-x|\leq 2^{-N-1}\}\, dy$.
The definition of $\rho^{\delta,N}$ is chosen such that 
its semi-discrete integral still equals one. With these definitions at hand we then define $\xi^{\delta,N}$ via $\xi^{\delta,N}=\rho^{\delta,N}\star_N \xi^N$. Here, 
the symbol $\star_N$ means that the convolution is discrete in space, i.e., is given by a Riemann sum. We can then define a new model 
$(\hat\Pi^{\delta,N},\hat\Gamma^{\delta,N})$ by copying the construction of the model $(\hat\Pi^N,\hat\Gamma^N)$ but each appearance of $\xi^N$ is replaced by $\xi^{\delta,N}$.
\begin{proposition}
	\label{prop:momentbounds}
	For any $\kappa>0$, for any $\tau\in\{\Xi,\<Xi2>,\<Xi3>,\<XiX>, \<Xi2X>,\<Xi4>\}$ there exists a $\theta \in (0,\kappa)$ such that
	for every integer $p$ one has for every $z=(t,x)\in\R\times 2^{-N}\St$,
	\begin{equation}
	\label{eq:momentbounds}
	\E[(\hat\Pi_z^{N}\tau)(\phi_z^{\lambda})^p]\lesssim \lambda^{p(|\tau|+\kappa)},\quad 
	\E\Big[\Big(\int (\hat\Pi_z^{N}\tau)(s,x)(\CS_{2,t}^{\lambda}\phi)(s)\, ds\Big)^p\Big]\lesssim \lambda^{p(|\tau|+\kappa)}\;,
	\end{equation}
	as well as
	\begin{equs}
		\label{eq:momentboundsdifference}
		&\E[|(\hat\Pi_z^{N}-\hat\Pi_z^{\delta,N})(\tau)(\phi_z^{\lambda})|^p]\lesssim \delta^{p\theta}\lambda^{p(|\tau|+\kappa-\theta)},\quad\text{and}\\
		&\E\Big[\Big|\int (\hat\Pi_z^{N}-\hat\Pi_z^{\delta,N}(\tau)(s,x)(\CS_{2,t}^{\lambda}\phi)(s)\, ds\Big|^p\Big]\lesssim \delta^{p\theta}\lambda^{p(|\tau|+\kappa-\theta)}\;,
	\end{equs}
	uniformly over $N \ge 0$ and $\delta\in (2^{-N},1)$. The first bounds in~\eqref{eq:momentbounds} and~\eqref{eq:momentboundsdifference} are moreover uniform in $\lambda\in (2^{-N},1]$ and all test functions $\phi$ as in~\eqref{eq:Pi}. The second bounds are uniform in $\lambda\in (0,2^{-N}]$ and all functions $\phi$ as in~\eqref{eq:seminorm}.
\end{proposition}
The rest of this section is devoted to the proof of Proposition~\ref{prop:momentbounds}.

It turns out that to establish Proposition~\ref{prop:momentbounds} it is actually enough to establish only~\eqref{eq:momentbounds}. Indeed, it follows from~\cite[Lemma 7.6]{MatetskiDiscrete}
that there is a $\theta'>0$ such that for all multi-indices 
$|k|\leq 2$ we have the estimate
\begin{equation}
\sup_{z=(t,x)\in\R^4: t\neq 0}|D^kK^N(z)-D^k( K^N\star_N\rho^{\delta,N})(z)|\lesssim \delta^{\theta'}
\|z\|_{\s,N}^{-3-|k|_\s-\theta'},
\end{equation}
where $\|z\|_{\s,N}\overset{\text{def}}{=}\|z\|_\s\vee2^{-N}$.
Combining this with the estimates in~\eqref{eq:momentbounds} yields the bounds in~\eqref{eq:momentboundsdifference} for all symbols $\tau\in\{\<Xi2>,\<Xi3>, \<Xi2X>,\<Xi4>\}$, i.e., for all symbols $\tau$ whose definition contains the integration map $\cI$. The second bound in~\eqref{eq:momentboundsdifference} for $\tau\in\{\Xi, \<XiX>\}$ follows from the estimate $\lambda^{|\tau|+\kappa}\leq \delta^{\theta}\lambda^{|\tau|+\kappa -\theta}$, which directly follows from the choice $\lambda\leq 2^{-N} < \delta$. Finally, to establish the first bound  in~\eqref{eq:momentboundsdifference} for $\tau\in\{\Xi, \<XiX>\}$ note that we have for instance the identity
\begin{equation}
\big(\hat{\Pi}_z^N- \hat{\Pi}_z^{\delta,N}\big)(\Xi)(\varphi_z^\lambda)
=\int \xi^N(y)\Phi^{\lambda,\delta}_z(y)\, dy,
\end{equation}
where
\begin{equation}
\Phi^{\lambda,\delta}_z(y)= \int [\varphi^\lambda_z(y)-\varphi^\lambda_z(\bar y)]\rho^{\delta,N}(\bar y- y)\, d\bar y,
\end{equation}
and a similar formula holds for $\tau=\<XiX>$.

\begin{lemma}
	\label{lem:deltatestfct}
	Given a test function $\varphi^\lambda_z$ as in the formulation of Proposition~\ref{prop:momentbounds} we define $\Phi^{\lambda,\delta}_z:\R^4\to\R$ via
	\begin{equation}
	\Phi^{\lambda,\delta}_z(y)= \int [\varphi^\lambda_z(y)-\varphi^\lambda_z(\bar y)]\rho^{\delta,N}(\bar y- y)\, d\bar y,
	\end{equation}
	where the integral over space is the discrete integral.
	Then $\Phi^{\lambda,\delta}_z$ is a test function whose support contains $z$ and has a diameter of order $\delta$ and there exists $\theta'>0$ such that 
	\begin{equation}
	\sup_{y\in\R^4}|\Phi^{\lambda,\delta}_z(y)|\lesssim \delta^{\theta'} \lambda^{-5-\theta'}.
	\end{equation}
\end{lemma}
The desired bound for $\lambda\geq \delta$ then follows from the above lemma and~\eqref{eq:momentbounds}. For $\lambda < \delta$ we first estimate the difference by the sum. Then, we use that $\lambda^{|\tau|+\kappa}\leq \delta^{\theta'}\lambda^{|\tau|+\kappa-\theta'}$ to obtain the desired bound for $\hat{\Pi}_z^N(\tau)$. Finally it is enough to note that $\varphi_z^\lambda \star_N\rho^{\delta,N}$ is a test-function localised at scale $\delta$ and that $\delta^{|\tau|+\kappa}\leq \delta^{\theta'}\lambda^{|\tau|+\kappa-\theta'}$ to obtain the corresponding bound for $\hat{\Pi}_z^{\delta,N}(\tau)$.

By stationarity, it is enough to show~\eqref{eq:momentbounds} for $z=0$.
We start with the case $\tau=\Xi$. To that end recall that $\bar\xi^N_t(x)= \xi_{2^{2N}t}(2^Nx)-\rho$. The second bound in~\eqref{eq:momentbounds} is a simple consequence of the fact that $|2^{3N/2}\bar\xi^N|$ is bounded  by $2^{3N/2}$. Thus, we can focus
on the first bound. First of all, note that
\begin{equation}
\label{eq:momentxi}
\E[(\hat\Pi_0^{N}\Xi)^p(\phi_0^{\lambda})]
=\int \E\Big[\prod_{i=1}^{p} \bar \xi_{s_i}^{N}(x_i)\Big]\prod_{i=1}^{p}\phi_0^{\lambda}(z_i)\, d z_i.
\end{equation} 
By Lemma~\ref{lem:diagram} with $m=1$, and using that cumulants of order two or higher of $\bar \xi^N$ and $\xi^N$ coincide, the left hand side of~\eqref{eq:momentxi} equals
\begin{equation}\label{eq:xisumoverpi}
\sum_{\pi\in\CP_{M}(M\times P)}
\prod_{B\in\pi}\int \E_c(\{\xi_{s_i}^N(x_i),\, i\in B\})\prod_{i\in B}\phi_0^{\lambda}(z_i)\, dz_i,
\end{equation}
where $M=\{1\}$ and $P=\{i:\, 1\leq i\leq p\}$.
Now, for each $\pi\in\CP_{M}(M\times P)$ and each $B\in\pi$, by Theorem~\ref{thm:cycle} the corresponding term above is bounded by a multiple times
\begin{equation}\label{eq:xibound}
\int \prod_{(u,v)\in\sigma}(\|z_u-z_v\|_{\s,N})^{-3/2 -\eta}\prod_{i\in B}\phi_0^{\lambda}(z_i)\, dz_i
\end{equation}
where $\sigma\in\sigma (B)$ is an arbitrarily chosen cycle. Note that in the present situation all choices of $\sigma\in\sigma(B)$ lead to the same result. 
As a result of a direct computation, as well as a consequence of Theorem~\ref{theo:ultimate1}, the 
term in~\eqref{eq:xibound} is bounded by a multiple times 
$\lambda^{-3/2|B|-\eta |B|}$. This together with the multiplicative structure in~\eqref{eq:momentxi} and the fact that we may chose $\eta>0$ arbitrarily small yields the desired estimate.
The bound on $\tau=\<XiX>$ follows in exactly the same way since this amounts to replacing the test function $\varphi_0^{\lambda}$ by $(t,x)\mapsto x_i\varphi_0^{\lambda}(t,x)$, where $i\in\{1,2,3,\}$ (since space is 
three dimensional, the symbol $\<XiX>$ actually stands for a vector of three symbols, one for each spatial coordinate). 

\subsection{Some generalities for the remaining symbols}
\label{Sgeneralities}

In this section we do some computations with the first non trivial symbol, namely $\tau=\<Xi2>$, and we explain afterwards how to generalise these computations to the remaining symbols. To that end we introduce some more graphical notation. In this notation nodes drawn as \tikz[baseline=-3] \node [dot] {}; represent dummy variables that are to be integrated out, and  \tikz[baseline=-3] \node [root] {}; is a special node that represents the origin.

Each line represents a kernel, with 
\tikz[baseline=-0.1cm]{\draw[white] (0,-0.2) -- (0,0.2);\draw[kernel] (0,0) to (1,0);}
representing the kernel $K^N$, 
\tikz[baseline=-0.1cm] \draw[kernel1] (0,0) to (1,0);
represents a factor $K^N(t-s,y-x) - K^N(-s,-x)$, where
$(s,x)$ and $(t,y)$ are the coordinates of the starting and end
point respectively, and
a double barred arrow 
\tikz[baseline=-0.1cm] \draw[kernel2] (0,0) to (1,0);
represents a factor 
$K^N(t-s,y-x) - K^N(-s,-x) - \sum_{i=1}^{3}y_i\, \partial_i K^N(-s,-x)$.
Moreover \tikz[baseline=-0.1cm] \draw[kprime] (0,0) to (1,0);
represents $\sum_{i=1}^{3}\partial_i K^N$
These four factors will be equipped with labels $(a_e,r_e)= (3,0)$,
$(a_e,r_e)= (3,1)$, $(a_e,r_e)= (3,2)$, and $(a_e,r_e)=(4,0)$, respectively. We recall here that $K^N$ can be decomposed as $K^N=\sum_{n=1}^{N} K_n$, and the spatial derivatives for the double barred arrow are only applied to the first $N-1$ terms.
Moreover,
\tikz[baseline=-0.1cm]{\draw[white] (0,-0.2) -- (0,0.2); \draw[testfcn] (1,0) to (0,0);}
represents a generic test function $\phi_\lambda$ rescaled to 
scale $\lambda$, and finally
\tikz[baseline=-0.1cm] \draw[rho] (0,0) to (1,0);
represents an occurrence of a noise term $\bar \xi^N$ evaluated at the time-space point to which 
the corresponding dotted line is attached to. Products of dotted lines have to be interpreted as Wick products.
Moreover, a red polygon with $p$ dots inside, for instance \begin{tikzpicture}[baseline=-4]
\node		(mid)  	at (0,0) {};
\node[cumu4]	(mid-cumu) 	at (mid) {};
\node[dot] at (mid.north west) {};
\node[dot] at (mid.south west) {};
\node[dot] at (mid.north east) {};
\node[dot] at (mid.south east) {};
\end{tikzpicture}
with $p=4$, represents a joint cumulant of $p$-th order of $\bar \xi^N$, with each dot 
representing one of the corresponding space-time locations. 
With this notation at hand we see for instance that we have the identity
\begin{equation}
\label{eq:xiixibasic}
\begin{aligned}
(\hat\Pi_0^{N}\<Xi2>)(\varphi_0^{\lambda})
&= \begin{tikzpicture}[scale=0.35,baseline=0.3cm]
\node at (-2,-1)  [root] (root) {};
\node at (-2,1)  [dot] (left) {};
\node at (-2,3)  [dot] (left1) {};
\node at (0,1) [] (variable1) {};
\node at (0,3) [] (variable2) {};

\draw[testfcn] (left) to  (root);

\draw[kernel1] (left1) to  node[midway, left] {\tiny $(3,1)$} (left);
\draw[rho] (variable2) to (left1); 
\draw[rho] (variable1) to (left); 
\end{tikzpicture}\;
- \;
\begin{tikzpicture}[scale=0.35,baseline=0.3cm]
\node at (-2,-1)  [root] (root) {};
\node at (-2,1)  [dot] (left) {};
\node at (-2,3)  [dot] (top) {};
\node[cumu2n] (a) at (-4,2){};
\draw[cumu2] (a) ellipse (12pt and 24pt);

\draw[testfcn] (left) to  (root);

\draw[kernel] (top) to [bend left=60] node[midway, right] {\tiny $(3,0)$} (root);
\draw[] (a.south) node[dot] {} to (left);
\draw[] (a.north)node[dot] {} to (top); 
\end{tikzpicture}\;.
\end{aligned}
\end{equation} 
At this point we refer the reader who is not yet familiar with the kind of renormalisation and notation which is used here to~\cite[Section 5.1.1]{WongZakai}, where in a very similar context the symbol $\tau=\<Xi2>$ is analysed.
Note that the second cumulant of the exclusion process (i.e.\ its covariance) is simply given by 
the discrete heat kernel, which is a kernel of order three in the sense of Definition~\ref{def:orderofkernel}.
Hence, an application of Theorem~\ref{theo:ultimate1} shows that the second term above is bounded by a multiple times $\lambda^{-1}$ which, being deterministic, is as desired.
It remains to prove a similar bound for the first graph in~\eqref{eq:xiixibasic}. For this, we note that  
its $p$-th moment is given by
\begin{equation}\label{eq:pthmomentXiIXi}
\int\E\Big[ \prod_{i=1}^{p}\varphi_0^{\lambda}(z_i)[K^N(z_i^{(1)}-z_i^{(2)})-K^N(-z_i^{(2)})]
\Wick{\bar \xi_{t_{i}^{(1)}}^N(x_{i}^{(1)})\bar \xi_{t_{i}^{(2)}}^N(x_{i}^{(2)})}\Big]\, dz_i,
\end{equation}
where $dz$ denotes the integral over all time-space points $z_i^{(j)} = (t_i^{(j)},x_i^{(j)})$ above (recall that the integral over space is indeed a Riemann sum).
With Lemma~\ref{lem:diagram} we then see that the above expression equals
\begin{equation}
\sum_{\pi}
\int \prod_{i=1}^{p}\varphi_0^{\lambda}(z_i)[K^N(z_i^{(1)}-z_i^{(2)})-K^N(-z_i^{(2)})]\prod_{B\in\pi}\E_c(\{\xi_t^N(x):\, (t,x)\in B\})\, dz_i,
\end{equation}
where $\pi$ runs over all partitions in $\cP_{\{1,2\}}(\{1,2\}\times\{1,2,\ldots, p\})$.
As an example, for $p=2$ and $p=4$ respectively, the following two graphs show up,
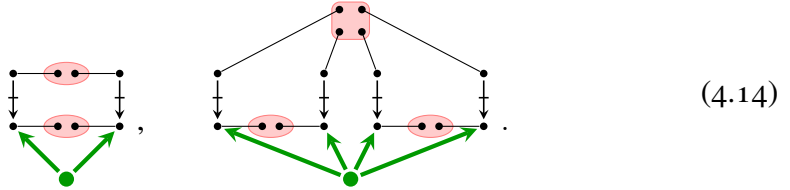
\begin{equation}
\label{eq:xiixi}
\begin{aligned}
\begin{tikzpicture}[scale=0.35,baseline=0.3cm]
\node at (0,-1)  [root] (root) {};
\node at (-2,1)  [dot] (left) {};
\node at (-2,3)  [dot] (left1) {};
\node at (2,1)   [dot] (right) {};
\node at (2,3)   [dot] (right1) {}; 
\node[cumu2n] (a) at (0,3){};
\draw[cumu2] (a) ellipse (24pt and 12pt);
\node[cumu2n] (b) at (0,1){};
\draw[cumu2] (b) ellipse (24pt and 12pt);

\draw[testfcn] (left) to  (root);
\draw[testfcn] (right) to (root);

\draw[kernel1] (left1) to   (left);
\draw[kernel1] (right1) to  (right);
\draw[] (a.west) node[dot] {} to (left1);
\draw[] (a.east) node[dot] {} to (right1);
\draw[] (b.west) node[dot] {} to (left);
\draw[] (b.east) node[dot] {} to (right);
\end{tikzpicture}\;,\qquad
\begin{tikzpicture}[scale=0.35,baseline=0.3cm]
\node at (0,-1)  [root] (root) {};
\node at (-5,1)  [dot] (vleft) {};
\node at (-5,3)  [dot] (vleft1) {};
\node at (-1,1)  [dot] (left)   {};
\node at (-1,3)  [dot] (left1) {};
\node at (1,1)   [dot] (right) {};
\node at (1,3)   [dot] (right1) {}; 
\node at (5,1)   [dot] (vright) {};
\node at (5,3)   [dot] (vright1){};
\node[cumu2n] (a) at (-3,1){};
\draw[cumu2] (a) ellipse (24pt and 12pt);
\node[cumu2n] (b) at (3,1){};
\draw[cumu2] (b) ellipse (24pt and 12pt);
\node[]		(cumu)  	at (0,5) {};
\node[cumu4]	(top-cumu) 	at (cumu) {};

\draw[testfcn] (left) to  (root);
\draw[testfcn] (right) to (root);
\draw[testfcn] (vright) to (root);
\draw[testfcn] (vleft) to (root);

\draw[kernel1] (left1) to   (left);
\draw[kernel1] (right1) to   (right);
\draw[kernel1] (vright1) to   (vright);
\draw[kernel1] (vleft1) to  (vleft);
\draw[] (a.west) node[dot] {} to (vleft);
\draw[] (a.east) node[dot] {} to (left);
\draw[] (b.west) node[dot] {} to (right);
\draw[] (b.east) node[dot] {} to (vright);
\draw[] (cumu.north west) node[dot] {} to (vleft1);
\draw[] (cumu.north east) node[dot] {}  to (vright1);
\draw[] (cumu.south west) node[dot] {} to (left1);
\draw[] (cumu.south east) node[dot] {} to (right1);
\end{tikzpicture}\;.
\end{aligned}
\end{equation} 
At this point, it only remains to make use of Theorem~\ref{thm:cycle} and to verify the assumptions of Theorem~\ref{theo:ultimate1} to obtain the desired bound for $\tau=\<Xi2>$. However, since it will turn out that for the remaining symbols the arguments will be very similar, we continue by laying the groundwork for bounding all terms altogether.
To that end it is convenient to have a shorthand for the graphical notation defined via a map $h$ that maps the symbols $\tau$ appearing below to its graphical representation, i.e.,
\begin{equ}\label{eq:h}
	h(\<Xi2>) = 
	\begin{tikzpicture}[scale=0.35,baseline=0.3cm]
	\node at (-2,-1)  [root] (root) {};
	\node at (-2,1)  [dot] (left) {};
	\node at (-2,3)  [dot] (left1) {};
	\node at (0,1) [] (variable1) {};
	\node at (0,3) [] (variable2) {};
	
	\draw[testfcn] (left) to  (root);
	
	\draw[kernel1] (left1) to  (left);
	\draw[rho] (variable2) to (left1); 
	\draw[rho] (variable1) to (left); 
	\end{tikzpicture}\;,\quad
	h(\<Xi3>)= \begin{tikzpicture}[scale=0.35,baseline=0.5cm]
	\node at (-2,-1)  [root] (root) {};
	\node at (-2,1)  [dot] (left) {};
	\node at (-2,3)  [dot] (left1) {};
	\node at (-2,5)  [dot] (left2) {};
	\node at (0,1) [] (variable1) {};
	\node at (0,3) [] (variable2) {};
	\node at (0,5) [] (variable3) {};
	
	\draw[testfcn] (left) to  (root);
	
	\draw[kernel1] (left1) to (left);
	\draw[kernel1] (left2) to (left1);
	\draw[rho] (variable3) to (left2); 
	\draw[rho] (variable2) to (left1); 
	\draw[rho] (variable1) to (left); 
	\end{tikzpicture}\;,\quad
	h(\<Xi4>)=	
	\begin{tikzpicture}[scale=0.35,baseline=0.5cm]
	\node at (-2,-1)  [root] (root) {};
	\node at (-2,1)  [dot] (left) {};
	\node at (-2,3)  [dot] (left1) {};
	\node at (-2,5)  [dot] (left2) {};
	\node at (-2,7)  [dot] (left3) {};
	\node at (0,1) [] (variable1) {};
	\node at (0,3) [] (variable2) {};
	\node at (0,5) [] (variable3) {};
	\node at (0,7) [] (variable4) {};
	
	\draw[testfcn] (left) to  (root);
	
	\draw[kernel2] (left1) to (left);
	\draw[kernel1] (left2) to (left1);
	\draw[kernel1] (left3) to (left2);
	\draw[rho] (variable4) to (left3);
	\draw[rho] (variable3) to (left2); 
	\draw[rho] (variable2) to (left1); 
	\draw[rho] (variable1) to (left); 
	\end{tikzpicture}\;,\quad
	h(\<Xi2X>)=
	\begin{tikzpicture}[scale=0.35,baseline=0.3cm]
	\node at (0,-1)  [root] (root) {};
	\node at (-1,1)  [dot] (left) {};
	\node at (1,1)  [dot] (left1) {};
	\node at (-1,3) [] (variable1) {};
	\node at (1,3) [] (variable2) {};
	
	\draw[testfcn] (left) to  (root);
	
	\draw[kernel2] (left1) to (left);
	\draw[rho] (variable2) to (left1); 
	\draw[rho] (variable1) to (left); 
	\draw[multx] (left1) to (root); 
	\end{tikzpicture}
	\;.
\end{equ}
Here, we used the notation \tikz[baseline=-0.1cm] \draw[multx] (0,0) to (1,0); 
for the kernel $(t,x) \mapsto x_i$ (which is of order $-1$).
In the sequel whenever we write that a set of time-space points gets contracted, we always mean that we replace the noise terms that are connected to them by a dotted line by a joint cumulant evaluated at the respective set of time-space points. Further below we analyse each term of the form $(\hat\Pi_0^N\tau)(\varphi_0^{\lambda})$, where $\tau$ is one of the symbols of Proposition~\ref{prop:momentbounds}. The result of that analysis is summarised in the following Remark.
\begin{remark}\label{rem:renorm}
	Let $\tau$ be one of the symbols in Proposition~\ref{prop:momentbounds}, then we can write
	\begin{equation}\label{eq:renorm}
	|(\hat\Pi_0^N\tau)(\varphi_0^{\lambda})|\leq\sum_{\hat\tau\in\CG_\tau} |\hat\tau|,
	\end{equation}
	where for each $\tau$ the set $\CG_\tau$ is a finite set of symbols, such that each $\hat\tau\in\CG_\tau$ is constructed in a way that obeys the following rules:
	\begin{itemize}
		\item[1.] consider the symbol $h(\tau)$ and partition its set of black nodes into two disjoint sets $A$ and $B$, where $A$ is either empty or contains at least two elements;
		\item[2.] consider a partition $\pi$ of $A$ such that each of its elements has cardinality at least two. Then for each set $a\in\pi$ contract all its elements;
		\item[3.] leave the nodes belonging to $B$ untouched;
		\item[4.] if after replacing each cumulant term by a cycle according to Theorem~\ref{thm:cycle}, where each kernel in this cycle has weight $a=3/2+\eta$, there is a subset of nodes $\bar\CCV$ of non-positive degree, i.e. is such that in the notation of Appendix~\ref{A}
		\begin{equation}
		\sum_{e\in\CCE_0(\bar\CCV)}\hat a_e\geq 5(|\bar\CCV|-1),
		\end{equation}
		then for at least one edge $e\in\CCE_0(\bar\CCV)$ we perform one of the following substitutions:
		\begin{equation}\label{e:subs}
		\begin{aligned}
		&\begin{tikzpicture}[scale=0.5, baseline=0cm]
		
		\draw[kernel1] (0,1) to (0,-1); 
		\node[cumu2n] (a) at (1,0){};
		\draw[cumu2] (a) ellipse (12pt and 24pt);

		\draw[] (a.north) node[dot] {} to (0,1);
		\draw[] (a.south) node[dot] {} to (0,-1);
		\end{tikzpicture}\to
		\begin{tikzpicture}[scale=0.5, baseline=0cm]
		
		\draw[kernelBig] (0,1) to (0,-1);
		
		\end{tikzpicture}\;, \quad
		\begin{tikzpicture}[scale=0.5, baseline=0cm]
		
		\draw[kernel1] (0,1) to (0,-1);

		\end{tikzpicture}\to
		\begin{tikzpicture}[scale=0.5, baseline=0cm]
		
		\node at (0,-1) [root] (root) {};
		
		\draw[kernel] (0,1) to [bend left=60] (0,-1);
		
		\end{tikzpicture}\;,\quad
		\begin{tikzpicture}[scale=0.5, baseline=0cm]
		
		\draw[kernel2] (0,1) to (0,-1);
		\node[cumu2n] (a) at (1,0){};
		\draw[cumu2] (a) ellipse (12pt and 24pt);

		\draw[] (a.north) node[dot] {} to (0,1);
		\draw[] (a.south) node[dot] {} to (0,-1);
		
		\end{tikzpicture}\to
		\begin{tikzpicture}[scale=0.5, baseline=0cm]
		
		\draw[kernelBig] (0,1) to (0,-1); 
		\end{tikzpicture}\;,\quad
		\begin{tikzpicture}[scale=0.5, baseline=0cm]
		
		\draw[kernel2] (0,1) to (0,-1);
		
		\end{tikzpicture}\to
		\begin{tikzpicture}[scale=0.5, baseline=0cm]
		
		\node at (0,-1) [root](root) {};
		
		\draw[kernel] (0,1) to [bend left=60] (0,-1);
		\end{tikzpicture}\;,\quad
		\begin{tikzpicture}[scale=0.5, baseline=0cm]
		
		\draw[kernel2] (0,1) to (0,-1);
		
		\end{tikzpicture}\to
		\begin{tikzpicture}[scale=0.5, baseline=0cm]
		
		\node at (0,-1) [root](root) {};
		
		\draw[kprime] (0,1) to [bend left=60] (0,-1);
		\end{tikzpicture}\;,\\
		&\begin{tikzpicture}[scale=0.5, baseline=0cm]

		\node at (0,0) (mid) {};
		\node[]		(cumu)  	at (1,0) {};
		\node[cumu3, rotate=90]  (top-cumu )at (cumu) {};

		\draw[] (cumu.north) node[dot] {} to (0,1);
		\draw[] (cumu.west) node[dot] {} to (0,0);
		\draw[] (cumu.south) node[dot] {} to (0,-1);
		
		\draw[kernel1] (0,1) to (0,0);
		\draw[kernel1] (0,0) to (0,-1);
		
		\end{tikzpicture}\to
		\begin{tikzpicture}[scale=0.5, baseline=0cm]
		
		\draw[kernelBig] (0,1) to node[labl, pos=0.45] {\tiny $2$} (0,-1);
		
		\end{tikzpicture}\;,\quad	
		\begin{tikzpicture}[scale=0.5, baseline=0cm]
		
		\node at (0,0) (mid) {};
		
		\node[]		(cumu)  	at (1,0) {};
		\node[cumu3, rotate=90]  (top-cumu) at (cumu) {};

		\draw[] (cumu.north) node[dot] {} to (0,1);
		\draw[] (cumu.west) node[dot] {} to (0,0);
		\draw[] (cumu.south) node[dot] {} to (0,-1);

		\draw[kernel1] (0,1) to (0,0);
		\draw[kernel2] (0,0) to (0,-1);
		
		\end{tikzpicture}\to
		\begin{tikzpicture}[scale=0.5, baseline=0cm]
		
		\draw[kernelBig] (0,1) to node[labl, pos=0.45] {\tiny $2$} (0,-1);
		
		\end{tikzpicture}\;.
		\end{aligned}
		\end{equation}
	\end{itemize}
	Here, the first zig-zag-line denotes the renormalised kernel \\
	$\Ren Q^{N}(s,x) = K^{N}(s,x)\E[\bar\xi_s^{N}(x)\bar\xi_0^{N}(0)] - c_N2^{3N}\delta_0(s,x)$. Since $ K^{N}(s,x)\E[\bar\xi_s^{N}(x)\bar\xi_0^{N}(0)]$ is an even function in space, and consequently kills polynomials of parabolic degree one, the kernel $\Ren Q^{N}$ comes equipped with a label $(6,-2)$. The second zig-zag-line denotes the renormalised kernel
	\begin{equation}
	\label{eq:Q2}
	\begin{aligned}
	&\Ren Q^{2,N}(\bar z) =
	\int
	K^{N}(-z)K^{N}(z-\bar z)\E_c(\{\bar\xi_0^{N}(0),\bar\xi_t^{N}(x),\bar\xi_{\bar t}^{N}(\bar x)\})\, dz - c_N^{(1)}\delta_0(\bar z),
	\end{aligned}
	\end{equation}
	where the constant $c_N^{(1)}$ was introduced in Section~\ref{S3.3} and we used the abbreviations $z=(t,x)$, and $\bar z=(\bar t,\bar x)$. One can now show that this kernel comes equipped with the label $(5.5,-1)$. We refer to~\eqref{e:defRen} for more details.
	Moreover, in case of the dotted line above a factor $x_i$ gets assigned to the test function.
\end{remark}
\begin{remark}
	We emphasize at this point that Remark~\eqref{rem:renorm} does not give a complete and precise description of the way the graphical representation of each of the symbols $h(\tau)$ is manipulated. Indeed, it turns out that some of the barred arrows will loose their bar, i.e., for some of the symbols we will do the substitution
	\begin{equation}\label{eq:substituteextra}
	\begin{tikzpicture}[scale=0.5, baseline=0cm]
	
	\draw[kernel1] (0,1) to (0,-1);

	\end{tikzpicture}\to
	\begin{tikzpicture}[scale=0.5, baseline=0cm]
	
	\draw[kernel] (0,1) to (0,-1);

	\end{tikzpicture}\;
	\end{equation}
	and similar for the kernel with two bars. The reason why this substitution is not listed in in the above remark is that any substitution of the type~\eqref{eq:substituteextra} is always accompanied with a substitution of the type~\eqref{e:subs}, and it turns out that these are really decisive when checking the applicability of Theorem~\ref{theo:ultimate1}.
\end{remark}
\begin{remark}
	Having the representation of~\eqref{eq:renorm} at hand to bound the $p$-th moment of the left hand side we use the bound
	\begin{equation}
	\Big(|(\hat\Pi_0^N\tau)(\varphi_0^{\lambda})|\Big)^p\leq C_p \sum_{\hat\tau\in\CG_\tau} (|\hat\tau|)^p,
	\end{equation}
	where $C_p$ is a constant depending solely on $p$. This bound allows us to focus on each term $\hat\tau$ separately. 
\end{remark}
We now ``prove'' Remark~\ref{rem:renorm}.
For the symbol $\tau=\<Xi2>$, the identity~\eqref{eq:xiixibasic} shows that Remark~\ref{rem:renorm} is true. The first symbol on the right hand side in~\eqref{eq:xiixibasic} corresponds to the choice where $A$ is the empty set, whereas the second symbol corresponds to the choice where $A$ contains both black nodes. Note at this point that the reason for the appearance of the second term in~\eqref{eq:xiixibasic} is exactly caused by the fact that the contraction of the two noise terms yields an object of negative degree, so that the kernel gets moved to the origin, which corresponds to the second substitution listed in~\eqref{e:subs}. We now proceed with the remaining symbols.
It follows as in~\cite[Sec.~5.3.1]{WongZakai} that one has the identity
\begin{equ}[e:decompPiXiTwo2]
	\bigl(\hat \Pi_0^{(N)} \<Xi2X>\bigr)(\phi_\lambda) = \;
	\begin{tikzpicture}[scale=0.35,baseline=0.3cm]
	\node at (0,-1)  [root] (root) {};
	\node at (-1,1)  [dot] (left) {};
	\node at (1,1)  [dot] (left1) {};
	\node at (-1,3) [] (variable1) {};
	\node at (1,3) [] (variable2) {};
	
	\draw[testfcn] (left) to  (root);
	
	\draw[kernel2] (left1) to (left);
	\draw[rho] (variable2) to (left1); 
	\draw[rho] (variable1) to (left); 
	\draw[multx] (left1) to (root); 
	\end{tikzpicture}
	\; + \;
	\begin{tikzpicture}[scale=0.35,baseline=0.3cm]
	\node at (0,-1)  [root] (root) {};
	\node at (-1,1)  [dot] (left) {};
	\node at (1,1) [dot] (right) {};
	
	\draw[testfcn] (left) to  (root);
	\draw[kernelBig] (left) to (right);
	\draw[multx] (right) to (root);
	\end{tikzpicture}
	\; - \;
	\begin{tikzpicture}[scale=0.35,baseline=0.3cm]
	\node at (0,-1)  [root] (root) {};
	\node at (-1,1)  [dot] (left) {};
	\node at (1,1) [dot] (right) {};
	\node[cumu2n] (a) at (0,3){};
	\draw[cumu2] (a) ellipse (24pt and 12pt);
	\draw[testfcn] (left) to  (root);
	
	\draw[kernelx] (right) to (root);
	\draw[] (a.west) node[dot] {} to (left);
	\draw[] (a.east) node[dot] {} to (right);
	\end{tikzpicture}
	\; - \;
	\begin{tikzpicture}[scale=0.35,baseline=0.3cm]
	\node at (0,-1)  [root] (root) {};
	\node at (-1,1)  [dot] (left) {};
	\node at (1,1) [dot] (right) {};
	\node[cumu2n] (a) at (0,3){};
	\draw[cumu2] (a) ellipse (24pt and 12pt);
	
	\draw[testfcnx] (left) to  (root);
	
	\draw[kprimex] (right) to (root);
	\draw[] (a.west) node[dot] {} to (left);
	\draw[] (a.east) node[dot] {} to (right);
	\end{tikzpicture}\;.
\end{equ}
Here we used the notation
\tikz[baseline=-0.1cm] \draw[testfcnx] (0,0) -- (1,0);
for the test function $(t,x) \mapsto x_i\phi^\lambda(t,x)$, where $i$ ranges between $1$ and $3$,
\tikz[baseline=-0.1cm] \draw[kernelx] (0,0) to (1,0); for the kernel
$(t,x) \mapsto x_iK(t,x)$, and similarly for \tikz[baseline=-0.1cm] \draw[kprimex] (0,0) to (1,0);. 
Here, $i$ ranges from $1$ to $3$. To compare the above expression with the construction outlined in Remark~\ref{rem:renorm} we note that the first term on the right hand side corresponds to the case where $A=\emptyset$, whereas the remaining terms correspond to the case where $A$ consists of the two black nodes. In that case the only admissible partition that can be formed out of $A$ is the partition whose only element is $A$ itself. Contracting the two black nodes then yields
\begin{equ}
	\begin{tikzpicture}[scale=0.35,baseline=0.3cm]
\node at (0,-1)  [root] (root) {};
\node at (-1,1)  [dot] (left) {};
\node at (1,1)  [dot] (left1) {};
\node[cumu2n] (a) at (0,3){};
\draw[cumu2] (a) ellipse (24pt and 12pt);

\draw[testfcn] (left) to  (root);

\draw[kernel2] (left1) to (left);
\draw[multx] (left1) to (root); 
\draw[] (a.west) node[dot] {} to (left);
\draw[] (a.east) node[dot] {} to (left1);
\end{tikzpicture}\,.
	\end{equ}
Since the double barred arrow comes with a label $a_e=3$ and the cumulant is simply the discrete heat kernel, it comes also with a weight $a_e=3$ (in this specific case there is no need to add an extra weight $\eta$), so that the sum of the two weights over the edge connecting the two black nodes is $6$. We are therefore in the situation in which we need to perform the substitutions outlined in~\eqref{e:subs}. The second symbol in~\eqref{e:decompPiXiTwo2} is a result of the third substitution in~\eqref{e:subs}, the third symbol results from the fourth substitution and the fourth symbol results from the fifth substitution. We recall that we remarked after~\eqref{e:subs} that in case the fifth substitution is performed we multiply the test function additionally with $x_i$.
Hence, this justifies Remark~\ref{rem:renorm} in this case. We turn to $\tau=\<Xi3>$. An expansion into Wick products together with the definition of the renormalised model yields
\begin{equs}
	\label{eq:Xi3dec}
	&(\hat\Pi_0^{N}\<Xi3>)(\phi_0^{\lambda})=\\
	&\begin{tikzpicture}[scale=0.35,baseline=0.5cm]
	\node at (-2,-1)  [root] (root) {};
	\node at (-2,1)  [dot] (left) {};
	\node at (-2,3)  [dot] (left1) {};
	\node at (-2,5)  [dot] (left2) {};
	\node at (0,1) [] (variable1) {};
	\node at (0,3) [] (variable2) {};
	\node at (0,5) [] (variable3) {};
	
	\draw[testfcn] (left) to  (root);
	
	\draw[kernel1] (left1) to (left);
	\draw[kernel1] (left2) to (left1);
	\draw[rho] (variable3) to (left2); 
	\draw[rho] (variable2) to (left1); 
	\draw[rho] (variable1) to (left); 
	\end{tikzpicture}
	\; + \;
	\left(
	\begin{tikzpicture}[scale=0.35,baseline=0.5cm]
	\node at (-2,-1)  [root] (root) {};
	\node at (-2,1)  [dot] (left) {};
	\node at (-2,3)  [dot] (left1) {};
	\node at (-2,5)  [dot] (left2) {};
	\node at (0,5) [] (variable3) {};
	\node[cumu2n] (a) at (0,2){};
	\draw[cumu2] (a) ellipse (12pt and 24pt);
	
	\draw[testfcn] (left) to  (root);6
	
	\draw[kernel1] (left1) to (left);
	\draw[kernel1] (left2) to (left1);
	\draw[rho] (variable3) to (left2); 
	\draw[] (a.north) node[dot] {} to (left1);
	\draw[] (a.south) node[dot] {} to (left);
	
	\end{tikzpicture}
	\; - c_N\times\;
	\begin{tikzpicture}[scale=0.35,baseline=0.5cm]
	\node at (-2,-1)  [root] (root) {};
	\node at (-2,1)  [dot] (left) {};
	\node at (-2,5)  [dot] (left2) {};
	\node at (0,5) [] (variable3) {};
	
	\draw[testfcn] (left) to  (root);
	
	\draw[kernel1] (left2) to (left);
	\draw[rho] (variable3) to (left2); 
	\end{tikzpicture}
	\right)
	\; + \;
	\begin{tikzpicture}[scale=0.35,baseline=0.5cm]
	\node at (-2,-1)  [root] (root) {};
	\node at (-2,1)  [dot] (left) {};
	\node at (-2,3)  [dot] (jnct) {};
	\node at (-0.75,3)  [dot] (left1) {};
	\node at (-2,5)  [dot] (left2) {};
	\node at (0.5,3) [] (variable2) {};
	\node[cumu2n] (a) at (-2,3){};
	\draw[cumu2] (a) ellipse (12pt and 24pt);
	
	\draw[testfcn] (left) to  (root);
	
	\draw[kernel1] (left1) to (left);
	\draw[kernel1] (left2) to (left1);
	\draw[rho] (variable2) to (left1); 
	\draw[] (a.north) node[dot] {} to (left2);
	\draw[] (a.south) node[dot] {} to (left);
	\end{tikzpicture}
	\; + \;
	\begin{tikzpicture}[scale=0.35,baseline=0.5cm]
	\node at (-2,-1)  [root] (root) {};
	\node at (-2,1)  [dot] (bottom) {};
	\node at (-2,3)  [dot] (mid) {};
	\node at (-2,5)  [dot] (top) {};
	\node at (0,1) [] (variable2) {};
	\node[cumu2n] (a) at (0,4){};
	\draw[cumu2] (a) ellipse (12pt and 24pt);
	
	\draw[testfcn] (bottom) to  (root);
	
	\draw[kernel1] (mid) to (bottom);
	\draw[kernel1] (top) to (mid);
	\draw[rho] (variable2) to (bottom); 
	\draw[] (a.north) node[dot] {} to (top);
	\draw[] (a.south) node[dot] {} to (mid); 
	\end{tikzpicture}\;
	+\;\left( 
	\begin{tikzpicture}[scale=0.35, baseline=0.5cm]
	\node at (-2,-1)  [root] (root) {};
	\node at (-2,1)  [dot] (bottom) {};
	\node at (-2,3)  [dot] (mid) {};
	\node at (-2,5)  [dot] (top) {};
	\node[]		(cumu)  	at (0,3) {};
	\node[cumu3,rotate=90]	(top-cumu) 	at (cumu) {};
	
	\draw[testfcn] (bottom) to  (root);
	
	\draw[kernel1] (mid) to (bottom);
	\draw[kernel1] (top) to (mid);
	\draw[] (cumu.north) node[dot] {} to (top);
	\draw[] (cumu.west) node[dot] {} to (mid);
	\draw[] (cumu.south) node[dot] {} to (bottom);
	\end{tikzpicture}\;
	-c_N^{(1)}\right)\\
	&=:\I +\II + \III + \IV + \V.
\end{equs}
With the definition of the second renormalised kernel at hand we can then write,
\begin{equ}
	\label{eq:Xi3dec2}
	\begin{aligned}
		&\II
		= 	\begin{tikzpicture}[scale=0.35,baseline=0.5cm]
		\node at (-2,-1)  [root] (root) {};
		\node at (-2,1)  [dot] (left) {};
		\node at (-2,3)  [dot] (left1) {};
		\node at (-2,5)  [dot] (left2) {};
		\node at (0,5) [] (variable3) {};
		
		\draw[testfcn] (left) to  (root);
		
		\draw[kernelBig] (left1) to (left);
		\draw[kernel1] (left2) to (left1);
		\draw[rho] (variable3) to (left2); 
		\end{tikzpicture}\;
		-\;
		\begin{tikzpicture}[scale=0.35,baseline=0.5cm]
		\node at (-2,-1)  [root] (root) {};
		\node at (-2,1)  [dot] (left) {};
		\node at (-2,3)  [dot] (left1) {};
		\node at (-2,5)  [dot] (left2) {};
		\node at (0,5) [] (variable3) {};
		\node[cumu2n] (a) at (-4,2){};
		\draw[cumu2] (a) ellipse (12pt and 24pt);
		
		\draw[testfcn] (left) to  (root);
		
		\draw[kernel1] (left2) to (left1);
		\draw[rho] (variable3) to (left2); 
		\draw[kernel] (left1) to [bend left=60] (root); 
		\draw[] (a.north) node[dot] {} to (left1);
		\draw[] (a.south) node[dot] {} to (left);
		\end{tikzpicture}\;	,
		\quad \IV=
		- \;
		\begin{tikzpicture}[scale=0.35,baseline=0.3cm]
		\node at (-2,-1)  [root] (root) {};
		\node at (-2,1)  [dot] (left) {};
		\node at (-2,3)  [dot] (left1) {};
		\node at (-2,5) [dot] (left2) {};
		\node at (-4,1)  [] (variable) {};
		\node[cumu2n] (a) at (-4,4){};
		\draw[cumu2] (a) ellipse (12pt and 24pt);
		
		\draw[testfcn] (left) to  (root);
		
		\draw[kernel1] (left1) to (left);
		\draw[kernel] (left2) to [bend left=60] (root); 
		\draw[rho] (left) to (variable); 
		\draw[] (a.north) node[dot] {} to (left2);
		\draw[] (a.south) node[dot] {} to (left1); 
		\end{tikzpicture}\; ,\\
		&\V= 
		-\;
		\begin{tikzpicture}[scale=0.35, baseline=0.5cm]
		\node at (-2,-1)  [root] (root) {};
		\node at (-2,1)  [dot] (left) {};
		\node at (-2,3)  [dot] (left1) {};
		\node at (-2,5)  [dot] (left2) {};
		\node[]		(cumu)  	at (-4,3) {};
		\node[cumu3,rotate=90]	(top-cumu) 	at (cumu) {};
		
		\draw[testfcn] (left) to  (root);
		
		\draw[kernel] (left1) to (left);
		\draw[kernel] (left2) to [bend left=60] (root);
		\draw[] (cumu.north) node[dot] {} to (left2); 
		\draw[] (cumu.west) node[dot] {} to (left1);
		\draw[] (cumu.south) node[dot] {} to (left); 
		\end{tikzpicture}\;
		-\;
		\begin{tikzpicture}[scale=0.35, baseline=0.5cm]
		\node at (-2,-1)  [root] (root) {};
		\node at (-2,1)  [dot] (left) {};
		\node at (-2,3)  [dot] (left1) {};
		\node at (-2,5)  [dot] (left2) {};
		\node[]		(cumu)  	at (-4,3) {};
		\node[cumu3,rotate=90]	(top-cumu) 	at (cumu) {};
		
		\draw[testfcn] (left) to  (root);
		
		\draw[kernel] (left1) to [bend left=60] (root);
		\draw[kernel1] (left2) to (left1);
		\draw[] (cumu.north) node[dot] {} to (left2); 
		\draw[] (cumu.west) node[dot] {} to (left1);
		\draw[] (cumu.south) node[dot] {} to (left); 
		\end{tikzpicture}\;.
	\end{aligned}
\end{equ}
Here, we used that $K^{N}$ kills polynomials
to get the equality for $\IV$ and the choice of $c_N^{(1)}$ to obtain the relation for $\V$. It is then straightforward to check that also this symbol falls into the framework of Remark~\ref{rem:renorm}.
We turn to the last symbol, namely to $\tau=\<Xi4>$. To do so, we apply the definition of $\hat\Pi^N\<Xi4>$ from Section~\ref{S3}. 
Expanding this expression, we see that it equals the sum of~\eqref{e:exprXi4} and~\eqref{eq:Xi4dec} below.  Note that the graphs in~\eqref{e:exprXi4} are actually those from~\cite[Equations~(5.20)--(5.22)]{WongZakai} and it is here that we make use of the renormalisation constants $c_N^{(2,2)}$ and $c_N^{(2,3)}$.
\begin{equs}
	\label{e:exprXi4}
	&\begin{tikzpicture}[scale=0.35,baseline=0.5cm]
	\node at (-2,-1)  [root] (root) {};
	\node at (-2,1)  [dot] (left) {};
	\node at (-2,3)  [dot] (left1) {};
	\node at (-2,5)  [dot] (left2) {};
	\node at (-2,7)  [dot] (left3) {};
	\node at (0,1) [] (variable1) {};
	\node at (0,3) [] (variable2) {};
	\node at (0,5) [] (variable3) {};
	\node at (0,7) [] (variable4) {};
	
	\draw[testfcn] (left) to  (root);
	
	\draw[kernel2] (left1) to (left);
	\draw[kernel1] (left2) to (left1);
	\draw[kernel1] (left3) to (left2);
	\draw[rho] (variable4) to (left3);
	\draw[rho] (variable3) to (left2); 
	\draw[rho] (variable2) to (left1); 
	\draw[rho] (variable1) to (left); 
	\end{tikzpicture}
	\;-\;
	\begin{tikzpicture}[scale=0.35,baseline=0.5cm]
	\node at (-2,-1)  [root] (root) {};
	\node at (-2,1)  [dot] (left) {};
	\node at (-2,3)  [dot] (left1) {};
	\node at (-2,5)  [dot] (left2) {};
	\node at (-2,7)  [dot] (left3) {};
	\node at (-4,1) [] (variable) {};
	\node at (0,7) [] (variable3) {};
	\node[cumu2n] (a) at (-4,4) {};
	\draw[cumu2] (a) ellipse (12pt and 24pt);

	\draw[testfcn] (left) to  (root);
	
	\draw[kernel2] (left1) to (left);
	\draw[kernel1] (left3) to (left2);
	\draw[kernel] (left2) to [bend left=60] (root);
	\draw[] (a.north) node[dot] {} to (left2);
	\draw[] (a.south) node[dot] {} to (left1);
	\draw[rho] (variable3) to (left3); 
	\draw[rho] (variable) to (left); 
	\end{tikzpicture}
	\;+\;
	\begin{tikzpicture}[scale=0.35,baseline=0.5cm]
	\node at (-2,-1)  [root] (root) {};
	\node at (-2,1)  [dot] (left) {};
	\node at (-2,3)  [dot] (left1) {};
	\node at (-2,5)  [dot] (left2) {};
	\node at (-2,7)  [dot] (left3) {};
	\node at (0,1) [] (variable1) {};
	\node at (0,5) [] (variable3) {};
	\node[cumu2n] (a) at (-4,5) {};
	\draw[cumu2] (a) ellipse (12pt and 24pt);

	\draw[testfcn] (left) to  (root);
	
	\draw[kernel2] (left1) to (left);
	\draw[kernel1] (left2) to (left1);
	\draw[kernel1] (left3) to (left2);
	\draw[rho] (variable3) to (left2); 
	\draw[] (a.north) node[dot] {} to (left3);
	\draw[] (a.south) node[dot] {} to (left1);
	\draw[rho] (variable1) to (left); 
	\end{tikzpicture}
	\;-\;
	\begin{tikzpicture}[scale=0.35,baseline=0.5cm]
	\node at (-2,-1)  [root] (root) {};
	\node at (-2,1)  [dot] (left) {};
	\node at (-2,3)  [dot] (left1) {};
	\node at (-2,5)  [dot] (left2) {};
	\node at (-2,7)  [dot] (left3) {};
	\node[cumu2n] (a) at (-4,6) {};
	\draw[cumu2] (a) ellipse (12pt and 24pt);

	\node at (-4,1) [] (variable) {};
	\node at (-4,3) [] (variable1) {};
	
	\draw[testfcn] (left) to  (root);
	
	\draw[kernel2] (left1) to (left);
	\draw[kernel1] (left2) to (left1);
	\draw[kernel] (left3) [bend left=60] to (root); 
	\draw[rho] (variable1) to (left1); 
	\draw[rho] (variable) to (left); 
	\draw[] (a.north) node[dot] {} to (left3);
	\draw[] (a.south) node[dot] {} to (left2);
	\end{tikzpicture}
	\\ &
	\;+\;
	\begin{tikzpicture}[scale=0.35,baseline=0.5cm]
	\node at (-2,-1)  [root] (root) {};
	\node at (-2,1)  [dot] (left) {};
	\node at (-2,3)  [dot] (left1) {};
	\node at (-2,5)  [dot] (left2) {};
	\node at (-2,7)  [dot] (left3) {};
	\node at (0,1) [] (variable) {};
	\node at (0,7) [] (variable3) {};
	
	\draw[testfcn] (left) to  (root);
	
	\draw[kernel2] (left1) to (left);
	\draw[kernelBig] (left2) to (left1);
	\draw[kernel1] (left3) to (left2);
	\draw[rho] (variable3) to (left3); 
	\draw[rho] (variable) to (left); 
	\end{tikzpicture}
	\;+\;
	\begin{tikzpicture}[scale=0.35,baseline=0.5cm]
	\node at (-2,-1)  [root] (root) {};
	\node at (-2,1)  [dot] (left) {};
	\node at (-2,3)  [dot] (left1) {};
	\node at (-2,5)  [dot] (left2) {};
	\node at (-2,7)  [dot] (left3) {};
	\node at (0,5) [] (variable2) {};
	\node at (0,7) [] (variable3) {};
	
	\draw[testfcn] (left) to  (root);
	
	\draw[kernelBig] (left1) to (left);
	\draw[kernel1] (left2) to (left1);
	\draw[kernel1] (left3) to (left2);
	\draw[rho] (variable3) to (left3); 
	\draw[rho] (variable2) to (left2); 
	\end{tikzpicture}
	\;-\;
	\begin{tikzpicture}[scale=0.35,baseline=0.5cm]
	\node at (-2,-1)  [root] (root) {};
	\node at (-2,1)  [dot] (left) {};
	\node at (-2,3)  [dot] (left1) {};
	\node at (-2,5)  [dot] (left2) {};
	\node at (-2,7)  [dot] (left3) {};
	\node at (0,5) [] (variable2) {};
	\node at (0,7) [] (variable3) {};
	\node[cumu2n] (a) at (-4,2) {};
	\draw[cumu2] (a) ellipse (12pt and 24pt);
	
	\draw[testfcn] (left) to  (root);

	\draw[kernel1] (left2) to (left1);
	\draw[kernel1] (left3) to (left2);
	\draw[kernel] (left1) [bend left=60] to (root);
	\draw[] (a.north) node[dot] {} to (left1);
	\draw[] (a.south) node[dot] {} to (left);
	\draw[rho] (variable3) to (left3); 
	\draw[rho] (variable2) to (left2); 
	\end{tikzpicture}
	\;-\;
	\begin{tikzpicture}[scale=0.35,baseline=0.5cm]
	\node at (-2,-1)  [root] (root) {};
	\node at (-2,1)  [dot] (left) {};
	\node at (-2,3)  [dot] (left1) {};
	\node at (-2,5)  [dot] (left2) {};
	\node at (-2,7)  [dot] (left3) {};
	\node at (0,5) [] (variable2) {};
	\node at (0,7) [] (variable3) {};
	\node[cumu2n] (a) at (-4,2) {};
	\draw[cumu2] (a) ellipse (12pt and 24pt);
	
	\draw[testfcnx] (left) to  (root);

	\draw[kernel1] (left2) to (left1);
	\draw[kernel1] (left3) to (left2);
	\draw[kprime] (left1) [bend left=60] to (root);
	\draw[] (a.north) node[dot] {} to (left1);
	\draw[] (a.south) node[dot] {} to (left);
	\draw[rho] (variable3) to (left3); 
	\draw[rho] (variable2) to (left2); 
	\end{tikzpicture}
	\;+\;
	\begin{tikzpicture}[scale=0.35,baseline=0.5cm]
	\node at (-2,-1)  [root] (root) {};
	\node at (-2,1)  [dot] (left) {};
	\node at (-2,3)  [dot] (left1) {};
	\node at (-2,5)  [dot] (left2) {};
	\node at (-2,7)  [dot] (left3) {};
	\node at (0,3) [] (variable2) {};
	\node at (0,7) [] (variable4) {};
	\node[cumu2n] (a) at (-4,3) {};
	\draw[cumu2] (a) ellipse (12pt and 24pt);

	\draw[testfcn] (left) to  (root);
	
	\draw[kernel2] (left1) to (left);
	\draw[kernel1] (left2) to (left1);
	\draw[kernel1] (left3) to (left2);
	\draw[rho] (variable4) to (left3); 
	\draw[] (a.north) node[dot] {} to (left2);
	\draw[] (a.south) node[dot] {} to (left);
	\draw[rho] (variable2) to (left1); 
	\end{tikzpicture}
	\;+\;
	\begin{tikzpicture}[scale=0.35,baseline=0.5cm]
	\node at (-2,-1)  [root] (root) {};
	\node at (-2,1)  [dot] (left) {};
	\node at (-2,3)  [dot] (left1) {};
	\node at (-2,5)  [dot] (left2) {};
	\node at (-2,7)  [dot] (left3) {};
	\node at (0,3) [] (variable2) {};
	\node at (0,5) [] (variable3) {};
	\node[cumu2n] (a) at (-4,4) {};
	\draw[cumu2] (a) ellipse (12pt and 24pt);

	\draw[testfcn] (left) to  (root);
	
	\draw[kernel2] (left1) to (left);
	\draw[kernel1] (left2) to (left1);
	\draw[kernel1] (left3) to (left2);
	\draw[rho] (variable3) to (left2); 
	\draw[] (a.north) node[dot] {} to (left3);
	\draw[] (a.south) node[dot] {} to (left);
	\draw[rho] (variable2) to (left1); 
	\end{tikzpicture}
	\\&
	\;+\;
	\begin{tikzpicture}[scale=0.35,baseline=0.5cm]
	\node at (-2,-1)  [root] (root) {};
	\node at (-2,1)  [dot] (left) {};
	\node at (-2,3)  [dot] (left1) {};
	\node at (-2,5)  [dot] (left2) {};
	\node at (-2,7)  [dot] (left3) {};
	\node[cumu2n] (a) at (-4,2) {};
	\node[cumu2n] (b) at (-4,6) {};
	\draw[cumu2] (a) ellipse (12pt and 24pt);
	\draw[cumu2] (b) ellipse (12pt and 24pt);

	\draw[testfcn] (left) to  (root);

	\draw[kernel] (left3) [bend left=60] to (root); 
	\draw[kernel1] (left2) to (left1);
	\draw[kernel] (left1) [bend left=60] to (root); 
	\draw[] (a.north) node[dot] {} to (left1);
	\draw[] (a.south) node[dot] {} to (left);
	\draw[] (b.north) node[dot] {} to (left3);
	\draw[] (b.south) node[dot] {} to (left2);
	\end{tikzpicture}
	\;+\;
	\begin{tikzpicture}[scale=0.35,baseline=0.5cm]
	\node at (-2,-1)  [root] (root) {};
	\node at (-2,1)  [dot] (left) {};
	\node at (-2,3)  [dot] (left1) {};
	\node at (-2,5)  [dot] (left2) {};
	\node at (-2,7)  [dot] (left3) {};
	\node[cumu2n] (a) at (-4,2) {};
	\node[cumu2n] (b) at (-4,6) {};
	\draw[cumu2] (a) ellipse (12pt and 24pt);
	\draw[cumu2] (b) ellipse (12pt and 24pt);

	\draw[testfcnx] (left) to  (root);

	\draw[kernel] (left3) [bend left=60] to (root); 
	\draw[kernel1] (left2) to (left1);
	\draw[kprime] (left1) [bend left=60] to (root); 
	\draw[] (a.north) node[dot] {} to (left1);
	\draw[] (a.south) node[dot] {} to (left);
	\draw[] (b.north) node[dot] {} to (left3);
	\draw[] (b.south) node[dot] {} to (left2);
	\end{tikzpicture}
	\;-\;
	\begin{tikzpicture}[scale=0.35,baseline=0.5cm]
	\node at (-2,-1)  [root] (root) {};
	\node at (-2,1)  [dot] (left) {};
	\node at (-2,3)  [dot] (left1) {};
	\node at (-2,5)  [dot] (left2) {};
	\node at (-2,7)  [dot] (left3) {};
	\node[cumu2n] (a)  at (-4,6)  {};
	\draw[cumu2] (a) ellipse (12pt and 24pt);	
	
	\draw[testfcn] (left) to  (root);
	
	\draw[kernelBig] (left1) to (left);
	\draw[kernel1] (left2) to (left1);
	\draw[] (a.north) node[dot] {} to (left3); 
	\draw[] (a.south) node[dot] {} to (left2); 
	\draw[kernel] (left3) [bend left=60] to (root);
	\end{tikzpicture}
	\;-\;
	\begin{tikzpicture}[scale=0.35,baseline=0.5cm]
	\node at (-2,-1)  [root] (root) {};
	\node at (-2,1)  [dot] (left) {};
	\node at (-2,3)  [dot] (left1) {};
	\node at (-2,5)  [dot] (left2) {};
	\node at (-2,7)  [dot] (left3) {};
	\node[cumu2n] (a) at (-4,4) {};
	\node[cumu2n] (b) at (-6,4) {};
	\draw[cumu2] (a) ellipse (12pt and 24pt);
	\draw[cumu2] (b) ellipse (12pt and 24pt);
	
	\draw[testfcn] (left) to  (root);
	
	\draw[kernel1] (left3) to (left2);
	\draw[kernel] (left2) [bend left=60] to (root);
	\draw[kernel2] (left1) to (left);
	\draw[] (a.north) node[dot] {} to (left2);
	\draw[] (a.south) node[dot] {} to (left1);
	\draw[] (b.north) node[dot] {} to (left3);
	\draw[] (b.south) node[dot] {} to (left);
	
	\end{tikzpicture}
	\\&
	\;-\;
	\begin{tikzpicture}[scale=0.35,baseline=0.5cm]
	\node at (-2,-1)  [root] (root) {};
	\node at (-2,1)  [dot] (left) {};
	\node at (-2,3)  [dot] (left1) {};
	\node at (-2,5)  [dot] (left2) {};
	\node at (-2,7)  [dot] (left3) {};
	\node[cumu2n] (a) at (-4,5) {};
	\node[cumu2n] (b) at (0,3) {};
	\draw[cumu2] (a) ellipse (12pt and 24pt);
	\draw[cumu2] (b) ellipse (12pt and 24pt);
	
	\draw[testfcn] (left) to  (root);
	
	\draw[kernel] (left1) [bend right=60] to (root);
	\draw[kernel1] (left2) to (left1);
	\draw[kernel1] (left3) to (left2);
	\draw[] (a.north) node[dot] {} to (left3);
	\draw[] (a.south) node[dot] {} to (left1);
	\draw[] (b.north) node[dot] {} to (left2);
	\draw[] (b.south) node[dot] {} to (left);
	\end{tikzpicture}
	\;-\;\begin{tikzpicture}[scale=0.35,baseline=0.5cm]
	\node at (-2,-1)  [root] (root) {};
	\node at (-2,1)  [dot] (left) {};
	\node at (-2,3)  [dot] (left1) {};
	\node at (-2,5)  [dot] (left2) {};
	\node at (-2,7)  [dot] (left3) {};
	\node[cumu2n] (a) at (-4,5) {};
	\node[cumu2n] (b) at (0,3) {};
	\draw[cumu2] (a) ellipse (12pt and 24pt);
	\draw[cumu2] (b) ellipse (12pt and 24pt);
	
	\draw[testfcnx] (left) to  (root);
	
	\draw[kprime] (left1) [bend right=60] to (root);
	\draw[kernel1] (left2) to (left1);
	\draw[kernel1] (left3) to (left2);
	\draw[] (a.north) node[dot] {} to (left3);
	\draw[] (a.south) node[dot] {} to (left1);
	\draw[] (b.north) node[dot] {} to (left2);
	\draw[] (b.south) node[dot] {} to (left);
	\end{tikzpicture}
	\;-\:
	\begin{tikzpicture}[scale=0.35,baseline=0.5cm]
	\node at (-2,-1)  [root] (root) {};
	\node at (-2,1)  [dot] (left) {};
	\node at (-2,3)  [dot] (left1) {};
	\node at (-2,5)  [dot] (left2) {};
	\node at (-2,7)  [dot] (left3) {};
	\node[cumu2n] (a) at (-4,5) {};
	\node[cumu2n] (b) at (0,3) {};
	\draw[cumu2] (a) ellipse (12pt and 24pt);
	\draw[cumu2] (b) ellipse (12pt and 24pt);
	
	\draw[testfcn] (left) to  (root);
	
	\draw[kernel] (left1)  to (left);
	\draw[kernel] (left2) [bend right=60] to (root);
	\draw[kernel1] (left3) to (left2);
	\draw[] (a.north) node[dot] {} to (left3);
	\draw[] (a.south) node[dot] {} to (left1);
	\draw[] (b.north) node[dot] {} to (left2);
	\draw[] (b.south) node[dot] {} to (left);
	\end{tikzpicture}\;
	-\:
	\begin{tikzpicture}[scale=0.35,baseline=0.5cm]
	\node at (-2,-1)  [root] (root) {};
	\node at (-2,1)  [dot] (left) {};
	\node at (-2,3)  [dot] (left1) {};
	\node at (-2,5)  [dot] (left2) {};
	\node at (-2,7)  [dot] (left3) {};
	\node[cumu2n] (a) at (-4,5) {};
	\node[cumu2n] (b) at (0,3) {};
	\draw[cumu2] (a) ellipse (12pt and 24pt);
	\draw[cumu2] (b) ellipse (12pt and 24pt);
	
	\draw[testfcn] (left) to  (root);
	
	\draw[kernel] (left1)  to (left);
	\draw[kernel] (left2) to (left1);
	\draw[kernel] (left3) [bend left=80]  to (root);
	\draw[] (a.north) node[dot] {} to (left3);
	\draw[] (a.south) node[dot] {} to (left1);
	\draw[] (b.north) node[dot] {} to (left2);
	\draw[] (b.south) node[dot] {} to (left);
	\end{tikzpicture}
	\\&
	\;-\;
	\begin{tikzpicture}[scale=0.35,baseline=0.5cm]
	\node at (-2,-1)  [root] (root) {};
	\node at (-2,1)  [dot] (left) {};
	\node at (-2,3)  [dot] (left1) {};
	\node at (-2,5)  [dot] (left2) {};
	\node at (-2,7)  [dot] (left3) {};
	\node[cumu2n] (a) at (-4,4) {};
	\draw[cumu2] (a) ellipse (12pt and 24pt);
	
	\draw[testfcn] (left) to  (root);
	
	\draw[kernel] (left1) [bend left =60] to (root);
	\draw[kernelBig] (left2) to (left1);
	\draw[kernel] (left3) to (left2);
	\draw[] (a.north) node[dot] {} to (left3);
	\draw[] (a.south) node[dot] {} to (left);
	\end{tikzpicture}
	\;-\;
	\begin{tikzpicture}[scale=0.35,baseline=0.5cm]
	\node at (-2,-1)  [root] (root) {};
	\node at (-2,1)  [dot] (left) {};
	\node at (-2,3)  [dot] (left1) {};
	\node at (-2,5)  [dot] (left2) {};
	\node at (-2,7)  [dot] (left3) {};
	\node[cumu2n] (a) at (-4,4) {};
	\draw[cumu2] (a) ellipse (12pt and 24pt);
	
	\draw[testfcnx] (left) to  (root);
	
	\draw[kprime] (left1) [bend left=60 ]to (root);
	\draw[kernelBig] (left2) to (left1);
	\draw[kernel] (left3) to (left2);
	\draw[] (a.north) node[dot] {} to (left3);
	\draw[] (a.south) node[dot] {} to (left);
	\end{tikzpicture}\;.
\end{equs}	
Recall that we used the notation
\tikz[baseline=-0.1cm] \draw[testfcnx] (0,0) -- (1,0);
for the test function $(t,x) \mapsto x_i\phi^\lambda(t,x)$, where $i$ ranges between $1$ and $3$.
Moreover, if we set $\tilde \phi(t,x) = x_i \phi(t,x)$, then $\tilde \phi$ is again an 
admissible test function and one has $x_i\phi^\lambda(t,x) = \lambda \tilde \phi^\lambda(t,x)$.
As a consequence, when applying Theorem~\ref{theo:ultimate1} 
to a graph with test function $\tilde \phi$,
one gains an additional power of $\lambda$. This however is exactly compensated by the fact 
that in this case one instance of the kernel $K^N$ is replaced by $\partial_i K^N$, thus lowering
the total degree of the graph by one.  
Since cumulants of third and fourth order do not vanish in our situation, new graphs that show up are
\begin{equs}{}
&\begin{tikzpicture}[scale=0.35,baseline=0.8cm]
\node at (-2,-1)  [root] (root) {};
\node at (-2,1)  [dot] (left) {};
\node at (-2,3)  [dot] (left1) {};
\node at (-2,5)  [dot] (left2) {};
\node at (-2,7)  [dot] (left3) {};
\node at (0,1) [] (variable1) {};
\node[] (a) at (0,5) {};
\node[cumu3,rotate =90] (cumu) at (a) {};

\draw[testfcn] (left) to  (root);

\draw[kernel2] (left1) to (left);
\draw[kernel1] (left2) to (left1);
\draw[kernel1] (left3) to (left2);
\draw[] (a.north) node[dot] {} to (left3);
\draw[] (a.west) node[dot] {} to (left2);
\draw[] (a.south) node[dot] {} to (left1);
\draw[rho] (variable1) to (left); 
\end{tikzpicture}
\;+\;
\begin{tikzpicture}[scale=0.35,baseline=0.8cm]
\node at (-2,-1)  [root] (root) {};
\node at (-2,1)  [dot] (left) {};
\node at (-2,3)  [dot] (left1) {};
\node at (-2,5)  [dot] (left2) {};
\node at (-2,7)  [dot] (left3) {};
\node[] (a) at (0,4) {};
\node[cumu3, rotate=90] (cumu) at (a) {};
\node at (-4,3) [] (variable2) {};

\draw[testfcn] (left) to  (root);

\draw[kernel2] (left1) to (left);
\draw[kernel1] (left2) to (left1);
\draw[kernel1] (left3) to (left2);
\draw[rho] (variable2) to (left1); 
\draw[] (a.north) node[dot] {} to (left3);
\draw[] (a.west) node[dot] {} to (left2);
\draw[] (a.south) node[dot] {} to (left);
\end{tikzpicture}
\;+\;
\begin{tikzpicture}[scale=0.35,baseline=0.8cm]
\node at (-2,-1)  [root] (root) {};
\node at (-2,1)  [dot] (left) {};
\node at (-2,3)  [dot] (left1) {};
\node at (-2,5)  [dot] (left2) {};
\node at (-2,7)  [dot] (left3) {};
\node[] (a)  at (0,4) {};
\node[cumu3, rotate=90] (cumu) at (a) {};
\node at (-4,5)  [] (variable3) {};

\draw[testfcn] (left) to  (root);

\draw[kernel2] (left1) to (left);
\draw[kernel1] (left2) to (left1);
\draw[kernel1] (left3) to (left2);
\draw[rho] (variable3) to (left2); 
\draw[] (a.north) node[dot] {} to (left3);
\draw[] (a.west) node[dot] {} to (left1);
\draw[] (a.south) node[dot] {} to (left);
\end{tikzpicture}
\;+\;
\left(\begin{tikzpicture}[scale=0.35,baseline=0.8cm]
\node at (-2,-1)  [root] (root) {};
\node at (-2,1)  [dot] (left) {};
\node at (-2,3)  [dot] (left1) {};
\node at (-2,5)  [dot] (left2) {};
\node at (-2,7)  [dot] (left3) {};
\node[] (a) at (0,3) {};
\node[cumu3, rotate=90] (cumu) at (a) {};
\node at (0,7) [] (variable4) {};

\draw[testfcn] (left) to  (root);

\draw[kernel2] (left1) to (left);
\draw[kernel1] (left2) to (left1);
\draw[kernel1] (left3) to (left2);
\draw[rho] (variable4) to (left3);
\draw[] (a.north) node[dot] {} to (left2);
\draw[] (a.west) node[dot] {} to (left1);
\draw[] (a.south) node[dot] {} to (left);
\end{tikzpicture}
\;-\;
c_N^{(1)}\times
\begin{tikzpicture}[scale=0.35,baseline=0.8cm]
\node at (-2,-1)  [root] (root) {};
\node at (-2,1)  [dot] (left) {};
\node at (-2,7)  [dot] (left3) {};
\node at (0,7) [] (variable4) {};

\draw[testfcn] (left) to  (root);

\draw[kernel1] (left3) to (left);
\draw[rho] (variable4) to (left3); 
\end{tikzpicture}\right)\;\\ &+\; 
\left(
\begin{tikzpicture}[scale=0.35,baseline=0.8cm]
\node at (-2,-1)  [root] (root) {};
\node at (-2,1)  [dot] (left) {};
\node at (-2,3)  [dot] (left1) {};
\node at (-2,5)  [dot] (left2) {};
\node at (-2,7)  [dot] (left3) {};
\node[] (a) at (0,4) {};
\node[cumu4,rotate=45] (cumu) at (a) {};

\draw[testfcn] (left) to  (root);

\draw[kernel2] (left1) to (left);
\draw[kernel1] (left2) to (left1);
\draw[kernel1] (left3) to (left2);
\draw[] (a.north) node[dot] {} to (left3);
\draw[] (a.west) node[dot] {} to (left2);
\draw[] (a.east) node[dot] {} to (left1);
\draw[] (a.south) node[dot] {} to (left);
\end{tikzpicture}
\;-\;
C_N^{(2,1)}\right)
=\I+\II+\III+\IV+\V.\label{eq:Xi4dec}
\end{equs}
We now use the fact that $K^N$ kills constants, to rewrite $\I$ as
\begin{equation}
\label{eq:Xi4Irewrite}
\I=\;-\;
\begin{tikzpicture}[scale=0.35,baseline=0.5cm]
\node at (-2,-1)  [root] (root) {};
\node at (-2,1)  [dot] (left) {};
\node at (-2,3)  [dot] (left1) {};
\node at (-2,5)  [dot] (left2) {};
\node at (-2,7)  [dot] (left3) {};
\node at (-4,1) [] (variable1) {};
\node[] (a) at (-4,5) {};
\node[cumu3, rotate=90] (cumu) at (a) {};

\draw[testfcn] (left) to  (root);

\draw[kernel2] (left1) to (left);
\draw[kernel] (left2) to (left1);
\draw[kernel] (left3) [bend left=60] to (root);
\draw[] (a.north) node[dot] {} to (left3);
\draw[] (a.east) node[dot] {} to (left2);
\draw[] (a.south) node[dot] {} to (left1);
\draw[rho] (variable1) to (left); 
\end{tikzpicture}
\;-\;
\begin{tikzpicture}[scale=0.35,baseline=0.5cm]
\node at (-2,-1)  [root] (root) {};
\node at (-2,1)  [dot] (left) {};
\node at (-2,3)  [dot] (left1) {};
\node at (-2,5)  [dot] (left2) {};
\node at (-2,7)  [dot] (left3) {};
\node at (-4,1) [] (variable1) {};
\node[] (a) at (-4,5) {};
\node[cumu3, rotate=90] (cumu) at (a) {};

\draw[testfcn] (left) to  (root);

\draw[kernel2] (left1) to (left);
\draw[kernel] (left2) [bend left=60] to (root);
\draw[kernel1] (left3)  to (left2);
\draw[] (a.north) node[dot] {} to (left3);
\draw[] (a.east) node[dot] {} to (left2);
\draw[] (a.south) node[dot] {} to (left1);
\draw[rho] (variable1) to (left); 
\end{tikzpicture}\;.
\end{equation}
Using the definition of $c_N^{(1)}$, we further rewrite $\IV$ as,
\begin{equation}
\label{eq:Xi4IVrewrite}
\IV= \begin{tikzpicture}[scale=0.35,baseline=0.5cm]
\node at (-2,-1)  [root] (root) {};
\node at (-2,1)  [dot] (left) {};
\node at (-2,5)  [dot] (left2) {};
\node at (-2,7)  [dot] (left3) {};
\node at (0,7) [] (variable3) {};

\draw[testfcn] (left) to  (root);

\draw[kernelBig] (left2) to node[labl, pos=0.45] {\tiny $2$}  (left);
\draw[kernel1] (left3) to (left2);
\draw[rho] (variable3) to (left3); 
\end{tikzpicture}
\;-\;
\begin{tikzpicture}[scale=0.35,baseline=0.5cm]
\node at (-2,-1)  [root] (root) {};
\node at (-2,1)  [dot]  (left) {};
\node at (-2,3)  [dot]  (left1) {};
\node at (-2,5)  [dot] (left2) {};
\node at (-2,7)  [dot] (left3) {};
\node at (0,7) [] (variable3) {};
\node[] (a) at (-4,3) {};
\node[cumu3, rotate=90] (cumu) at (a) {};

\draw[testfcn] (left) to  (root);
\draw[kernel]  (left1) to (left);
\draw[kernel]  (left2) [bend left=60] to (root);
\draw[kernel1] (left3) to (left2);
\draw[] (a.north) node[dot] {} to (left2);
\draw[] (a.east) node[dot] {} to (left1);
\draw[] (a.south) node[dot] {} to (left);
\draw[rho] (variable3) to (left3);
\end{tikzpicture}
\;-\;
\begin{tikzpicture}[scale=0.35,baseline=0.5cm]
\node at (-2,-1)  [root] (root) {};
\node at (-2,1)  [dot]  (left) {};
\node at (-2,3)  [dot]  (left1) {};
\node at (-2,5)  [dot] (left2) {};
\node at (-2,7)  [dot] (left3) {};
\node at (0,7) [] (variable3) {};
\node[] (a) at (-4,3) {};
\node[cumu3, rotate=90] (cumu) at (a) {};

\draw[testfcn] (left) to  (root);
\draw[kernel]  (left1) [bend left=60] to (root);
\draw[kernel1]  (left2)  to (left1);
\draw[kernel1] (left3) to (left2);
\draw[] (a.north) node[dot] {} to (left2);
\draw[] (a.east) node[dot] {} to (left1);
\draw[] (a.south) node[dot] {} to (left);
\draw[rho] (variable3) to (left3);
\end{tikzpicture}
\;-\;
\begin{tikzpicture}[scale=0.35,baseline=0.5cm]
\node at (-2,-1)  [root] (root) {};
\node at (-2,1)  [dot]  (left) {};
\node at (-2,3)  [dot]  (left1) {};
\node at (-2,5)  [dot] (left2) {};
\node at (-2,7)  [dot] (left3) {};
\node at (0,7) [] (variable3) {};
\node[] (a) at (-4,3) {};
\node[cumu3, rotate=90] (cumu) at (a) {};

\draw[testfcnx] (left) to  (root);
\draw[kprime]  (left1) [bend left=60] to (root);
\draw[kernel1]  (left2)  to (left1);
\draw[kernel1] (left3) to (left2);
\draw[] (a.north) node[dot] {} to (left2);
\draw[] (a.east) node[dot] {} to (left1);
\draw[] (a.south) node[dot] {} to (left);
\draw[rho] (variable3) to (left3);
\end{tikzpicture}.
\end{equation}
Finally, using the definition of $c_N^{(2,1)}$, we write
\begin{equation}
\label{eq:Xi4Vrewrite} 
\V= \;-\;
\begin{tikzpicture}[scale=0.35,baseline=0.5cm]
\node at (-2,-1)  [root] (root) {};
\node at (-2,1)  [dot] (left) {};
\node at (-2,3)  [dot] (left1) {};
\node at (-2,5)  [dot] (left2) {};
\node at (-2,7)  [dot] (left3) {};
\node[]  (a) at (-4,4) {};
\node[cumu4, rotate=45] (cumu) at (a) {};

\draw[testfcn] (left) to  (root);

\draw[kernel] (left1) to (left);
\draw[kernel] (left2) to (left1);
\draw[kernel] (left3) [bend left=60] to (root);
\draw[] (a.north) node[dot] {} to (left3);
\draw[] (a.east) node[dot] {} to (left2);
\draw[] (a.west) node[dot] {} to (left1);
\draw[] (a.south) node[dot] {} to (left);
\end{tikzpicture}
\;-\;
\begin{tikzpicture}[scale=0.35,baseline=0.5cm]
\node at (-2,-1)  [root] (root) {};
\node at (-2,1)  [dot] (left) {};
\node at (-2,3)  [dot] (left1) {};
\node at (-2,5)  [dot] (left2) {};
\node at (-2,7)  [dot] (left3) {};
\node[]  (a) at (-4,4) {};
\node[cumu4, rotate=45] (cumu) at (a) {};

\draw[testfcn] (left) to  (root);

\draw[kernel] (left1) to (left);
\draw[kernel] (left2) [bend left=60] to (root);
\draw[kernel1] (left3) to (left2);
\draw[] (a.north) node[dot] {} to (left3);
\draw[] (a.east) node[dot] {} to (left2);
\draw[] (a.west) node[dot] {} to (left1);
\draw[] (a.south) node[dot] {} to (left);
\end{tikzpicture}
\;-\;
\begin{tikzpicture}[scale=0.35,baseline=0.5cm]
\node at (-2,-1)  [root] (root) {};
\node at (-2,1)  [dot] (left) {};
\node at (-2,3)  [dot] (left1) {};
\node at (-2,5)  [dot] (left2) {};
\node at (-2,7)  [dot] (left3) {};
\node[]  (a) at (-4,4) {};
\node[cumu4, rotate=45] (cumu) at (a) {};

\draw[testfcn] (left) to  (root);

\draw[kernel] (left1) [bend left=60] to (root);
\draw[kernel1] (left2) to (left1);
\draw[kernel1] (left3) to (left2);
\draw[] (a.north) node[dot] {} to (left3);
\draw[] (a.east) node[dot] {} to (left2);
\draw[] (a.west) node[dot] {} to (left1);
\draw[] (a.south) node[dot] {} to (left);
\end{tikzpicture}
\;-\;
\begin{tikzpicture}[scale=0.35,baseline=0.5cm]
\node at (-2,-1)  [root] (root) {};
\node at (-2,1)  [dot] (left) {};
\node at (-2,3)  [dot] (left1) {};
\node at (-2,5)  [dot] (left2) {};
\node at (-2,7)  [dot] (left3) {};
\node[]  (a) at (-4,4) {};
\node[cumu4, rotate=45] (cumu) at (a) {};

\draw[testfcnx] (left) to  (root);

\draw[kprime] (left1) [bend left=60] to (root);
\draw[kernel1] (left2) to (left1);
\draw[kernel1] (left3) to (left2);
\draw[] (a.north) node[dot] {} to (left3);
\draw[] (a.east) node[dot] {} to (left2);
\draw[] (a.west) node[dot] {} to (left1);
\draw[] (a.south) node[dot] {} to (left);
\end{tikzpicture}
\end{equation}
We see that also all graphs above fall into the framework of Remark~\ref{rem:renorm}. Hence, we can conclude.
\begin{remark}\label{rem:correction}
	In the next section we will derive conditions on the graph appearing in the current section under which we can apply Theorem~\ref{theo:ultimate1} to finally deduce the desired bounds. It will however turn out that not all graphs satisfy these conditions, namely there are some that fail~\eqref{eq:5thcondition} below. Fortunately it is possible to decompose them in such a way that they satisfy all conditions needed. In what follows we list these graphs and we show how to decompose them so that Proposition~\ref{prop:pcondition} applies to them. For the sake of transparency we also draw the node in red that correspond to the choice of $\bar\CCV$ below, that makes~\eqref{eq:5thcondition} fail.
	The first graph that fails~\eqref{eq:5thcondition} is the last graph in~\eqref{eq:Xi3dec2} which we may rewrite as
	\begin{equation}\label{eq:1stbadgraph}
	\;
	\begin{tikzpicture}[scale=0.35, baseline=0.5cm]
	\node at (-2,-1)  [root] (root) {};
	\node at (-2,1)  [dot] (left) {};
	\node at (-2,3)  [dot, red] (left1) {};
	\node at (-2,5)  [dot] (left2) {};
	\node[]		(cumu)  	at (-4,3) {};
	\node[cumu3,rotate=90]	(top-cumu) 	at (cumu) {};
	
	\draw[testfcn] (left) to  (root);
	
	\draw[kernel] (left1) to [bend left=60] (root);
	\draw[kernel1] (left2) to (left1);
	\draw[] (cumu.north) node[dot] {} to (left2); 
	\draw[] (cumu.west) node[dot] {} to (left1);
	\draw[] (cumu.south) node[dot] {} to (left); 
	\end{tikzpicture}\;=
	\;
	\begin{tikzpicture}[scale=0.35, baseline=0.5cm]
	\node at (-2,-1)  [root] (root) {};
	\node at (-2,1)  [dot] (left) {};
	\node at (-2,3)  [dot] (left1) {};
	\node at (-2,5)  [dot] (left2) {};
	\node[]		(cumu)  	at (-4,3) {};
	\node[cumu3,rotate=90]	(top-cumu) 	at (cumu) {};
	
	\draw[testfcn] (left) to  (root);
	
	\draw[kernel] (left1) to [bend left=60] (root);
	\draw[kernel] (left2) to (left1);
	\draw[] (cumu.north) node[dot] {} to (left2); 
	\draw[] (cumu.west) node[dot] {} to (left1);
	\draw[] (cumu.south) node[dot] {} to (left); 
	\end{tikzpicture}\;-
	\;
	\begin{tikzpicture}[scale=0.35, baseline=0.5cm]
	\node at (-2,-1)  [root] (root) {};
	\node at (-2,1)  [dot] (left) {};
	\node at (-2,3)  [dot] (left1) {};
	\node at (-2,5)  [dot] (left2) {};
	\node[]		(cumu)  	at (-4,3) {};
	\node[cumu3,rotate=90]	(top-cumu) 	at (cumu) {};
	
	\draw[testfcn] (left) to  (root);
	
	\draw[kernel] (left1) to [bend left=60] (root);
	\draw[kernel] (left2) to [bend left=60] (root);
	\draw[] (cumu.north) node[dot] {} to (left2); 
	\draw[] (cumu.west) node[dot] {} to (left1);
	\draw[] (cumu.south) node[dot] {} to (left); 
	\end{tikzpicture}\;.
	\end{equation}
	The second graph that does not meet the conditions of Proposition~\ref{prop:pcondition} is the second graph in~\eqref{eq:Xi4Irewrite}, which we may write as
	\begin{equation}\label{eq:2ndbadgraph}
	\;
	\begin{tikzpicture}[scale=0.35,baseline=0.5cm]
	\node at (-2,-1)  [root] (root) {};
	\node at (-2,1)  [dot] (left) {};
	\node at (-2,3)  [dot] (left1) {};
	\node at (-2,5)  [dot,red] (left2) {};
	\node at (-2,7)  [dot] (left3) {};
	\node at (-4,1) [] (variable1) {};
	\node[] (a) at (-4,5) {};
	\node[cumu3, rotate=90] (cumu) at (a) {};

	\draw[testfcn] (left) to  (root);
	
	\draw[kernel2] (left1) to (left);
	\draw[kernel] (left2) [bend left=60] to (root);
	\draw[kernel1] (left3)  to (left2);
	\draw[] (a.north) node[dot] {} to (left3);
	\draw[] (a.east) node[dot] {} to (left2);
	\draw[] (a.south) node[dot] {} to (left1);
	\draw[rho] (variable1) to (left); 
	\end{tikzpicture}\;=
	\;
	\begin{tikzpicture}[scale=0.35,baseline=0.5cm]
	\node at (-2,-1)  [root] (root) {};
	\node at (-2,1)  [dot] (left) {};
	\node at (-2,3)  [dot] (left1) {};
	\node at (-2,5)  [dot] (left2) {};
	\node at (-2,7)  [dot] (left3) {};
	\node at (-4,1) [] (variable1) {};
	\node[] (a) at (-4,5) {};
	\node[cumu3, rotate=90] (cumu) at (a) {};

	\draw[testfcn] (left) to  (root);
	
	\draw[kernel2] (left1) to (left);
	\draw[kernel] (left2) [bend left=60] to (root);
	\draw[kernel] (left3)  to (left2);
	\draw[] (a.north) node[dot] {} to (left3);
	\draw[] (a.east) node[dot] {} to (left2);
	\draw[] (a.south) node[dot] {} to (left1);
	\draw[rho] (variable1) to (left); 
	\end{tikzpicture}\;-\;
	\begin{tikzpicture}[scale=0.35,baseline=0.5cm]
	\node at (-2,-1)  [root] (root) {};
	\node at (-2,1)  [dot] (left) {};
	\node at (-2,3)  [dot] (left1) {};
	\node at (-2,5)  [dot] (left2) {};
	\node at (-2,7)  [dot] (left3) {};
	\node at (-4,1) [] (variable1) {};
	\node[] (a) at (-4,5) {};
	\node[cumu3, rotate=90] (cumu) at (a) {};

	\draw[testfcn] (left) to  (root);
	
	\draw[kernel2] (left1) to (left);
	\draw[kernel] (left2) [bend left=60] to (root);
	\draw[kernel] (left3) [bend left=60] to (root);
	\draw[] (a.north) node[dot] {} to (left3);
	\draw[] (a.east) node[dot] {} to (left2);
	\draw[] (a.south) node[dot] {} to (left1);
	\draw[rho] (variable1) to (left); 
	\end{tikzpicture}\;.
	\end{equation}
	Moreover, we need to rewrite the third graph in~\eqref{eq:Xi4IVrewrite} as follows
	\begin{equation}\label{eq:3rdbadgraph}
	\;
	\begin{tikzpicture}[scale=0.35,baseline=0.5cm]
	\node at (-2,-1)  [root] (root) {};
	\node at (-2,1)  [dot]  (left) {};
	\node at (-2,3)  [dot, red]  (left1) {};
	\node at (-2,5)  [dot] (left2) {};
	\node at (-2,7)  [dot] (left3) {};
	\node at (0,7) [] (variable3) {};
	\node[] (a) at (-4,3) {};
	\node[cumu3, rotate=90] (cumu) at (a) {};

	\draw[testfcn] (left) to  (root);
	\draw[kernel]  (left1) [bend left=60] to (root);
	\draw[kernel1]  (left2)  to (left1);
	\draw[kernel1] (left3) to (left2);
	\draw[] (a.north) node[dot] {} to (left2);
	\draw[] (a.east) node[dot] {} to (left1);
	\draw[] (a.south) node[dot] {} to (left);
	\draw[rho] (variable3) to (left3);
	\end{tikzpicture}
	\;=
	\;
	\begin{tikzpicture}[scale=0.35,baseline=0.5cm]
	\node at (-2,-1)  [root] (root) {};
	\node at (-2,1)  [dot]  (left) {};
	\node at (-2,3)  [dot]  (left1) {};
	\node at (-2,5)  [dot] (left2) {};
	\node at (-2,7)  [dot] (left3) {};
	\node at (0,7) [] (variable3) {};
	\node[] (a) at (-4,3) {};
	\node[cumu3, rotate=90] (cumu) at (a) {};

	\draw[testfcn] (left) to  (root);
	\draw[kernel]  (left1) [bend left=60] to (root);
	\draw[kernel]  (left2)  to (left1);
	\draw[kernel1] (left3) to (left2);
	\draw[] (a.north) node[dot] {} to (left2);
	\draw[] (a.east) node[dot] {} to (left1);
	\draw[] (a.south) node[dot] {} to (left);
	\draw[rho] (variable3) to (left3);
	\end{tikzpicture}
	\;-
	\;
	\begin{tikzpicture}[scale=0.35,baseline=0.5cm]
	\node at (-2,-1)  [root] (root) {};
	\node at (-2,1)  [dot]  (left) {};
	\node at (-2,3)  [dot]  (left1) {};
	\node at (-2,5)  [dot] (left2) {};
	\node at (-2,7)  [dot] (left3) {};
	\node at (0,7) [] (variable3) {};
	\node[] (a) at (-4,3) {};
	\node[cumu3, rotate=90] (cumu) at (a) {};

	\draw[testfcn] (left) to  (root);
	\draw[kernel]  (left1) [bend left=60] to (root);
	\draw[kernel]  (left2) [bend left=60]  to (root);
	\draw[kernel1] (left3) to (left2);
	\draw[] (a.north) node[dot] {} to (left2);
	\draw[] (a.east) node[dot] {} to (left1);
	\draw[] (a.south) node[dot] {} to (left);
	\draw[rho] (variable3) to (left3);
	\end{tikzpicture}
	\;,
	\end{equation}
	finally the last three graphs in the representation of $\V$ above need to be rewritten as 
	\begin{equation}\label{eq:4thbadgraph}
	\begin{tikzpicture}[scale=0.35,baseline=0.5cm]
	\node at (-2,-1)  [root] (root) {};
	\node at (-2,1)  [dot] (left) {};
	\node at (-2,3)  [dot] (left1) {};
	\node at (-2,5)  [dot, red] (left2) {};
	\node at (-2,7)  [dot] (left3) {};
	\node[]  (a) at (-4,4) {};
	\node[cumu4, rotate=45] (cumu) at (a) {};

	\draw[testfcn] (left) to  (root);
	
	\draw[kernel] (left1) to (left);
	\draw[kernel] (left2) [bend left=60] to (root);
	\draw[kernel1] (left3) to (left2);
	\draw[] (a.north) node[dot] {} to (left3);
	\draw[] (a.east) node[dot] {} to (left2);
	\draw[] (a.west) node[dot] {} to (left1);
	\draw[] (a.south) node[dot] {} to (left);
	\end{tikzpicture}
	\; =
	\begin{tikzpicture}[scale=0.35,baseline=0.5cm]
	\node at (-2,-1)  [root] (root) {};
	\node at (-2,1)  [dot] (left) {};
	\node at (-2,3)  [dot] (left1) {};
	\node at (-2,5)  [dot] (left2) {};
	\node at (-2,7)  [dot] (left3) {};
	\node[]  (a) at (-4,4) {};
	\node[cumu4, rotate=45] (cumu) at (a) {};

	\draw[testfcn] (left) to  (root);
	
	\draw[kernel] (left1) to (left);
	\draw[kernel] (left2) [bend left=60] to (root);
	\draw[kernel] (left3) to (left2);
	\draw[] (a.north) node[dot] {} to (left3);
	\draw[] (a.east) node[dot] {} to (left2);
	\draw[] (a.west) node[dot] {} to (left1);
	\draw[] (a.south) node[dot] {} to (left);
	\end{tikzpicture}
	\;-
	\begin{tikzpicture}[scale=0.35,baseline=0.5cm]
	\node at (-2,-1)  [root] (root) {};
	\node at (-2,1)  [dot] (left) {};
	\node at (-2,3)  [dot] (left1) {};
	\node at (-2,5)  [dot] (left2) {};
	\node at (-2,7)  [dot] (left3) {};
	\node[]  (a) at (-4,4) {};
	\node[cumu4, rotate=45] (cumu) at (a) {};

	\draw[testfcn] (left) to  (root);
	
	\draw[kernel] (left1) to (left);
	\draw[kernel] (left2) [bend left=60] to (root);
	\draw[kernel] (left3) to [bend left =60] (root);
	\draw[] (a.north) node[dot] {} to (left3);
	\draw[] (a.east) node[dot] {} to (left2);
	\draw[] (a.west) node[dot] {} to (left1);
	\draw[] (a.south) node[dot] {} to (left);
	\end{tikzpicture},
	\;
	\end{equation}
	\begin{equation}\label{eq:5thbadgraph}
	\begin{tikzpicture}[scale=0.35,baseline=0.5cm]
	\node at (-2,-1)  [root] (root) {};
	\node at (-2,1)  [dot] (left) {};
	\node at (-2,3)  [dot, red] (left1) {};
	\node at (-2,5)  [dot, blue!50!white] (left2) {};
	\node at (-2,7)  [dot] (left3) {};
	\node[]  (a) at (-4,4) {};
	\node[cumu4, rotate=45] (cumu) at (a) {};

	\draw[testfcn] (left) to  (root);
	
	\draw[kernel] (left1) [bend left=60] to (root);
	\draw[kernel1] (left2) to (left1);
	\draw[kernel1] (left3) to (left2);
	\draw[] (a.north) node[dot] {} to (left3);
	\draw[] (a.east) node[dot] {} to (left2);
	\draw[] (a.west) node[dot] {} to (left1);
	\draw[] (a.south) node[dot] {} to (left);
	\end{tikzpicture}
	\; =
	\begin{tikzpicture}[scale=0.35,baseline=0.5cm]
	\node at (-2,-1)  [root] (root) {};
	\node at (-2,1)  [dot] (left) {};
	\node at (-2,3)  [dot] (left1) {};
	\node at (-2,5)  [dot] (left2) {};
	\node at (-2,7)  [dot] (left3) {};
	\node[]  (a) at (-4,4) {};
	\node[cumu4, rotate=45] (cumu) at (a) {};

	\draw[testfcn] (left) to  (root);
	
	\draw[kernel] (left1) [bend left=60] to (root);
	\draw[kernel] (left2) to (left1);
	\draw[kernel] (left3) to (left2);
	\draw[] (a.north) node[dot] {} to (left3);
	\draw[] (a.east) node[dot] {} to (left2);
	\draw[] (a.west) node[dot] {} to (left1);
	\draw[] (a.south) node[dot] {} to (left);
	\end{tikzpicture}
	\;-
	\begin{tikzpicture}[scale=0.35,baseline=0.5cm]
	\node at (-2,-1)  [root] (root) {};
	\node at (-2,1)  [dot] (left) {};
	\node at (-2,3)  [dot] (left1) {};
	\node at (-2,5)  [dot] (left2) {};
	\node at (-2,7)  [dot] (left3) {};
	\node[]  (a) at (-4,4) {};
	\node[cumu4, rotate=45] (cumu) at (a) {};

	\draw[testfcn] (left) to  (root);
	
	\draw[kernel] (left1) [bend left=60] to (root);
	\draw[kernel] (left2) to [bend left=60] (root);
	\draw[kernel] (left3) to (left2);
	\draw[] (a.north) node[dot] {} to (left3);
	\draw[] (a.east) node[dot] {} to (left2);
	\draw[] (a.west) node[dot] {} to (left1);
	\draw[] (a.south) node[dot] {} to (left);
	\end{tikzpicture}
	\;-
	\begin{tikzpicture}[scale=0.35,baseline=0.5cm]
	\node at (-2,-1)  [root] (root) {};
	\node at (-2,1)  [dot] (left) {};
	\node at (-2,3)  [dot] (left1) {};
	\node at (-2,5)  [dot] (left2) {};
	\node at (-2,7)  [dot] (left3) {};
	\node[]  (a) at (-4,4) {};
	\node[cumu4, rotate=45] (cumu) at (a) {};

	\draw[testfcn] (left) to  (root);
	
	\draw[kernel] (left1) [bend left=60] to (root);
	\draw[kernel] (left2) to (left1);
	\draw[kernel] (left3) to [bend left=60] (root);
	\draw[] (a.north) node[dot] {} to (left3);
	\draw[] (a.east) node[dot] {} to (left2);
	\draw[] (a.west) node[dot] {} to (left1);
	\draw[] (a.south) node[dot] {} to (left);
	\end{tikzpicture}
	\;+
	\begin{tikzpicture}[scale=0.35,baseline=0.5cm]
	\node at (-2,-1)  [root] (root) {};
	\node at (-2,1)  [dot] (left) {};
	\node at (-2,3)  [dot] (left1) {};
	\node at (-2,5)  [dot] (left2) {};
	\node at (-2,7)  [dot] (left3) {};
	\node[]  (a) at (-4,4) {};
	\node[cumu4, rotate=45] (cumu) at (a) {};

	\draw[testfcn] (left) to  (root);
	
	\draw[kernel] (left1) [bend left=60] to (root);
	\draw[kernel] (left2) to [bend left=60] (root);
	\draw[kernel] (left3) to [bend left=60] (root);
	\draw[] (a.north) node[dot] {} to (left3);
	\draw[] (a.east) node[dot] {} to (left2);
	\draw[] (a.west) node[dot] {} to (left1);
	\draw[] (a.south) node[dot] {} to (left);
	\end{tikzpicture},
	\;
	\end{equation}
	and 
	\begin{equation}\label{eq:6thbadgraph}
	\begin{tikzpicture}[scale=0.35,baseline=0.5cm]
	\node at (-2,-1)  [root] (root) {};
	\node at (-2,1)  [dot] (left) {};
	\node at (-2,3)  [dot] (left1) {};
	\node at (-2,5)  [dot, red] (left2) {};
	\node at (-2,7)  [dot] (left3) {};
	\node[]  (a) at (-4,4) {};
	\node[cumu4, rotate=45] (cumu) at (a) {};

	\draw[testfcnx] (left) to  (root);
	
	\draw[kprime] (left1) [bend left=60] to (root);
	\draw[kernel1] (left2) to (left1);
	\draw[kernel1] (left3) to (left2);
	\draw[] (a.north) node[dot] {} to (left3);
	\draw[] (a.east) node[dot] {} to (left2);
	\draw[] (a.west) node[dot] {} to (left1);
	\draw[] (a.south) node[dot] {} to (left);
	\end{tikzpicture}
	\; =
	\begin{tikzpicture}[scale=0.35,baseline=0.5cm]
	\node at (-2,-1)  [root] (root) {};
	\node at (-2,1)  [dot] (left) {};
	\node at (-2,3)  [dot] (left1) {};
	\node at (-2,5)  [dot] (left2) {};
	\node at (-2,7)  [dot] (left3) {};
	\node[]  (a) at (-4,4) {};
	\node[cumu4, rotate=45] (cumu) at (a) {};

	\draw[testfcnx] (left) to  (root);
	
	\draw[kprime] (left1) [bend left=60] to (root);
	\draw[kernel] (left2) to (left1);
	\draw[kernel] (left3) to (left2);
	\draw[] (a.north) node[dot] {} to (left3);
	\draw[] (a.east) node[dot] {} to (left2);
	\draw[] (a.west) node[dot] {} to (left1);
	\draw[] (a.south) node[dot] {} to (left);
	\end{tikzpicture}
	\;-
	\begin{tikzpicture}[scale=0.35,baseline=0.5cm]
	\node at (-2,-1)  [root] (root) {};
	\node at (-2,1)  [dot] (left) {};
	\node at (-2,3)  [dot] (left1) {};
	\node at (-2,5)  [dot] (left2) {};
	\node at (-2,7)  [dot] (left3) {};
	\node[]  (a) at (-4,4) {};
	\node[cumu4, rotate=45] (cumu) at (a) {};

	\draw[testfcnx] (left) to  (root);
	
	\draw[kprime] (left1) [bend left=60] to (root);
	\draw[kernel] (left2) to [bend left=60] (root);
	\draw[kernel] (left3) to (left2);
	\draw[] (a.north) node[dot] {} to (left3);
	\draw[] (a.east) node[dot] {} to (left2);
	\draw[] (a.west) node[dot] {} to (left1);
	\draw[] (a.south) node[dot] {} to (left);
	\end{tikzpicture}
	\;-
	\begin{tikzpicture}[scale=0.35,baseline=0.5cm]
	\node at (-2,-1)  [root] (root) {};
	\node at (-2,1)  [dot] (left) {};
	\node at (-2,3)  [dot] (left1) {};
	\node at (-2,5)  [dot] (left2) {};
	\node at (-2,7)  [dot] (left3) {};
	\node[]  (a) at (-4,4) {};
	\node[cumu4, rotate=45] (cumu) at (a) {};

	\draw[testfcnx] (left) to  (root);
	
	\draw[kprime] (left1) [bend left=60] to (root);
	\draw[kernel] (left2) to (left1);
	\draw[kernel] (left3) to [bend left=60] (root);
	\draw[] (a.north) node[dot] {} to (left3);
	\draw[] (a.east) node[dot] {} to (left2);
	\draw[] (a.west) node[dot] {} to (left1);
	\draw[] (a.south) node[dot] {} to (left);
	\end{tikzpicture}
	\;+
	\begin{tikzpicture}[scale=0.35,baseline=0.5cm]
	\node at (-2,-1)  [root] (root) {};
	\node at (-2,1)  [dot] (left) {};
	\node at (-2,3)  [dot] (left1) {};
	\node at (-2,5)  [dot] (left2) {};
	\node at (-2,7)  [dot] (left3) {};
	\node[]  (a) at (-4,4) {};
	\node[cumu4, rotate=45] (cumu) at (a) {};

	\draw[testfcnx] (left) to  (root);
	
	\draw[kprime] (left1) [bend left=60] to (root);
	\draw[kernel] (left2) to [bend left=60] (root);
	\draw[kernel] (left3) to [bend left=60] (root);
	\draw[] (a.north) node[dot] {} to (left3);
	\draw[] (a.east) node[dot] {} to (left2);
	\draw[] (a.west) node[dot] {} to (left1);
	\draw[] (a.south) node[dot] {} to (left);
	\end{tikzpicture}.
	\;
	\end{equation}
	Above, in~\eqref{eq:5thbadgraph}, the blue node is another possible choice of $\bar\CCV$ in~Proposition~\ref{prop:pcondition} that makes~\eqref{eq:5thcondition} fail. Thus, in these cases the problematic choices for $\bar\CCV$ are $\bar\CCV_1=\{\text{red node}\}$, $\bar\CCV_2=\{\text{blue node}\}$ and $\bar\CCV_3=\bar\CCV_1\cup\bar\CCV_2$.
	Fortunately, it will turn out that all graphs above on the respective right hand sides fall into the framework of Proposition~\ref{prop:pcondition}. Therefore, whenever in the sequel we refer to a generic graph of this section we never refer to one of the graphs on the left hand sides above, but  always to one on the right hand side above.
\end{remark}
\subsection{Condition on basic graphs}
In this section we derive a general criterion under which the $p$-th moment of any graph listed in the previous section satisfies Assumption~\ref{ass:mainGraph}, so that Theorem~\ref{theo:ultimate1} applies to it. To that end, fix a graph $\CG$ as in the previous section obtained from one of the symbols $h(\tau)$, and replace each cumulant term by a weighted simple cycle with weight $3/2+\eta$. We denote its set of black nodes to which a noise term is attached (i.e. a dotted line in our graphical notation) by $\CCV_\xi$ (note that $\CCV_\xi$ is nothing but the set $B$ in the notation of Remark~\ref{rem:renorm}). Its complement is denoted by $\CCV_\xi^c$.
The $p$-th moment of it may be described as follows (see \eqref{eq:pthmomentXiIXi} for a concrete example):
\begin{claim}
	\item[1.] Take $p$ copies of the graph under consideration. 
	\item[2.] Fix a partition $\pi$ of $\CCV_\xi^{\otimes p}$ such that each $B\in\pi$ contains at least two elements that are from different copies of $\CG$ (this additional condition is a consequence of Lemma~\ref{lem:diagram});
	\item[3.] For each $B\in\pi$ draw a red polygon with $|B|$ dots inside and connect these $|B|$ dots to the elements of $B$;
	\item[4.] Sum over all expressions obtained in this way.
\end{claim} 
The above procedure gives an exact description of the $p$-th moment of any graph appearing in this work. Since all the inequalities we aim at proving are strict inequalities and since the $\eta$ in the weight $3/2+\eta$ of the simple cycles can be chosen arbitrarily small we will at times ignore the additional summand $\eta$.
Before we formulate the next result we refer the reader to Appendix~\ref{A}, where some of the notation used in the proposition below is introduced.
\begin{proposition}\label{prop:pcondition}
Fix a graph $\CG$ with vertex set $\CCV$ resulting from one of the symbols $h(\tau)$ and replace each cumulant term by a simple weighted cycle with weight $3/2+\eta$. Denote its set of black nodes that are attached to a noise term by $\CCV_\xi$. Let $\bar\CCV$ be a subset of  $\CCV_0$, and assume that
	\begin{equation}\label{eq:3rdcondition}
	\begin{aligned}
	&(i)\, \sum_{e\in\hat\CCE_0(\bar \CCV)}\hat a_e +\frac32|\bar \CCV\cap \CCV_\xi| < 5\big(|\bar \CCV|-\frac12\big),\, \text{and}\\
	&(ii)\, \text{if }\,|\bar \CCV|\geq 3,\, \text{then }\,\sum_{e\in\hat\CCE_0(\bar \CCV)}\hat a_e < 5(|\bar \CCV|-1).
	\end{aligned}
	\end{equation}
	Then the $p$-th moment of $\CG$ satisfies the second item in Assumption~\ref{ass:mainGraph}. 
	If for all $\bar\CCV\subseteq \CCV$ with $0\in\bar\CCV$ one has
	\begin{equation}\label{eq:4thcondition}
	\begin{aligned}
	(i)\, &\sum_{e  \in \hat \CCE_0 (\bar \CCV) } \hat a_e 
	+\sum_{e  \in \hat\CCE^{\uparrow}_+ (\bar \CCV)  } (\hat a_e+ r_e -1)  - \sum_{e  \in \hat\CCE^{\downarrow}_+ (\bar \CCV)  }  r_e + \frac32|\bar\CCV\cap\CCV_\xi| < 5\big(|\bar\CCV|-\frac12\big),\\
	(ii)\, &\sum_{e  \in \hat \CCE_0 (\bar \CCV) } \hat a_e 
	+\sum_{e  \in \hat\CCE^{\uparrow}_+ (\bar \CCV)  } (\hat a_e+ r_e -1)  - \sum_{e  \in \hat\CCE^{\downarrow}_+ (\bar \CCV)  }  r_e < 5(|\bar\CCV|-1),
	\end{aligned}
	\end{equation}
	where the first condition is subject to $\bar\CCV\cap\CCV_\xi\neq\emptyset$, and the second is subject to $|\bar\CCV|\geq 2$, then the $p$-th moment of $\CG$ satisfies the third item of Assumption~\ref{ass:mainGraph}.
	Finally, if for all $\bar\CCV\subseteq \CCV\setminus\CCV_\star$,
	\begin{equation}\label{eq:5thcondition}
	\sum_{e\in \hat\CCE(\bar\CCV)\setminus \hat\CCE^{\downarrow}_+(\bar \CCV) }  \hat a_e 
	+\sum_{e\in \hat \CCE^{\uparrow}_+(\bar\CCV)}  r_e
	- \sum_{e \in \hat \CCE^\downarrow_+(\bar \CCV)} (r_e-1) +\frac32|\bar\CCV\cap\CCV_\xi|
	> 5|\bar \CCV|\;,
	\end{equation}
	then the $p$-th moment of $\CG$ satisfied the last item of Assumption~\ref{ass:mainGraph}.
\end{proposition}

\begin{proof}
	Fix a graph $\CG$ as in the proposition that satisfies Condition~\eqref{eq:3rdcondition} and consider $p$ copies of it. Note that in the case $p=1$ the second condition of Assumption~\ref{ass:mainGraph} follows immediately from the second item in~\eqref{eq:3rdcondition}. Hence, we may assume that $p\geq 2$. Let $\bar\CCV\subseteq\CCV_{0}$ be of cardinality at least three. Here, $\CCV_{0}$ refers to the union of vertices of all $p$ copies. We write $\bar\CCV=\cup_{j=1}^{p}\bar\CCV_j$, where $\bar\CCV_j$ is the set of vertices belonging to the $j$-th copy of $\CG$. 
	Now recall that we replace each cumulant term by a simple weighted cycle. We refer to each kernel in this cycle as a cumulant kernel. Note that there are two types of cumulant kernels present, those that are part of the definition of each copy of $\CG$, i.e., those connecting nodes in $\CCV_\xi^c$, and those that are created when forming the $p$-th moment, i.e., those that connect nodes from possibly different copies of $\CCV_\xi$.
	Let $e=\{v,w\}\in\hat\CCE_0(\bar\CCV)$ be an edge such that $v,w\in\CCV_\xi$. We note that if there is a cumulant kernel associated to that edge it has weight $3/2+$. Moreover, the weight of a cumulant kernel connecting the two nodes $v$ and $w$ is only counted if both $v$ and $w$ belong to $\hat\CCE_0(\bar\CCV)$. Since, moreover cumulant kernels come in cycles we see that each $v\in \bar\CCV_j\cap\CCV_\xi$ contributes at most a weight $3/2+$ through its cumulant kernel. Hence,
	\begin{equation}
	\begin{aligned}
	\sum_{e\in\hat\CCE_0(\bar\CCV)}\hat a_e \leq
	&\sum_{\substack{j=1:\\ \bar\CCV_j\cap\CCV_\xi\neq\emptyset}}^{p} \Big(\sum_{e\in\hat\CCE_0(\bar\CCV_j)}\hat a_e+ \frac 32| \bar\CCV_j\cap\CCV_\xi|\Big) +\sum_{\substack{j=1:\\ \bar\CCV_j\cap\CCV_\xi=\emptyset}}^p \sum_{e\in\hat\CCE_0(\bar\CCV_j)}\hat a_e\\
	&< 5\sum_{\substack{j=1:\\ \bar\CCV_j\cap\CCV_\xi\neq\emptyset}}^{p}(|\bar\CCV_j|-\frac12)
	+ 5\sum_{\substack{j=1:\\ \bar\CCV_j\cap\CCV_\xi=\emptyset}}^p(|\bar\CCV_j|-\frac12)\\
	&< 5(|\bar\CCV|-1),
	\end{aligned}
	\end{equation} 
	where in the last line we used that $p\geq 2$. Therefore, the second condition follows.
	We now turn to the proof of the third condition in Assumption~\ref{ass:mainGraph}. To that end let $\bar\CCV$ be of cardinality at least two such that it contains the origin. As above we can assume that $p\geq 2$, since the case $p=1$ is covered by the second condition above. We write $\bar\CCV\setminus\{0\}= \cup_{j=1}^{p} \bar\CCV_j$, where $\bar\CCV_j$ are the vertices belonging to the $p$-th copy of $\CG$, and neither of them contains the origin. We can then argue in the same way as we did for the third condition. The only difference is that each summand with $\bar\CCV_j\cap\CCV_\xi=\emptyset$ can now be bounded from above by $5(|\CCV_j\cup\{0\}|-1)$, since $|\CCV_j\cup\{0\}|\geq 2$ in case that $\CCV_j\neq\emptyset$.
	We turn to the last part of the proof. To that end we assume that $\CG$ satisfies~\eqref{eq:5thcondition}, and we fix $\bar\CCV\subseteq\CCV\setminus\CCV_\star$. We then note that for each $v\in \CCV_\xi$ there are two cumulant kernels emerging from it, and moreover $v$ can belong to at most one edge $e$ with $e\in\hat\CCE_+^{\downarrow}(\bar\CCV)$.  Therefore, writing $\bar\CCV=\cup_{j=1}^{p}\bar\CCV_j$ as above, we may conclude that
	\begin{equation}
	\begin{aligned}
	&\sum_{e\in\hat\CCE(\bar\CCV)\setminus \hat\CCE^{\downarrow}_+(\bar \CCV) }  \hat a_e 
	+\sum_{e\in \hat \CCE^{\uparrow}_+(\bar\CCV)}  r_e
	- \sum_{e \in \hat \CCE^\downarrow_+(\bar \CCV)} (r_e-1)\\
	&\geq \sum_{j=1}^{p}\Big(\sum_{e\in \hat\CCE(\bar\CCV_j)\setminus \hat\CCE^{\downarrow}_+(\bar \CCV_j) }  \hat a_e 
	+\sum_{e\in \hat \CCE^{\uparrow}_+(\bar\CCV_j)}  r_e
	- \sum_{e \in \hat \CCE^\downarrow_+(\bar \CCV_j)} (r_e-1)+ \frac32|\bar\CCV_j\cap\CCV_\xi|\Big)\\
	&> 5|\bar\CCV|.
	\end{aligned}
	\end{equation}
	This yields the claim.
\end{proof}

\subsection{Verifying the assumptions of Theorem~\ref{theo:ultimate1}}
In this section we show that almost all graphs that can be built from $p$-th moments of any of the graphs in Section~\ref{Sgeneralities} satisfy Assumption~\ref{ass:mainGraph}. 
To that end note that by simple inspection one may conclude that the first assumption is indeed satisfied. To show the remaining conditions we make use of Proposition~\ref{prop:pcondition}. Fix a graph $\CG$ from Section~\ref{Sgeneralities}.
We first show that $\CG$ satisfies the first item in~\eqref{eq:3rdcondition}.
To that end fix $\bar\CCV\subset\CCV$. If $\bar\CCV\cap\CCV_\xi=\emptyset$, then since the second item in Equation~\eqref{eq:3rdcondition} is more restrictive in that case, it is enough to consider the case in which $|\bar\CCV|\leq 2$. In this case however the first item in~\eqref{eq:3rdcondition} can easily seen to be fulfilled since it simply amounts to the statement that $\hat a_e< 7.5$ for any edge $e$ in the graph, which one can verify via a direct inspection. Hence, we may assume that $\bar\CCV$ has a non-empty intersection with $\CCV_\xi$.
For any edge $e\in\hat\CCE(\bar\CCV)$ we write $\hat a_e=\hat a_e(B) + \hat a_e(C)$, where $\hat a_e(B)$ is the weight of the edge when erasing all cumulant kernels. Note that for $e\in\hat\CCE_0(\bar\CCV)$ in almost all cases we have that $\hat a_e(B)=3$. The only exceptions are the renormalised kernels, which have weight $5.5$ and $6$ respectively (the kernel depicted by a dotted line has weight four, however it is always connected to the origin and $\bar\CCV$ is chosen such that $0\notin\bar\CCV$). In this case however, we can artificially split the kernel into two  and simply \emph{redefine} $\hat a_e(B)\overset{\text{def}}{=}3$. Note that in this case $\hat a_e(C)\in\{2.5+, 3+\}$, which can be interpreted as $e_-$ and $e_+$ both having a cumulant kernel attached to them with weight $\in\{2.5/2 +, 3/2 +\}$.
Since cycles are only created by cumulant kernels, we see that 
\begin{equation}\label{eq:sumwithoutcycles}
\sum_{e\in\hat\CCE_0(\bar\CCV)} \hat a_e(B) \leq 3(|\bar\CCV|-1),
\end{equation}
and since $\hat a_e(C)\leq 3/2 +$, whenever $e=(e_-,e_+)$ is such that $e_-,e_+\in\CCV_\xi^c$,
\begin{equation}
\sum_{\substack{e\in\hat\CCE_0(\bar\CCV):\\ e_-,e_+\in\bar\CCV\cap\CCV_\xi^c}}\hat a_e(C)\leq 
\frac32|\bar\CCV\cap\CCV_\xi^c|\, 
\end{equation}
where we omitted a term of order $\eta$ on the right hand side. 
Thus,
\begin{equation}
\begin{aligned}
\sum_{e\in\hat\CCE_0(\bar\CCV)} \hat a_e +\frac32|\bar\CCV\cap\CCV_\xi|
&=\sum_{e\in\hat\CCE_0(\bar\CCV)} \hat a_e(B) + \sum_{\substack{e\in\hat\CCE_0(\bar\CCV):\\ e_-,e_+\in\bar\CCV\cap\CCV_\xi^c}}\hat a_e(C) + \frac32|\bar\CCV\cap\CCV_\xi|\\
&\leq 3(|\bar\CCV|-1) +\frac32|\bar\CCV| < 5(|\bar\CCV|-\frac12),
\end{aligned}
\end{equation}
and the last inequality is true whenever $|\bar\CCV|>-1$, which is of course satisfied.
We turn to the second item in~\eqref{eq:3rdcondition}. Since cumulant kernels form cycles, it is possible that
\begin{equation}
\sum_{e\in\hat\CCE_0(\bar\CCV)} \hat a_e(C)=\frac32|\bar\CCV|\,.
\end{equation}
In this case we can not use the same arguments as above. Indeed, 
$4.5|\bar\CCV| -3 < 5(|\bar\CCV|-1)$ if and only if $4 < |\bar\CCV|$.
However, if the second item in Equation~\ref{eq:3rdcondition} is violated this means that the graph is of non-positive degree. In this case the fourth item of
Remark~\ref{rem:renorm} applies. There are two cases to be considered here. Either a kernel gets connected to the origin, in which case the left hand side in ~\eqref{eq:sumwithoutcycles} gets reduced by three (since $0\notin\bar\CCV$), and we may conclude as for the first item. Or it gets replaced by a renormalised kernel. In this case however note that if $v,w\in\bar\CCV$ are the start and end point of the renormalised kernel, then we always have that $v,w\notin\CCV_\xi$. Thus, we therefore see that we always have that
\begin{equation}
\sum_{e\in\hat\CCE_0(\bar\CCV)} \hat a_e(C)\leq \frac32(|\bar\CCV|-1),
\end{equation}
and so we may conclude using that $|\bar\CCV|\geq 3$. 
We now show that almost all graphs $\CG$ satisfy~\eqref{eq:4thcondition}. Let $\bar\CCV\subset\CCV$ contain zero. We first focus on the first item. Since $\bar\CCV$ contains the origin we see that
\begin{equation}
\sum_{e\in\hat\CCE_0(\bar\CCV)}\one\{\hat a_e(B) >0\} + \sum_{e\in\hat\CCE_+^{\uparrow}(\bar\CCV)}\one\{\hat a_e(B) >0\}\leq |\bar\CCV|-1.
\end{equation}
Moreover, in all graphs under consideration there is at most one edge $e$ that has label $4$ or for which
$r_e=2$, and both cases do not appear simultaneously. Since edges with $r_e>0$ are never connected to the origin, and since vertices that have a kernel connected to the origin are never the starting point of a kernel with a $r_e>0$ we can thus conclude that
\begin{equation}
\sum_{e\in\hat\CCE_0(\bar\CCV)}\hat a_e(B) + 
\sum_{e\in\hat\CCE_{+}^{\uparrow}(\bar\CCV)}(\hat a_e(B)+r_e-1)\leq 3(|\bar\CCV|-1) +1.
\end{equation} 
Finally, since cumulant kernels are not connected to the origin as well, we see that
\begin{equs}
\sum_{e\in\hat\CCE_0(\bar\CCV)}&\hat a_e + 
\sum_{e\in\hat\CCE_{+}^{\uparrow}(\bar\CCV)}(\hat a_e+r_e-1)+\frac32|\bar\CCV\cap\CCV_\xi|\\
&= \sum_{e\in\hat\CCE_0(\bar\CCV)}\hat a_e(B) + \sum_{\substack{e\in\hat\CCE_0(\bar\CCV):\\
		e_-,e_+\in\bar\CCV_\cap\CCV_\xi^{c}}}\hat a_e(C)+
\sum_{e\in\hat\CCE_{+}^{\uparrow}(\bar\CCV)}(\hat a_e(B)+r_e-1)+\frac32|\bar\CCV\cap\CCV_\xi|\\
&\leq 3(|\bar\CCV|-1) +1 +\frac32(|\bar\CCV|-1)\\
&< 5(|\bar\CCV|-\frac12), \label{eq:4thcondverified}
\end{equs} 
and the last inequality holds as soon as $|\bar\CCV|> -2$.
We turn to the second part of~\eqref{eq:4thcondition}. As before this case turns out to be more complicated. We distinguish between several cases.
\begin{itemize}
	\item[1.] For all $e\in\hat\CCE_0(\bar\CCV)$ we have $\hat a_e(B)\leq 3$ and
	for all $e\in\hat\CCE_{+}^{\uparrow}(\bar\CCV)$ we have that $r_e\leq 1$. In this case, very similar to~\eqref{eq:4thcondverified}, the left hand side of the second item in~\eqref{eq:4thcondition} is at most
	\begin{equation}
	3(|\bar\CCV|-1)+\frac32(|\bar\CCV|-1)< 5(|\bar\CCV|-1),
	\end{equation}
	so we can conclude.
	\item[2.] There is $e\in\hat\CCE_0(\bar\CCV)$ with $\hat a_e(B)=4$ or there is $e\in\hat\CCE_{+}^{\uparrow}(\bar\CCV)$ with $r_e=2$. Recall that both cases do not appear simultaneously. Moreover, we assume that
	\begin{equation}
	\sum_{e\in\hat\CCE_0(\bar\CCV)}\hat a_e(C)\leq \frac32(|\bar\CCV|-2).
	\end{equation}
	In this case the left hand side of the second item in~\eqref{eq:4thcondition} is bounded by
	$3(|\bar\CCV|-1)+1 +3/2(|\bar\CCV|-2)$, which is strictly bounded from above by $5(|\bar\CCV|-1)$ whenever $|\bar\CCV|>0$.
	\item[3.] There is $e\in\hat\CCE_0(\bar\CCV)$ with $\hat a_e(B)=4$ or there is $e\in\hat\CCE_{+}^{\uparrow}(\bar\CCV)$ with $r_e=2$.  Moreover we assume that,
	\begin{equation}\label{eq:3rdcase4thcond}
	\sum_{e\in\hat\CCE_0(\bar\CCV)}\hat a_e(C)= \frac32(|\bar\CCV|-1),
	\end{equation}
	and we note that the right hand side in~\eqref{eq:3rdcase4thcond} is always an upper bound for the corresponding left hand side (since $0\in\bar\CCV$ and cumulant kernels are never connected to the origin).
	Hence, our previous analysis yields
	\begin{equation}\label{eq:3rdcase4thcondverified}
	\sum_{e\in\hat\CCE_0(\bar\CCV)}\hat a_e +\sum_{e\in\hat\CCE_{+}^{\uparrow}(\bar\CCV)}(\hat a_e + r_e-1)
	\leq 3(|\bar\CCV|-1) + 1+\frac32(|\bar\CCV|-1),
	\end{equation}
	which is strictly bounded from above by $5(|\bar\CCV|-1)$ if and only if $|\bar\CCV|>3$. This shows that our arguments do not apply in that case and we need to proceed
	by direct inspection.
	To that end we note that we can restrict ourselves to the case that for each $e\in\hat\CCE_0(\bar\CCV)$ one has that  $\hat a_e(C)>0$, so that in particular $\bar\CCV\cap\CCV_\xi=\emptyset$.
	This is a direct consequence of~\eqref{eq:3rdcase4thcond}.
	Moreover, since~\eqref{eq:3rdcase4thcondverified} holds for $|\bar\CCV|\geq 4$, we only need to focus on the case $|\bar\CCV|\leq 3$. If $|\bar\CCV|=2$, recalling that $0\in\bar\CCV$, and since cumulant kernels are never connected to the origin it would  follow that
	\begin{equation}
	\sum_{e\in\hat\CCE_0(\bar\CCV)}\hat a_e(C)=0.
	\end{equation}
    This case however was already treated in the second item above, so that we can conclude. 
	Hence, it is sufficient to only consider those cases for which $|\bar\CCV|=3$. Since then $|\bar\CCV\setminus\{0\}|=2$ we also only need to consider graphs with cumulant terms of order two. By inspection one sees that for most of such graphs the right hand side of~\eqref{eq:4thcondverified} is strictly larger than its left hand side. Therefore, we are really left with the following graphs
	\begin{equs}
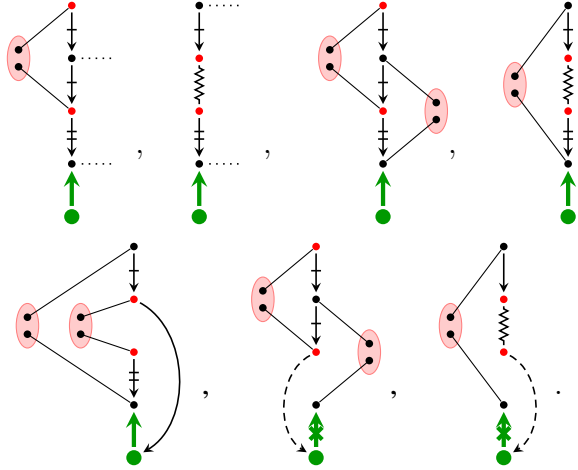
		\begin{aligned}
			&\begin{tikzpicture}[scale=0.35,baseline=0.5cm]
			\node at (-2,-1)  [root] (root) {};
			\node at (-2,1)  [dot] (left) {};
			\node at (-2,3)  [dot, red] (left1) {};
			\node at (-2,5)  [dot] (left2) {};
			\node at (-2,7)  [dot,red] (left3) {};
			\node at (0,1) [] (variable1) {};
			\node at (0,5) [] (variable3) {};
			\node[cumu2n] (a) at (-4,5) {};
			\draw[cumu2] (a) ellipse (12pt and 24pt);

			\draw[testfcn] (left) to  (root);
			
			\draw[kernel2] (left1) to (left);
			\draw[kernel1] (left2) to (left1);
			\draw[kernel1] (left3) to (left2);
			\draw[rho] (variable3) to (left2); 
			\draw[] (a.north) node[dot] {} to (left3);
			\draw[] (a.south) node[dot] {} to (left1);
			\draw[rho] (variable1) to (left); 
			\end{tikzpicture},\;\quad
			\begin{tikzpicture}[scale=0.35,baseline=0.5cm]
			\node at (-2,-1)  [root] (root) {};
			\node at (-2,1)  [dot] (left) {};
			\node at (-2,3)  [dot,red] (left1) {};
			\node at (-2,5)  [dot,red] (left2) {};
			\node at (-2,7)  [dot] (left3) {};
			\node at (0,1) [] (variable) {};
			\node at (0,7) [] (variable3) {};
			
			\draw[testfcn] (left) to  (root);
			
			\draw[kernel2] (left1) to (left);
			\draw[kernelBig] (left2) to (left1);
			\draw[kernel1] (left3) to (left2);
			\draw[rho] (variable3) to (left3); 
			\draw[rho] (variable) to (left); 
			\end{tikzpicture},
			\;\quad
			\begin{tikzpicture}[scale=0.35,baseline=0.5cm]
			\node at (-2,-1)  [root] (root) {};
			\node at (-2,1)  [dot] (left) {};
			\node at (-2,3)  [dot,red] (left1) {};
			\node at (-2,5)  [dot] (left2) {};
			\node at (-2,7)  [dot,red] (left3) {};
			\node[cumu2n] (a) at (-4,5) {};
			\node[cumu2n] (b) at (0,3) {};
			\draw[cumu2] (a) ellipse (12pt and 24pt);
			\draw[cumu2] (b) ellipse (12pt and 24pt);
			
			\draw[testfcn] (left) to  (root);
			
			\draw[kernel2] (left1) to (left);
			\draw[kernel1] (left2) to (left1);
			\draw[kernel1] (left3) to (left2);
			\draw[] (a.north) node[dot] {} to (left3);
			\draw[] (a.south) node[dot] {} to (left1);
			\draw[] (b.north) node[dot] {} to (left2);
			\draw[] (b.south) node[dot] {} to (left);
			\end{tikzpicture},\;\quad
			\begin{tikzpicture}[scale=0.35,baseline=0.5cm]
			\node at (-2,-1)  [root] (root) {};
			\node at (-2,1)  [dot] (left) {};
			\node at (-2,3)  [dot,red] (left1) {};
			\node at (-2,5)  [dot,red] (left2) {};
			\node at (-2,7)  [dot] (left3) {};
			\node[cumu2n] (a) at (-4,4) {};
			\draw[cumu2] (a) ellipse (12pt and 24pt);
			
			\draw[testfcn] (left) to  (root);
			
			\draw[kernel2] (left1) to (left);
			\draw[kernelBig] (left2) to (left1);
			\draw[kernel1] (left3) to (left2);
			\draw[] (a.north) node[dot] {} to (left3);
			\draw[] (a.south) node[dot] {} to (left);
			\end{tikzpicture}\\
			\;&\;
			\begin{tikzpicture}[scale=0.35,baseline=0.5cm]
			\node at (-2,-1)  [root] (root) {};
			\node at (-2,1)  [dot] (left) {};
			\node at (-2,3)  [dot, red] (left1) {};
			\node at (-2,5)  [dot, red] (left2) {};
			\node at (-2,7)  [dot] (left3) {};
			\node[cumu2n] (a) at (-4,4) {};
			\node[cumu2n] (b) at (-6,4) {};
			\draw[cumu2] (a) ellipse (12pt and 24pt);
			\draw[cumu2] (b) ellipse (12pt and 24pt);
			
			\draw[testfcn] (left) to  (root);
			
			\draw[kernel1] (left3) to (left2);
			\draw[kernel] (left2) [bend left=60] to (root);
			\draw[kernel2] (left1) to (left);
			\draw[] (a.north) node[dot] {} to (left2);
			\draw[] (a.south) node[dot] {} to (left1);
			\draw[] (b.north) node[dot] {} to (left3);
			\draw[] (b.south) node[dot] {} to (left);
			
			\end{tikzpicture}\;,\quad\;
			\begin{tikzpicture}[scale=0.35,baseline=0.5cm]
			\node at (-2,-1)  [root] (root) {};
			\node at (-2,1)  [dot] (left) {};
			\node at (-2,3)  [dot, red] (left1) {};
			\node at (-2,5)  [dot] (left2) {};
			\node at (-2,7)  [dot, red] (left3) {};
			\node[cumu2n] (a) at (-4,5) {};
			\node[cumu2n] (b) at (0,3) {};
			\draw[cumu2] (a) ellipse (12pt and 24pt);
			\draw[cumu2] (b) ellipse (12pt and 24pt);
			
			\draw[testfcnx] (left) to  (root);
			
			\draw[kprime] (left1) [bend right=60] to (root);
			\draw[kernel1] (left2) to (left1);
			\draw[kernel1] (left3) to (left2);
			\draw[] (a.north) node[dot] {} to (left3);
			\draw[] (a.south) node[dot] {} to (left1);
			\draw[] (b.north) node[dot] {} to (left2);
			\draw[] (b.south) node[dot] {} to (left);
			\end{tikzpicture}\;,\quad\;
			\begin{tikzpicture}[scale=0.35,baseline=0.5cm]
			\node at (-2,-1)  [root] (root) {};
			\node at (-2,1)  [dot] (left) {};
			\node at (-2,3)  [dot, red] (left1) {};
			\node at (-2,5)  [dot, red] (left2) {};
			\node at (-2,7)  [dot] (left3) {};
			\node[cumu2n] (a) at (-4,4) {};
			\draw[cumu2] (a) ellipse (12pt and 24pt);
			
			\draw[testfcnx] (left) to  (root);
			
			\draw[kprime] (left1) [bend left=60 ]to (root);
			\draw[kernelBig] (left2) to (left1);
			\draw[kernel] (left3) to (left2);
			\draw[] (a.north) node[dot] {} to (left3);
			\draw[] (a.south) node[dot] {} to (left);
			\end{tikzpicture}\;.
		\end{aligned}
	\end{equs}
	where the red nodes are the nodes corresponding to the choices of $\bar\CCV\setminus\{0\}$ that need to be considered. Thus, by direct inspection (note that at this point we really need the summand $-r_e$) we see that all graphs but the last satisfy the second item in~\eqref{eq:4thcondition}. As for the last graph, Theorem~\ref{theo:ultimate1} does not apply, but as in~\cite[Equation (5.22)]{WongZakai} one may treat this graph by hand.
	We now show that all graphs satisfy~\eqref{eq:5thcondition}. To that end we first note that all graphs containing only cumulant terms up to order $2$ satisfy~\eqref{eq:5thcondition} as a consequence of~\cite[Section5]{WongZakai}. Indeed, in this work the aforementioned type of graphs coincide modulo its labelling with the those graphs in the current setting. Then, given any such graph $\CG$ the following ``doubling'' procedure was performed in~\cite{WongZakai}:
	For any copy $\CG_i$, $i\in\{1,2\}$, denote by $\CCV_{\xi,i}$ the set of vertices of $\CG_i$ to which there is a noise term attached. Then, consider a bijective mapping $f:\CCV_{\xi,1}\to\CCV_{\xi,2}$ and for any $v_1\in\CCV_{\xi,1}$ connect $v_1$ with $f(v_1)$ and equip the edge $(v_1, f(v_1))$ with label $3$. Denote the resulting graph with vertex set $\CCV_f=\CCV_1\cup\CCV_2$ by $\CG_f$. Then, for any bijection $f$, and any $\bar\CCV_f\subset\CCV_f\setminus\CCV_{f,\star}$ it was shown that
	\begin{equation}
	\sum_{e\in\hat\CCE(\bar\CCV_f)\setminus\hat\CCE_{+}^{\downarrow}(\bar\CCV_f)}\hat a_e^{\rm{HP}}
	+\sum_{e\in\hat\CCE_{+}^{\uparrow}(\bar\CCV_f)}r_e-\sum_{e\in\hat\CCE_{+}^{\downarrow}(\bar\CCV_f)}(r_e-1)
	> 3|\bar\CCV_f|,
	\end{equation}
	where $\hat a_e^{\rm{HP}}$ denotes the labelling of the edges in~\cite{WongZakai}.
    Taking as bijection $f$ the function that maps each $v_1\in\CCV_{\xi,1}$ to its copy in $\CCV_{\xi,2}$, we see that
	\begin{equation}
	\begin{aligned}
	\sum_{e\in\hat\CCE(\bar\CCV_{f,1})\setminus\hat\CCE_{+}^{\downarrow}(\bar\CCV_{f,1})}&\hat a_e^{\rm{HP}}(B)
	+\sum_{e\in\hat\CCE_{+}^{\uparrow}(\bar\CCV_{f,1})}r_e-\sum_{e\in\hat\CCE_{+}^{\downarrow}(\bar\CCV_{f,1})}(r_e-1)+\frac32|\bar\CCV_{f,1}\cap\CCV_{\xi,1}|\\
	&>3|\bar\CCV_{f,1}|
	\end{aligned}
	\end{equation}
	with hopefully obvious notation. Now, using that $|\{e\in\hat\CCE(\bar\CCV_{f,1})\setminus\hat\CCE_{+}^{\downarrow}(\bar\CCV_{f,1})\}|\geq |\bar\CCV_{f,1}|$, since $\bar\CCV_{f}\subseteq \CCV_f\setminus\CCV_{f,\star}$, so that any $v\in \bar\CCV_{f,1}$ has an outgoing admissible edge, and that for any edge $e\in\CG$ we have the identity $\hat a_e(B)= \hat a_e^{\rm{HP}}(B)+2$, we can conclude.
	It therefore remains to check~\eqref{eq:5thcondition} for all graphs $\CG$ that contain cumulant terms of order three or higher. We present the necessary calculations for two examples, all remaining graphs may be treated in exactly the same way. First, we consider the first graph in~\eqref{eq:Xi3dec2} in the presentation of $\V$. The red nodes will denote the choice of $\bar\CCV\subseteq\CCV_\star$. Possible choices are
	\begin{equation}
	\;
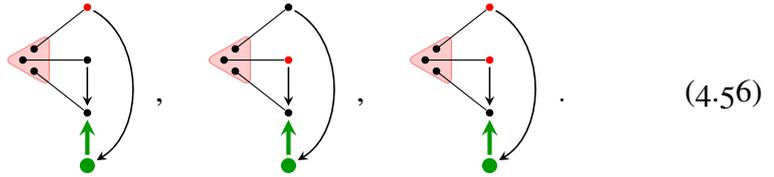
	\begin{tikzpicture}[scale=0.35, baseline=0.5cm]
	\node at (-2,-1)  [root] (root) {};
	\node at (-2,1)  [dot] (left) {};
	\node at (-2,3)  [dot] (left1) {};
	\node at (-2,5)  [dot, red] (left2) {};
	\node[]		(cumu)  	at (-4,3) {};
	\node[cumu3,rotate=90]	(top-cumu) 	at (cumu) {};
	
	\draw[testfcn] (left) to  (root);
	
	\draw[kernel] (left1) to (left);
	\draw[kernel] (left2) to [bend left=60] (root);
	\draw[] (cumu.north) node[dot] {} to (left2); 
	\draw[] (cumu.west) node[dot] {} to (left1);
	\draw[] (cumu.south) node[dot] {} to (left); 
	\end{tikzpicture}\;,\quad
	\;
	\begin{tikzpicture}[scale=0.35, baseline=0.5cm]
	\node at (-2,-1)  [root] (root) {};
	\node at (-2,1)  [dot] (left) {};
	\node at (-2,3)  [dot, red] (left1) {};
	\node at (-2,5)  [dot] (left2) {};
	\node[]		(cumu)  	at (-4,3) {};
	\node[cumu3,rotate=90]	(top-cumu) 	at (cumu) {};
	
	\draw[testfcn] (left) to  (root);
	
	\draw[kernel] (left1) to (left);
	\draw[kernel] (left2) to [bend left=60] (root);
	\draw[] (cumu.north) node[dot] {} to (left2); 
	\draw[] (cumu.west) node[dot] {} to (left1);
	\draw[] (cumu.south) node[dot] {} to (left); 
	\end{tikzpicture}\;,\quad
	\;
	\begin{tikzpicture}[scale=0.35, baseline=0.5cm]
	\node at (-2,-1)  [root] (root) {};
	\node at (-2,1)  [dot] (left) {};
	\node at (-2,3)  [dot, red] (left1) {};
	\node at (-2,5)  [dot, red] (left2) {};
	\node[]		(cumu)  	at (-4,3) {};
	\node[cumu3,rotate=90]	(top-cumu) 	at (cumu) {};
	
	\draw[testfcn] (left) to  (root);
	
	\draw[kernel] (left1) to (left);
	\draw[kernel] (left2) to [bend left=60] (root);
	\draw[] (cumu.north) node[dot] {} to (left2); 
	\draw[] (cumu.west) node[dot] {} to (left1);
	\draw[] (cumu.south) node[dot] {} to (left); 
	\end{tikzpicture}\;.
	\end{equation}
	Rewriting now the cumulant term as a cycle where each edge has label $3/2 +$ we see that for the above choices of $\bar\CCV$ the left hand side in~\eqref{eq:5thcondition}
	becomes $6 +$ for the first and second graph and $10.5 +$ for the last graph, which is as desired. 
	The second example we look at is the second graph in~\eqref{eq:Xi4IVrewrite}, for which the following choices of $\bar\CCV$ are possible
	\begin{equs}{}
	&
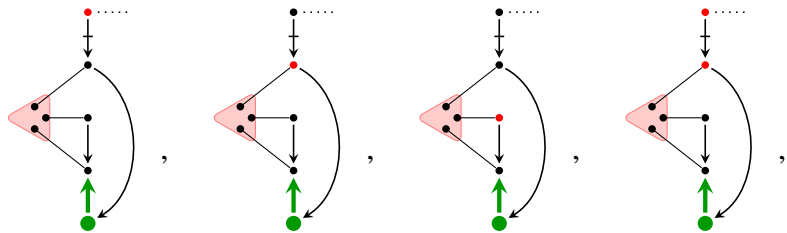
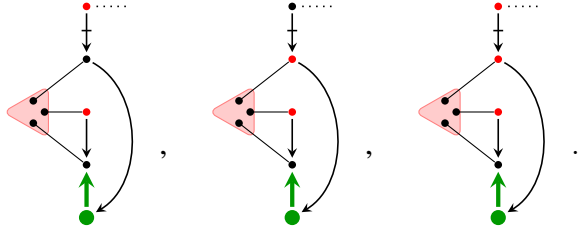
\begin{tikzpicture}[scale=0.35,baseline=0.5cm]
	\node at (-2,-1)  [root] (root) {};
	\node at (-2,1)  [dot]  (left) {};
	\node at (-2,3)  [dot]  (left1) {};
	\node at (-2,5)  [dot] (left2) {};
	\node at (-2,7)  [dot, red] (left3) {};
	\node at (0,7) [] (variable3) {};
	\node[] (a) at (-4,3) {};
	\node[cumu3, rotate=90] (cumu) at (a) {};

	\draw[testfcn] (left) to  (root);
	\draw[kernel]  (left1) to (left);
	\draw[kernel]  (left2) [bend left=60] to (root);
	\draw[kernel1] (left3) to (left2);
	\draw[] (a.north) node[dot] {} to (left2);
	\draw[] (a.east) node[dot] {} to (left1);
	\draw[] (a.south) node[dot] {} to (left);
	\draw[rho] (variable3) to (left3);
	\end{tikzpicture}
	\;, \quad
	\;
	\begin{tikzpicture}[scale=0.35,baseline=0.5cm]
	\node at (-2,-1)  [root] (root) {};
	\node at (-2,1)  [dot]  (left) {};
	\node at (-2,3)  [dot]  (left1) {};
	\node at (-2,5)  [dot, red] (left2) {};
	\node at (-2,7)  [dot] (left3) {};
	\node at (0,7) [] (variable3) {};
	\node[] (a) at (-4,3) {};
	\node[cumu3, rotate=90] (cumu) at (a) {};

	\draw[testfcn] (left) to  (root);
	\draw[kernel]  (left1) to (left);
	\draw[kernel]  (left2) [bend left=60] to (root);
	\draw[kernel1] (left3) to (left2);
	\draw[] (a.north) node[dot] {} to (left2);
	\draw[] (a.east) node[dot] {} to (left1);
	\draw[] (a.south) node[dot] {} to (left);
	\draw[rho] (variable3) to (left3);
	\end{tikzpicture}
	\;,\quad
	\;
	\begin{tikzpicture}[scale=0.35,baseline=0.5cm]
	\node at (-2,-1)  [root] (root) {};
	\node at (-2,1)  [dot]  (left) {};
	\node at (-2,3)  [dot, red]  (left1) {};
	\node at (-2,5)  [dot] (left2) {};
	\node at (-2,7)  [dot] (left3) {};
	\node at (0,7) [] (variable3) {};
	\node[] (a) at (-4,3) {};
	\node[cumu3, rotate=90] (cumu) at (a) {};

	\draw[testfcn] (left) to  (root);
	\draw[kernel]  (left1) to (left);
	\draw[kernel]  (left2) [bend left=60] to (root);
	\draw[kernel1] (left3) to (left2);
	\draw[] (a.north) node[dot] {} to (left2);
	\draw[] (a.east) node[dot] {} to (left1);
	\draw[] (a.south) node[dot] {} to (left);
	\draw[rho] (variable3) to (left3);
	\end{tikzpicture}
	\;,\quad
	\;
	\begin{tikzpicture}[scale=0.35,baseline=0.5cm]
	\node at (-2,-1)  [root] (root) {};
	\node at (-2,1)  [dot]  (left) {};
	\node at (-2,3)  [dot]  (left1) {};
	\node at (-2,5)  [dot, red] (left2) {};
	\node at (-2,7)  [dot, red] (left3) {};
	\node at (0,7) [] (variable3) {};
	\node[] (a) at (-4,3) {};
	\node[cumu3, rotate=90] (cumu) at (a) {};

	\draw[testfcn] (left) to  (root);
	\draw[kernel]  (left1) to (left);
	\draw[kernel]  (left2) [bend left=60] to (root);
	\draw[kernel1] (left3) to (left2);
	\draw[] (a.north) node[dot] {} to (left2);
	\draw[] (a.east) node[dot] {} to (left1);
	\draw[] (a.south) node[dot] {} to (left);
	\draw[rho] (variable3) to (left3);
	\end{tikzpicture}
	\;,\\
	\;
	&\begin{tikzpicture}[scale=0.35,baseline=0.5cm]
	\node at (-2,-1)  [root] (root) {};
	\node at (-2,1)  [dot]  (left) {};
	\node at (-2,3)  [dot, red]  (left1) {};
	\node at (-2,5)  [dot] (left2) {};
	\node at (-2,7)  [dot, red] (left3) {};
	\node at (0,7) [] (variable3) {};
	\node[] (a) at (-4,3) {};
	\node[cumu3, rotate=90] (cumu) at (a) {};

	\draw[testfcn] (left) to  (root);
	\draw[kernel]  (left1) to (left);
	\draw[kernel]  (left2) [bend left=60] to (root);
	\draw[kernel1] (left3) to (left2);
	\draw[] (a.north) node[dot] {} to (left2);
	\draw[] (a.east) node[dot] {} to (left1);
	\draw[] (a.south) node[dot] {} to (left);
	\draw[rho] (variable3) to (left3);
	\end{tikzpicture}
	\;,\quad
	\;
	\begin{tikzpicture}[scale=0.35,baseline=0.5cm]
	\node at (-2,-1)  [root] (root) {};
	\node at (-2,1)  [dot]  (left) {};
	\node at (-2,3)  [dot, red]  (left1) {};
	\node at (-2,5)  [dot, red] (left2) {};
	\node at (-2,7)  [dot] (left3) {};
	\node at (0,7) [] (variable3) {};
	\node[] (a) at (-4,3) {};
	\node[cumu3, rotate=90] (cumu) at (a) {};

	\draw[testfcn] (left) to  (root);
	\draw[kernel]  (left1) to (left);
	\draw[kernel]  (left2) [bend left=60] to (root);
	\draw[kernel1] (left3) to (left2);
	\draw[] (a.north) node[dot] {} to (left2);
	\draw[] (a.east) node[dot] {} to (left1);
	\draw[] (a.south) node[dot] {} to (left);
	\draw[rho] (variable3) to (left3);
	\end{tikzpicture}
	\;,\quad
	\;
	\begin{tikzpicture}[scale=0.35,baseline=0.5cm]
	\node at (-2,-1)  [root] (root) {};
	\node at (-2,1)  [dot]  (left) {};
	\node at (-2,3)  [dot, red]  (left1) {};
	\node at (-2,5)  [dot, red] (left2) {};
	\node at (-2,7)  [dot, red] (left3) {};
	\node at (0,7) [] (variable3) {};
	\node[] (a) at (-4,3) {};
	\node[cumu3, rotate=90] (cumu) at (a) {};

	\draw[testfcn] (left) to  (root);
	\draw[kernel]  (left1) to (left);
	\draw[kernel]  (left2) [bend left=60] to (root);
	\draw[kernel1] (left3) to (left2);
	\draw[] (a.north) node[dot] {} to (left2);
	\draw[] (a.east) node[dot] {} to (left1);
	\draw[] (a.south) node[dot] {} to (left);
	\draw[rho] (variable3) to (left3);
	\end{tikzpicture}
	\;.
	\end{equs}
	In the above case the right hand side in~\eqref{eq:5thcondition} yields $5.5 +$ for the first graph, $6 +$ for the second and third, $10.5 +$ for the fourth and sixth graph, $11.5 +$ for the fifth and $15+$ for the last graph. Note here, that for the last graph we for the first time really needed to make use of the fact that the weight on cumulant kernels is $3/2+$, a fact that we ignored so far most of the time. Proceeding in the same way for the remaining graphs we can conclude that all but one graph satisfy the conditions of Proposition~\ref{prop:pcondition} and therefore its $p$-th moment satisfy the conditions of Theorem~\ref{theo:ultimate1}. For the graph that is not amenable to an application of Theorem~\ref{theo:ultimate1} we already argued that it can be treated by hand.
\end{itemize}

\subsection{Finishing the proof of Proposition~\ref{prop:momentbounds}}
In this section we finish the proof of Proposition~\eqref{prop:momentbounds}. To do so we apply Theorem~\ref{theo:ultimate1} to the $p$-th moment of any graph that appears on the right hand side of~\eqref{eq:renorm}. We also note at this point that the first and second bound in~\eqref{eq:momentbounds} may be shown in exactly the same way. We first look at the graphs corresponding to the symbol $\<Xi2>$. 
Having $p$ copies of it we note that $|\CCV\setminus\CCV_\star|=p$. Moreover, the symbol at hand is constructed from two copies of the noise, therefore the sum of the labels of all cumulant terms that appear in the computation of the $p$-th moment equals $2p\times 3/2 +$. Finally, $\<Xi2>$ is constructed from a single random walk kernel, which has label $3$. Hence, we see that $\tilde\alpha=5p-3p-3p-=-p-$, as desired. Here, the last minus sign has to be understood as a additional summand of the form $-\kappa$, where $\kappa$ can be chosen to be arbitrarily small.
We can proceed in almost the exact way for the symbol $\<Xi2X>$, with the only difference that the polynomial term increases the degree by $1$. On the level of calculating the corresponding value of $\tilde\alpha$ in this case it follows from the fact that the test function and the kernel connecting to the origin are (in some cases) multiplied by a factor $x_i$ which increases the order of the respective kernel by $1$.
Regarding the two remaining symbols $\<Xi3>$ and $\<Xi4>$ we note that $\tilde\alpha$ equals in this case
\begin{equation}
\tilde\alpha=
\begin{cases}
5\times 2p-3p\times \frac32-2p\times 3-, &\mbox{ if }\tau=\<Xi3>,\\
5\times 3p- 4p\times \frac32 - 3p\times 3-, &\mbox{ if }\tau=\<Xi4>,
\end{cases}
\end{equation}
where the first term in the two lines above equal $|\CCV\setminus\CCV_\star|$, the second term comes from the total weight of all cumulant kernels, and the last one from the total weight of all random walk kernels.
We can therefore finish the proof of Proposition~\ref{prop:momentbounds}.

\section{Identification of the limit}
\label{S:Identification}
In this section we will show that there is a coupling between $\xi^N$ and $Y$, where $Y$ denotes the Ornstein--Uhlenbeck process introduced in~\eqref{eq:Y}, such that for each fixed $T>0$ and $\bar \eta \in (0,\eta\wedge \frac12)$ the sequence of solutions $u^N$ converges in probability in the space $C_{N}^{\bar \eta,T}$ to $u$, where $u$ is the solution to the PAM equation with noise $Y$.
Given the function $\rho$ introduced at the beginning of Section~\ref{SBRM} we recall that we introduced the smooth discrete noise $\xi^{\delta,N}$ via $\xi^{\delta,N}= \rho^{\delta,N}\star_N \xi^N$. We can then introduce the equation
\begin{equation}
\label{eq:PAM_discrete_smooth}
\partial_t u^{\delta,N} u(t,x) = (2^{2N}\Delta u^{\delta,N})(t,x) - [\xi^{\delta,N}(t,x) - C_{\delta,N}]u^{\delta,N}(t,x),
\end{equation}
defined on $\R_+\times 2^{-N}\St$. Here, the sequence of constants is defined as in Section~\ref{S3.3} with the only difference that each appearance of $\xi^N$ is replaced by $\xi^{\delta,N}$. We then note that denoting by $U^{\delta,N}$ the solution to the corresponding abstract fixed point problem introduced in~\eqref{eq:fixedpoint} and by $\hat\CR^{\delta,N}$ the reconstruction operator associated to the model $(\hat{\Pi}^{\delta,N},\hat{\Gamma}^{\delta,N})$ we have the identity $u^{\delta,N}= \hat{\CR}^{\delta,N} U^{\delta,N}$. 
Thus, as a consequence of the reconstruction theorem, the abstract fixed point theorem in~\cite{ErhardHairerRegularity}, 
the fact that by linearity the fixed point problem has a global in time solution, and by Proposition~\ref{prop:momentbounds} we see that for any $T>0$, and any $p\geq 2$, and $\bar\eta\in (0,\eta\wedge \frac12)$
\begin{equation}
\lim_{\delta\to 0}\lim_{N\to\infty}\E[\|u^N;u^{\delta,N}\|^p_{C_N^{\bar\eta,T}}]=0,
\end{equation}
which readily implies that $\|u^N;u^{\delta,N}\|_{C_N^{\bar\eta,T}}$ converges to zero in probability provided that we first send $N$ to infinity and then $\delta$ to zero.
To continue we define a smooth version of $Y$ via $Y^\delta= \rho^{\delta}\star Y$, and associate the equation
\begin{equation}
\partial_t u^{\delta}(t,x) = \Delta u^{\delta}(t,x) -(Y^\delta -C^\delta) u^\delta(t,x)
\end{equation}
defined on $\R_+\times \T^3$ to it. Here the sequence of constants is essentially defined as in Section~\ref{S3.3}, but all integrals are continuous and the heat kernel is the usual heat kernel associated to $(\partial_t -\Delta)$, and the noise is given by $Y^\delta$. It then follows from the methods developed in~\cite{Ajay} and~\cite{BHZalg} that the function $u^{\delta}$ converges in probability in $\CC_\s^{\bar\eta,T}$ to a limit $u$ as $\delta\to 0$ (for a proof using methods that rely on less heavy machinery we also refer to~\cite{WongZakai}). Thus, to conclude the argument it suffices to show that there exists a coupling such that
\begin{equation} \label{eq:smoothapprox}
\lim_{\delta\to 0}\lim_{N\to\infty}\|u^{\delta, N}; u^\delta\|_{C_N^{\bar\eta,T}}=0
\end{equation}
in probability.
To that end note that the input noise $Y^\delta$ is smooth, hence $u^\delta$ is as well. It therefore would be enough to show that for all $\delta >0$,
\begin{equation}\label{eq:Linftyapprox}
\lim_{N\to\infty} \sup_{(t,x)\in [0,T]\times 2^{-N}\St}|\xi^{\delta,N}(t,x)- Y^\delta (t,x)|=0
\end{equation}
in probability.
To that end we first note that it was shown in~\cite{Ravi92} that $\xi^N$ converges to $Y$ in law in the space $\cS'(\R\times\T^3)$ of all Schwartz distributions. Defining $\tilde{\xi}^{\delta,N}$ via $\tilde{\xi}^{\delta,N}= \xi^N \star_N \rho^\delta$, it therefore follows that $\tilde{\xi}^{\delta,N}$ converges in law to $Y^\delta$.
Using the bounds
\begin{equation}
\sup_{z\in\R^4}|D^k \rho^{\delta,N}(z) - D^k \rho^{\delta}(z)|\lesssim 
2^{-N} \delta^{-|\s|-|k|_\s -1},
\end{equation}
one readily shows that the same holds for $\xi^{\delta,N}$, so it only remains to establish tightness of $\xi^{\delta,N}$ in the space of continuous functions $C([0,T]\times \T^3)$. Since $\xi^{\delta,N}$ is defined only on $[0,\infty)\times 2^{-N}\St$, we first extend it to $[0,\infty)\times\T^3$ by linear interpolation (or more precisely by trilinear interpolation since the spatial dimension is three). Tightness of $\xi^{\delta,N}$ then follows from 
Kolmogorov's continuity criterion combined with Theorem~\ref{thm:cumu}, so that $\xi^{\delta,N}$ converges law in 
$C([0,T]\times \T^3)$ to $Y^\delta$. To finish the argument, it is sufficient to note that Skorokhod 
representation theorem guarantees the existence of a coupling between $\xi^{\delta,N}$ and $Y^\delta$ 
for which the convergence holds in probability, thus establishing~\eqref{eq:Linftyapprox} as claimed.

\appendix

\section{A bound on graphs}
\label{A}
In this section we adapt Theorem A.3 from Hairer, Quastel~\cite{KPZJeremy} to the present setting. Since the proof is very similar, we will only point out some of the main differences.
Before we formulate the result we describe the setup.
The basic ingredient is a finite directed multigraph $\CCG=(\CCV, \CCE)$ with edges $e\in\CCE$ labelled by pairs  $(a_e,r_e)\in\R_{+}\times\Z$ and kernels $K_e:\R\times\T_N^d\to\R$ which are compactly supported and associated to $e\in\CCE$.
If $r_e< 0$, then we will in addition be given a collection of real numbers $\{I_{e,k}\}_{|k|_\frak{s}<|r_e|}$. We further denote by $e_{\star,1},\ldots, e_{\star,M}$ the set of distinguished edges that connect a distinguished vertex $0\in\CCV$ to $M$ distinguished vertices $v_{\star,1},\ldots, v_{\star, M}$. All those edges come with a label $(a_e,r_e)=(0,0)$. We write $\CCV_0= \CCV\setminus\{0\}$ and $\CCV_{\star}=\{0,v_{\star,1},\ldots, v_{\star,M}\}$.
Given a directed edge $e\in\CCE$, we write $e_{\pm}$ for the two vertices so that $e=(e_{-},e_{+})$ is directed from $e_{-}$ to $e_{+}$. In case there is more than one edge connecting $e_{-}$ to $e_{+}$ we always assume that at most one can have nonzero renormalisation $r_e$ and in that case $r_e$ must be positive. We identify the multigraph with a graph $(\CCV,\hat\CCE)$ where the multi-edges from $e_-$ to $e_+$ are concatenated to one edge whose label $(\hat a_e,r_e)$ is the sum of the labels of the original multi-edges. The rest of the assumptions will be stated in terms of the graph $(\CCV, \hat\CCE)$.
A subset $\bar\CCV \subset \CCV$ has \textit{outgoing edges} $\hat\CCE^\uparrow (\bar \CCV)  =  
 \{ e\in \CCE : e\cap \bar\CCV = e_-\}$, \textit{incoming edges} $\hat\CCE^\downarrow (\bar \CCV)  =   \{ e\in \CCE : e\cap \bar\CCV = e_+\}$, \textit{internal edges} $\hat\CCE_0 (\bar \CCV)  =   \{ e\in \CCE : e\cap \bar\CCV = e\}$, and
\textit{incident edges} $\hat\CCE (\bar \CCV)  =   \{ e\in \CCE : e\cap \bar\CCV \neq \emptyset\}$.
We will also use
$
\hat\CCE_+ (\bar \CCV)  =   \{ e\in \CCE(\bar \CCV) :  r_e>0\}$ to denote the edges with positive renormalization and we set $\hat\CCE_+^\uparrow = \hat\CCE_+ \cap \hat\CCE^\uparrow$ and $\hat\CCE_+^\downarrow = \hat\CCE_+ \cap \hat\CCE^\downarrow$. 
\begin{assumption}\label{ass:mainGraph}
The resulting directed graph $(\CCV,\hat\CCE)$ with labels $(\hat a_e,r_e)$ satisfies:  
No edge with $r_e \neq 0$
connects two elements in $\CCV_\star$ and $0 \in e \Rightarrow r_e = 0$; no more than one edge with 
negative renormalization $r_e<0$ may emerge from the same vertex; and
\begin{claim}
\item[1.] For all $e\in \hat\CCE$, one has $ \hat a_e + (r_e \wedge 0) < |\s|$\;; 
\item[2.] For every subset $\bar \CCV \subset \CCV_0$ of cardinality at least $3$,  
\begin{equ}[e:assEdges]
\sum_{e  \in \hat\CCE_0(\bar \CCV)} \hat a_e < (|\bar \CCV| - 1)|\s|\;;
\end{equ}
\item[3.]  For every subset $\bar \CCV \subset \CCV$ containing $0$ of cardinality at least $2$,  
\begin{equ}[e:assEdges2]
\begin{aligned}
\sum_{e  \in \hat \CCE_0 (\bar \CCV) } \hat a_e 
&+\sum_{e  \in \hat\CCE^{\uparrow}_+ (\bar \CCV)  } (\hat a_e+ r_e -1)  - \sum_{e  \in \hat\CCE^{\downarrow}_+ (\bar \CCV)  }  r_e
&< (|\bar \CCV| - 1)|\s|\;;
\end{aligned}
\end{equ}
\item[4.]  For every non-empty subset $\bar \CCV \subset \CCV\setminus \CCV_\star$,
\begin{equ}[e:assEdges3]
 \sum_{e\in \hat\CCE(\bar\CCV)\setminus \hat\CCE^{\downarrow}_+(\bar \CCV) }  \hat a_e 
 +\sum_{e\in \hat \CCE^{\uparrow}_+(\bar\CCV)}  r_e
- \sum_{e \in \hat \CCE^\downarrow_+(\bar \CCV)} (r_e-1)
> |\bar \CCV||\s|\;.
\end{equ}
\end{claim}
\end{assumption} 
Next, we describe the role of $r_e$. If $r_e<0$, then we associate to $K_e$ the distribution
\begin{equ}[e:defRen]
\bigl(\Ren K_e\bigr)(\phi) = \int_{\R\times\T_N^d} K_e(z) \Bigl(\phi(z) - \sum_{|k|_\s < |r_e|} {z^k \over k!} D^k\phi(0)\Bigr)\,dz + \sum_{|k|_\s < |r_e|} {I_{e,k} \over k!} D^k\phi(0).
\end{equ}  
The above describes a distributional kernel $\hat K_e$ acting on smooth $\phi$ on $\R\times\T_N^d\times \R\times\T_N^d$ by
\begin{equ}
 \hat K_e (  \phi)
\eqdef\tfrac12\int  
\Ren K_e( \phi_{z})\,dz\;,
\end{equ} where $\phi_{ z}(\bar z)\eqdef\phi((z+\bar z)/2,(z-\bar z)/2)$.

For $r_e \ge 0$, we define
\begin{equ}[e:defKhatn2]
\hat K_e(z_{e_-}, z_{e_+}) = 
K_e(z_{e_+}-z_{e_-}) - \sum_{|j|_\s < r_e} {z_{e_+}^j \over j!} D^j K_e(-z_{e_-})\;.
\end{equ}
We now describe the role of $a_e$.

\begin{assumption}\label{ass:kernelsRenorm}
Given $K_e$ as above, we assume that there exist $\{K_e^{(n)}\}_{0\leq n\leq  N}$ satisfying:
\begin{claim}
\item  $K_e(z) = \sum_{0\leq n \leq N} K_e^{(n)}(z)$ for all $z \neq 0$. Here, $K_e^{(n)}$ maps $\R\times \T_N^d$ or $\R^{d+1}$ into $\R$;
\item  $\bigl(\Ren K_e\bigr)(\phi) = \sum_{0\leq n \leq N} \int_{\R\times\T_N^d} K_e^{(n)}(z)\,\phi(z)\,dz$ for smooth test functions $\phi$;
\item For all $0\leq n < N$ the functions $K_e^{(n)}$ are supported in the annulus
$2^{-(n+2)}\le  \|z\|_\s \le  2^{-n}$ and $\mathrm{supp}(K_e^{(N)})\subseteq \{z\colon\, \|z\|_\s\leq 2^{-N}\}$;
\item for all $p<\infty$ and some $C<\infty$ 
\begin{equation}\label{e:boundKn}
\begin{aligned}
\sup_{|k| \le p, 0\leq n\leq N-1 } 2^{-(a_e + |k|_\s)n}
|D^k K_e^{(n)}(z)| &\le C \quad \text{and}\\
2^{-a_eN} |K_e^{(N)}(z)| & \le C \, ;
\end{aligned}
\end{equation}
\item if $r_e < 0$, then $
\int_{\R\times\T_N^d} P(z) K_e^{(n)}(z)\,dz = 0$
for all $0<n\leq N$ and all polynomials $P$ with scaled degree strictly less than $|r_e|$.
\end{claim}
\end{assumption}
\noindent
In the sequel for a kernel $K_e$ we denote the largest constant $C$ appearing in the fourth item by $\|K_e\|_{a_e;p}$.
For a smooth test function $\phi$ and $z=(t,x)$, let
$
\phi_{\lambda,\mu}(z) = \lambda^{-|\s\setminus\s_1|}\mu^{-|\s_1|} \phi(t/\mu^{|\s_1|},x/\lambda)$.  The key quantity of interest is the generalized convolution
\begin{equ}\label{genconv}
\CI^\CCG(\phi_{\lambda,\mu},K) \eqdef  \int_{(\R\times\T_N^d)^{\CCV_0}} \prod_{e\in \CCE}{\hat K}_e(z_{e_-}, z_{e_+})\prod_{i=1}^M \phi_{\lambda,\mu}(z_{v_{\star,i}}) \,dz\,
.
\end{equ}
In case $\lambda=\mu$ we simply write $\CI^\CCG(\phi_{\lambda},K)$.
\begin{theorem}\label{theo:ultimate1}  Let $\CCG=(\CCV,\CCE)$ be a finite directed multigraph with labels $\{a_e,r_e\}_{e\in \CCE}$ and kernels
$\{K_e\}_{e\in \CCE}$ with  the resulting graph satisfying Assumption~\ref{ass:mainGraph} and its preamble and the kernels satisfying Assumption~\ref{ass:kernelsRenorm}.
Then, there exists  $C$ depending only on the 
structure of the graph $(\CCV,\CCE)$ and the labels $r_e$ such that 
\begin{equ}\label{eq:genconvBound}
\CI^\CCG(\phi_{\lambda,\mu},K) 
\le C\lambda^{\tilde \alpha}\;,
\end{equ}
if either $2^{-N}\leq \lambda=\mu$ or $\mu <2^{-N}=\lambda$.
Here,
\begin{equ}\tilde \alpha = |\s||\CCV\setminus \CCV_\star| - \sum_{e\in \CCE} a_e. 
\end{equ} 
\end{theorem}

The proof follows from an adaptation of the arguments in~\cite[Theorem A.3]{KPZJeremy}, and it relies on a combination of volume estimates and supremum estimates. The latter are often derived via an application of Taylors formula. Regarding the supremum estimates we can essentially copy those from~\cite{KPZJeremy}, the only difference is that our kernels do not explode but its derivatives are bounded by some positive powers of $2^N$. The difference appears mostly in the volume estimate. However, in the case $\mu=\lambda$ the adaptations are only minor, hence we will mostly comment on the case $\mu < 2^{-N}=\lambda$ in the sequel.
To prepare the proof we enhance (as in~\cite{KPZJeremy}) the set of edges in our graph to include any $(v,w)\in\CCV^2$ that is not already contained in $\CCE$. To all such new edges we assign the kernel $\hat K_{(v,w)}\equiv 1$ and the corresponding label is $(a_e,r_e)=(0,0)$.
Given $\Psi:\R^{d+1}\to\R$ smooth and supported in $\{z:\, \|z\|_\s\in [3/8,1]\}$ and such that $\sum_{n\in\Z}\Psi(2^nz)=1$ for all $z\neq 0$ and letting $\Psi^{(n)}(z)=\Psi(2^nz)$ we further define
\begin{equation}
\label{eq:Psiminus}
\Psi^{(-)}(z) = \sum_{n\leq 0}\Psi^{(n)}(z)\quad\mbox{and}\quad
\Psi^{(\geq N)}(z)= \sum_{n\geq N}\Psi^{(n)}(z).
\end{equation}
Let $\CCE_{\star}= \{e\in\CCE\colon\, \CCV_{\star} \cap e = e\}$.
\begin{definition}\label{def:khat}
For  $\fn_1:\CCE\setminus\CCE_{\star} \to \{0,1,\ldots,N-1,N\}^{3}$ and $e\in\CCE\setminus\CCE_{\star}$ define  $\hat K_e^{(\fn_1(e) )}(z_1,z_2)$ as  follows:
{\em If} $r_e \le 0$, then  $\hat K_e^{(\fn_1(e))} = 0$ unless 
$\fn_1(e) = (k,0,0)$ in which case  $\hat K_e^{(\fn_1(e))}(z_1,z_2) = K_e^{(k)}(z_2-z_1)$;
{\em if} $r_e > 0$, then $\hat K_e^{(\fn_1(e))} = 0$ unless $\fn_1(e)=(k,p,m)\in\{0,1,\ldots,N-1\}^3$,
in which case $\hat K_e^{(\fn_1(e))} =  K_e^{(k,p,m)}$, where
\begin{equation}\label{eq:defKhatn}
K_e^{(k,p,m)}(z_1,z_2) = \Psi^{(k)}(z_2-z_1)\Psi^{(p)}(z_1)\Psi^{(m)}(z_2) 
\Bigl(K_e(z_2-z_1) - \sum_{|j|_\s < r_e} {z_2^j \over j!} D^j K_e(-z_1)\Bigr)\; ;
\end{equation}
or $\fn_1(e)=(N,0,0)$, then $\hat K_e^{(\fn_1(e))}(z_1,z_2) = K_e^{(N)}(z_2-z_1)$.
Let $\fn_2:\CCE_{\star}\to \{0,1,\ldots, \lceil |\log_2(\mu)|\rceil\}^3$ and $e\in\CCE_{\star}$.
We define $\hat K_e^{(\fn_2(e))}$ as $\hat K_e^{(\fn_2)} = 0$ unless 
$\fn_2(e) = (k,0,0)$ in which case  $\hat K_e^{(\fn_2(e))}(z_1,z_2) = \Psi^{(k)}(z_2-z_1)K_e^{(k)}(z_2-z_1)$.
\end{definition}
\begin{remark}
The reason to define $\hat K_e^{(\fn_1(e))}$ separately for $\fn_1(e)=(N,0,0)$ is that at that scale $2^{-N}$ our version of the discrete heat kernel is not differentiable since we extended it to be zero for all negative times, so that at $t=0$ differentiability fails.	
	\end{remark}
\begin{remark}
The reason to distinguish between the set of edges $\CCE_{\star}$ and the remaining ones is that in~\eqref{genconv} the scaled test functions only act on $\CCV_{\star}$. Along these the scaling is different.
Note also that Assumption~\ref{ass:mainGraph} implies that $r_e=0$ for $e\in\CCE_\star$.
In case $\lambda=\mu$ the map $\fn_2$ is not necessary and one extends $\fn_1$ in a natural way to act on all of $\CCE$.
\end{remark}
We are now in a position to define 
\begin{equ}\label{def:kayhat}
\hat K^{(\fn)}(z) = \prod_{e \in \CCE} \hat K_e^{(\fn_e)}(z_{e_-},z_{e_+}),
\end{equ}
where $z_0=0$ and with $\fn_1,\fn_2$ from above we set $\fn_e=\fn_1(e)$ if $e\in\CCE\setminus\CCE_{\star}$ and $\fn_e= \fn_2(e)$ otherwise. We write $\fn=\fn_1\cup\fn_2$ if $\fn$ is given in such a form.
For $0<\mu < \lambda=2^{-N} \leq 1$ we let
\begin{equ}
\CN_{\lambda,\mu}\eqdef\{ 
\fn=\fn_1\sqcup\fn_2 \colon\, 2^{-|\fn_2(e_{\star,i})|}\leq \mu,\,  i=1,\ldots,M\}\;,
\end{equ}
Of course the above definition forces $\fn_2$ to be almost constant.
Let
\begin{equ}[e:bigsum]
\CI_{\lambda,\mu}^\CCG(K) \eqdef \sum_{\fn \in \CN_{\lambda,\mu}} \int_{(\R\times\T_N^d)^{\CCV_0}} \hat K^{(\fn )}(z)\,dz\;.
\end{equ}
As in~\cite{KPZJeremy} one can show that Theorem~\ref{theo:ultimate1} follows from the following lemma.
\begin{lemma}
\label{lem:wantedBound2}
Under the same assumptions and with the same notation as in Theorem~\ref{theo:ultimate1}, the 
inequality 
\begin{equ}\label{e:wantedBound2}
|\CI_{\lambda,\mu}^\CCG(K)| \leq C\lambda^\alpha (\mu^{|\s_1|}2^{N|\s_1|})^M\prod_{e\in\CCE}\|K_e\|_{a_e;p},
\end{equ}
holds. 
If $\lambda=\mu$, then we even have
\begin{equ}\label{e:wantedBound21}
|\CI_{\lambda,\mu}^\CCG(K)| \leq C\lambda^\alpha \prod_{e\in\CCE}\|K_e\|_{a_e;p},
\end{equ}
In both cases $\alpha =|\s| |\CCV_0| - \sum_{e \in \CCE} a_e$.
\end{lemma}
Indeed, this follows upon observing that the scaled testfunctions are kernels $K_e$ with $a_e=0$ and
$\|K_e\|_{a_e;p}\lesssim \lambda^{-|\s\setminus\s_1|}\mu^{-|\s_1|}$.
We now explain how to adapt the proof of~\eqref{e:wantedBound2} to the current setting.
First of all, as in Section A.2 in~\cite{KPZJeremy}, given $z\in(\R\times\T_N^d)^{\CCV}$
with $z_0=0$ one can construct a labelled rooted binary tree $T$ whose leaves are the $v\in\CCV$.
The inner nodes of this tree are denoted by $\nu,\omega$ etc., whereas the leaves are denoted by $v,w\in\CCV$ etc. We denote by $\nu\wedge \omega$ the most recent ancestor of $\nu$ and $\omega$.
Similarly as in~\cite{KPZJeremy} the inner nodes come with a label
\begin{equ}\label{eq:label}
\ell_{\nu} = \max_{v\wedge w = \nu}\lfloor -\log_2 \|z_v-z_w\|_{\s,N}\rfloor \,,
\end{equ}
unless, if $\nu$ is the most recent ancestor of two elements in $\CCV_{\star}$, then we assign the pair of labels
\begin{equ}\label{eq:labeltime}
(\ell_{\nu}^{T}, \ell_{\nu}) = (\lceil |\log_2(\mu)|\rceil,N) \,
\end{equ}
to it. Here, the upper index $T$ stands for time. The reasoning behind these labels is that at scale $2^{-N}$ space does not vary anymore, and since all our kernels do not change their order of explosion below scales of order $2^{-N}$ we do not need to control the time variable beyond scale $\mu$ which is imposed by the problem at hand.  Moreover, one can construct the tree in such a way that $\ell_\nu^\alpha\geq \ell_\omega^\beta$ provided that $\nu\geq\omega$, where $\alpha,\beta$ denote label that can either be empty or equal to $T$. Furthermore those labels can be chosen such that
\begin{equation}
\|z_v-z_w\|_{\s}\lesssim 2^{-\ell_{v\wedge w}}\,,
\end{equation}
unless $v,w\in\CCV_\star$ in which case $\ell_{v\wedge w}$ has to be replaced by $\ell_{v \wedge w}^T$. We denote the set of labelled rooted binary trees obtained in such a way by $\T(\CCV)$ and we denote its elements by $(T,\ell,\ell^{T})$.
\begin{definition} 
For $c= \log|\CCV|+2$, let  $\CN(T,\ell,\ell^T)$ consist of all functions $\fn =\fn_1\cup\fn_2$ as alluded to above such that
\begin{itemize}
\item for every edge $e = (v,w)\in\CCE\setminus\CCE_{\star}$ if $r_e \le 0$, one has $\fn_e = (k,0,0)$ with
$|k - \ell_{v\wedge w}| \le c$ and if $r_e > 0$, one has $\fn_e = (k,p,m)$ with
$|k - \ell_{v\wedge w}| \le c$, $|p - \ell_{v\wedge 0}| \le c$, and
$|m - \ell_{w\wedge 0}| \le c$;
\item for every edge $e=(v,w)\in\CCE_{\star}$ with $\fn_e=(k,0,0)$  one has 
$|k - \ell_{v\wedge w}| \le c$.
\end{itemize} 
\end{definition}
Denote by $\T_{\lambda,\mu}(\CCV)$ the set of those labelled trees in $\T(\CCV)$ such that  $2^{-\ell_{v\wedge w}^T}\leq \mu$ for all $v,w\in\CCV_\star$. In a way similar to~\cite{KPZJeremy} we obtain the following lemma.
\begin{lemma}
With the above notation we have that
\begin{equation}
|\CI_{\lambda,\mu}^{G}(K)|\lesssim \sum_{(T,\ell,\ell^T)\in \T_{\lambda,\mu}(\CCV)}\sum_{\fn\in\cN(T,\ell,\ell^T)}
\Big|\int_{(\R\times\T_N^d)^{\CCV_0}}\hat{K}^{(\fn)}(z)\, dz\Big|.
\end{equation}
\end{lemma}
Given a rooted binary tree $T$ with distinguished leaves $0,v_{\star,1},\ldots, v_{\star,M}$ and a distinguished inner node $\nu_{\star}$ we denote by $T^\circ$ the set of inner nodes of $T$. 
We consider the set $\cN_{\lambda,\mu}(T^\circ)$ of all integer labelings as above. 
Finally, given two functions $\eta, \eta^{T}:T^\circ\to \R$ we write
\begin{equation}
\label{eq:Ilambdamu}
\cI_{\lambda,\mu}(\eta,\eta^{T}) = \sum_{(\ell,\ell^T)\in\cN_{\lambda,\mu}(T^\circ)}\prod_{\nu\in T^\circ}2^{-\ell_{\nu}\eta_{\nu}-\ell_{\nu}^{T}\eta_{\nu}^{T}}.
\end{equation}
Setting $|\eta|= \sum_{\nu\in T^\circ}\eta_\nu$ and likewise for $\eta^T$ we have the following result.
\begin{lemma}
\label{lem:eta}
Assume that $\eta$ satisfies the following two properties:
\begin{enumerate}
\item For every $\nu \in T^\circ$, one has $\sum_{\upsilon \ge \nu} \eta_\upsilon > 0$.
\item For every $\nu \in T^\circ$ such that $\nu\le \nu_\star$, one has 
$\sum_{\upsilon \not\ge \nu} \eta_\upsilon < 0$, provided that this sum contains at least one term.
\end{enumerate}
If moreover $\eta^T$ is such that $\eta_{\nu}^T\neq 0 \Leftrightarrow \ell_{\nu}^T>0$ and in this case $\eta_{\nu}^T$ is positive, then one has $\CI_{\lambda,\mu}(\eta,\eta^T) \lesssim \lambda^{|\eta|}\mu^{|\eta^T|}$, uniformly over $0<\mu <2^{-N}= \lambda$.
\end{lemma}
\begin{proof}
We denote by $\cN_\lambda(T^\circ)$ the set of labelings as above, with the difference that all labels $\ell_{\nu}^T$ are not present and such that $2^{-\ell_\nu}\leq \lambda$ if $\nu=v_\star\wedge w_\star$ for $v_\star,w_\star\in \CCV_\star$.
We then see that
\begin{equation}
\cI_{\lambda,\mu}(\eta,\eta^{T})\leq  \Big(\sum_{\ell\in\cN_{\lambda}(T^\circ)}\prod_{\nu\in T^\circ}2^{-\ell_{\nu}\eta_{\nu}}\Big)
\times \prod_{\substack{\nu=v_\star\wedge w_\star,\\ v_\star,w_\star\in\CCV_\star}}2^{-\ell_{\nu}^T\eta_{\nu}^T}.
\end{equation}
It has been shown in~\cite[Lemma A.10]{KPZJeremy} that the first term on the right hand side above is bounded by a multiple times $\lambda^{|\eta|}$. Since $\ell_\nu^T = \lceil |\log_2(\mu)|\rceil$, the bound on the second term follows at once. 
\end{proof}
We denote by $T^{\circ}_{\star}$ the set of nodes $\nu$ such that $\nu=v_{\star}\wedge w_{\star}$ for some $v_{\star},w_{\star}\in\CCV_{\star}$.  Moreover, given a labelled tree $(T,\ell,\ell^T)$ we denote by $\CD(T,\ell,\ell^T)$ the subset of those $z\in (\R\times\T_N^d)^{\CCV}$ such that $\|z_v-z_w\|_\s\leq |\CCV|2^{-\ell_{v\wedge w}}$ unless $v,w\in T^{\circ}_\star$ in which case we assume that $\|z_{v}-z_{w}\|_\s\leq |\CCV|2^{-\ell_{v\wedge w}^T}$.
\begin{lemma}
Assume that there is a collection of functions $\tilde{K}^{(\fn)}$ indexed by labelled trees $(T,\ell,\ell^T)$ and $\fn\in\cN(T,\ell,\ell^T)$ such that for labels
\begin{equation}
\supp (\tilde{K}^{(\fn)})\subseteq \CD(T,\ell,\ell^T), \quad
\text{and}\quad 
\int_{(\R\times\T_N^d)^{\CCV}}K^{(\fn)}(z)\, dz= \int_{(\R\times\T_N^d)^{\CCV}}\tilde{K}^{(\fn)}(z)\, dz,
\end{equation}
then
\begin{equation}
\label{eq:secondmainstep}
|\CI_{\lambda,\mu}^\CCG(K)|\lesssim 
 \sum_{(T,\ell,\ell^T) \in \T_{\lambda,\mu}(\CCV)} \Bigl(\prod_{v \in T^\circ\setminus T^{\circ}_{\star}} 2^{-\ell_v |\s|}\prod_{v\in T^{\circ}_{\star}}2^{-\ell_v^T|\s_1|-\ell_v|\s\setminus\s_1|}\Bigr) \sup_{\fn \in \CN(T,\ell,\ell^T)} \sup_{z}|\tilde K^{(\fn)}(z)|\;.
\end{equation}
\end{lemma}
\begin{proof}
The proof follows in the same way as in~\cite[Lemma A.13]{KPZJeremy}. The only difference being that the volume of
$\CD(T,\ell,\ell^T)$ is bounded from above by a multiple times 
\begin{equation}
\prod_{v \in T^\circ\setminus T^{\circ}_{\star}} 2^{-\ell_v |\s|}\prod_{v\in T^{\circ}_{\star}}2^{-\ell_v^T|\s_1|-\ell_v|\s\setminus\s_1|},
\end{equation}
which can be shown as in the aforementioned article.
\end{proof}
\begin{remark}
In case $\lambda=\mu$ the arguments in~\cite{KPZJeremy} show that the volume bound above takes the form
\begin{equation}
\prod_{v \in T^\circ} 2^{-\ell_v |\s|}.
\end{equation}
\end{remark}
To proceed we define
\begin{equation}
\label{eq:eta}
\eta(\nu)=
|\s|+ \sum_{e\in \hat\CCE}\eta_e(\nu),
\end{equation} 
where
\begin{equation}
\begin{aligned}
\eta_e(v)   &= - \hat{a}_e \one_{e_\uparrow} (v)
+ r_e  (\one_{e_+\wedge 0} (v) 
- \one_{e_\uparrow} (v)\bigr) \one_{r_e >0, e_+\wedge 0 > e_\uparrow} \\
&\quad + (1-r_e -  \hat{a}_e) (\one_{e_-\wedge 0} (v)
- \one_{e_\uparrow} (v)\bigr) \one_{r_e >0, e_-\wedge 0 > e_\uparrow}\;.
\end{aligned}
\end{equation}
Here we used the notation $e_\uparrow$ as a shorthand for $e_+\wedge e_-$.
Moreover, we introduce an additional map $\eta^T:T^{\circ}\to \R$ which is given by
\begin{equation}
\label{eq:etaT}
\eta^T(\nu)= 
\begin{cases}
|\s_1|,\, &\mbox{if }\nu\in T_*^{\circ}\;,\\
0, &\mbox{otherwise.}
\end{cases}
\end{equation}
We then have the following result.
\begin{lemma} \label{lem:naivebound}
With the above notation we have the bound
\begin{equation}
\label{eq:wantedProdBoundKHat}
\begin{aligned}
\Bigl(\prod_{v \in T^\circ\setminus T^{\circ}_{\star}} &2^{-\ell_v |\s|}\prod_{v\in T^{\circ}_{\star}}2^{-\ell_v^T|\s_1|-\ell_v|\s\setminus\s_1|}\Bigr)\sup_z |\hat K^{(\fn)}(z)|\\
& \lesssim \prod_{v \in T^\circ} 2^{-\ell_v \eta(v)-\ell_v^{T}\eta^T(v)}\prod_{v\in T_*^{\circ}}2^{\ell_v|\s_1|}
\leq \prod_{v \in T^\circ} 2^{-\ell_v \eta(v)-\ell_v^{T}\eta^T(v)}2^{N|\s_1|M} \;,
\end{aligned}
\end{equation}
uniformly over all $\fn \in \CN(T,\ell,\ell^T)$. Here, the label $\ell_v^T$ is decreed to be zero if $v\notin T^{\circ}_\star$.
\end{lemma}
\begin{proof}
The second inequality in~\eqref{eq:wantedProdBoundKHat} is a direct consequence of the fact that $2^{-\ell_v}= 2^{-N}$ for all $v\in T_*^{\circ}$. 
Using the multiplicative structure on both sides of~\eqref{eq:wantedProdBoundKHat}, we see that the first inequality follows if for each edge $e=(e_-,e_+)\in \hat\CCE$
\begin{equation}
\label{eq:prodstructure}
\sup_z |\hat K_{e}^{(\fn_e)}(z_{e_-}, z_{e_+})|\lesssim \prod_{v\in T^{\circ}}2^{-\ell_v \eta_e(v)}.
\end{equation}
To see that this is true one may proceed as in the proof of~\cite[Lemma A.15]{KPZJeremy}.
\end{proof}
From here on, the arguments follow very closely those in~\cite[Appendix A]{KPZJeremy}. We therefore omit the details. 
We note however that it first sight it seems like it is the case that our choice of $\eta$ needs to deviate from the choice made in~\cite{KPZJeremy}. Indeed, with $\tilde{\eta}_e$ as in~\cite[Equation (A.27)]{KPZJeremy}, Equation~\ref{eq:secondmainstep} seems to imply that one needs to choose
\begin{equation}
\label{eq:choiceofeta}
\eta(\nu)= \begin{cases}
|\s|+ \sum_{e\in \hat\CCE}\tilde{\eta}_e(\nu),\, &\mbox{if }\nu\in T^\circ\setminus T_{\star}^{\circ}\;,\\
|\s\setminus \s_1|+ \sum_{e\in \hat\CCE}\tilde{\eta}_e(\nu),\, &\mbox{if }\nu\in T_{\star}^\circ\;,
\end{cases}
\end{equation}

However, going back to~\ref{eq:secondmainstep} one can simply rewrite the product as
\begin{equ}
\prod_{v\in T_\star^{\circ}}2^{-\ell_v|\s|} \Big(\frac{\mu}{2^{-N}}\Big)^{M|\s_1|}\,,	
	\end{equ}
which puts us into the setting of~\cite{KPZJeremy} with the only difference arising from the factor involving $\eta$ and $2^{-N}$, which appears in that form in Lemma~\ref{lem:wantedBound2}.

\section{Proof of Proposition~\ref{prop:renconstant}}
\label{C}

To prepare for the proof of Proposition~\ref{prop:renconstant}, we first note that as a consequence of~\cite[Section 5]{MatetskiDiscrete} for all $(t,x)\in\R\times \T_N^3$ we can write $K^N(t,x)= 2^{3N}p_{t2^{2N}}(2^Nx) + R^N(t,x)$, where $R^{N}$ is a smooth, compactly supported functions, which is bounded uniformly in $N$. Hence, modulo an additive constant that is uniformly bounded in $N$, 
we can replace each appearance of $K^N$ in~\eqref{eq:renconst1} and~\eqref{eq:c1c22c23} by $ 2^{3N}p_{\cdot 2^{2N}}(2^N\cdot)$.
The proof of Proposition~\ref{prop:renconstant} then relies on the following two lemmas.
\begin{lemma}
	\label{lem:SRWest}
	Uniformly in all time-space points $z=(t,x)\in\R_+\times\Z_N^3$ we have the estimate
	\begin{equation}\label{eq:SRWest}
	p_t(x)\lesssim 1\wedge \|z\|_\s^{-3}.
	\end{equation}
	Moreover, for all $i\in\{1,2,3\}$
	\begin{equation}\label{eq:gradientest}
	|p_t(x+e_i)-p_t(x)|\lesssim 1\wedge \|z\|_\s^{-4},
	\end{equation}
	where $e_i$ denotes the $i$-th unit vector.
\end{lemma}
\begin{proof}
	For the proof of \eqref{eq:SRWest} we make use of~\cite[Theorem 2.5.6]{LawlerLimic}.
	\begin{theorem}[Theorem 2.5.6 in~\cite{LawlerLimic}]
		\label{thm:heatkernelest}
		Let $\|x\|\leq t/2$, then 
		\begin{equation}
		p_t(x) = \frac{1}{(4\pi t)^{3/2}}\exp\Big(-\frac{\|x\|^2}{4t}\Big)\exp\Big(\CO\Big(\frac{1}{\sqrt{t}}+\frac{\|x\|^3}{t^2}\Big)\Big).
		\end{equation}
	\end{theorem}
	We note that the above theorem is only useful if $|t|\geq 1$, and  $\|x\|^3\leq t^2$. However, using standard large deviations estimates the proof of~\eqref{eq:SRWest} follows.
	We turn to the proof of~\eqref{eq:gradientest}.
	Denote by $K$ the usual (continuous) heat kernel. It was shown in~\cite[Theorem 2.3.6]{LawlerLimic} via an analysis of characteristic functions that for all $n\in\N$ and $x\in\Z^3$,
	\begin{equation}\label{eq:disgradientest}
	|p_n(x+e_i)-p_n(x)-(K_n(x+e_i)-K_n(x))|\lesssim n^{-3},
	\end{equation}
	where $p_n$ denotes the transition probability of a discrete time simple random walk. Using that for continuous time simple random walk the characteristic function is given by
	\begin{equation}
	\Phi(\theta)= \frac{1}{3}\sum_{i=1}^{3}\cos (\theta_i),\quad\theta\in\R^3,
	\end{equation}
	it is not difficult to reproduce the proof of~\eqref{eq:disgradientest} for the continuous time simple random walk for all times $t\geq 1$. Note that $|K_t(x+e_i)-K_t(x)|\lesssim \|z\|_\s^{-4}$ and that
	\begin{equation}
	t^{-3}\leq (\sqrt{t}+|x|)^{-4},
	\end{equation}
	whenever $t\geq 1$ or $|x|\leq t^{3/4}$. Thus, for this range of time-space points we see that~\eqref{eq:gradientest} is satisfied. Thus, once again we can finish the proof using standard large deviation estimates.
\end{proof}
We are now in a position to prove Proposition~\ref{prop:renconstant}.
\begin{proof}
	First of all we note that for $z=(t,x)\in\R\times \T_N^3$ we have that $\E_c(\xi^N(0),\xi^N(z))= 2^{3N}p_{t2^{2N}}(2^Nx)$.
	Hence,
	\begin{equation}
	c_N\approx 2^{-3N} \int \sum_{x\in\Z_N^3} (2^{3N}p_{t2^{2N}}(x))^2\, dt,
	\end{equation}
	where the equality is only approximately true, since we removed the contribution of $R^{N}$ from $K^N$. 
	Using the substitution $\hat t = t2^{2N}$, we can rewrite the last term above as
	\begin{equation}\label{eq:CN}
	2^{N}\int \sum_{x\in\Z_N^3}(p_t(x))^2\, dt.
	\end{equation}
	Therefore, as a consequence of Lemma~\ref{lem:SRWest} we see that the integral above is bounded by some proportionality constant times
	\begin{equation}
\int_\R \sum_{x\in\Z^3} 1\wedge (\sqrt{t}+|x|)^{-6}\, dt,
	\end{equation}
	which is finite, and terminates the analysis of $c_N$.

	We turn to the analysis of $c_N^{(1)}$. We may assume that $z_1,z_2$ are time-space points such that $0>t_1>t_2$, since otherwise at least one of the kernels appearing in the definition would be zero. It then follows from a direct computation (alternatively also from the results in Section~\ref{S4.3}) that
	\begin{equation}
	\E_c(\xi^N(z_1),\xi^N(z_2),\xi^N(0))= 2^{4.5N} p_{-t_12^{2N}}(-2^Nx_1) p_{(t_1-t_2)2^{2N}}(2^{N}(x_1-x_2)).
	\end{equation}
	Thus, we can write
	\begin{equation}
	\begin{aligned}
	c_N^{(1)} \approx 2^{4.5N}\int \sum_{x_1,x_2\in\Z_N^3}
	(p_{-t_12^{2N}}(-x_1))^2 (p_{(t_1-t_2)2^{2N}}(x_1-x_2))^2\, dt_1\, dt_2,
	\end{aligned}
	\end{equation}
	where the power $4.5$ is really the sum of all weights in front of the transition probabilities above, and the Riemann factor $2^{-6N}$.
	Substituting $\hat t_1= 2^{2N}t_1$ and $\hat t_2=2^{2N}t_2$, the above becomes
	\begin{equation}
	\begin{aligned}
	2^{0.5N}\int \sum_{x_1,x_2\in\Z_N^3}
	(p_{-t_1}(-x_1))^2 (p_{(t_1-t_2)}(x_1-x_2))^2\, dt_1\, dt_2.
	\end{aligned}
	\end{equation}
	Using Lemma~\ref{lem:SRWest} it is plain to see that the integral term above is bounded in $N$.
	 Therefore, we can conclude that $c_N^{(1)}\approx \beta_N2^{0.5N}$, as announced.
	We proceed to analyse $c_N^{(2,1)}$.
	To that end we note that as a consequence of the analysis in Section~\ref{S4.3} for time-space points $z_0=(t_0,x_0),\ldots, z_3=(t_3,x_3)$ with $0=t_0>t_1>t_2>t_3$ and $x_0=0$ we can write
	\begin{equation}
	\E_c(\xi^N(z_i),\, i\in\{1,2,3,4\})= \I + \II+ \III,
	\end{equation}
	with
	\begin{equation}\label{eq:4thcumu}
	\begin{aligned}
	\I &= 2^{6N}\prod_{i=0}^{2}p_{(t_{i}-t_{i+1})2^{2N}}(x_{i+1}-x_i),\\
	\II&= 2^{6N}\sum_{y_1,y_2\in\Z_N^3} p_{(t_2-t_3)2^{2N}}(y_2-x_3)p_{c;(t_1-t_2)2^{2N}}^{y_2,x_2}(y_1,x_1)
	p_{-t_12^{2N}}(-y_1),\\
	\III &= 2^{6N}\sum_{y_1,y_2\in\Z_N^3} p_{(t_2-t_3)2^{2N}}(y_2-x_3)p_{c;(t_1-t_2)2^{2N}}^{y_2,x_2}(x_1,y_1)
	p_{-t_12^{2N}}(-y_1)\,,
	\end{aligned}
	\end{equation}
	where the difference between the second and third term is that in the middle term $x_1$ and $y_1$ are swapped.
	Here, $p_{c;\cdot}$ denotes the connected transition probability, which was introduced in~\eqref{eq:connectedtransitionproba}. For the purposes of the current analysis it is enough to know that (see also Proposition~\ref{prop:productofmartingales})
	\begin{equation}
	p_{c;t}^{x_1,x_2}(y_1,y_2)= \sum_{e} \int_0^t p_{s}^{x_1,x_2}(e_-,e_+) \nabla p_{t-s}^{y_1}(e_-)\nabla p_{t-s}^{y_2}(e_+)\, ds,
	\end{equation}
	where the sum is over all directed edges $e=(e_-,e_+)$ of $\Z_N^3$. 
	As was already done for $c_N$ and $c_N^{(1)}$ we use for each $i\in\{1,2,3\}$ the substitution $\hat t_i=t_i2^{2N}$. Then collecting all scaling factors in front of the various transition probabilities, space integrals and the factor $2^{6N}$ in~\eqref{eq:4thcumu} we see that they cancel each other out. 
	Now, Lemma~\ref{lem:SRWest} shows that the contribution coming from the first term in~\eqref{eq:4thcumu} yields an expression that is bounded in $N$.
	Therefore it remains to plug in $\II$ and $\III$ into the first line of~\eqref{eq:c1c22c23}. Summing over $y_1$ and $y_2$, using Lemmas~\ref{lem:SRWest}, \ref{lem:scaling} and~\ref{lem:convolutioninspace} we see that we obtain the expression
	\begin{equation}
	\II,\III\lesssim
	\tikzsetnextfilename{boundII,III}
	\begin{tikzpicture}[scale=0.35,baseline=1.3cm]
	\node at (-2,1)  [root] (left) {};
	\node at (-2,3)  [dot] (left1) {};
	\node at (-2,5)  [dot] (left2) {};
	\node at (-2,7)  [dot] (left3) {};
	\node at (0,4)   [dot] (right1) {};

	\draw[kernel] (left1) to node[midway, left] {\tiny 3} (left);
	\draw[kernel] (left2) to node[midway, left] {\tiny 3} (left1);
	\draw[kernel] (left3) to node[midway, left] {\tiny 3} (left2);
	\draw[] (left3) to node[midway, above] {\tiny 3} (right1);
	\draw[] (left2) to node[midway, above] {\tiny 3} (right1);
	\draw[] (left1) to node[midway, below] {\tiny 4} (right1);
	\draw[] (left) to node[midway, below] {\tiny 4} (right1);
	\end{tikzpicture},
	\end{equation}
	where the corresponding integrals over space are indeed proper sums without any scaling factors in front, and the green node denotes the origin. The label $a\in\{3,4\}$ means that the kernel associated to the respective edge decays like $\|z\|_\s^{-a}$ at infinity  To show that this expression is finite we make repeated use of~\cite[Lemma 10.14]{Regularity}. By going carefully through its proof and using that we do not need any estimates on derivatives one can show that this result also carries over to our case.
	Hence, first integrating over the top node above, we end up with a kernel $\tilde p$ such that $\tilde p(z)\lesssim \|z\|_\s^{-1}$, so that in our graphical notation we obtain the term
	\begin{equation}
	\tikzsetnextfilename{boundIIbis}
	\begin{tikzpicture}[scale=0.35,baseline=1.3cm]
	\node at (-2,1)  [root] (left) {};
	\node at (-2,3)  [dot] (left1) {};
	\node at (-2,5)  [dot] (left2) {};
	\node at (0,4)   [dot] (right1) {};

	\draw[kernel] (left1) to node[midway, left] {\tiny 3} (left);
	\draw[kernel] (left2) to node[midway, left] {\tiny 3} (left1);
	\draw[] (left2) to node[midway, above] {\tiny 4} (right1);
	\draw[] (left1) to node[midway, below] {\tiny 4} (right1);
	\draw[] (left) to node[midway, below] {\tiny 4} (right1);
	\end{tikzpicture},
	\end{equation}
	where the topmost label $4$ is really the sum of the labels $3$ and $1$.
	One may now proceed in the same manner by successively integrating over all nodes on the left hand side from the top to the bottom, which finally yields
	a term that is bounded from above by some proportionality constant times
	\begin{equation}
	\int\sum_{x\in\Z^3} 1\wedge (\sqrt{t}+|x|)^{-8}\, dt,
	\end{equation}
	which is indeed finite.
	We now analyse $c_N^{(2,2)}$. We can write 
	\begin{equation}
	\begin{aligned}
	c_N^{(2,2)}\approx \sum\int p_{-t_1}(-x_1) 
	p_{t_1-t_2}(x_1-x_2)p_{t_2-t_3}(x_2-x_3)
	p_{-t_2}(-x_2)p_{t_1-t_3}(x_1-x_3).
	\end{aligned}
	\end{equation}
	Summing over $x_3$, and using the symmetry in space of the discrete heat kernel we see that it is sufficient to estimate
	\begin{equation}\label{eq:startingpoint}
	\sum\int p_{-t_1}(-x_1)p_{t_1-t_2}(x_1-x_2)p_{t_1+t_2-2t_3}(x_1-x_2)p_{-t_2}(-x_2).
	\end{equation}
	We next decompose the domain of integration as $\cA_1\cup\cA_2$, 
	where $\cA_1=\{t\in\R:\, |t_1|\leq 1\}\cup\{t\in\R: |t_1|>1, |t_1-t_2|\leq 1\}$, and
	$\cA_{2}=\{t\in\R:\, |t_1|>1, |t_1-t_2|> 1\}$.
	We first analise the integral over $\cA_1$.
	Observe that
	\begin{equation}\label{eq:Green}
	\int p_{t_1+t_2-2t_3}(x_1-x_2)\, dt_3 \leq G(0),
	\end{equation}
	where $G(0)$ denotes the Green's function of simple random walk at the origin. Thus, summing over $x_2$, we see that it is enough to show the boundedness of
	\begin{equation}
	\sum_{x_1}\int_{\cA_1} p_{-t_1}(-x_1)p_{t_1-2t_2}(x_1)\, dt_1\, dt_2\,.
	\end{equation}
	Summing first over $x_1$ and using that $G(0)$ is finite it is straightforward to show that the above expression is finite.
	We now integrate over $\cA_{2}$. 
	Note that the form of the integral yields that $t_1>t_2>t_3$ so that $|t_1-t_3|> 1$.
	To proceed we additionally assume for the moment that the domain of summation is such that $\|x_1-x_2\|_\infty\leq (t_1-t_2)^{2/3}$. In particular, by Theorem~\ref{thm:heatkernelest} we can make use of the estimate
	\begin{equation}\label{eq:heatkernelest}
	p_{t_1+t_2-2t_3}(x_1-x_2)\lesssim \frac{1}{(t_1+t_2-2t_3)^{3/2}} \exp\Big(-\frac{\|x_1-x_2\|^2}{2(t_1+t_2-2t_3)}\Big).
	\end{equation}
	Now in the same way as in~\cite[Equations (4.28)-(4.29)]{ErhardPoisat} we can show that there are constants $\theta_1,$ and $\theta_2$ such that
	\begin{equation}
	\theta_1(t_1-t_2)^{-1/2}\leq\int_{t_3\leq  t_2-1}p_{t_1+t_2-2t_3}(x_1-x_2)\, dt_3\leq \theta_2(t_1-t_2)^{-1/2}.
	\end{equation}
	Thus, summing the remaing terms in~\eqref{eq:startingpoint} over $\|x_1-x_2\|_{\infty}\leq (t_1-t_2)^{2/3}$, we obtain a term that is some constant times
	\begin{equation}
	\int_{|t_1|> 1, |t_1-t_2|> 1}\frac{1}{(-t_2)^{3/2}}\frac{1}{\sqrt{t_1-t_2}}\, dt_1\, dt_2
	\sim \log (2^{2N}),
	\end{equation}
	where we used that the domain of integration of $t_1$ and $t_2$ is such that both have an absolute value of at most $2^{2N}$.
	It is not hard to see that the contribution from~\eqref{eq:startingpoint} is bounded in $N$ when the domain of summation is restricted to those values of $x_1$ and $x_2$ such that $\|x_1-x_2\|_\infty\geq (t_1-t_2)^{2/3}$. Indeed, this follows from large deviation estimates.
	We can therefore conclude the analysis of $c_N^{(2,2)}$.
	We now turn to $c_N^{(2,3)}$. Similar as above we can rewrite it as
	\begin{equation}
	\sum_{x_1,x_3\in\Z_N^3}\int_{t_1,t_3}p_{-t_1}(-x_1)(\Ren \tilde Q^N\star p)(z_1-z_3)p_{-t_3}(-x_3),
	\end{equation}
	where $\Ren \tilde Q^N(z)= p_t(x)^2-2^{-3N}c_N \delta_0(z)$, where the convolution in space is a sum which is not rescaled.
	It now follows from the analysis in~\cite[Lemma 2.1]{WongZakai} that uniformly in $z$,
	$|(\Ren \tilde Q^N\star p)(z)|\lesssim (|z|^{-3}\wedge|z|^{-5})$, from which the desired boundedness of $c_N^{(2,3)}$ follows.
	We can therefore conclude the proof of Proposition~\ref{prop:renconstant}.
\end{proof}

\endappendix

\bibliographystyle{Martin}

\bibliography{refs}

\end{document}